\documentclass[a4paper,10pt]{article}

\usepackage{geometry}
\usepackage{hyperref}
\usepackage{theoremref}

\usepackage{times}
\usepackage{amsthm}
\usepackage{scalerel}
\usepackage{amsfonts}
\usepackage{amssymb}
\usepackage{mathdots}
\usepackage{color}
\usepackage{mathrsfs}
\usepackage{multirow}

\usepackage{amsmath}
\usepackage{bm}
\usepackage{stmaryrd}
\SetSymbolFont{stmry}{bold}{U}{stmry}{m}{n}

\usepackage{tikz-cd}
\usepackage{tikz}
\usetikzlibrary{matrix,arrows,decorations.pathmorphing}
\usetikzlibrary{positioning}
\usepackage[all]{xy}
\usepackage{graphicx}
\usepackage{subfigure}

\usepackage[toc,page]{appendix}

\usepackage[usestackEOL]{stackengine}
\usepackage{dutchcal}

\newcommand{\diag}{\tmop{diag}}

\newcommand{\e}{\mathrm{e}}

\newtheorem{thm}{Theorem}[section]
\newtheorem{lem}[thm]{Lemma}
\newtheorem{cor}[thm]{Corollary}
\newtheorem{pro}[thm]{Proposition}

\theoremstyle{definition}

\newtheorem{rmk}[thm]{Remark}
\newtheorem{defi}[thm]{Definition}

\numberwithin{equation}{section}

\usepackage[normalem]{ulem}

\newcommand{\assign}{:=}

\newcommand{\tmop}[1]{\ensuremath{\operatorname{#1}}}

\newcommand{\udots}{{\mathinner{
\mskip1mu\raise1pt\vbox{\kern7pt\hbox{.}}
\mskip2mu\raise4pt\hbox{.}
\mskip2mu\raise7pt\hbox{.}\mskip1mu}}}

\newenvironment{prf}
    {\proof}
    {\hspace*{\fill}$\Box$\medskip}
    
\usepackage{mathtools}
\usepackage{empheq}
\usepackage{cite}

\hypersetup{
    colorlinks=true,
    linkcolor=blue,
    filecolor=magenta,    
    urlcolor=cyan,
    citecolor=cyan,}

\title{Monodromy and Isomonodromy for Linear $q$-Difference Systems with Coefficient Matrix $Ax+B$}

\author{Jinghong Lin, Qian Tang and Xiaomeng Xu}
\date{August 2026}

\newcommand{\Addresses}{%
  \begingroup
  \bigskip
  \footnotesize

  \noindent\textsc{
  Academy of Mathematics and Systems Science,
  Chinese Academy of Sciences,
  Beijing 100190, China}\par\nopagebreak
  \textit{E-mail address}:
  \texttt{linjinghong@amss.ac.cn}\par

  \medskip

  \noindent\textsc{
  Department of Mathematical Sciences,
  Tsinghua University,
  Beijing 100084, China}\par\nopagebreak
  \textit{E-mail address}:
  \texttt{919130201@qq.com}\par

  \medskip

  \noindent\textsc{
  School of Mathematical Sciences \&
  Beijing International Center for Mathematical Research,
  Peking University,
  Beijing 100871, China}\par\nopagebreak
  \textit{E-mail address}:
  \texttt{xxu@bicmr.pku.edu.cn}\par

  \endgroup
}

\begin{document}

\maketitle

\abstract{
In this paper,
	we construct canonical fundamental solutions
	for linear $q$-difference systems with coefficient matrix $Ax+B$,
	develop the associated isomonodromy theory,
	and study the differential limit as $q\to1$.
We determine the asymptotic behavior
	of generic isomonodromic deformations,
	and derive explicit formulas for their monodromy data
	in terms of their asymptotic leading terms.
We also show that
	the differential limit as $q\to1$ recovers
	the canonical fundamental solutions,
	the monodromy data, and
	the isomonodromic deformations
	of the corresponding differential systems.
}

\tableofcontents

\section{Introduction}

In this paper,
	we study the monodromy data of
	$n\times n$ linear $q$-difference systems
	with coefficient matrix $Ax+B$.
For comparison with the differential case,
	we write these systems throughout the paper
	in the following form:
\begin{align}\label{Eq:intro}
\frac{\mathrm{d}_q}{\mathrm{d}_q x}Y(x)
	:= \frac{Y(qx) - Y(x)}{qx - x}
	= \left(U + \frac{\Phi}{x} \right) \cdot Y(x),\qquad
U=\diag(u_1,\ldots,u_n),
\end{align}
	where $Y(x)$ is an $n\times n$ matrix-valued function.
We construct the canonical fundamental solutions of
	\eqref{Eq:intro}
	at $x=0$ and $x=\infty$
	(see Section~\ref{Sect:CanSol}).
The connection matrix between these
	canonical fundamental solutions,
	together with the local exponent data,
	defines the monodromy data
	(see Section~\ref{Sect:qMonodromy}).
	
A deformation of \eqref{Eq:intro}
	is called isomonodromic
	if the monodromy data are preserved.
For the deformations considered here,
	the eigenvalues of $U$ are the deformation parameters,
	and $\Phi$ is determined
	as an $n\times n$ matrix-valued function of these parameters.
Equivalently,
	the isomonodromy equations arise as
	the compatibility conditions for
	a system of linear $q$-difference equations.
For \eqref{Eq:intro},
	the deformation equations for $\Phi$ and
	its diagonalizing matrix $W$,
	with respect to the deformation parameters $u_1,\ldots,u_n$,
	take the following form:
\begin{equation}
\label{Intro:qiso}
\left\{
\begin{aligned}
\frac{\partial_q}{\partial_q u_k}\Phi
&=
	\left(
	\frac{
		V_{n;k}^{(n)}(u;\Phi,\alpha_k)-I
	}{
		(q-1)u_k
	}
	\cdot\Phi
	-
	\Phi\cdot
	\frac{
		V_{n;k}^{(n)}(u;\Phi,\alpha_k)-I
	}{
		(q-1)u_k
	}
	\right)
	\cdot V_{n;k}^{(n)}(u;\Phi,\alpha_k)^{-1},
	\\
\frac{\partial_q}{\partial_q u_k}W
&=
	\frac{
		V_{n;k}^{(n)}(u;\Phi,\alpha_k)-I
	}{
		(q-1)u_k
	}
	\cdot W,
\end{aligned}
\right.
\qquad
k=1,\ldots,n.
\end{equation}
(See Section~\ref{Sect:qisoeq}
	for the compatibility and further properties).
Here the exponent data
	$\boldsymbol{\alpha} = \diag
	(\alpha_1,\ldots,\alpha_n)$
	form part of the monodromy data and satisfy
\begin{align*}
	\mathrm{Spec}\bigl(
	U^{-1}(I+(q-1)\Phi)
	\bigr)
	&=
	\left\{
	u_1^{-1}q^{\alpha_1},
	\ldots,
	u_n^{-1}q^{\alpha_n}
	\right\}.
\end{align*}
The matrix $V_{n;k}^{(n)}(u;\Phi,\alpha_k)$
	is defined by
\begin{align*}
V_{n;k}^{(n)}(u;\Phi,\alpha_k)_{\hat{k}\hat{k}}
	&=
	I,
	\\
V_{n;k}^{(n)}(u;\Phi,\alpha_k)_{k\hat{k}}
	&=
	q^{-\alpha_k}\Phi_{k\hat{k}}
	\frac{(q-1)u_k}{
		u_kI-q^{-1}U_{\hat{k}\hat{k}}
	},
	\\
V_{n;k}^{(n)}(u;\Phi,\alpha_k)_{\hat{k}k}
	&=
	\frac{(q-1)q^{-\alpha_k}u_k}{
		q^{-\alpha_k}u_k
		\bigl(I+(q-1)\Phi_{\hat{k}\hat{k}}\bigr)
		-U_{\hat{k}\hat{k}}
	}
	\Phi_{\hat{k}k},
	\\
V_{n;k}^{(n)}(u;\Phi,\alpha_k)_{kk}
	&=
	1+
	\bigl(V_{n;k}^{(n)}\bigr)_{k\hat{k}}
	\bigl(V_{n;k}^{(n)}\bigr)_{\hat{k}k}.
\end{align*}
For an $n\times n$ matrix $M$,
	the $k$-th column with its $k$-th entry removed
	is denoted by $M_{\hat{k}k}$.
The $k$-th row with its $k$-th entry removed
	is denoted by $M_{k\hat{k}}$.
The submatrix obtained by deleting
	the $k$-th row and the $k$-th column
	is denoted by $M_{\hat{k}\hat{k}}$.
Our main purpose is to establish
	the isomonodromy theory for these linear $q$-difference systems
	and use this theory to solve
	the monodromy problem.

We study the asymptotic behavior of
	solutions of the isomonodromy equations \eqref{Intro:qiso}
	and use their asymptotic leading terms
	to reconstruct generic isomonodromic deformations.
Moreover, these asymptotic leading terms allow us to reconstruct
	the monodromy data of
	the original linear $q$-difference system \eqref{Eq:intro}
	with generic coefficients.
In the differential limit as $q\to1$,
	the canonical fundamental solutions and monodromy data
	of the linear $q$-difference system \eqref{Eq:intro}
	as well as the generic solutions of
	the isomonodromy equations \eqref{Intro:qiso}
	converge to the corresponding objects
	in the differential setting.
	
\subsection{Background and Motivation}
	
The theory of monodromy-preserving deformations
	goes back to Riemann's study of
	linear differential equations with prescribed monodromy,
	and was developed by Schlesinger, Fuchs, and Garnier
	\cite{iwasaki_moduli_1991}.
Sato, Miwa, and Jimbo studied
	solutions of the Schlesinger equations
	with prescribed asymptotic leading terms
	under an eigenvalue condition
	\cite{sato_holonomic_1979}.
Jimbo, Miwa, and Ueno developed
	a general theory of monodromy-preserving deformations
	for linear differential systems with rational coefficients
	and singularities of arbitrary Poincar\'e rank
	\cite{jimbo_monodromy_1981}.

The theory of $q$-difference equations
	is a cousin theory of differential equations
	and has also attracted considerable attention.
As early as 1913,
	Birkhoff tried to treat differential,
	$h$-difference and $q$-difference equations
	in a unified way by formulating
	a generalized Riemann problem.
Since monodromy matrices cannot be defined
	for $q$-difference equations
	in the same direct way as in the differential case,
	Birkhoff introduced the connection matrix
	as an analog of monodromy data
	\cite{birkhoff1913generalized}.
Ramis, Zhang, and others developed
	the theories of $q$-Gevrey asymptotics
	and $q$-Borel--Laplace summation
	for linear $q$-difference equations
	\cite{zhang_developpements_1999,zhang2000theta,
	marotte2000multi,ramis2002gevrey}.
These methods are among the main analytic tools
	used in this paper.
Ramis, Sauloy, and Zhang later established
	the local analytic classification
	of irregular linear $q$-difference equations
	and essentially completed Birkhoff's program
	for $q$-difference equations
	\cite{ramis_local_2013}.
	
The $q$-isomonodromy equations studied here
	are nonlinear $q$-difference equations
	arising as compatibility conditions
	for linear $q$-difference systems
	and thus have a natural integrable structure.
Jimbo and Sakai derived $q$-Painlev\'e VI
	from connection-preserving deformations
	of a rank-two linear $q$-difference system
	and recovered Painlev\'e VI in the differential limit as $q\to1$
	\cite{jimbo_q-analog_1996}.
Kakei and Kikuchi subsequently obtained
	a $3\times3$ formulation of $q$-Painlev\'e VI
	from the multi-component $q$-KP hierarchy
	\cite{kakei_q-analogue_2006}.
Mano later studied the asymptotic behavior
	of a class of solutions near boundary points
	and expressed the associated connection matrix
	as a product of two explicitly computable
	connection matrices for simpler
	linear $q$-difference systems
	\cite{mano_asymptotic_2010}.

Sauloy established the $q\to1$ confluence
	of canonical fundamental solutions and connection matrices
	for regular singular systems
	\cite{sauloy2000regular}.
For irregular systems,
	the confluence problem is more delicate
	because the corresponding differential formal solutions
	are generally divergent.
The confluence problem for basic hypergeometric functions
	has been studied by Zhang, Morita, and Adachi
	\cite{zhang_sommation_2002,
	morita_connection_2013,
	adachi_q-borel_2019}.
Related confluence results
	for broader classes of $q$-Gevrey formal solutions
	were obtained by Di Vizio--Zhang and Dreyfus
	\cite{di_vizio_q_2009,dreyfus_confluence_2015}.

The $q$-isomonodromy equations studied in this paper
	are precisely the general $n\times n$ version
	of the equations considered in
	\cite{kakei_q-analogue_2006}.
Thus, when $n=3$,
	our results apply to $q$-Painlev\'e VI.
Under the shrinking condition,
	our monodromy formula for generic isomonodromic deformations
	recovers Mano's result
	\cite{mano_asymptotic_2010}.
As an application of our explicit monodromy formula,
	we establish the differential limit of
	generic isomonodromic deformations as $q\to1$.
The formula also gives the differential limit of the associated canonical fundamental solutions
	while avoiding a direct analysis
	in the irregular setting.
For the main technique used here,
	we refer the reader to
	\cite{xu_closure_2021,tang_boundary_nodate}.

Beyond our interest in the analytic theory
	of linear $q$-difference systems,
	a primary motivation for this work is to understand
	the Poisson geometry of
	the $q$-Riemann--Hilbert--Birkhoff
	($q$-RHB) correspondence.
We aim to establish
	a $q$-difference analog of Boalch's results
	on differential Riemann--Hilbert--Birkhoff
	(RHB) maps
	\cite{boalch_stokes_2001,boalch_symplectic_2001}.
In the differential limit as $q\to1$,
	the systems considered here converge to
	meromorphic differential systems
	with a pole of order two at infinity.
In the corresponding differential setting,
	Boalch proved that for each fixed $U$
	the monodromy map
	from $\mathfrak{gl}_n(\mathbb C)^*$
	to the Poisson--Lie group $G^*$ dual to
	$G=\mathrm{GL}_n(\mathbb C)$ is Poisson.
Here $\mathfrak{gl}_n(\mathbb C)^*$
	has its standard complex Poisson structure.
The group $G^*$ is identified with
	the space of monodromy/Stokes data
	and carries its canonical complex
	Poisson--Lie group structure
	scaled by a factor of $2\pi \mathrm{i}$.
The Poisson structure on $G^*$
	is determined by the standard classical $r$-matrix
	and is the semiclassical limit
	of the corresponding quantum group structure.

In our previous work
	\cite{tang_boundary_nodate,xu_closure_2021},
	we obtained explicit formulas
	for the differential monodromy data
	in terms of the asymptotic leading terms
	of the corresponding isomonodromic deformations.
These formulas give an explicit description
	of the differential RHB map
	and allow us to study its Poisson property
	through concrete identities for special functions.
The $q$-monodromy data studied here
	are $q$-analogs of
	the differential monodromy data.
This suggests asking whether
	the $q$-monodromy space carries
	a natural Poisson structure
	associated with an elliptic $r$-matrix.
Theorem~\ref{Thm}
	gives an explicit formula for the $q$-RHB map.
Theorem~\ref{Thm:qto1Confluence} shows that
	this formula recovers
	the corresponding differential RHB formula
	in the differential limit as $q\to1$.
The next step is to identify
	the relevant Poisson structures
	on the coefficient and $q$-monodromy spaces
	and to prove that this explicit map is Poisson.
We leave this problem to future work.

\subsection{Main Results}

In this subsection,
	we explain our main ideas and
	state our main results.

In Section~\ref{Sect:Basic},
	we introduce the canonical fundamental solutions
		$Y^{[\infty]}(x;U,\Phi;\boldsymbol{\beta})$
		in a form different from
		the classical Birkhoff--Guenther form
		\cite{birkhoff1941note}.
The monodromy data
	$(C(x),\Lambda;\boldsymbol{\alpha})$
	considered in this paper
	are defined by these solutions.
We will show that
	the canonical fundamental solutions and
	monodromy data
	defined in this way
	converge to the corresponding objects
	in the differential setting as $q\to1$.
We also determine the possible locations of
	the poles of the canonical fundamental solutions
	(see Proposition~\ref{Pro:SolY})
	and establish a uniqueness theorem
	for reconstruction from the monodromy data
	(see Proposition~\ref{Pro:qUni}).
These two results are key ingredients
	in the proofs of our main results.

In Section~\ref{Sect:qisoeq},
	we introduce the formal exponents
	$\boldsymbol{\alpha}$ as a necessary part of
	the monodromy data.
They are analogous to the formal monodromy
	of the corresponding differential system,
	but arise naturally only when we consider
	isomonodromic deformations.
By comparing the pole sets of
	the canonical fundamental solutions,
	we directly derive
	the isomonodromy equations \eqref{Intro:qiso}
	considered in this paper.
Theorems~\ref{Thm:qIsoMD} and~\ref{Thm:qiso}
	give the explicit form
	of the isomonodromy equations
	and their global solutions.
	
In Section~\ref{Sect:Asym},
	we use Picard iteration to construct
	isomonodromic deformations
	with a prescribed asymptotic leading term $(\Phi_0,W_0)$
	that satisfies the shrinking condition
	(see Theorem~\ref{Thm:AsymBehavior}).
We call these solutions
	shrinking solutions.

In Section~\ref{Sect:Dec},
	we analyze the pole sets of
	the canonical fundamental solutions
	and factor them into canonical fundamental solutions
	of simpler systems
	(see Equation
	\eqref{Eq:FirstFactorization}--
	\eqref{Eq:ThirdFactorization}).
This gives a factorization of
	the monodromy data of the shrinking solutions
	(see Theorem~\ref{Thm:qCaterpillar}).
We then derive explicit formulas for
	the monodromy data of the shrinking solutions.
	
In Section~\ref{Sect:Applications},
	we introduce the Gelfand--Zeitlin eigenvalue--minor coordinates
	associated with a prescribed asymptotic leading term
	$(\Phi_0,W_0)$, namely
\begin{subequations}\label{Eq:EMcorr}
\begin{align}
\operatorname{Spec}\bigl([\Phi_0^{[k]}]^q\bigr)
&=
\{\lambda_1^{(k)},\ldots,\lambda_k^{(k)}\},
&&
1\leqslant k\leqslant n,
\\
\eta_j^{(m)}
&:=
\frac{
	\det\bigl(
	\Phi_0-[\lambda_j^{(m)}]_qI
	\bigr)^{1,\ldots,m}_{1,\ldots,m-1,m+1}
}{
	\displaystyle
	\prod_{\substack{1\leqslant s\leqslant m\\s\neq j}}
	\bigl(
	[\lambda_s^{(m)}]_q-[\lambda_j^{(m)}]_q
	\bigr)
},
&&
1\leqslant j\leqslant m\leqslant n-1,
\\
\eta_j^{(n)}
&:=
(W_0)_{nj},
&&
1\leqslant j\leqslant n.
\end{align}
\end{subequations}
Here $M^{[k]}$ denotes
	the upper-left $k\times k$ submatrix of $M$.
We then start from generic monodromy data
	and invert the explicit monodromy formulas
	in these coordinates to reconstruct
	a shrinking solution with the same monodromy data
	(see Lemma~\ref{Lem:InverseMonodromy}).
The uniqueness theorem for reconstruction
	from the monodromy data
	then shows that the shrinking solutions
	form a generic family.
Theorems~\ref{Thm:AsymBehavior},
	\ref{Thm:qCaterpillar}, and
	\ref{Thm:GenericShrinking}
	can be summarized in the following theorem.
	
\begin{thm}\label{Thm}
Let $(\Phi(u),W(u);\boldsymbol{\alpha})$
	be a generic solution of
	the $q$-isomonodromy equations \eqref{Intro:qiso}.
Then there exist constant matrices
	$(\Phi_0,W_0)$
such that, for each $1\leqslant k\leqslant n-1$,
	the shrinking condition
\begin{align}
	\max_{1\leqslant i,j\leqslant k}
	\left|
		\operatorname{Re}
		\bigl(\lambda_i^{(k)}-\lambda_j^{(k)}\bigr)
		-
		\frac{\arg q}{\ln|q|}
		\operatorname{Im}
		\bigl(\lambda_i^{(k)}-\lambda_j^{(k)}\bigr)
	\right|
	<
	1-\varepsilon_k
\end{align}
holds for some $0<\varepsilon_k<1$,
	and the following limits hold:
\begin{subequations}
\begin{align}
\lim_{u_{k+1}\to\infty}\cdots\lim_{u_n\to\infty}
	[\delta_k\Phi(u)]^q
	&=
	[\delta_k\Phi_k(u)]^q,
	\\
\lim_{u_2\to\infty}\cdots\lim_{u_n\to\infty}
	u_1^{\operatorname{ad}\boldsymbol{\alpha}}
	\left(
	\mathop{\prod}\limits_{s=1}^{\overset{\longrightarrow}{n-1}}
	\left(
		\frac{u_{s+1}}{u_s}
	\right)^{
		\operatorname{ad}[\delta_s\Phi_s(u)]^q
	}
	\right)
	\Phi(u)
	&=
	\Phi_0,
	\\
\lim_{u_2\to\infty}\cdots\lim_{u_n\to\infty}
	u_1^{\boldsymbol{\alpha}}
	\left(
	\mathop{\prod}\limits_{s=1}^{\overset{\longrightarrow}{n-1}}
	\left(
		\frac{u_{s+1}}{u_s}
	\right)^{
		[\delta_s\Phi_s(u)]^q
	}
	\right)
	u_n^{-[\Phi(u)]^q}
	W(u)
	&=
	W_0.
\end{align}
\end{subequations}
The formal exponents
	$\boldsymbol{\alpha}
	=\operatorname{diag}(\alpha_1,\ldots,\alpha_n)$
satisfy
\begin{align}
q^{\alpha_k}
&=
\frac{
	\det\bigl(I+(q-1)\Phi_0^{[k]}\bigr)
}{
	\det\bigl(I+(q-1)\Phi_0^{[k-1]}\bigr)
},
\qquad
k=1,\ldots,n.
\end{align}
For $|q|<1$,
	the central connection matrix $C(x)$ satisfies
\begin{align}
\bigl(C(x)^{-1}\bigr)_{rj}
&=
\frac{1}{\eta_r^{(n)}}
\frac{
\theta_q\bigl(
-q^{-\lambda_r^{(n)}}(1-q)xu_j
\bigr)
}{
\theta_q\bigl(
-q^{-\alpha_j}(1-q)xu_j
\bigr)
}
\bigl((1-q)xu_j\bigr)^{
\alpha_j-\lambda_r^{(n)}
}
\sum_{\boldsymbol{i}\in\mathcal P_{rj}}
\frac{
\displaystyle
\prod_{\substack{1\leqslant\ell\leqslant j\\
\ell\neq i_j}}
\Gamma_q\bigl(
\lambda_\ell^{(j)}-\lambda_{i_j}^{(j)}
\bigr)
}{
\displaystyle
\prod_{1\leqslant\ell\leqslant j-1}
\Gamma_q\bigl(
\lambda_\ell^{(j-1)}-\lambda_{i_j}^{(j)}
\bigr)
}
\notag\\
&\qquad\times
\prod_{k=j+1}^{n}
	\Bigg(
	\frac{
	[\lambda_{i_k}^{(k)}]_q
	-
	[\lambda_{i_{k-1}}^{(k-1)}]_q
	}{
	\eta_{i_{k-1}}^{(k-1)}
	}
\prod_{\substack{1\leqslant\ell\leqslant k\\
\ell\neq i_k}}
\frac{
\Gamma_q\bigl(
\lambda_\ell^{(k)}-\lambda_{i_k}^{(k)}
\bigr)
}{
\Gamma_q\bigl(
\lambda_\ell^{(k)}
-\lambda_{i_{k-1}}^{(k-1)}
\bigr)
}
\prod_{\substack{1\leqslant\ell\leqslant k-1\\
\ell\neq i_{k-1}}}
\frac{
\Gamma_q\bigl(
\lambda_\ell^{(k-1)}
-\lambda_{i_{k-1}}^{(k-1)}
\bigr)
}{
\Gamma_q\bigl(
\lambda_\ell^{(k-1)}-\lambda_{i_k}^{(k)}
\bigr)
}
\notag\\
&\qquad\qquad\times
\frac{
\theta_q\left(
-q^{
\lambda_{i_{k-1}}^{(k-1)}
-\lambda_{i_k}^{(k)}
}
\frac{u_k}{u_j}
\right)
}{
\theta_q\left(-\frac{u_k}{u_j}\right)
}
\left(\frac{u_k}{u_j}\right)^{
\lambda_{i_{k-1}}^{(k-1)}
-\lambda_{i_k}^{(k)}
}
\Bigg).
\label{Eq:ExplicitQMonodromyEntry}
\end{align}
For $|q|>1$,
	the central connection matrix $C(x)$ satisfies
\begin{align}
C(x)_{jr}
&=
\eta_r^{(n)}
\frac{
\theta_{q^{-1}}\bigl(
q^{-\lambda_r^{(n)}}(q-1)xu_j
\bigr)
}{
\theta_{q^{-1}}\bigl(
q^{-\alpha_j}(q-1)xu_j
\bigr)
}
\bigl((q-1)xu_j\bigr)^{
\lambda_r^{(n)}-\alpha_j
}
\sum_{\boldsymbol{i}\in\mathcal P_{rj}}
\frac{
\displaystyle
\prod_{\substack{1\leqslant\ell\leqslant j\\
\ell\neq i_j}}
\Gamma_{q^{-1}}\bigl(
1+\lambda_{i_j}^{(j)}-\lambda_\ell^{(j)}
\bigr)
}{
\displaystyle
\prod_{1\leqslant\ell\leqslant j-1}
\Gamma_{q^{-1}}\bigl(
1+\lambda_{i_j}^{(j)}-\lambda_\ell^{(j-1)}
\bigr)
}
\notag\\
&\quad\times
\prod_{k=j+1}^{n}
\Bigg(
\frac{
\eta_{i_{k-1}}^{(k-1)}
}{
[\lambda_{i_k}^{(k)}]_q
-
[\lambda_{i_{k-1}}^{(k-1)}]_q
}
\prod_{\substack{1\leqslant\ell\leqslant k\\
\ell\neq i_k}}
\frac{
\Gamma_{q^{-1}}\bigl(
1+\lambda_{i_k}^{(k)}-\lambda_\ell^{(k)}
\bigr)
}{
\Gamma_{q^{-1}}\bigl(
1+\lambda_{i_{k-1}}^{(k-1)}
-\lambda_\ell^{(k)}
\bigr)
}
\prod_{\substack{1\leqslant\ell\leqslant k-1\\
\ell\neq i_{k-1}}}
\frac{
\Gamma_{q^{-1}}\bigl(
1+\lambda_{i_{k-1}}^{(k-1)}
-\lambda_\ell^{(k-1)}
\bigr)
}{
\Gamma_{q^{-1}}\bigl(
1+\lambda_{i_k}^{(k)}
-\lambda_\ell^{(k-1)}
\bigr)
}
\notag\\
&\qquad\qquad\times
\frac{
\theta_{q^{-1}}\left(
-q^{
\lambda_{i_k}^{(k)}
-\lambda_{i_{k-1}}^{(k-1)}
}
\frac{u_j}{u_k}
\right)
}{
\theta_{q^{-1}}\left(
-\frac{u_j}{u_k}
\right)
}
\left(\frac{u_j}{u_k}\right)^{
\lambda_{i_{k-1}}^{(k-1)}
-\lambda_{i_k}^{(k)}
}
\Bigg).
\label{Eq:ExplicitQMonodromyEntryGreater}
\end{align}
Here $\mathcal P_{rj}$ denotes the set of sequences
$\boldsymbol{i}
=
(i_j,i_{j+1},\ldots,i_n=r)$
such that
$1\leqslant i_k\leqslant k$.
\end{thm}
\noindent
For the notation and definitions used here, see
	Sections~\ref{Sect:qNotation} and~\ref{Sect:Asym}.

Finally, Theorem~\ref{Thm:qto1Confluence} shows,
	by considering the differential limit of
	the asymptotic leading terms of
	generic isomonodromic deformations as $q\to1$,
	that the canonical fundamental solutions
	and the isomonodromic deformations
	also converge to the corresponding objects
	in the differential setting.
	
\subsection{Further Directions}

For $n=3$, the equations studied here include
	the $3\times3$ formulation of $q$-Painlev\'e VI
	\cite{kakei_q-analogue_2006}.
Thus our construction extends
	the associated isomonodromy and monodromy problems
	to higher rank.
Theorem~\ref{Thm:qto1Confluence} shows that
	the isomonodromic deformations
	and canonical fundamental solutions studied here
	converge as $q\to1$
	to the isomonodromic deformations
	and canonical fundamental solutions
	of the corresponding differential systems.
We consider two directions for further study.

The first direction concerns $q$-Painlev\'e VI.
Its geometric Riemann--Hilbert correspondence
	and monodromy spaces have been studied in
	\cite{ohyama_space_2020,
	joshi_monodromy_2023,
	roffelsen_q-painleve_2024}.
In the differential limit as $q\to1$,
	the $q$-Painlev\'e VI monodromy space
	was identified with
	the Painlev\'e VI monodromy manifold
	in \cite{joshi_segre_2026}.
Solutions of $q$-Painlev\'e VI
	that are meromorphic at the origin
	were studied in
	\cite{ohyama_analytic_2009,ohyama_meromorphic_2015}.
Special-function solutions of
	$q$-Painlev\'e equations were constructed in
	\cite{sakai_casorati_1998,
	kajiwara_hypergeometric_2004,
	tsuda_masuda_2006,
	joshi_shi_2011,joshi_shi_2012}.
Connections with gap probabilities
	and orthogonal polynomials
	have also been studied in
	\cite{ormerod_connection_2011,
	knizel_moduli_2016,
	hu_gap_2020}.
CFT and combinatorial formulas
	for the $\tau$-functions
	of $q$-Painlev\'e equations
	were obtained in
	\cite{jimbo_cft_2017,
	bershtein_qdeformed_2017,
	matsuhira_combinatorial_2019}.
It would be interesting to determine
	which of these results and constructions
	admit higher-rank analogs
	for the equations considered here.
One may also ask whether
	the explicit monodromy formula
	in Theorem~\ref{Thm}
	can be extended to suitable nongeneric cases
	and used to derive
	higher-rank special solutions,
	their connection formulas, and
	$\tau$-function expansions.
	
The second direction concerns representation theory.
In the differential setting,
	the Stokes matrices of certain systems
	with a pole of order two
	are related to Yangians and quantum groups
	\cite{tang2025stokes,xu_representations_2020}.
Related constructions for quantum supergroups
	have recently been obtained in
	\cite{li_stokes_2026}.
Explicit differential monodromy formulas
	of the type recovered here in the differential limit as $q\to1$
	are also related to
	the Gelfand--Zeitlin integrable system
	and crystal bases
	\cite{xu_closure_2021}.
The spectral coordinates used
	in the WKB analysis of Stokes matrices
	are related to cluster structures
	\cite{alekseev_wkb_2024}.
It is natural to ask whether
	the explicit $q$-monodromy formulas
	of Theorem~\ref{Thm}
	admit similar representation-theoretic interpretations.
	
A broader question is whether
	the method developed here
	extends beyond degree-one linear $q$-difference systems.
This requires extending both
	the factorization of canonical fundamental solutions
	into those of simpler systems
	and the reconstruction of generic isomonodromic deformations
	from their monodromy data.
The $h$-difference systems studied in
	\cite{borodin_isomonodromy_2004,
	arinkin_moduli_2006,arinkin_tau_2009}
	and the $q$-difference equations
	arising in quantum $K$-theory
	\cite{ruan_quantum_2022,
	ruan_level_2026,ruan_toric_2022}
	provide two natural settings
	for investigating extensions of our method.
For systems with polynomial coefficient matrices
	of higher degree,
	the first problem is whether
	their canonical fundamental solutions
	can still be factorized
	into those of simpler systems.
If such a factorization exists,
	one may then ask whether
	the monodromy data of these systems
	can be used to reconstruct
	generic isomonodromic deformations.
This would extend
	the inverse monodromy construction developed here
	to a more general class of systems.
	
\newpage

The logical relations among the main results
	are summarized in the following diagram.
	
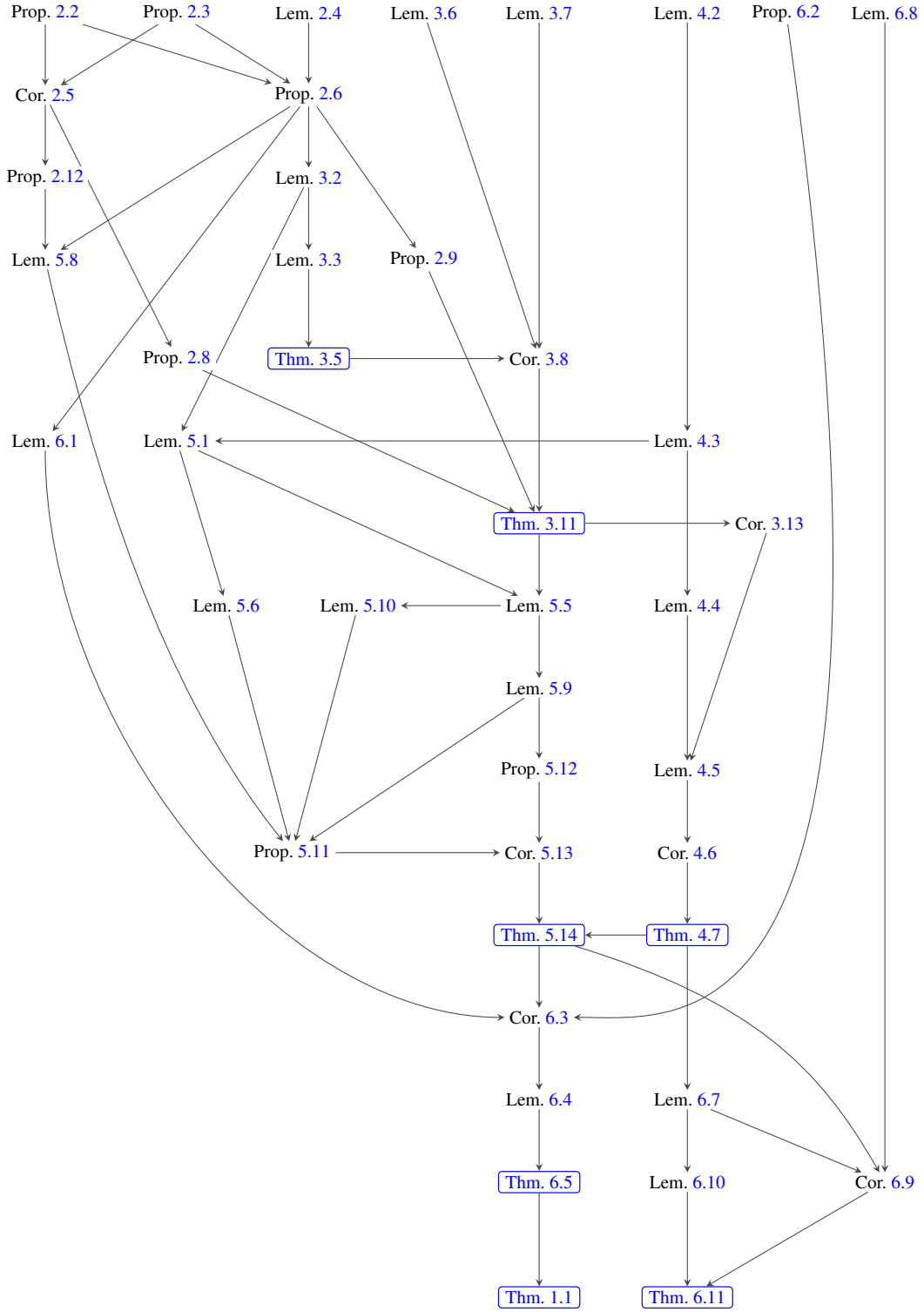
\begingroup
\pgfdeclarelayer{diagrambackground}
\pgfsetlayers{diagrambackground,main}
\begin{center}
\begin{tikzpicture}[
	result/.style={
		font=\footnotesize,
		inner xsep=2pt,
		inner ysep=1.2pt,
		fill=white
	},
	theorem/.style={
		result,
		draw=blue,
		line width=0.45pt,
		rounded corners=1.5pt,
		inner xsep=3pt,
		inner ysep=1.8pt
	},
	dependency/.style={
		->,
		>=stealth,
		thin,
		draw=black!72,
		rounded corners=2pt
	}
]
\node[result] (OneStep) at (11,19.5) {Prop.~\ref{Pro:OneStepConnectionEntries}};
\node[result] (CanZero) at (-0.25,19.5) {Prop.~\ref{Can:zero}};
\node[result] (CanInf) at (1.75,19.5) {Prop.~\ref{Can:inf}};
\node[result] (BG) at (3.75,19.5) {Lem.~\ref{Lem:BG}};
\node[result] (EquiLAG) at (7.25,19.5) {Lem.~\ref{Lem:EquiLAG}};
\node[result] (UniLAG) at (5.5,19.5) {Lem.~\ref{Lem:UniLAG}};
\node[result] (Picard) at (9.5,19.5) {Lem.~\ref{Lem:Picard}};
\node[result] (QToOne) at (12.5,19.5) {Lem.~\ref{Lem:qto1}};
\node[result] (Basic) at (-0.25,18.25) {Cor.~\ref{Cor:Basic}};
\node[result] (SolY) at (3.75,18.25) {Prop.~\ref{Pro:SolY}};
\node[result] (RecStep) at (9.5,13) {Lem.~\ref{Lem:RecStep1}};
\node[result] (QBorel) at (-0.25,17) {Prop.~\ref{Pro:qBorelSummability}};
\node[result] (CTransform) at (1.75,14.25) {Prop.~\ref{Pro:CTransformation}};
\node[result] (BetaEq) at (3.75,17) {Lem.~\ref{Lem:betaEq}};
\node[result] (QUni) at (5.5,15.75) {Prop.~\ref{Pro:qUni}};
\node[result] (RecImportant) at (9.5,10.5) {Lem.~\ref{Lem:RecImprotant}};
\node[result] (CTheta) at (-0.25,13) {Lem.~\ref{Lem:CThetaStructure}};
\node[result] (QBasic) at (3.75,15.75) {Lem.~\ref{Lem:qBasic}};
\node[result] (SummedInf) at (-0.25,15.75) {Lem.~\ref{Lem:SummedZetaInf}};
\node[theorem] (QIsoMD) at (3.75,14.25) {{{\color{blue}Thm.}~\ref{Thm:qIsoMD}}};
\node[result] (TargetSys) at (1.75,13) {Lem.~\ref{Lem:TargetSysNew}};
\node[result] (UniVk) at (7.25,14.25) {Cor.~\ref{Cor:UniVk}};
\node[theorem] (QIso) at (7.25,11.75) {{{\color{blue}Thm.}~\ref{Thm:qiso}}};
\node[result] (CanZetaInf) at (2.5,10.5) {Lem.~\ref{Lem:CanZetaInf}};
\node[result] (CanXnZero) at (7.25,10.5) {Lem.~\ref{Lem:CanXn0}};
\node[result] (Comp) at (10.75,11.75) {Cor.~\ref{CompCond}};
\node[result] (SecondCommon) at (7.25,9.25) {Lem.~\ref{Lem:SecondCommonSolution}};
\node[result] (ZeroStability) at (4.5,10.5) {Lem.~\ref{Lem:ZeroCanonicalStability}};
\node[result] (Phin) at (9.5,8) {Lem.~\ref{Lem:Phin}};
\node[result] (ThirdFactor) at (7.25,8) {Prop.~\ref{Lem:ThirdFactorization}};
\node[result] (SecondFactor) at (3.5,6.75) {Prop.~\ref{Lem:SecondFactorization}};
\node[result] (Reck) at (9.5,6.75) {Cor.~\ref{Cor:Reck}};
\node[result] (CentralFirst) at (7.25,6.75) {Cor.~\ref{Cor:CentralConnectionFirstStep}};
\node[theorem] (Asym) at (9.5,5.5) {{{\color{blue}Thm.}~\ref{Thm:AsymBehavior}}};
\node[theorem] (QCaterpillar) at (7.25,5.5) {{{\color{blue}Thm.}~\ref{Thm:qCaterpillar}}};
\node[result] (QIsoConfluence) at (9.5,3) {Lem.~\ref{Lem:qIsoConfluence}};
\node[result] (ConnectionCoordinates) at (7.25,4.25) {Cor.~\ref{Cor:ConnectionFromEigenminorCoordinates}};
\node[result] (InverseConfluence) at (9.5,1.75) {Lem.~\ref{Lem:InverseMonodromyConfluence}};
\node[result] (CentralConfluence) at (12.5,1.75) {Cor.~\ref{Cor:CentralConnectionConfluence}};
\node[result] (InverseMonodromy) at (7.25,3) {Lem.~\ref{Lem:InverseMonodromy}};
\node[theorem] (GenericShrinking) at (7.25,1.75) {{{\color{blue}Thm.}~\ref{Thm:GenericShrinking}}};
\node[theorem] (MainTheorem) at (7.25,0) {{{\color{blue}Thm.}~\ref{Thm}}};
\node[theorem] (ConfluenceTheorem) at (9.5,0) {{{\color{blue}Thm.}~\ref{Thm:qto1Confluence}}};

\begin{pgfonlayer}{diagrambackground}
\draw[dependency] (Asym) -- (QCaterpillar);
\draw[dependency] (Asym) -- (QIsoConfluence);
\draw[dependency] (BG) -- (SolY);
\draw[dependency] (Basic) -- (CTransform);
\draw[dependency] (Basic) -- (QBorel);
\draw[dependency] (BetaEq) -- (QBasic);
\draw[dependency] (BetaEq) -- (TargetSys);
\draw[dependency] (CTheta) .. controls (-0.25,8.75) and (3.25,4.25) .. (ConnectionCoordinates);
\draw[dependency] (CTransform) -- (QIso);
\draw[dependency] (CanInf) -- (Basic);
\draw[dependency] (CanInf) -- (SolY);
\draw[dependency] (CanXnZero) -- (SecondCommon);
\draw[dependency] (CanXnZero) -- (ZeroStability);
\draw[dependency] (CanZero) -- (Basic);
\draw[dependency] (CanZero) -- (SolY);
\draw[dependency] (CanZetaInf) -- (SecondFactor);
\draw[dependency] (CentralConfluence) -- (ConfluenceTheorem);
\draw[dependency] (CentralFirst) -- (QCaterpillar);
\draw[dependency] (Comp) -- (Phin);
\draw[dependency] (ConnectionCoordinates) -- (InverseMonodromy);
\draw[dependency] (EquiLAG) -- (UniVk);
\draw[dependency] (GenericShrinking) -- (MainTheorem);
\draw[dependency] (InverseConfluence) -- (ConfluenceTheorem);
\draw[dependency] (InverseMonodromy) -- (GenericShrinking);
\draw[dependency] (OneStep) .. controls (13.25,3.5) and (9.5,4.25) .. (ConnectionCoordinates);
\draw[dependency] (Phin) -- (Reck);
\draw[dependency] (Picard) -- (RecStep);
\draw[dependency] (QBasic) -- (QIsoMD);
\draw[dependency] (QBorel) -- (SummedInf);
\draw[dependency] (QCaterpillar) .. controls (10.5,4.5) and (11.5,3.5) .. (CentralConfluence);
\draw[dependency] (QCaterpillar) -- (ConnectionCoordinates);
\draw[dependency] (QIso) -- (CanXnZero);
\draw[dependency] (QIso) -- (Comp);
\draw[dependency] (QIsoConfluence) -- (CentralConfluence);
\draw[dependency] (QIsoConfluence) -- (InverseConfluence);
\draw[dependency] (QIsoMD) -- (UniVk);
\draw[dependency] (QToOne) -- (CentralConfluence);
\draw[dependency] (QUni) -- (QIso);
\draw[dependency] (RecImportant) -- (Phin);
\draw[dependency] (RecStep) -- (RecImportant);
\draw[dependency] (RecStep) -- (TargetSys);
\draw[dependency] (Reck) -- (Asym);
\draw[dependency] (SecondCommon) -- (SecondFactor);
\draw[dependency] (SecondCommon) -- (ThirdFactor);
\draw[dependency] (SecondFactor) -- (CentralFirst);
\draw[dependency] (SolY) -- (BetaEq);
\draw[dependency] (SolY) -- (CTheta);
\draw[dependency] (SolY) -- (QUni);
\draw[dependency] (SolY) -- (SummedInf);
\draw[dependency] (SummedInf) .. controls (0.5,12.5) and (1.5,9.25) .. (SecondFactor);
\draw[dependency] (TargetSys) -- (CanXnZero);
\draw[dependency] (TargetSys) -- (CanZetaInf);
\draw[dependency] (ThirdFactor) -- (CentralFirst);
\draw[dependency] (UniLAG) -- (UniVk);
\draw[dependency] (UniVk) -- (QIso);
\draw[dependency] (ZeroStability) -- (SecondFactor);
\end{pgfonlayer}
\end{tikzpicture}
\end{center}
\endgroup

\section{Linear $q$-Difference Systems}
\label{Sect:Basic}

Throughout this paper, we assume that $q$ is nonzero and
$|q|\neq 1$.
In this section, we introduce
the canonical fundamental solutions and
the $q$-monodromy data of
the following \textbf{linear $q$-difference system}:
\begin{align}\label{q1}
	T_x Y(x) := Y(qx) = (A_1x + A_0) \cdot Y(x).
\end{align}
We then develop the basic theory needed in this paper.
General background on linear $q$-difference systems can be found in
\cite{ramis_local_2013, sauloy_basic_2024}.
To compare system \eqref{q1} with its
differential limit as $q\to1$,
we consistently write it in the following form:
\begin{align}\label{qLinear}
\frac{\mathrm{d}_q}{\mathrm{d}_q x}Y(x)
	:= \frac{Y(qx) - Y(x)}{qx - x}
	= \left(U + \frac{\Phi}{x} \right) \cdot Y(x).
\end{align}
We always assume that $U$ is diagonal, with equal diagonal entries grouped into blocks. More precisely,
\begin{itemize}
	\item We write
	\begin{align*}
	U & =
	\mathrm{diag}(
		u_1 I_{m_1}, \ldots, 
		u_n I_{m_n}) =
		u_1E_1+\cdots+u_nE_n,
	\end{align*}
	where $u_1,\ldots,u_n$ are pairwise distinct,
	with respective multiplicities $m_1,\ldots,m_n$,
	and
	\[
	E_k
	:=
	\mathrm{diag}
	\bigl(0,\ldots,\overset{k\text{-th}}{I_{m_k}},\ldots,0\bigr);
	\]
	
	\item We regard system \eqref{qLinear} as an $n\times n$ block system
	with block sizes $(m_1,\ldots,m_n)$ and use this block decomposition
	throughout. In particular, we write 
	\[
		\Phi
		=
		\begin{pmatrix}
			\Phi_{11} & \cdots & \Phi_{1n} \\
			\vdots & \ddots & \vdots \\
			\Phi_{n1} & \cdots & \Phi_{nn}
		\end{pmatrix},
		\qquad
		\Phi_{ij}\in\mathrm{Mat}_{m_i\times m_j}(\mathbb{C});
	\]

	\item System \eqref{qLinear} has total matrix size $m\times m$, where $m=m_1+\cdots+m_n$.
\end{itemize}
Since the functions considered below are generally multivalued,
we work on the logarithmic Riemann surface
$\widetilde{\mathbb{C}^{\ast}}$.
A point of this surface is specified by its modulus and an argument in
$\mathbb{R}$, rather than modulo $2\pi$.

The remainder of this section is organized as follows.
In Section~\ref{Sect:qNotation},
	we introduce the basic concepts and
	notation used throughout the paper.
In Section~\ref{Sect:CanSol},
	we construct the canonical fundamental solutions of
	the linear $q$-difference system \eqref{qLinear} and determine
	the locations of their possible poles.
In Section~\ref{Sect:qMonodromy},
	we introduce the $q$-monodromy data and
	study their uniqueness properties.
In Section~\ref{Sect:qBorel},
	we treat divergent canonical formal fundamental solutions by
	$q$-Borel summation.
Finally, in Section~\ref{Sect:ExampCan2},
	we present an explicit example in the $2\times2$ case.

\subsection{Notation}
\label{Sect:qNotation}

For a matrix $A$, we use the following notation:
\begin{itemize}
\item $A^{[k]}$ denotes
	the upper-left $k\times k$ submatrix of $A$;
\item $A_{\ast k}$ and $A_{k\ast}$
	denote the $k$-th column and the $k$-th row of $A$, 
	respectively;
\item $A_{\hat{k}\hat{k}}$ denotes the submatrix obtained from $A$ by deleting its $k$-th row and $k$-th column.
\end{itemize}
We also denote
\begin{align*}
	[A]_q &:=
	\frac{q^A - I}{q-1} :=
	\frac{\mathrm{e}^{A \ln q} - I}{q-1},\\
	[A]^q &:=
	\log_q\left(I+(q-1)A \right) :=
	\frac{ \ln\left(I+(q-1)A \right) }{ \ln q },\quad
	\text{for invertible $I+(q-1)A$}.
\end{align*}
The notation $[A]^q$ is generally multivalued, and we choose one of its
branches whenever needed. Throughout the paper, we fix the argument of
$q$. With this convention, the identities
\begin{align*}
	q^{[A]^q}=I+(q-1)A,\qquad
	[[A]^q]_q=A,
\end{align*}
hold for every chosen branch.

We shall also use the following terminology.
A matrix $A$ is called \textbf{$q$-nonresonant} if
$I+(q-1)A$ is invertible and
\begin{align*}
	\frac{1+(q-1)\lambda_i}
	{1+(q-1)\lambda_j}
	\notin q^{\mathbb{Z}\setminus\{0\}},\quad
	\forall \lambda_i,\lambda_j\in\mathrm{Spec}(A).
\end{align*}
For $a\in\widetilde{\mathbb{C}^{\ast}}$, the set
\[
	aq^{\mathbb{Z}}
	:=
	\{aq^k\mid k\in\mathbb{Z}\}
\]
is called the \textbf{$q$-spiral} through $a$.
We denote the space of $q$-spirals by
$\widetilde{\mathbb{C}^{\ast}}/q^{\mathbb{Z}}$.
A scalar- or matrix-valued function $C(x)$
is called a \textbf{$q$-constant} with respect to $x$ if
\[
	C(qx)=C(x).
\]

\begin{defi}\label{Def:q}
We use the following notation for
the \textbf{$q$-Pochhammer symbol},
the \textbf{$q$-Gamma function},
the \textbf{$q$-exponential function} and
the \textbf{Jacobi theta function}:
\begin{align*}
	(a;q)_\infty
	&:=
	\begin{cases}
		\displaystyle
		\prod_{k=0}^{\infty}(1-aq^k),
		& |q|<1,\\[6pt]
		\displaystyle
		\prod_{k=1}^{\infty}\frac{1}{1-aq^{-k}},
		& |q|>1,
	\end{cases}\\
	\Gamma_q(x)
	&:=
	\begin{cases}
		\displaystyle
		(1-q)^{1-x}
		\frac{(q;q)_\infty}{(q^x;q)_\infty},
		& |q|<1,\\[8pt]
		\displaystyle
		q^{\binom{x-1}{2}}
		(1-q^{-1})^{1-x}
		\frac{(q^{-1};q^{-1})_\infty}
		{(q^{-x};q^{-1})_\infty},
		& |q|>1,
	\end{cases}\\
	e_q(x)
	&:=
	\sum_{n=0}^{\infty}
	\frac{x^n}{[n]!_q},
	\qquad
	[n]!_q
	:=
	\prod_{k=1}^{n}[k]_q,\\
	\theta_q(x)
	&:=
	(q;q)_\infty
	(-x;q)_\infty
	\left(-\frac{q}{x};q\right)_\infty,
	\qquad
	|q|<1.
\end{align*}
In the definition of the $q$-Gamma function, we choose
$\arg(1-q)\in\left(-\frac{\pi}{2},\frac{\pi}{2}\right)$
when $|q|<1$ and
$\arg(1-q^{-1})\in
\left(-\frac{\pi}{2},\frac{\pi}{2}\right)$
when $|q|>1$.
\end{defi}

The preceding extension of the $q$-Pochhammer symbol
to the case $|q|>1$ is natural because the $q$-exponential function
has the same product representation for both $|q|<1$ and $|q|>1$:
\begin{align}
\label{Eq:eqisqP}
	e_q(x) =
	\frac{1}{((1-q)x;q)_\infty}.
\end{align}
This convention therefore allows us to treat these two cases
uniformly throughout the paper.
We shall also frequently use the following identities:
\begin{subequations}
\begin{alignat}{2}
\label{Eq:qP}
(qx;q)_\infty
	&=
	(1-x)^{-1}\cdot(x;q)_\infty,
&\qquad
(x;q^{-1})_\infty
	&=
	\frac{1}{(qx;q)_\infty},
\\
\Gamma_q(x+1)
	&=
	[x]_q\cdot\Gamma_q(x),
&\qquad
\Gamma_{q^{-1}}(x)
	&=
	q^{-\binom{x-1}{2}}\Gamma_q(x),
\\
e_q(qx)
	&=
	(1+(q-1)x)\cdot e_q(x),
&\qquad
e_{q^{-1}}(x)
	&=
	e_q(-x)^{-1},
\label{Eq:eq}\\
\theta_q(qx)
	&=
	x^{-1}\cdot\theta_q(x)
	=
	\theta_q(x^{-1}),
&\qquad
\theta_q(x)
	&=
	\sum_{k\in\mathbb{Z}}
	q^{\binom{k}{2}}x^k,
	\qquad |q|<1.
\end{alignat}
\end{subequations}
Consequently, for $|q|<1$ and $a,b\in\mathbb{C}$,
the meromorphic function
$x^{a-b}\frac{\theta_q(q^a x)}{\theta_q(q^b x)}$
on $\widetilde{\mathbb{C}^{\ast}}$
is a $q$-constant with respect to $x$.

\subsection{Canonical Fundamental Solutions}
\label{Sect:CanSol}

In this subsection, we introduce
	the canonical fundamental solutions
	$Y^{[0]}(x)$ and $Y^{[\infty]}(x)$ of
	the linear $q$-difference system \eqref{qLinear}
	at $x=0$ and $x=\infty$, respectively.
We also determine the possible locations of their poles
	on $\mathbb{C}^{\ast}$.
This pole structure will play a key role
	in this paper.

For more general linear systems of the form
\begin{align*}
	T_xY(x)=(A_Nx^N+\cdots+A_0)\cdot Y(x),
\end{align*}
	the classical theory of linear $q$-difference systems
	provides canonical fundamental solutions at $x=0$ and $x=\infty$
	\cite{birkhoff1913generalized,
	adams1928linear,
	birkhoff1941note,
	ramis_local_2013,
	sauloy_basic_2024}.
For system \eqref{qLinear},
	this construction takes the following form.

\begin{pro}\label{Can:zero}
Suppose that $\Phi$ is $q$-nonresonant and that
a branch $[\Phi]^q$ has been fixed.
Then there exists a unique formal fundamental solution of
    system \eqref{qLinear}
	of the following form:
	\begin{subequations}
    \begin{align}
    	\label{Solatzero}
    	Y^{[0]}(x;U,\Phi) & =
    		H^{[0]}(x;U,\Phi)\cdot
    		x^{[\Phi]^q}, \\
    	\label{Series:Solatzero}
    	H^{[0]}(x;U,\Phi) & =
    		I+\sum_{\nu=1}^\infty H_\nu^{[0]}(U,\Phi)\cdot x^{\nu}.
    \end{align}
	\end{subequations}
    We call \eqref{Solatzero}
    the \textbf{canonical fundamental solution} of
    the linear $q$-difference system \eqref{qLinear} at $x=0$.
	Moreover,
	the formal power series $H^{[0]}(x)$ is convergent.
\end{pro}

\begin{prf}
For simplicity,
	we denote $H_\nu^{[0]}(U,\Phi)$
	by $H_\nu^{[0]}$.
By substituting \eqref{Solatzero} and
\eqref{Series:Solatzero} into \eqref{qLinear}, we obtain
\begin{align*}
	H^{[0]}_{\nu} \cdot 
	([\nu]_q I+q^\nu\Phi) -
	\Phi \cdot H^{[0]}_{\nu} & = 
		U H^{[0]}_{\nu-1},\quad
		\nu\geqslant 1,\qquad
		H^{[0]}_0 = I.
\end{align*}
Hence, when $\Phi$ is $q$-nonresonant,
this recurrence uniquely determines the sequence
$(H^{[0]}_\nu)_{\nu\in\mathbb{N}}$,
and there exist constants $C, R > 0$ such that
\begin{align*}
	\lVert H^{[0]}_{\nu} \rVert < C R^\nu |[\nu]!_q|^{-1}.
\end{align*}
As a result,
$H^{[0]}(x)$ converges in a neighborhood of $x=0$.
\end{prf}

\noindent
Although $[\Phi]^q$ in general cannot be chosen canonically,
this ambiguity does not affect the power series $H^{[0]}(x)$.
It only changes the factor $x^{[\Phi]^q}$
by right multiplication by a $q$-constant.
Thus, the lack of a canonical choice for $[\Phi]^q$
has no essential effect on this construction.

The canonical fundamental solution of system \eqref{qLinear}
at $x=\infty$ provided by the classical Birkhoff--Guenther theory
does not have a well-behaved differential limit as $q\to1$
and is not defined in this form when $u_k=0$ for some $k$.
For these reasons, we introduce
the canonical fundamental solution of system \eqref{qLinear}
at $x=\infty$ of the following form, which will exhibit
more natural analytic behavior.
A similar construction has appeared in
\cite{kakei_q-analogue_2006}.

\begin{pro}\label{Can:inf}
Fix a block-diagonal matrix:
\begin{align*}
	\boldsymbol{\beta}
	=
	\mathrm{diag}(\beta_1,\ldots,\beta_n).
\end{align*}
Suppose that
$u_i\notin u_jq^{\mathbb{Z}\setminus\{0\}}$
for all $i\neq j$.
If $u_k=0$ for some $k$, assume in addition that
$[\beta_k]_q=\Phi_{kk}$ is $q$-nonresonant.
Then there exists a unique formal fundamental solution of
system \eqref{qLinear}
of the following form:
\begin{subequations}
\begin{align}
	\label{Solatinf}
	Y^{[\infty]}(x;U,\Phi;\boldsymbol{\beta})
	&=
	H^{[\infty]}(x;U,\Phi;\boldsymbol{\beta})
	\cdot
	x^{\boldsymbol{\beta}}
	e_q(q^{-\boldsymbol{\beta}}Ux),\\
	\label{Series:Solatinf}
	H^{[\infty]}(x;U,\Phi;\boldsymbol{\beta})
	&=
	I+\sum_{\nu=1}^{\infty}
	H_\nu^{[\infty]}(U,\Phi;\boldsymbol{\beta})
	\cdot x^{-\nu}.
\end{align}
\end{subequations}
We call \eqref{Solatinf}
the \textbf{canonical fundamental solution
with formal monodromy $\boldsymbol{\beta}$}
of the linear $q$-difference system \eqref{qLinear}
at $x=\infty$.
Moreover, for each $k$ such that $u_k\neq0$, the corresponding
formal power series
\[
\begin{cases}
H^{[\infty]}(x)_{\ast k}, & |q|<1,\\
\left(H^{[\infty]}(x)^{-1}\right)_{k\ast}, & |q|>1,
\end{cases}
\]
is convergent.
\end{pro}

\begin{prf}
For simplicity,
we denote
$H_\nu^{[\infty]}(U,\Phi;\boldsymbol{\beta})$
by $H_\nu^{[\infty]}$.
By applying \eqref{Eq:eq} and
substituting \eqref{Solatinf} and
\eqref{Series:Solatinf} into \eqref{qLinear}, we obtain
\begin{align}\label{Eq:HinfRec}
	U H^{[\infty]}_{\nu+1}
	-
	q^{-(\nu+1)}H^{[\infty]}_{\nu+1}U
	&=
	H^{[\infty]}_{\nu}
	\cdot[-\nu I+\boldsymbol{\beta}]_q
	-
	\Phi\cdot H^{[\infty]}_{\nu},
	\quad
	\nu\geqslant0,
	\qquad
	H^{[\infty]}_0=I.
\end{align}
The condition
$u_i\notin u_jq^{\mathbb{Z}\setminus\{0\}}$
uniquely determines every block of
$H^{[\infty]}_{\nu+1}$,
except possibly the $(k,k)$-block when $u_k=0$.
In this case, the condition
$[\beta_k]_q=\Phi_{kk}$
makes the $(k,k)$-block of the recurrence for $\nu=0$
identically satisfied, while the $q$-nonresonance of $\Phi_{kk}$
uniquely determines $(H^{[\infty]}_{\nu})_{kk}$
from the $(k,k)$-block of the recurrence for every $\nu\geqslant1$.
Hence, under the assumptions of the proposition,
this recurrence uniquely determines the sequence
$(H^{[\infty]}_\nu)_{\nu\in\mathbb{N}}$.

Taking the $k$-th block column
in \eqref{Eq:HinfRec} gives
\begin{align*}
(H_{\nu+1}^{[\infty]})_{\ast k}
&=
\frac{1}{U-q^{-(\nu+1)}u_kI}
\left(
(H_\nu^{[\infty]})_{\ast k}
[-\nu I+\beta_k]_q
-
\Phi
(H_\nu^{[\infty]})_{\ast k}
\right).
\end{align*}
When $|q|<1$ and $u_k\neq0$,
	the linear operators on the right-hand side are uniformly
	bounded with respect to $\nu$.
Therefore, there exist constants $C,R>0$ such that
\begin{align*}
\left\|
\left(H_\nu^{[\infty]}\right)_{\ast k}
\right\|
&\leqslant CR^\nu.
\end{align*}
The case $|q|>1$ is treated similarly by applying the preceding
argument to the inverse formal power series.
\end{prf}

On the other hand, the canonical fundamental solution considered
in the classical Birkhoff--Guenther theory
has a simpler pole structure, and
$Y^{[\infty]}(x;U,\Phi;\boldsymbol{\beta})$
is obtained from it by right multiplication by an explicit
$q$-constant.
It is therefore useful as an auxiliary function
in studying the basic properties of
$Y^{[\infty]}(x;U,\Phi;\boldsymbol{\beta})$.
For this purpose, we introduce the following lemma.

\begin{lem}\label{Lem:BG}
Suppose that
	$u_i\notin u_jq^{\mathbb{Z}\setminus\{0\}}$
	for all $i\neq j$,
	and $u_1,\ldots,u_n$ are all nonzero.
Then there exists a unique formal fundamental solution of
	system \eqref{qLinear}
	of the following form:
\begin{subequations}
\begin{align}
	\label{BGatinf}
	Y^{[\infty]}_{\mathrm{BG}}(x;U,\Phi)
	&=
	H^{[\infty]}_{\mathrm{BG}}(x;U,\Phi)
	\cdot
	x^{\frac{1}{2}(\log_qx-1)}
	x^{\log_q((q-1)U)},\\
	H^{[\infty]}_{\mathrm{BG}}(x;U,\Phi)
	&=
	I+\sum_{\nu=1}^{\infty}
	H^{[\infty]}_{\mathrm{BG},\nu}(U,\Phi)
	\cdot x^{-\nu}.
\end{align}
\end{subequations}
We call \eqref{BGatinf}
the \textbf{Birkhoff--Guenther canonical fundamental solution}
of the linear $q$-difference system \eqref{qLinear}
at $x=\infty$.
Moreover,
the formal power series
$H^{[\infty]}_{\mathrm{BG}}(x;U,\Phi)$
is convergent, and
\begin{subequations}
\begin{align}
	\label{YBGtoYbeta}
	Y^{[\infty]}(x;U,\Phi;\boldsymbol{\beta})
	&=
	Y^{[\infty]}_{\mathrm{BG}}(x;U,\Phi)
	\cdot
	\frac{
		x^{\boldsymbol{\beta}}
		e_q(q^{-\boldsymbol{\beta}}Ux)
		e_q\left(
		\frac{q}{(q-1)^2}
		U^{-1}q^{\boldsymbol{\beta}}x^{-1}
		\right)
	}{
		x^{\log_q((q-1)U)}
		x^{\frac{1}{2}(\log_qx-1)}
	},\\
	\label{BGtobeta}
	H^{[\infty]}(x;U,\Phi;\boldsymbol{\beta})
	&=
	H^{[\infty]}_{\mathrm{BG}}(x;U,\Phi)
	\cdot
	e_q\left(
		\frac{q}{(q-1)^2}
		U^{-1}q^{\boldsymbol{\beta}}x^{-1}
	\right).
\end{align}
\end{subequations}
\end{lem}

\noindent
Given an initial value at a fixed point $x_0$, the linear
$q$-difference system \eqref{qLinear} determines the values
of the corresponding solution only along the $q$-spiral
$x_0q^{\mathbb{Z}}$.
By contrast, the canonical fundamental solutions
given in \eqref{Solatzero} and \eqref{Solatinf}
are meromorphic functions of $x$ near $x=0$ and $x=\infty$,
respectively, on the logarithmic Riemann surface.
This difference can be understood in the sense that,
by prescribing the asymptotic behavior at the expansion points,
we uniquely extend the solutions on $q$-spirals
that satisfy the required conditions
to meromorphic solutions in neighborhoods of the expansion points.

The uniqueness of the canonical fundamental solutions in
Propositions~\ref{Can:zero} and~\ref{Can:inf}
directly yields the following corollary.

\begin{cor}\label{Cor:Basic}
For any $r\in\mathbb{C}^\ast$ and
$D\in\mathrm{GL}_m(\mathbb{C})$ satisfying $DU=UD$,
we have
\begin{align*}
	D^{-1}H^{[0]}(rx;U,\Phi)D
	&=
	H^{[0]}(x;rU,D^{-1}\Phi D),\\
	D^{-1}H^{[\infty]}(rx;U,\Phi;\boldsymbol{\beta})D
	&=
	H^{[\infty]}
	\bigl(x;rU,D^{-1}\Phi D;
	D^{-1}\boldsymbol{\beta}D\bigr).
\end{align*}
Writing $M^{-\top}:=(M^{-1})^\top$, we also have
\begin{subequations}\label{Eq:qInverseH}
\begin{align}
	H^{[0]}(x;U,\Phi;q)^{-\top}
	&=
	H^{[0]}(x;-U,-q\Phi^\top;q^{-1}),\\
	H^{[\infty]}(x;U,\Phi;\boldsymbol{\beta};q)^{-\top}
	&=
	H^{[\infty]}
	\bigl(x;-U,-q\Phi^\top;
	-\boldsymbol{\beta}^\top;q^{-1}\bigr).
\end{align}
\end{subequations}
\end{cor}

\noindent
The identities in \eqref{Eq:qInverseH} show that
it suffices to establish the theory for $|q|<1$
and determine the behavior of
$H^{[0]}(x)^{-1}$ and $H^{[\infty]}(x)^{-1}$,
since the corresponding theory for $|q|>1$
then follows automatically.

As we shall see, the formal power series
$H^{[0]}(x)$ and, when $u_1,\ldots,u_n$ are all nonzero,
$H^{[\infty]}(x)$
not only converge in neighborhoods of their respective expansion points
but also admit single-valued meromorphic continuations to
$\mathbb{C}^{\ast}$.

\begin{pro}\label{Pro:SolY}
Suppose that the assumptions of
Propositions~\ref{Can:zero} and~\ref{Can:inf} hold,
and that $u_1,\ldots,u_n$ are all nonzero.
The power series
$H^{[0]}(x;U,\Phi)$ and
$H^{[\infty]}(x;U,\Phi;\boldsymbol{\beta})$
defined in \eqref{Series:Solatzero} and
\eqref{Series:Solatinf}, respectively,
admit single-valued meromorphic continuations to
$\mathbb{C}^{\ast}$.
Their determinants are given by
\begin{subequations}\label{Det:H}
\begin{align}
\label{Det:zero}
\det H^{[0]}(x;U,\Phi)
&=
\det e_q\bigl((I+(q-1)\Phi)^{-1}Ux\bigr),\\
\label{Det:inf}
\det H^{[\infty]}(x;U,\Phi;\boldsymbol{\beta})
&=
\frac{
	\det\left(
		\dfrac{q}{1-q}
		U^{-1}(I+(q-1)\Phi)x^{-1};q
	\right)_\infty
}{
	\det\left(
		\dfrac{q}{1-q}
		U^{-1}q^{\boldsymbol{\beta}}x^{-1};q
	\right)_\infty
},\\
\label{Det:infBG}
\det H^{[\infty]}_{\mathrm{BG}}(x;U,\Phi) &=
	\det\left(
		\dfrac{q}{1-q}
		U^{-1}(I+(q-1)\Phi)x^{-1};q
	\right)_\infty.
\end{align}
\end{subequations}
The poles of these functions and their inverses
can occur only at the following points
\begin{subequations}\label{Pole:All}
\begin{align}
\label{Pole:H0}
\operatorname{Pole}\bigl(H^{[0]}(x)\bigr)
	&\subseteq
	\begin{cases}
		\dfrac{1}{1-q}
		\mathrm{Spec}\bigl(U^{-1}(I+(q-1)\Phi)\bigr)
		q^{-\mathbb{N}},
		& |q|<1,\\
		\varnothing,
		& |q|>1,
	\end{cases}\\
\label{Pole:H0inv}
\operatorname{Pole}\bigl(H^{[0]}(x)^{-1}\bigr)
	&\subseteq
	\begin{cases}
		\varnothing,
		& |q|<1,\\
		\dfrac{1}{1-q}
		\mathrm{Spec}\bigl(U^{-1}(I+(q-1)\Phi)\bigr)
		q^{\mathbb{N}+1},
		& |q|>1,
	\end{cases}\\
\label{Pole:Hinf}
\operatorname{Pole}
	\bigl(H^{[\infty]}(x)_{\ast k}\bigr)
	&\subseteq
	\begin{cases}
		\dfrac{1}{1-q}
		\mathrm{Spec}\bigl(u_k^{-1}q^{\beta_k}\bigr)
		q^{\mathbb{N}+1},
		& |q|<1,\\
		\dfrac{1}{1-q}
		\mathrm{Spec}\bigl(U^{-1}(I+(q-1)\Phi)\bigr)
		q^{-\mathbb{N}},
		& |q|>1,
	\end{cases}
	\qquad
	1\leqslant k\leqslant n,\\
\label{Pole:Hinfinv}
\operatorname{Pole}
	\bigl((H^{[\infty]}(x)^{-1})_{k\ast}\bigr)
	&\subseteq
	\begin{cases}
		\dfrac{1}{1-q}
		\mathrm{Spec}\bigl(U^{-1}(I+(q-1)\Phi)\bigr)
		q^{\mathbb{N}+1},
		& |q|<1,\\
		\dfrac{1}{1-q}
		\mathrm{Spec}\bigl(u_k^{-1}q^{\beta_k}\bigr)
		q^{-\mathbb{N}},
		& |q|>1,
	\end{cases}
	\qquad
	1\leqslant k\leqslant n,\\
\label{Pole:HBG}
\operatorname{Pole}
	\bigl(H^{[\infty]}_{\mathrm{BG}}(x)\bigr)
	&\subseteq
	\begin{cases}
		\varnothing,
		& |q|<1,\\
		\dfrac{1}{1-q}
		\mathrm{Spec}\bigl(U^{-1}(I+(q-1)\Phi)\bigr)
		q^{-\mathbb{N}},
		& |q|>1,
	\end{cases},\\
\label{Pole:HBGinv}
\operatorname{Pole}
	\bigl(H^{[\infty]}_{\mathrm{BG}}(x)^{-1}\bigr)
	&\subseteq
	\begin{cases}
		\dfrac{1}{1-q}
		\mathrm{Spec}\bigl(U^{-1}(I+(q-1)\Phi)\bigr)
		q^{\mathbb{N}+1},
		& |q|<1,\\
		\varnothing,
		& |q|>1,
	\end{cases}.
\end{align}
\end{subequations}
Here $\ast k$ and $k\ast$ denote the $k$-th block column
and the $k$-th block row, respectively.
Moreover, for generic data
$(U,\Phi;\boldsymbol{\beta})$, all the inclusions in
\eqref{Pole:All} become equalities.
\end{pro}

\begin{prf}
Set
\[
	A_0:=I+(q-1)\Phi,
	\qquad
	A(x):=(q-1)Ux+A_0.
\]
From \eqref{Solatzero}, \eqref{BGatinf}, and \eqref{Eq:eq},
we obtain
\begin{subequations}\label{ShiftH}
\begin{alignat}{2}
	H^{[0]}(qx)
	&=
	A(x)\cdot H^{[0]}(x)\cdot A_0^{-1},
&\qquad
	H^{[\infty]}_{\mathrm{BG}}(qx)
	&=
	A(x)\cdot H^{[\infty]}_{\mathrm{BG}}(x)\cdot
	\bigl((q-1)Ux\bigr)^{-1},
\\
	H^{[0]}(qx)^{-1}
	&=
	A_0\cdot H^{[0]}(x)^{-1}\cdot A(x)^{-1},
&\qquad
	H^{[\infty]}_{\mathrm{BG}}(qx)^{-1}
	&=
	(q-1)Ux\cdot
	H^{[\infty]}_{\mathrm{BG}}(x)^{-1}\cdot
	A(x)^{-1}.
\end{alignat}
\end{subequations}
Taking determinants gives
\begin{align*}
	\frac{\det H^{[0]}(qx)}
	{\det H^{[0]}(x)} =
	\det\bigl(I+(q-1)A_0^{-1}Ux\bigr),\qquad
	\frac{\det H^{[\infty]}_{\mathrm{BG}}(qx)}
	{\det H^{[\infty]}_{\mathrm{BG}}(x)} =
	\det\left(
		I+\frac{1}{(q-1)x}U^{-1}A_0
	\right).
\end{align*}
Using \eqref{Eq:eqisqP}
and
$\underset{x\to 0}{\lim} \det H^{[0]}(x)
=\underset{x\to \infty}{\lim} \det H^{[\infty]}_{\mathrm{BG}}(x)
=1$,
we obtain
\begin{align*}
	\det H^{[0]}(x)=
	\det e_q(A_0^{-1}Ux),\qquad
	\det H^{[\infty]}_{\mathrm{BG}}(x)=
	\det\left(
		\frac{q}{1-q}U^{-1}A_0x^{-1};q
	\right)_\infty.
\end{align*}
This proves \eqref{Det:zero} and \eqref{Det:infBG}.
Applying \eqref{BGtobeta} in Lemma~\ref{Lem:BG}
then gives \eqref{Det:inf}.

The matrix $A(x)$ is singular precisely when
\[
	x\in
	\frac{1}{1-q}
	\mathrm{Spec}\bigl(U^{-1}A_0\bigr).
\]
Since $H^{[0]}(x)^{\pm1}$ and
$H^{[\infty]}_{\mathrm{BG}}(x)^{\pm1}$
are analytic at $x=0$ and $x=\infty$, respectively,
iterating \eqref{ShiftH} gives their single-valued
meromorphic continuations and
\eqref{Pole:H0}, \eqref{Pole:H0inv},
\eqref{Pole:HBG}, and \eqref{Pole:HBGinv}.
Applying \eqref{BGtobeta} in Lemma~\ref{Lem:BG},
together with \eqref{Eq:eqisqP},
then gives the meromorphic continuation of
$H^{[\infty]}(x)$ and
\eqref{Pole:Hinf}, \eqref{Pole:Hinfinv}.
\end{prf}

The iterative continuation argument used above for
$H^{[0]}$ and $H_{\mathrm{BG}}^{[\infty]}$
is standard for canonical gauge transformations of polynomial
$q$-difference systems; see
\cite{ohyama_space_2020} and
\cite{sauloy_basic_2024}.

\subsection{$q$-Monodromy Data}
\label{Sect:qMonodromy}

To define the $q$-monodromy data of the linear
$q$-difference system \eqref{qLinear},
we require that the canonical fundamental solutions
\eqref{Solatzero} and \eqref{Solatinf}
be well defined and convergent near
$x=0$ and $x=\infty$, respectively.
To ensure this, we impose the following conditions:
\begin{itemize}
	\item $\Phi$ is $q$-nonresonant,
		and we fix a branch of $[\Phi]^q$ throughout this section;
	\item $u_i\notin u_jq^{\mathbb{Z}\setminus\{0\}}$
	for all $i\neq j$, and
	$u_1,\ldots,u_n$ are all nonzero.
\end{itemize}
We work under these assumptions whenever the $q$-monodromy data
are considered in the remainder of this subsection.

Choose $W\in\mathrm{GL}_m(\mathbb C)$ and $\Lambda$ such that
\begin{align}\label{Jordan}
	W^{-1}\cdot [\Phi]^q \cdot W &= \Lambda,
\end{align}
where $\Lambda$ is the Jordan normal form of $[\Phi]^q$.
The corresponding \textbf{Floquet solution} at
$x=0$ is
\begin{align*}
	Y^{[0]}(x;U,\Phi)W &=
	H^{[0]}(x;U,\Phi)W\cdot x^\Lambda =
	\left(
		W+\sum_{\nu=1}^{\infty}
		H_\nu^{[0]}(U,\Phi)W\cdot x^\nu
	\right)\cdot x^\Lambda.
\end{align*}
The \textbf{central connection matrix} associated with
$W$ is
\begin{align}\label{Def:Cbeta}
	C(x;U,\Phi;\boldsymbol{\beta},W) &:=
	Y^{[\infty]}(x;U,\Phi;\boldsymbol{\beta})^{-1}
	\cdot
	Y^{[0]}(x;U,\Phi)W.
\end{align}
Since $Y^{[\infty]}(x)$ and $Y^{[0]}(x)$ are fundamental solutions of
the same linear $q$-difference system, $C(x)$ is a
$q$-constant with respect to $x$. In particular, it is constant
along each $q$-spiral.

\begin{rmk}
\(C(x;U,\Phi;\boldsymbol{\beta},W)
=
\begin{cases}
x^{-\boldsymbol{\beta}}
\dfrac{\theta_q((q-1)q^{-\boldsymbol{\beta}}Ux)}
{\theta_q((q-1)Ux)}
C(x;U,\Phi;0,W), & |q|<1,\\
x^{-\boldsymbol{\beta}}
\dfrac{\theta_{q^{-1}}((1-q^{-1})Ux)}
{\theta_{q^{-1}}((1-q^{-1})q^{-\boldsymbol{\beta}}Ux)}
C(x;U,\Phi;0,W), & |q|>1
\end{cases}
\).
\end{rmk}

\begin{pro}\label{Pro:CTransformation}
For any $r\in\widetilde{\mathbb C^\ast}$ and
$D\in\mathrm{GL}_m(\mathbb{C})$ satisfying $DU=UD$,
we have
\begin{subequations}\label{Eq:CTransformation}
\begin{align}
C(
	x;rU,D^{-1}\Phi D;
	D^{-1}\boldsymbol{\beta}D,D^{-1}Wr^\Lambda
) &=
D^{-1}r^{\boldsymbol{\beta}}\cdot
C(rx;U,\Phi;\boldsymbol{\beta},W),
\label{Eq:CScaling}\\
C(x;U,\Phi;\boldsymbol{\beta},W;q)^{-\top} &=
C(
	x;-U,-q\Phi^\top;
	-\boldsymbol{\beta}^\top,W^{-\top};q^{-1}
).
\label{Eq:CDuality}
\end{align}
\end{subequations}
\end{pro}

\begin{prf}
The result follows directly from
	the definition of the central connection matrix \eqref{Def:Cbeta} and
	Corollary~\ref{Cor:Basic}.
\end{prf}
	
Although for fixed $U$, the parameter spaces of
$(C(x),\Lambda)$ and $(\Phi,W)$ have the same dimension $m^2+m$
on the generic locus, the data $(C(x),\Lambda)$ determine
$(\Phi,W)$ only up to a discrete ambiguity.
An analogous discrete ambiguity occurs in the differential setting
and is described by Schlesinger transformations, which preserve the
monodromy data while shifting the exponents of formal monodromy by
integers \cite{jimbo_monodromy_1981-1}.
To describe this ambiguity, write
\begin{align*}
	\mathrm{Spec}(U^{-1} (I+(q-1)\Phi)) & =
	\{
	u_k^{-1} q^{\alpha^{(k)}_i(\Phi)}:
	1\leqslant k \leqslant n,
	1\leqslant i \leqslant m_k
	\}.
\end{align*}
For subsets $A,B\subseteq \mathbb C^\ast$, we write
\[
	A/B:=\{a/b:a\in A,\ b\in B\}.
\]
The following proposition makes the remaining ambiguity precise.

\begin{pro}\label{Pro:qUni}
Suppose that
\begin{subequations}
\begin{align}
\label{UniC}
C(x;U,\Phi_1;\boldsymbol{\beta},W_1) & =
	C(x;U,\Phi_2;\boldsymbol{\beta},W_2),\\
\label{UniW}
W_1^{-1} [\Phi_1]^q W_1 & =
	W_2^{-1} [\Phi_2]^q W_2.
\end{align}
\end{subequations}
Then we have
\begin{subequations}\label{qSpiraleqplus}
\begin{align}
\label{qSpiraleq}
\bigcup_{k,i}
	\{ u_k^{-1} q^{\alpha^{(k)}_i(\Phi_1)} \} q^{\mathbb{Z}} & =
\bigcup_{k,i}
	\{ u_k^{-1} q^{\alpha^{(k)}_i(\Phi_2)} \} q^{\mathbb{Z}},
	\text{ in the sense of multisets}, \\
\label{ShiftBalance}
\prod_{k,i}
	q^{\alpha^{(k)}_i(\Phi_1)} & =
\prod_{k,i}
	q^{\alpha^{(k)}_i(\Phi_2)}.
\end{align}
\end{subequations}
If we further require
\begin{align}\label{NotUniPhi}
	\frac{
	\mathrm{Spec}( U^{-1}(I+(q-1)\Phi_2) )
		}{
	\mathrm{Spec}( U^{-1}(I+(q-1)\Phi_1) )
		} \cap
	q^{\mathbb{N}+1} & =
	\varnothing,
\end{align}
then $\Phi_1 = \Phi_2$ and $W_1 = W_2$.
\end{pro}

\begin{prf}
For convenience, we introduce the auxiliary
Birkhoff--Guenther connection matrix
\begin{align*}
C_{\mathrm{BG}}(x;U,\Phi,W)
&:=
Y_{\mathrm{BG}}^{[\infty]}(x;U,\Phi)^{-1}
\cdot
Y^{[0]}(x;U,\Phi)W.
\end{align*}
By \eqref{YBGtoYbeta} in Lemma~\ref{Lem:BG}, the factor relating
$C(x;U,\Phi;\boldsymbol{\beta},W)$ to
$C_{\mathrm{BG}}(x;U,\Phi,W)$ depends only on
$U$ and $\boldsymbol{\beta}$. Hence, \eqref{UniC} implies
\begin{align*}
C_{\mathrm{BG}}(x;U,\Phi_1,W_1)
&=
C_{\mathrm{BG}}(x;U,\Phi_2,W_2).
\end{align*}
Applying the determinant formulas \eqref{Det:H} in
Proposition~\ref{Pro:SolY}, we obtain
\begin{align*}
\det\left(
C_{\mathrm{BG}}(x;U,\Phi,W)W^{-1}
\right)
&=
\begin{cases}
\chi_q(x) \cdot (q;q)_\infty^m
\displaystyle
\prod_{k,i}
\theta_q\bigl(
(q-1)u_kq^{-\alpha_i^{(k)}(\Phi)}x
\bigr)^{-1},
& |q|<1,\\[14pt]
\chi_q(x) \cdot (q^{-1};q^{-1})_\infty^{-m}
\displaystyle
\prod_{k,i}
\theta_{q^{-1}}\bigl(
(1-q^{-1})u_kq^{-\alpha_i^{(k)}(\Phi)}x
\bigr),
& |q|>1,
\end{cases}\\
\chi_q(x)
&:=
x^{
\operatorname{tr}\Lambda
-\operatorname{tr}\log_q((q-1)U)
-\frac{m}{2}(\log_qx-1)
}.
\end{align*}
Thus, the determinant expressions corresponding to
	$(\Phi_1,W_1)$ and $(\Phi_2,W_2)$ have
	the same multiset of poles when $|q|<1$, and
	the same multiset of zeros when $|q|>1$,
	which yields \eqref{qSpiraleq}.
Moreover, taking determinants in \eqref{UniW}
	yields \eqref{ShiftBalance}.

Equation~\eqref{BGtobeta} in Lemma~\ref{Lem:BG}, together with
\eqref{UniC} and \eqref{UniW}, gives
\begin{align}\nonumber
	F(x)& \assign
	H^{[\infty]}(x;U,\Phi_2;\boldsymbol{\beta}) \cdot 
	H^{[\infty]}(x;U,\Phi_1;\boldsymbol{\beta})^{-1} \\
	& =
	H^{[\infty]}_{\mathrm{BG}}(x;U,\Phi_2) \cdot 
	H^{[\infty]}_{\mathrm{BG}}(x;U,\Phi_1)^{-1} =
	H^{[0]}(x;U,\Phi_2)W_2 \cdot
	W_1^{-1}H^{[0]}(x;U,\Phi_1)^{-1}.
\label{Eq:UniqueMono}
\end{align}
By Proposition~\ref{Pro:SolY},
the locations of the possible poles on both sides of
\eqref{Eq:UniqueMono} are
\begin{align*}
	\begin{cases}
		\displaystyle
		\frac{1}{1-q}\bigcup_{k,i}
			\left\{
			u_k^{-1}q^{\alpha_i^{(k)}(\Phi_1)}
			\right\}q^{\mathbb N+1}
		\text{ and }
		\displaystyle
		\frac{1}{1-q}\bigcup_{k,i}
			\left\{
			u_k^{-1}q^{\alpha_i^{(k)}(\Phi_2)}
			\right\}q^{-\mathbb N},
		& |q|<1,\\
		\displaystyle
		\frac{1}{1-q}\bigcup_{k,i}
			\left\{
			u_k^{-1}q^{\alpha_i^{(k)}(\Phi_2)}
			\right\}q^{-\mathbb N}
		\text{ and }
		\displaystyle
		\frac{1}{1-q}\bigcup_{k,i}
			\left\{
			u_k^{-1}q^{\alpha_i^{(k)}(\Phi_1)}
			\right\}q^{\mathbb N+1},
		& |q|>1.
	\end{cases}
\end{align*}
By \eqref{NotUniPhi}, the only possible common singularities are
$x=0$ and $x=\infty$.
Since
\begin{align*}
	\lim_{x\to\infty}F(x)=I,\qquad
	\lim_{x\to0}F(x)=W_2W_1^{-1},
\end{align*}
it follows from Liouville's theorem that $F(x)$
is constant and equal to $I$.
Hence, $W_1=W_2$, and \eqref{UniW} implies
$\Phi_1=\Phi_2$.
This completes the proof.
\end{prf}


More precisely, on the generic locus,
	\eqref{qSpiraleqplus} characterize
	all possible discrete choices of $(\Phi,W)$
	compatible with fixed $C(x)$ and $\Lambda$.
Let
\begin{align*}
	\bigl(n_i^{(k)}\bigr)_{
		1\leqslant k\leqslant n,\ 
		1\leqslant i\leqslant m_k}
	\in\mathbb Z^m,\qquad
	\sum_{k,i}n_i^{(k)}=0.
\end{align*}
Then, for generic $(\Phi_1,W_1)$, there exists a unique pair
$(\Phi_2,W_2)$ yielding the same central connection matrix $C(x)$
and the same $\Lambda$ such that
\begin{align*}
	\alpha_i^{(k)}(\Phi_2)
	&=
	\alpha_i^{(k)}(\Phi_1)+n_i^{(k)},
	\qquad
	1\leqslant k\leqslant n,\quad
	1\leqslant i\leqslant m_k.
\end{align*}
The existence follows by composing the $q$-isomonodromic shifts
described in Section~\ref{Sect:qisoeq} with permutations of the
$\alpha_i^{(k)}$, while the uniqueness follows from
Proposition~\ref{Pro:qUni} under the genericity condition
\eqref{NotUniPhi}.

Thus, to reconstruct $\Phi$ and $W$ uniquely
	from $C(x)$ and $\Lambda$,
	one must also record these spectral parameters.
We introduce these discrete data as follows.

\begin{defi}\label{Def:alpha}
We call
\begin{align*}
	\boldsymbol{\alpha}
	&=
	\mathrm{diag}(\alpha_1,\ldots,\alpha_n),
	\qquad
	\alpha_k
	=
	\mathrm{diag}
	\bigl(\alpha^{(k)}_1,\ldots,\alpha^{(k)}_{m_k}\bigr)
\end{align*}
	the \textbf{formal exponents} of
	the linear $q$-difference system \eqref{qLinear}
	at $x=\infty$
	if $u_1,\ldots,u_n$ are all nonzero and
\begin{align}
\label{Eq:Defalpha}
	\mathrm{Spec}\bigl(U^{-1}(I+(q-1)\Phi)\bigr)
	&=
	\{
		u_k^{-1}q^{\alpha_i^{(k)}}:
		1\leqslant k\leqslant n,\
		1\leqslant i\leqslant m_k
	\}.
\end{align}
We call $(U,\Phi;W;\boldsymbol{\alpha})$ a
	\textbf{$q$-monodromy system} if
	$W\in\mathrm{GL}_m(\mathbb C)$ is chosen as
	in \eqref{Jordan} and $\boldsymbol{\alpha}$
	satisfies \eqref{Eq:Defalpha}.
Its \textbf{$q$-monodromy data} with respect to the formal monodromy
$\boldsymbol{\beta}$ are defined by
\begin{align*}
	\mathrm{MD}_{\boldsymbol{\beta}}
	(U,\Phi;W;\boldsymbol{\alpha})
	&:=
	\bigl(
	C(x;U,\Phi;\boldsymbol{\beta},W),
	\Lambda;\boldsymbol{\alpha}
	\bigr).
\end{align*}
\end{defi}
\noindent
Note that the definition of the $q$-monodromy data also depends on
a choice of branch for $[\Phi]^q$, which we fix throughout.
Once this branch is fixed, the corresponding $\Lambda$ and the central
connection matrix $C(x)$ are uniquely determined.


Moreover, after an appropriate ordering, the formal exponents
$\alpha_i^{(k)}$
converge to elements of $\mathrm{Spec}(\Phi_{kk})$ as $q\to1$.
This follows from the Gershgorin Circle Theorem
\cite{gershgorin1931uber}.
In particular, when $u_1,\ldots,u_n$ all have multiplicity one,
$\boldsymbol{\alpha}$ recovers the formal monodromy
$\mathrm{diag}(\Phi_{11},\ldots,\Phi_{nn})$ of the differential system
in the differential limit as $q\to1$, and hence provides its natural $q$-analog.
When multiplicities occur, the corresponding block-valued $q$-analog
requires additional technical details, which we leave for future work.
Furthermore, in the multiplicity-one case, if we take the formal
monodromy to be $\boldsymbol{\beta}=\boldsymbol{\alpha}$, then
$Y^{[\infty]}(x)$ also converges, as $q\to1$, to the corresponding
canonical fundamental solution of the differential system.

\subsection{$q$-Borel Summation}
\label{Sect:qBorel}

In Proposition~\ref{Can:inf}, we constructed a canonical formal fundamental
solution of the linear $q$-difference system \eqref{qLinear} at $x=\infty$,
\begin{align*}
Y^{[\infty]}(x;U,\Phi;\boldsymbol{\beta})
&=
H^{[\infty]}(x;U,\Phi;\boldsymbol{\beta})\cdot
x^{\boldsymbol{\beta}}
e_q(q^{-\boldsymbol{\beta}}Ux).
\end{align*}
When $u_1,\ldots,u_n$ are all nonzero, the formal power series
$H^{[\infty]}(x;U,\Phi;\boldsymbol{\beta})$ is convergent.
By contrast, if $u_k=0$ for some $k$ and
$[\beta_k]_q=\Phi_{kk}$, the formal power series
$H^{[\infty]}(x;U,\Phi;\boldsymbol{\beta})$ need not converge.
In this case, we obtain an actual solution by $q$-Borel summation.

To treat the cases $|q|<1$ and $|q|>1$ uniformly, set
\begin{align*}
Q
:=
\begin{cases}
q^{-1}, & |q|<1,\\
q, & |q|>1.
\end{cases}
\end{align*}

\begin{defi}\label{Def:qBorel}
Let
\[
\widehat f(x)=\sum_{m=0}^{\infty}a_mx^{-m}
\]
be a scalar- or matrix-valued formal power series at $x=\infty$.
Its \textbf{$q$-Borel transform} is defined by
\begin{align}\label{Def:qBorelTransform}
(\mathcal B_q\widehat f)(\xi)
&:=
\sum_{m=0}^{\infty}
Q^{-\binom{m}{2}}a_m\xi^m.
\end{align}
Let $d\in\mathbb C^\ast$ and denote by
	$dq^{\mathbb Z}=dQ^{\mathbb Z}$ its $q$-spiral.
Suppose that $\mathcal B_q\widehat f$
	converges near $\xi=0$,
admits analytic continuation to a neighborhood of
$d^{-1}q^{\mathbb Z}$,
and has $Q$-exponential growth of order one there in the sense of
\cite[Definition~1.1]{dreyfus2015building}.
The \textbf{$q$-Borel sum} of $\widehat f$ along $d$ is defined by
\begin{align}\label{Def:qBorelSum}
(\mathcal S_d\widehat f)(x)
&:=
\sum_{\xi\in d^{-1}q^{\mathbb Z}}
\frac{(\mathcal B_q\widehat f)(\xi)}
{\theta_{Q^{-1}}(x\xi)}.
\end{align}
This sum depends only on the $q$-spiral $dq^{\mathbb Z}$.
When these conditions are satisfied, we say that
$\widehat f$ is \textbf{$q$-Borel summable along $d$}.
The possible poles of $\mathcal S_d\widehat f$ lie on
$-dq^{\mathbb Z}$ and have order at most one.
\end{defi}

The standard $q$-Borel--Laplace estimate
\cite[Definition~1.8 and Proposition~1.14]{dreyfus2015building}
gives the following basic property of the $q$-Borel sum.
For every sufficiently small $\varepsilon>0$, there exist
$R,L,M>0$ such that
\begin{align}\label{Eq:qGevreyAsymptotic}
\left\|
(\mathcal S_d\widehat f)(x)
-
\sum_{m=0}^{N-1}a_mx^{-m}
\right\|
&\leqslant
LM^N|Q|^{\binom{N}{2}}|x|^{-N}
\end{align}
for every $N\geqslant1$ and every $x$ satisfying
$|x|>R$ and
\[
|x+dq^\ell|>
\varepsilon|dq^\ell|,
\qquad
\ell\in\mathbb Z.
\]
Thus, $\mathcal S_d\widehat f$ admits $\widehat f$ as its
$q$-Gevrey asymptotic expansion of order one at $x=\infty$
away from the $q$-spiral $-dq^{\mathbb Z}$.

The following result is a specialization of the classical
$q$-summation theory developed in
\cite{zhang_developpements_1999}.

\begin{pro}\label{Pro:qBorelSummability}
Under the assumptions of Proposition~\ref{Can:inf}, suppose that
$u_k=0$ for some $k$.
Set
\begin{align}\label{Def:SingularDirections}
\Sigma_k
:=
\frac{1}{q-1}
\bigcup_{i\neq k}
u_i^{-1}
\bigl(1+(q-1)\mathrm{Spec}(\Phi_{kk})\bigr)
q^{\mathbb Z}
\subseteq
\mathbb C^\ast/q^{\mathbb Z}.
\end{align}
For every
$dq^{\mathbb Z}
\in(\mathbb C^\ast/q^{\mathbb Z})\setminus\Sigma_k$,
the formal power series
$H^{[\infty]}(x;U,\Phi;\boldsymbol{\beta})$ is
$q$-Borel summable along $d$.
Define
\begin{align}
H_d^{[\infty]}(x;U,\Phi;\boldsymbol{\beta})
&:=
\bigl(\mathcal S_dH^{[\infty]}\bigr)
(x;U,\Phi;\boldsymbol{\beta}),\notag\\
Y_d^{[\infty]}(x;U,\Phi;\boldsymbol{\beta})
&:=
H_d^{[\infty]}(x;U,\Phi;\boldsymbol{\beta})\cdot
x^{\boldsymbol{\beta}}
e_q(q^{-\boldsymbol{\beta}}Ux).
\label{qBorelCanInf}
\end{align}
Then $Y_d^{[\infty]}(x)$ is
	a fundamental solution of
	the linear $q$-difference system \eqref{qLinear}
	at $x=\infty$.
Moreover,
	$H_d^{[\infty]}(x)$ admits $H^{[\infty]}(x)$ as its
	$q$-Gevrey asymptotic expansion of order one at $x=\infty$
	in the sense of \eqref{Eq:qGevreyAsymptotic}.
We call $Y_d^{[\infty]}$ the
	\textbf{canonical fundamental solution at $x=\infty$
	in the direction $d$}.
\end{pro}

\begin{prf}
We first consider the case $|q|<1$. Set
$G(\xi):=
\bigl(\mathcal B_qH^{[\infty]}\bigr)(\xi)$
and denote by $G_{\ast k}$ the $k$-th block column of $G$.
The coefficient recurrence in the proof of
Proposition~\ref{Can:inf} shows that $G$ converges near $\xi=0$.
Applying the $q$-Borel transform to that recurrence gives
\begin{align}\label{Eq:qBorelContinuation}
(I+(q-1)\Phi) \xi \cdot G_{\ast k}(q\xi) &=
\xi G_{\ast k}(\xi) \cdot (I+(q-1)\Phi_{kk})
-(q-1)U \cdot G_{\ast k}(\xi).
\end{align}
For $i\neq k$, the linear map
\begin{align*}
X\longmapsto
\xi X \cdot (I+(q-1)\Phi_{kk}) - (q-1)u_i \cdot X
\end{align*}
is invertible unless
\begin{align*}
\xi \in \frac{(q-1)u_i}{1+(q-1)\mathrm{Spec}(\Phi_{kk})}.
\end{align*}
Since $q\xi$ is closer to the origin than $\xi$,
\eqref{Eq:qBorelContinuation} analytically continues
$G_{\ast k}$ from a neighborhood of $\xi=0$, with possible
singularities only on
\begin{align*}
(q-1)\bigcup_{i\neq k}
\frac{u_i}{1+(q-1)\mathrm{Spec}(\Phi_{kk})}
q^{-\mathbb N}.
\end{align*}
The remaining block columns of $H^{[\infty]}$ are convergent. 
Therefore, for every $dq^{\mathbb Z}\notin\Sigma_k$,
	the $q$-spiral $d^{-1}q^{\mathbb Z}$
	avoids all possible
singularities of $G$.

Iterating \eqref{Eq:qBorelContinuation} also gives
the required $Q$-exponential growth in a neighborhood of
every such spiral.
Together with the standard estimate for the theta function
\cite[Lemma~1.3]{dreyfus2015building}, this shows that
$H^{[\infty]}$ is $q$-Borel summable along every
$dq^{\mathbb Z}\notin\Sigma_k$ and that
$H_d^{[\infty]} :=
\mathcal S_dH^{[\infty]}$
is meromorphic on $\mathbb C^\ast$.
By \cite[Lemma~1.4 and Remark~1.5]{dreyfus2015building},
$Y_d^{[\infty]}$ is a solution of \eqref{qLinear}.
The asymptotic estimate \eqref{Eq:qGevreyAsymptotic} gives
$H_d^{[\infty]}(x)=I+O(x^{-1})$ as $x\to\infty$
away from $-dq^{\mathbb Z}$.
Hence $Y_d^{[\infty]}$ is a fundamental solution.

The case $|q|>1$ reduces to the case $|q|<1$ by applying
the relation \eqref{Eq:qInverseH} in
Corollary~\ref{Cor:Basic}.
For this reduction, note that the preceding argument for $|q|<1$
applies equally to the inverse formal power series
$H^{[\infty]}(x)^{-1}$.
Moreover,
\begin{align*}
\frac{1}{q^{-1}-1}
\bigcup_{i\neq k}
(-u_i)^{-1}
\bigl(
1+(q^{-1}-1)\mathrm{Spec}(-q\Phi_{kk}^{\top})
\bigr)
q^{\mathbb Z}
&=
\Sigma_k.
\end{align*}
This completes the proof.
\end{prf}

Suppose in addition that $\Phi$ is $q$-nonresonant and
$W,\Lambda$ are chosen as in \eqref{Jordan}.
We define the
\textbf{central connection matrix in the direction $d$} by
\begin{align}\label{Def:Cd}
	C_d(x;U,\Phi;\boldsymbol{\beta},W)
	&:=
	Y_d^{[\infty]}(x;U,\Phi;\boldsymbol{\beta})^{-1}
	\cdot
	Y^{[0]}(x;U,\Phi)W.
\end{align}

\subsection{Explicit Rank-Two Formulas}
\label{Sect:ExampCan2}

Consider the following $2\times2$ linear $q$-difference system:
\begin{equation}\label{rank2}
	\frac{\mathrm{d}_q}{\mathrm{d}_qx}Y(x)
	=
	\left(U+\Phi x^{-1}\right)\cdot Y(x),\quad
	U=
	\begin{pmatrix}
		u_1 & 0\\
		0 & u_2
	\end{pmatrix},\quad
	\Phi=
	\begin{pmatrix}
		\varphi_{11} & \varphi_{12}\\
		\varphi_{21} & \varphi_{22}
	\end{pmatrix}.
\end{equation}
Assume that $u_1,u_2\in\mathbb C^\ast$,
$u_1\notin u_2q^{\mathbb Z}$, and $\Phi$ is $q$-nonresonant.
Suppose further that $\Phi$ has distinct eigenvalues
$[\lambda_1]_q$ and $[\lambda_2]_q$,
neither of which is equal to $\varphi_{11}$.
Set
\begin{align*}
	W :=
	\begin{pmatrix}
		\displaystyle
		\frac{\varphi_{12}}
		{[\lambda_1]_q-\varphi_{11}}
		&
		\displaystyle
		\frac{\varphi_{12}}
		{[\lambda_2]_q-\varphi_{11}}\\
		1&1
	\end{pmatrix},\qquad
	\Lambda :=
	\mathrm{diag}(\lambda_1,\lambda_2).
\end{align*}
Then $W$ is invertible and
\begin{align*}
	W^{-1} [\Phi]^q W
	&=
	\Lambda.
\end{align*}
Choose $\alpha_1$ and $\alpha_2$ such that
\begin{align*}
	\mathrm{Spec}\bigl(U^{-1}(I+(q-1)\Phi)\bigr)
	&=
	\left\{
		u_1^{-1}q^{\alpha_1},
		u_2^{-1}q^{\alpha_2}
	\right\}.
\end{align*}
These choices determine the $q$-monodromy system
$(U,\Phi;W;\boldsymbol{\alpha})$ considered below.

We now compute the canonical fundamental solutions of the
$q$-monodromy system
$(U,\Phi;W;\boldsymbol{\alpha})$
and the corresponding $q$-monodromy data.
For the formal monodromy $\boldsymbol{\alpha}$,
these fundamental solutions admit explicit expressions
in terms of the basic hypergeometric series
\begin{align*}
	{}_2\varphi_1\left(
	\begin{array}{c}
		a_1,a_2\\
		b
	\end{array};
	x
	\right)
	&:=
	\sum_{k=0}^{\infty}
	\frac{(a_1;q)_k(a_2;q)_k}
	{(b;q)_k(q;q)_k}x^k.
\end{align*}

\begin{pro}
For the formal monodromy $\boldsymbol{\alpha}$, the canonical
fundamental solution of the linear $q$-difference system
\eqref{rank2} at $x=\infty$ is
	\begin{align*}
		Y^{[\infty]}(x) & =
		H^{[\infty]}(x) \cdot
		x^{\boldsymbol{\alpha}}e_{q}(q^{-\boldsymbol{\alpha}}U x),\\
		H^{[\infty]}(x)_{11} & =
		{}_2\varphi_1 \left(
		\begin{array}{c} 
			q^{\lambda_1-\alpha_1}, q^{\lambda_2-\alpha_1} \\ 
			\frac{u_2}{u_1} 
		\end{array}; 
		-\frac{u_1^{-1}q^{\alpha_1+1}}{q-1} x^{-1}
		\right),\\
		H^{[\infty]}(x)_{22} & =
		{}_2\varphi_1 \left(
		\begin{array}{c}
			q^{\lambda_1-\alpha_2}, q^{\lambda_2-\alpha_2} \\ 
			\frac{u_1}{u_2}
		\end{array}; 
		-\frac{u_2^{-1}q^{\alpha_2+1}}{q-1} x^{-1}
		\right),\\
		H^{[\infty]}(x)_{21} & =
		\frac{q\varphi_{21}}{u_1-qu_2}x^{-1} \cdot
		{}_2\varphi_1 \left(
		\begin{array}{c}
			q^{\lambda_1-\alpha_1+1}, q^{\lambda_2-\alpha_1+1} \\ 
			q^2\frac{u_2}{u_1}
		\end{array}; 
		-\frac{u_1^{-1}q^{\alpha_1+1}}{q-1} x^{-1}\right),\\
		H^{[\infty]}(x)_{12} & =
		\frac{q\varphi_{12}}{u_2-qu_1}x^{-1} \cdot
		{}_2\varphi_1 \left(
		\begin{array}{c}
			q^{\lambda_1-\alpha_2+1}, q^{\lambda_2-\alpha_2+1} \\ 
			q^2\frac{u_1}{u_2}
		\end{array}; 
		-\frac{u_2^{-1}q^{\alpha_2+1}}{q-1} x^{-1}\right).
	\end{align*}
The Floquet solution at $x=0$ with respect to $W$ is
\begin{align*}
	Y^{[0]}(x)W & =
	H^{[0]}(x)W \cdot
	x^\Lambda,\\
	( H^{[0]}(x)W )_{11} & =
		\frac{\varphi_{12}}{[\lambda_1]_q-\varphi_{11}}\cdot 
		{}_2\varphi_1\left(
		\begin{matrix}
			q^{\lambda_1 - \alpha_1},
			q^{\lambda_1 - \alpha_1 + 1} \frac{u_1}{u_2} \\
			q^{\lambda_1 - \lambda_2 + 1}
		\end{matrix};
		-\frac{q-1}{u_2^{-1}q^{\alpha_2}}x \right)
		e_q(q^{-\alpha_1}u_1 x)\\
	& =
	\frac{\varphi_{12}}{[\lambda_1]_q-\varphi_{11}}\cdot 
		{}_2\varphi_1\left(
		\begin{matrix}
			q^{\lambda_1 - \alpha_2} \frac{u_2}{u_1},
			q^{\lambda_1 - \alpha_2 + 1} \\
			q^{\lambda_1 - \lambda_2 + 1}
		\end{matrix};
		-\frac{q-1}{u_1^{-1}q^{\alpha_1}}x \right)
		e_q(q^{-\alpha_2}u_2 x),\\
	( H^{[0]}(x)W )_{12} & =
		\frac{\varphi_{12}}{[\lambda_2]_q-\varphi_{11}}\cdot
		{}_2\varphi_1\left(
		\begin{matrix}
			q^{\lambda_2 - \alpha_1},
			q^{\lambda_2 - \alpha_1 + 1}\frac{u_1}{u_2} \\
			q^{\lambda_2 - \lambda_1 + 1}
		\end{matrix};
		-\frac{q-1}{u_2^{-1}q^{\alpha_2}}x\right)
		e_q(q^{-\alpha_1}u_1 x)\\
	& =
		\frac{\varphi_{12}}{[\lambda_2]_q-\varphi_{11}}\cdot
		{}_2\varphi_1\left(
		\begin{matrix}
			q^{\lambda_2 - \alpha_2} \frac{u_2}{u_1},
			q^{\lambda_2 - \alpha_2 + 1} \\
			q^{\lambda_2 - \lambda_1 + 1}
		\end{matrix};
		-\frac{q-1}{u_1^{-1}q^{\alpha_1}}x\right)
		e_q(q^{-\alpha_2}u_2 x),\\
	( H^{[0]}(x)W )_{21} & =
		{}_2\varphi_1\left(
		\begin{matrix}
			q^{\lambda_1 - \alpha_1 + 1},
			q^{\lambda_1 - \alpha_1} \frac{u_1}{u_2} \\
			q^{\lambda_1 - \lambda_2 + 1}
		\end{matrix};
		-\frac{q-1}{u_2^{-1}q^{\alpha_2}}x\right)
		e_q(q^{-\alpha_1}u_1 x)\\
	& =
		{}_2\varphi_1\left(
		\begin{matrix}
			q^{\lambda_1 - \alpha_2 + 1} \frac{u_2}{u_1},
			q^{\lambda_1 - \alpha_2} \\
			q^{\lambda_1 - \lambda_2 + 1}
		\end{matrix};
		-\frac{q-1}{u_1^{-1}q^{\alpha_1}}x\right)
		e_q(q^{-\alpha_2}u_2 x),\\
	( H^{[0]}(x)W )_{22} & =
		{}_2\varphi_1\left(
		\begin{matrix}
			q^{\lambda_2 - \alpha_1 + 1},
			q^{\lambda_2 - \alpha_1} \frac{u_1}{u_2} \\
			q^{\lambda_2 - \lambda_1 + 1}
		\end{matrix};
		-\frac{q-1}{u_2^{-1}q^{\alpha_2}}x\right)
		e_q(q^{-\alpha_1}u_1 x)\\
	& =
		{}_2\varphi_1\left(
		\begin{matrix}
			q^{\lambda_2 - \alpha_2 + 1} \frac{u_2}{u_1},
			q^{\lambda_2 - \alpha_2} \\
			q^{\lambda_2-\lambda_1+1}
		\end{matrix};
		-\frac{q-1}{u_1^{-1}q^{\alpha_1}}x\right)
		e_q(q^{-\alpha_2}u_2 x).
\end{align*}

\end{pro}

\begin{prf}
The result follows by substituting the defining series of
${}_2\varphi_1$ into the coefficient recurrences in the proofs of
Propositions~\ref{Can:zero} and~\ref{Can:inf}.
\end{prf}

For $|q|<1$, the connection formula for the basic
hypergeometric series ${}_2\varphi_1$
\cite{gasper_basic_2004}
allows us to evaluate explicitly the central connection matrix $C(x)$
of the $q$-monodromy system
$(U,\Phi;W;\boldsymbol{\alpha})$.
Furthermore, when the formal monodromy is taken to be
$\boldsymbol{\beta}=\mathrm{diag}(\beta_1,\beta_2)$,
the central connection matrix satisfies
\begin{equation*}
C(x;\boldsymbol{\beta})
=
\frac{
	\theta_q\left(
	\frac{(q-1)U}{q^{\boldsymbol{\beta}}}x
	\right)
}{
	\theta_q\left(
	\frac{(q-1)U}{q^{\boldsymbol{\alpha}}}x
	\right)
}
x^{\boldsymbol{\alpha}-\boldsymbol{\beta}}
\cdot C(x;\boldsymbol{\alpha}).
\end{equation*}
Consequently, the central connection matrix with respect to
the formal monodromy $\boldsymbol{\beta}$ is given by the following
proposition.

\begin{pro}\label{Pro:EntryC2}
Suppose that $|q|<1$.
The entries of
	the central connection matrix $C(x;\boldsymbol{\beta})$
	with respect to the formal monodromy $\boldsymbol{\beta}$
	are given by
\begin{subequations}
\begin{align}
\label{qC11}
	C_{11} & =
	\frac{\varphi_{12}}{[\lambda_1]_q - \varphi_{11}} \cdot
	\frac{
		(q^{\alpha_1 - \lambda_2 + 1} ;q)_\infty
		(q^{\alpha_2 - \lambda_2 + 1} \frac{u_1}{u_2} ;q)_\infty
	}{
		(q^{\lambda_1 - \lambda_2 + 1} ;q)_\infty
		(q \frac{u_1}{u_2} ;q)_\infty
	} \cdot
	\frac{
		\theta_q \left(
		\frac{(q-1) u_1}{q^{\beta_1}}
		x\right)
	}{
		\theta_q \left(
		\frac{(q-1) u_1}{q^{\alpha_1}}
		x \right)
	}
	\frac{
		\theta_q \left(
		\frac{(q-1) u_2}{q^{\lambda_2}}
		x\right)
	}{
		\theta_q \left(
		\frac{(q-1) u_2}{q^{\alpha_2}}
		x \right)
	}
	x^{\lambda_1 - \beta_1}
	,\\
	C_{12} & =
	\frac{\varphi_{12}}{[\lambda_2]_q - \varphi_{11}} \cdot
	\frac{
		(q^{\alpha_1 - \lambda_1 + 1} ;q)_\infty
		(q^{\alpha_2 - \lambda_1 + 1} \frac{u_1}{u_2} ;q)_\infty
	}{
		(q^{\lambda_2 - \lambda_1 + 1} ;q)_\infty
		(q \frac{u_1}{u_2} ;q)_\infty
	} \cdot
	\frac{
		\theta_q \left(
		\frac{(q-1) u_1}{q^{\beta_1}}
		x \right)
	}{
		\theta_q \left(
		\frac{(q-1) u_1}{q^{\alpha_1}}
		x \right)
	}
	\frac{
		\theta_q \left(
		\frac{(q-1) u_2}{q^{\lambda_1}}
		x \right)
	}{
		\theta_q \left(
		\frac{(q-1) u_2}{q^{\alpha_2}}
		x \right)
	}
	x^{\lambda_2 - \beta_1}
	,\\
	C_{21} & =
	\frac{
		(q^{\alpha_2 - \lambda_2 + 1} ;q)_\infty
		(q^{\alpha_1 - \lambda_2 + 1} \frac{u_2}{u_1} ;q)_\infty
	}{
		(q^{\lambda_1 - \lambda_2 + 1} ;q)_\infty
		(q \frac{u_2}{u_1} ;q)_\infty
	} \cdot
	\frac{
		\theta_q \left(
		\frac{(q-1) u_2}{q^{\beta_2}}
		x \right)
	}{
		\theta_q \left(
		\frac{(q-1) u_2}{q^{\alpha_2}}
		x \right)
	}
	\frac{
		\theta_q \left(
		\frac{(q-1) u_1}{q^{\lambda_2}}
		x \right)
	}{
		\theta_q \left(
		\frac{(q-1) u_1}{q^{\alpha_1}}
		x \right)
	}
	x^{\lambda_1 - \beta_2}
	,\\
\label{qC22}
	C_{22} & =
	\frac{
		(q^{\alpha_2 - \lambda_1 + 1} ;q)_\infty
		(q^{\alpha_1 - \lambda_1 + 1} \frac{u_2}{u_1} ;q)_\infty
	}{
		(q^{\lambda_2 - \lambda_1 + 1} ;q)_\infty
		(q \frac{u_2}{u_1} ;q)_\infty
	} \cdot
	\frac{
		\theta_q \left(
		\frac{(q-1) u_2}{q^{\beta_2}}
		x \right)
	}{
		\theta_q \left(
		\frac{(q-1) u_2}{q^{\alpha_2}}
		x \right)
	}
	\frac{
		\theta_q \left(
		\frac{(q-1) u_1}{q^{\lambda_1}}
		x \right)
	}{
		\theta_q \left(
		\frac{(q-1) u_1}{q^{\alpha_1}}
		x \right)
	}
	x^{\lambda_2 - \beta_2}.
\end{align}
\end{subequations}
The inverse is explicitly given by
\begin{subequations}
\begin{align}
\label{qCinv11}
(C^{-1})_{11} & =
	\frac{\varphi_{21}}{[\lambda_1]_q - \varphi_{22}} \cdot
	\frac{
		(q^{\alpha_2 - \lambda_1} ;q)_\infty
		(q^{\alpha_1 - \lambda_1} \frac{u_2}{u_1} ;q)_\infty
	}{
		(q^{\lambda_2 - \lambda_1} ;q)_\infty
		(\frac{u_2}{u_1} ;q)_\infty
	} \cdot
	\frac{
		\theta_q \left(
		\frac{(q-1) u_1}{q^{\lambda_1}}
		x\right)
	}{
		\theta_q \left(
		\frac{(q-1) u_1}{q^{\beta_1}}
		x \right)
	}
	x^{\beta_1 - \lambda_1}
	,\\
(C^{-1})_{21} & =
	\frac{\varphi_{21}}{[\lambda_2]_q - \varphi_{22}} \cdot
	\frac{
		(q^{\alpha_2 - \lambda_2} ;q)_\infty
		(q^{\alpha_1 - \lambda_2} \frac{u_2}{u_1} ;q)_\infty
	}{
		(q^{\lambda_1 - \lambda_2} ;q)_\infty
		(\frac{u_2}{u_1} ;q)_\infty
	} \cdot
	\frac{
		\theta_q \left(
		\frac{(q-1) u_1}{q^{\lambda_2}}
		x \right)
	}{
		\theta_q \left(
		\frac{(q-1) u_1}{q^{\beta_1}}
		x \right)
	}x^{\beta_1 - \lambda_2}
	,\\
(C^{-1})_{12} & =
	\frac{
		(q^{\alpha_1 - \lambda_1} ;q)_\infty
		(q^{\alpha_2 - \lambda_1} \frac{u_1}{u_2} ;q)_\infty
	}{
		(q^{\lambda_2 - \lambda_1} ;q)_\infty
		(\frac{u_1}{u_2} ;q)_\infty
	} \cdot
	\frac{
		\theta_q \left(
		\frac{(q-1) u_2}{q^{\lambda_1}}
		x \right)
	}{
		\theta_q \left(
		\frac{(q-1) u_2}{q^{\beta_2}}
		x \right)
	}x^{\beta_2 - \lambda_1}
	,\\
\label{qCinv22}
(C^{-1})_{22} & =
	\frac{
		(q^{\alpha_1 - \lambda_2} ;q)_\infty
		(q^{\alpha_2 - \lambda_2} \frac{u_1}{u_2} ;q)_\infty
	}{
		(q^{\lambda_1 - \lambda_2} ;q)_\infty
		(\frac{u_1}{u_2} ;q)_\infty
	} \cdot
	\frac{
		\theta_q \left(
		\frac{(q-1) u_2}{q^{\lambda_2}}
		x \right)
	}{
		\theta_q \left(
		\frac{(q-1) u_2}{q^{\beta_2}}
		x \right)
	}x^{\beta_2 - \lambda_2}.
\end{align}
\end{subequations}
\end{pro}

The case $|q|>1$ follows from the inverse-transpose relation
\eqref{Eq:CDuality} in Proposition~\ref{Pro:CTransformation},
which reduces it to the case $|q|<1$.

\section{$q$-Isomonodromic Deformations}
\label{Sect:qisoeq}

In this section, we first study $q$-isomonodromic deformations of
the $q$-monodromy system
$(U,\Phi;W;\boldsymbol{\alpha})$.
More precisely, as $U$ varies, our aim is to construct
a $q$-monodromy system
$(U,\Phi(u);W(u);\boldsymbol{\alpha})$
with fixed $q$-monodromy data
$(C(x),\Lambda;\boldsymbol{\alpha})$.
The resulting $q$-isomonodromy equations also define a broader
notion of deformation for which the associated $q$-monodromy data
need not be defined.
For the part of the discussion involving $q$-monodromy data,
we impose the following conditions:
\begin{itemize}
	\item $\Phi$ is $q$-nonresonant,
		and we fix a branch of $[\Phi]^q$ throughout this section;
	\item $u_i\notin u_jq^{\mathbb{Z}\setminus\{0\}}$
	for all $i\neq j$, and
	$u_1,\ldots,u_n$ are all nonzero.
\end{itemize}
The discussion following Proposition~\ref{Pro:qUni} shows that,
	once the ordering data of the formal exponents
	$\boldsymbol{\alpha}$ are fixed,
	such a deformation, if it exists, is unique.
	
In the proof of Proposition~\ref{Pro:qUni},
	equation~\eqref{qSpiraleq} is obtained by comparing
	the zeros and poles of the determinants of
	the central connection matrices.
This shows that the locations of the singularities of $C(x)$
	as a function of~$x$ depend on~$U$.
Consequently, a deformation preserving $C(x)$ cannot vary
	continuously with respect to~$u$.
The deformation parameter~$u$ is therefore naturally restricted to
	the product of the $q$-spirals through an initial point~$u^0$.
For this reason, throughout this section we assume that $\Phi(u)$
	is defined on the lattice
\[
	u_1^0q^{\mathbb{Z}}\times\cdots\times u_n^0q^{\mathbb{Z}},
\]
where
\[
	u^0=(u_1^0,\ldots,u_n^0)
\]
has nonzero components.
We denote by $T_k$ the $q$-shift in the $k$-th variable.
Its action on $\Phi$ is given by
\begin{align*}
	(T_k\Phi)(u_1,\ldots,u_n)
	&:=\Phi(u_1,\ldots,qu_k,\ldots,u_n).
\end{align*}
In Section~\ref{Sect:Asym}, we shall use prescribed asymptotic behavior
	in the deformation variables to extend these solutions
	to continuously varying~$u$.

We first define $q$-isomonodromic deformations by requiring the
$q$-monodromy data to remain invariant.

\begin{defi}\label{Def:qID}
We call the triple 
	$(\Phi(u;\boldsymbol{\beta}),W(u;\boldsymbol{\beta});
	\boldsymbol{\alpha})$
	a \textbf{$q$-isomonodromic deformation} with respect to
	the formal monodromy $\boldsymbol{\beta}$ if there exists
	a choice of branches
	$[\Phi(u;\boldsymbol{\beta})]^q$
	for which the $q$-monodromy data remain invariant as $u$ varies:
	\begin{itemize}		
		\item The central connection matrix $C(x)$ remains invariant.
		For every $u$ and $k=1,\ldots,n$, we have
		\begin{align}
			C(x;U,\Phi(u;\boldsymbol{\beta});
				\boldsymbol{\beta},W(u;\boldsymbol{\beta})) = 
			C(x;T_kU,T_k\Phi(u;\boldsymbol{\beta});
				\boldsymbol{\beta},T_kW(u;\boldsymbol{\beta}));
		\end{align}
		
		\item There exists a matrix $\Lambda$ independent of $u$
		such that
		$W(u;\boldsymbol{\beta})$ conjugates
		$[\Phi(u;\boldsymbol{\beta})]^q$ to $\Lambda$:
		\begin{align}\label{qEigen}
			W(u;\boldsymbol{\beta})^{-1}
			[\Phi(u;\boldsymbol{\beta})]^q
			W(u;\boldsymbol{\beta})  = \Lambda;
		\end{align}
		
		\item The spectrum of
		$U^{-1}(I+(q-1)\Phi(u;\boldsymbol{\beta}))$
		is determined by the fixed formal exponents
		$\boldsymbol{\alpha}$ for every $u$:
		\begin{align}
			\mathrm{Spec}(U^{-1}(I+(q-1)\Phi(u;\boldsymbol{\beta}))) & =
			\{
				u_k^{-1} q^{\alpha^{(k)}_i}:
				1\leqslant k \leqslant n,
				1\leqslant i \leqslant m_k
			\}.
		\end{align}		
	\end{itemize}
\end{defi}

The remainder of this section is organized as follows.
In Section~\ref{Sect:qIsoEquation},
	we derive the explicit form of the $q$-isomonodromy equations.
In Section~\ref{Sect:LocalQIso},
	we establish the local uniqueness of their solutions
	in a single deformation direction.
In Section~\ref{Sect:GlobalQIso},
	we establish the local existence of solutions
	in the multiplicity-one case by constructing them explicitly,
	prove the compatibility
	of the deformation in different directions,
	and obtain global solutions.
Finally, in Section~\ref{Sect:Eqiso22},
	we present an explicit example in the $2\times2$ case.


\subsection{$q$-Isomonodromy Equations}
\label{Sect:qIsoEquation}

In this subsection, we derive the equations governing the
	$q$-isomonodromic deformation associated with the shift
	$u_k\mapsto qu_k$.
The shifted $q$-monodromy system
	$(T_kU,T_k\Phi;T_kW;\boldsymbol{\alpha})$
	is required to have the same $q$-monodromy data with respect to
	the formal monodromy $\boldsymbol{\beta}$ as
	$(U,\Phi;W;\boldsymbol{\alpha})$.
The resulting equations are called the
	$q$-isomonodromy equations.

For simplicity, we use the following notation for
	the canonical fundamental solutions:
	\begin{alignat*}{2}
    	T_kY^{[0]}(x) & :=
    		Y^{[0]}(x;T_kU,T_k\Phi),&\quad
    	T_kY^{[\infty]}(x) & :=
    		Y^{[\infty]}(x;T_kU,T_k\Phi;\boldsymbol{\beta}),\\
    	T_kH^{[0]}(x) & :=
    		H^{[0]}(x;T_kU,T_k\Phi),&\quad
    	T_kH^{[\infty]}(x) & :=
    		H^{[\infty]}(x;T_kU,T_k\Phi;\boldsymbol{\beta}).
	\end{alignat*}
The following lemma gives the equations that must be satisfied
whenever such a deformation exists.
	
\begin{lem}\label{Lem:betaEq}
Let $(U,\Phi;W;\boldsymbol{\alpha})$ be a
	$q$-monodromy system.
Suppose that there exists another $q$-monodromy system
	$(T_kU,T_k\Phi;T_kW;\boldsymbol{\alpha})$
	such that
\begin{align}\label{Eq:Mdbeta}
	\mathrm{MD}_{\boldsymbol{\beta}}
	(U,\Phi;W;\boldsymbol{\alpha})
	&=
	\mathrm{MD}_{\boldsymbol{\beta}}
	(T_kU,T_k\Phi;T_kW;\boldsymbol{\alpha}),
\end{align}
and that the following regularized condition holds:
\begin{align}\label{Eq:TkPandP}
	\frac{
		\mathrm{Spec}\bigl(
			(T_kU)^{-1}(I+(q-1)T_k\Phi)
		\bigr)
	}{
		\mathrm{Spec}\bigl(
			U^{-1}(I+(q-1)\Phi)
		\bigr)
	}
	\cap q^{\mathbb{N}+1}
	&=
	\varnothing.
\end{align}
Then there exists a matrix $V_k$ such that
\begin{subequations}\label{qCompaWu}
\begin{align}\label{qCompa}
	\left(T_kU x+I+(q-1)
		q^{\mathrm{ad}(\boldsymbol{\beta}E_k)}T_k\Phi\right)
	\left(UE_kx +V_k\right)
	&=
	\left(UE_kqx + V_k\right)
	\left(Ux+I+(q-1)\Phi\right),\\
	\label{Eq:Wu}
	q^{\boldsymbol{\beta}E_k} T_kW
	&=
	V_k\cdot W.
\end{align}
\end{subequations}
Moreover, $V_k$ satisfies
\begin{subequations}\label{VkidDet}
\begin{align}
\label{Vk:Id}
	(V_k)_{\hat{k}\hat{k}}
	&=
	I,\\
\label{Det:beta}
	\det\left(E_kx+V_k\right)
	&=
	\prod_{i=1}^{m_k} (x+q^{\alpha_i^{(k)}}).
\end{align}
\end{subequations}
\end{lem}

\begin{prf}
It follows from the definition~\eqref{Def:Cbeta}
of the central connection matrix and~\eqref{Eq:Mdbeta} that
\begin{align}\label{Eq:Rxu}
	R(x) &:=
	T_kY^{[\infty]}(x) \cdot
	Y^{[\infty]}(x)^{-1} =
	\left(T_kY^{[0]}(x)\,T_kW\right) \cdot
	\left(Y^{[0]}(x)W\right)^{-1}.
\end{align}
Using \eqref{YBGtoYbeta} in Lemma~\ref{Lem:BG}
and \eqref{Eq:eq},
we can rewrite \eqref{Eq:Rxu} as
\begin{align}\label{Eq:BGRxu}
	R(x)
	&=
	Y^{[\infty]}_{\mathrm{BG}}(x;T_kU,T_k\Phi)
	\left(
		I-E_k+(q-1)q^{-\boldsymbol{\beta}}UE_k
	\right)
	Y^{[\infty]}_{\mathrm{BG}}(x;U,\Phi)^{-1}.
\end{align}
Applying Proposition~\ref{Pro:SolY}
to \eqref{Eq:Rxu} and~\eqref{Eq:BGRxu}
shows that the possible poles of $R(x)$ lie in
\begin{align*}
	\begin{cases}
		\displaystyle
		\frac{1}{1-q}
		\mathrm{Spec}\bigl(
			(T_kU)^{-1}(I+(q-1)T_k\Phi)
		\bigr)q^{-\mathbb{N}}
			\cap
		\displaystyle
		\frac{1}{1-q}
		\mathrm{Spec}\bigl(
			U^{-1}(I+(q-1)\Phi)
		\bigr)q^{\mathbb{N}+1},
		& |q|<1,\\[8pt]
		\displaystyle
		\frac{1}{1-q}
		\mathrm{Spec}\bigl(
			U^{-1}(I+(q-1)\Phi)
		\bigr)q^{\mathbb{N}+1}
			\cap
		\displaystyle
		\frac{1}{1-q}
		\mathrm{Spec}\bigl(
			(T_kU)^{-1}(I+(q-1)T_k\Phi)
		\bigr)q^{-\mathbb{N}},
		& |q|>1.
	\end{cases}
\end{align*}
The intersection in either case is empty
by~\eqref{Eq:TkPandP}.
Therefore, $R(x)$ has no poles in $\mathbb{C}^{\ast}$.
Propositions~\ref{Can:inf} and~\ref{Can:zero},
together with \eqref{Eq:Rxu} and \eqref{Eq:eq}, give
\begin{align*}
	R(x)
	&=
	q^{-\boldsymbol{\beta}E_k}
	\left((q-1)UE_kx+V_k\right)
	+O(x^{-1}),
	\qquad x\to\infty,\\
	R(x)
	&=
	T_kW\cdot W^{-1}+O(x),
	\qquad x\to0,
\end{align*}
where
\begin{align}\label{Eq:VkHinf}
	V_k
	&=
	q^{\boldsymbol{\beta}E_k}
	+(q-1)
	\left(
		(q^{\operatorname{ad}(\boldsymbol{\beta}E_k)}
		T_kH_1^{[\infty]})\cdot UE_k
		-
		UE_k\cdot H_1^{[\infty]}
	\right).
\end{align}
The pole comparison and these expansions show that
$R(x)$ is the rational function
\begin{align}\label{Eq:RPolynomial}
	R(x)
	&=
	q^{-\boldsymbol{\beta}E_k}
	\left((q-1)UE_kx+V_k\right).
\end{align}
Comparing \eqref{Eq:RPolynomial} with the expansion at $x=0$
proves~\eqref{Eq:Wu}.
Since
$\left(q^{\boldsymbol{\beta}E_k}\right)_{\hat{k}\hat{k}} = I$
and $UE_k$ is supported on the $k$-th block,
\eqref{Eq:VkHinf} also proves~\eqref{Vk:Id}.

Applying~\eqref{qLinear} to the two factors
in~\eqref{Eq:Rxu} gives
\begin{align*}
	R\left(\frac{qx}{q-1}\right)
	&=
	\left(T_kUx+I+(q-1)T_k\Phi\right)
	R\left(\frac{x}{q-1}\right)
	\left(Ux+I+(q-1)\Phi\right)^{-1}.
\end{align*}
Substituting~\eqref{Eq:RPolynomial} into this identity
gives~\eqref{qCompa}.
Finally, Propositions~\ref{Can:inf} and~\ref{Pro:SolY},
together with \eqref{Eq:eqisqP} and~\eqref{Eq:qP}, give
\begin{align*}
	\det R(x)
	&=
	q^{- \mathrm{tr} \beta_k}
	\prod_{i=1}^{m_k}
	\left(
		(q-1)u_kx+q^{\alpha_i^{(k)}}
	\right).
\end{align*}
Comparison with the determinant of~\eqref{Eq:RPolynomial},
followed by replacing $(q-1)u_kx$ with $x$,
gives~\eqref{Det:beta}.
This completes the proof.
\end{prf}

Conversely, the following lemma shows that any solution of
	the above equations preserves the $q$-monodromy data.

\begin{lem}\label{Lem:qBasic}
Let $(U,\Phi;W;\boldsymbol{\alpha})$ be a
	$q$-monodromy system.
If there exist matrices $V_k$, $T_k\Phi$, and $T_kW$
satisfying \eqref{qCompaWu} and~\eqref{VkidDet}, then
$(T_kU,T_k\Phi;T_kW;\boldsymbol{\alpha})$
is a $q$-monodromy system and
\begin{align}
	\mathrm{MD}_{\boldsymbol{\beta}}
	(U,\Phi;W;\boldsymbol{\alpha})
	&=
	\mathrm{MD}_{\boldsymbol{\beta}}
	(T_kU,T_k\Phi;T_kW;\boldsymbol{\alpha}).
\end{align}
Here the branch $[T_k\Phi]^q$ is chosen as
\begin{align}\label{Eq:TkBranch}
	[T_k\Phi]^q
	&:=
	T_kW \cdot W^{-1} [\Phi]^q W \cdot T_kW^{-1}.
\end{align}
\end{lem}

\begin{prf}
Write the $q$-monodromy data of the original system as
\begin{align*}
	\mathrm{MD}_{\boldsymbol{\beta}}
	(U,\Phi;W;\boldsymbol{\alpha})
	&=
	(C(x),\Lambda;\boldsymbol{\alpha}).
\end{align*}
Setting $x=0$ in~\eqref{Det:beta} shows that $V_k$
is invertible.
Equation~\eqref{Eq:Wu} then shows that $T_kW$ is invertible.
Comparing the constant terms in~\eqref{qCompa}
and using~\eqref{Eq:Wu}, we obtain
\begin{align*}
	(T_kW)^{-1}(T_k\Phi)(T_kW)
	&=
	W^{-1}\Phi W
	=
	[\Lambda]_q.
\end{align*}
In particular, $T_k\Phi$ is $q$-nonresonant.
We choose the branch $[T_k\Phi]^q$ as in~\eqref{Eq:TkBranch}.
It then satisfies
\begin{align*}
	(T_kW)^{-1}[T_k\Phi]^q(T_kW)
	&=
	W^{-1}[\Phi]^qW
	=
	\Lambda.
\end{align*}
Denote
\begin{align*}
	T_kC(x)
	&:=
	\left(T_kY^{[\infty]}(x)\right)^{-1}
	T_kY^{[0]}(x)T_kW.
\end{align*}
It remains to prove that
$T_kC(x)=C(x)$ and that $\boldsymbol{\alpha}$
satisfies the spectral condition for
$(T_kU,T_k\Phi;T_kW;\boldsymbol{\alpha})$
to be a $q$-monodromy system.

Denote
\begin{align*}
	R(x)
	&:=
	q^{-\boldsymbol{\beta}E_k}
	\left((q-1)UE_kx+V_k\right).
\end{align*}
By~\eqref{Eq:Wu}, equation~\eqref{Eq:TkBranch} gives
\begin{align*}
	[T_k\Phi]^q
	&=
	R(0)[\Phi]^qR(0)^{-1}.
\end{align*}
It follows from~\eqref{qCompa} that
$R(x)Y^{[0]}(x)R(0)^{-1}$ and 
$R(x)Y^{[\infty]}(x)$
are fundamental solutions of the linear $q$-difference system
determined by $(T_kU,T_k\Phi)$ at $x=0$ and $x=\infty$,
respectively.
Using the preceding branch relation and~\eqref{Eq:eq},
we write these solutions as
\begin{alignat}{2}
\nonumber
	\widetilde{Y}^{[0]}(x)
	&:=
	\widetilde{H}^{[0]}(x)x^{[T_k\Phi]^q},
	&\qquad
	\widetilde{Y}^{[\infty]}(x)
	&:=
	\widetilde{H}^{[\infty]}(x)
	x^{\boldsymbol{\beta}}
	e_q(q^{-\boldsymbol{\beta}}T_kUx),\\
\label{Eq:Htilde}
	\widetilde{H}^{[0]}(x)
	&:=
	R(x)H^{[0]}(x)R(0)^{-1},
	&
	\widetilde{H}^{[\infty]}(x)
	&:=
	R(x)H^{[\infty]}(x)
	\left(
		I+(q-1)q^{-\boldsymbol{\beta}}UE_kx
	\right)^{-1}.
\end{alignat}
By Proposition~\ref{Can:zero} and~\eqref{Eq:Htilde},
\[
	\widetilde{H}^{[0]}(0)=I.
\]
Comparing the $(k,\hat{k})$-block of
the recurrence~\eqref{Eq:HinfRec} at $\nu=0$
in the proof of Proposition~\ref{Can:inf}
with the corresponding block of the coefficient of $x$
in~\eqref{qCompa} gives
\[
	(V_k)_{k\hat{k}}
	=
	(1-q)u_k(H_1^{[\infty]})_{k\hat{k}}.
\]
Together with~\eqref{Vk:Id}, this identity
and~\eqref{Eq:Htilde} give
\[
	\widetilde{H}^{[\infty]}(\infty)=I.
\]
The uniqueness in Propositions~\ref{Can:zero}
and~\ref{Can:inf} therefore gives
\begin{alignat*}{2}
	T_kY^{[0]}(x)
	&=
	\widetilde{Y}^{[0]}(x)
	=
	R(x)Y^{[0]}(x)R(0)^{-1},
	&\qquad
	T_kY^{[\infty]}(x)
	&=
	\widetilde{Y}^{[\infty]}(x)
	=
	R(x)Y^{[\infty]}(x).
\end{alignat*}
Consequently,
\begin{align*}
	T_kC(x)
	&=
	\left(R(x)Y^{[\infty]}(x)\right)^{-1}
	R(x)Y^{[0]}(x)R(0)^{-1}T_kW\\
	&=
	Y^{[\infty]}(x)^{-1}Y^{[0]}(x)W
	=
	C(x).
\end{align*}		
Taking determinants in~\eqref{qCompa}
and using~\eqref{Det:beta} gives the following identity
between multisets:
	\begin{align*}
	\mathrm{Spec}( (T_kU)^{-1} (I+(q-1)T_k\Phi) ) \cup
	\left\{ 
		u_k^{-1} q^{\alpha^{(k)}_i}
	\right\} =
	\left\{
		(qu_k)^{-1} q^{\alpha^{(k)}_i}
	\right\}\cup
	\mathrm{Spec}( U^{-1} (I+(q-1)\Phi) ).
	\end{align*}
Hence, $\boldsymbol{\alpha}$ are formal exponents of
$(T_kU,T_k\Phi)$.
Together with $T_kC(x)=C(x)$,
this proves the result.
\end{prf}

Thus, a single-step $q$-isomonodromic deformation in the
$u_k$-direction is equivalently characterized by the following
equations.

\begin{defi}\label{Def:qIsoEq}
The following equations for $(T_k\Phi,T_kW;V_k)$
are called the \textbf{$q$-isomonodromy equations}
in the $u_k$-direction with respect to the formal monodromy
$\boldsymbol{\beta}$:
\begin{subequations}\label{qisouk}
\begin{align}
\label{qisoeq}
	\left(
		T_kUx+I+(q-1)
		q^{\operatorname{ad}(\boldsymbol{\beta}E_k)}
		T_k\Phi
	\right)
	\left(UE_kx+V_k\right)
	&=
	\left(UE_kqx+V_k\right)
	\left(Ux+I+(q-1)\Phi\right),\\
\label{qisoeqW}
	q^{\boldsymbol{\beta}E_k}T_kW
	&=
	V_k\cdot W.
\end{align}
Here $V_k$ is required to satisfy the
\textbf{normalization condition}
\begin{align}\label{Cond:NorVk}
	(V_k)_{\hat{k}\hat{k}}
	&=
	I.
\end{align}
\end{subequations}
\end{defi}

\begin{thm}\label{Thm:qIsoMD}
Let $(U,\Phi;W;\boldsymbol{\alpha})$ be a
$q$-monodromy system.
Suppose that $T_k\Phi$ satisfies the following
\textbf{regularized condition}:
\begin{align}\label{Cond:RegTkPhi}
	\frac{
		\mathrm{Spec}\bigl(
			(T_kU)^{-1}(I+(q-1)T_k\Phi)
		\bigr)
	}{
		\mathrm{Spec}\bigl(
			U^{-1}(I+(q-1)\Phi)
		\bigr)
	}
	\cap q^{\mathbb{N}+1}
	&=
	\varnothing.
\end{align}
Then
$(T_kU,T_k\Phi;T_kW;\boldsymbol{\alpha})$
is a $q$-monodromy system with the same $q$-monodromy data
if and only if there exists a matrix $V_k$ satisfying the
\textbf{spectral condition} determined by
$\boldsymbol{\alpha}$,
\begin{align}\label{Cond:SpecVk}
	\det(E_kx+V_k)
	&=
	\prod_{i=1}^{m_k}
	\left(x+q^{\alpha_i^{(k)}}\right),
\end{align}
such that
$(T_k\Phi,T_kW;V_k)$ satisfies the
$q$-isomonodromy equations~\eqref{qisouk}
in the $u_k$-direction.
Here the branch $[T_k\Phi]^q$ is chosen as
\begin{align}\label{Eq:TkBranchThm}
	[T_k\Phi]^q
	&:=
	T_kW\cdot W^{-1}[\Phi]^qW\cdot T_kW^{-1}.
\end{align}
\end{thm}

\begin{prf}
This follows directly from
Lemmas~\ref{Lem:betaEq} and~\ref{Lem:qBasic}.
\end{prf}

The regularized condition~\eqref{Cond:RegTkPhi} imposed on $T_k\Phi$ is essential here.
For the local uniqueness of solutions to the
$q$-isomonodromy equations in the $u_k$-direction,
we impose the following separation conditions
on the formal exponents:
\begin{subequations}
\begin{align}
	\mathrm{Spec}(u_k^{-1}q^{\alpha_k})
	\cap
	\left(\bigcup_{s\neq k}
	\mathrm{Spec}(u_s^{-1}q^{\alpha_s})
	\right)
	&=
	\varnothing,\\
	\mathrm{Spec}(q^{\alpha_k})
	\cap
	\mathrm{Spec}(q^{-1}q^{\alpha_k})
	&=
	\varnothing.
\end{align}
\end{subequations}
Under the regularized condition~\eqref{Cond:RegTkPhi},
these conditions are equivalent to
\begin{align*}
	\frac{
		\mathrm{Spec}\bigl(
			(T_kU)^{-1}(I+(q-1)T_k\Phi)
		\bigr)
	}{
		\mathrm{Spec}\bigl(
			(T_kU)^{-1}(I+(q-1)T_k\Phi)
		\bigr)
	}
	\cap q^{\mathbb{N}+1}
	&=
	\varnothing.
\end{align*}
By Proposition~\ref{Pro:qUni}, this is precisely the genericity
condition that guarantees the uniqueness of $T_k\Phi$ and $T_kW$.

\subsection{Local Uniqueness}
\label{Sect:LocalQIso}

In this subsection, we establish a local uniqueness result
for solutions of
the $q$-isomonodromy equations introduced in
Definition~\ref{Def:qIsoEq}.
The local existence problem will be treated in the next subsection.
Related $q$-isomonodromy equations were derived by
Kakei and Kikuchi from the multi-component $q$-KP hierarchy
in~\cite{kakei_q-analogue_2006}.
A related algebraic factorization
is discussed in~\cite{MYMthesis}.

Equation~\eqref{qisoeq} leads to the following
linear-algebraic problem:
\begin{align*}
	(xI+A_2)(xE_{22}+G_1)
	&=
	(xE_{22}+G_2)(xI+A_1),
\end{align*}
where
$E_{22}=\left(\begin{smallmatrix}0&0\\0&I\end{smallmatrix}\right)$.
For a given~$A_1$, the problem is to determine
$G_1$, $G_2$, and~$A_2$ under suitable conditions.
We begin with two linear-algebraic lemmas.

\begin{lem}\label{Lem:UniLAG}
Fix a $2\times2$ block matrix $A$ and
a submultiset $S$ of $\mathrm{Spec}(A)$ satisfying
\begin{align}\label{Cond:Sdisj}
	S\cap\bigl(\mathrm{Spec}(A)\setminus S\bigr)
	&=
	\varnothing.
\end{align}
Suppose that there exists an invertible $2\times2$ block matrix~$G$
such that $G_{11}$ is invertible and
\begin{align}
	GAG^{-1} =
	\begin{pmatrix}
		\Lambda_1 & 0\\
		J_{21} & \Lambda_2
	\end{pmatrix},\qquad
	\mathrm{Spec}(\Lambda_2)
	&=
	S.
\end{align}
Then $(G^{-1})_{22}$ is invertible, and the matrices
\begin{align}
	G_{11}^{-1}\Lambda_1G_{11},\qquad
	(G^{-1})_{22}\Lambda_2(G^{-1})_{22}^{-1},\qquad
	G_{11}^{-1}G_{12}
	=
	-(G^{-1})_{12}(G^{-1})_{22}^{-1}
\end{align}
are independent of the choice of~$G$.
In particular, they are uniquely determined by $A$ and~$S$.
\end{lem}

\begin{prf}
Since $G$ and $G_{11}$ are invertible, the Schur complement formula
shows that $(G^{-1})_{22}$ is invertible.
Let $\widehat G$ be another matrix satisfying the assumptions, and write
\begin{align*}
	\widehat G A\widehat G^{-1} =
	\begin{pmatrix}
		\widehat\Lambda_1 & 0\\
		\widehat J_{21} & \widehat\Lambda_2
	\end{pmatrix},\qquad
	\mathrm{Spec}(\widehat\Lambda_2) =
	S.
\end{align*}
Set $H=\widehat GG^{-1}$.
Comparing the $(1,2)$-blocks of
\begin{align*}
	H
	\begin{pmatrix}
		\Lambda_1 & 0\\
		J_{21} & \Lambda_2
	\end{pmatrix}
	&=
	\begin{pmatrix}
		\widehat\Lambda_1 & 0\\
		\widehat J_{21} & \widehat\Lambda_2
	\end{pmatrix}
	H
\end{align*}
gives $H_{12}\Lambda_2 =
\widehat\Lambda_1H_{12}$.
Since
\begin{align*}
	\mathrm{Spec}(\widehat\Lambda_1) =
	\mathrm{Spec}(A)\setminus S,\qquad
	\mathrm{Spec}(\Lambda_2)
	&=
	S,
\end{align*}
condition~\eqref{Cond:Sdisj}
implies $H_{12}=0$.
Thus $H$ is block lower triangular.
Since $H$ is invertible, both $H_{11}$ and $H_{22}$ are invertible.
It follows that
\begin{alignat*}{2}
	\widehat G_{11}
	&=
	H_{11}G_{11},
	&\qquad \widehat\Lambda_1
	&=
	H_{11}\Lambda_1H_{11}^{-1},\\
	(\widehat G^{-1})_{22}
	&=
	(G^{-1})_{22}H_{22}^{-1},
	&\qquad \widehat\Lambda_2
	&=
	H_{22}\Lambda_2H_{22}^{-1},\\
	\widehat G_{12}
	&=
	H_{11}G_{12}.
\end{alignat*}
These identities show that
\begin{align*}
	\widehat G_{11}^{-1}\widehat\Lambda_1\widehat G_{11}
	&=
	G_{11}^{-1}\Lambda_1G_{11},\\
	(\widehat G^{-1})_{22}\widehat\Lambda_2
	(\widehat G^{-1})_{22}^{-1}
	&=
	(G^{-1})_{22}\Lambda_2(G^{-1})_{22}^{-1},\\
	\widehat G_{11}^{-1}\widehat G_{12}
	&=
	G_{11}^{-1}G_{12}.
\end{align*}
Finally, the $(1,2)$-block of $GG^{-1}=I$ gives
\begin{align*}
	G_{11}^{-1}G_{12}
	&=
	-(G^{-1})_{12}(G^{-1})_{22}^{-1}.
\end{align*}
This proves the result.
\end{prf}

\begin{lem}\label{Lem:EquiLAG}
Let $A_1$, $A_2$, $G_1$, and $G_2$ be $2\times2$ block matrices,
with $G_1$, $G_2$, $(G_1)_{11}$, and $(G_2)_{11}$ invertible.
Set
$E_{22}=\left(\begin{smallmatrix}0&0\\0&I\end{smallmatrix}\right)$.
Assume that
\begin{alignat}{2}\label{DiagA1A2}
	G_1A_1G_1^{-1}
	&=
	\begin{pmatrix}
		(J_1)_{11} & 0\\
		(J_1)_{21} & (G_1^{-1})_{22}^{-1}
	\end{pmatrix},
	&\qquad
	G_2^{-1}A_2G_2
	&=
	\begin{pmatrix}
		(J_2)_{11} & (J_2)_{12}\\
		0 & (G_2^{-1})_{22}^{-1}
	\end{pmatrix},
\end{alignat}
and
\begin{subequations}\label{Cond:G}
\begin{align}
\label{Cond:G11}
	G_{11}
	:=
	(G_1)_{11}
	&=
	(G_2)_{11},\\
\label{Cond:G11Prop}
	(J_1)_{11}
	&=
	G_{11}(J_2)_{11}G_{11}^{-1},\\
\label{Cond:G21Exp}
	(G_1)_{21}
	&=
	(A_1)_{21}+(G_2)_{21},\\
\label{Cond:G12Exp}
	(G_2)_{12}
	&=
	(A_2)_{12}+(G_1)_{12}.
\end{align}
\end{subequations}
Then
\begin{align}\label{Eq:BasicLem}
	(xI+A_2)(xE_{22}+G_1)
	&=
	(xE_{22}+G_2)(xI+A_1).
\end{align}
Conversely, if
\begin{align}\label{Cond:G1G2disj}
	\mathrm{Spec}((G_1^{-1})_{22})
	\cap
	\mathrm{Spec}((G_2^{-1})_{22})
	&=
	\varnothing,
\end{align}
then equation~\eqref{Eq:BasicLem} implies
\eqref{DiagA1A2} and~\eqref{Cond:G}.
Moreover,
\begin{subequations}\label{Exp:J11}
\begin{align}
	(J_1)_{11}
	&=
	G_{11}
	\left(
		(A_1)_{11}
		+
		G_{11}^{-1}(G_1)_{12}(A_1)_{21}
	\right)
	G_{11}^{-1},\\
	(J_2)_{11}
	&=
	G_{11}^{-1}
	\left(
		(A_2)_{11}
		+
		(A_2)_{12}(G_2)_{21}G_{11}^{-1}
	\right)
	G_{11}.
\end{align}
\end{subequations}
\end{lem}

\begin{prf}
The forward implication follows by direct block multiplication using
\eqref{DiagA1A2} and~\eqref{Cond:G}.
For the converse, the block matrix inverse formula shows that
$(G_i^{-1})_{22}$ is invertible for $i=1,2$.
By comparing the linear and constant terms in
\eqref{Eq:BasicLem}, we can rewrite it equivalently as
\begin{align*}
	A_2E_{22}+G_1
	=
	E_{22}A_1+G_2,\qquad
	G_2^{-1}A_2
	=
	A_1G_1^{-1}.
\end{align*}
Hence, we obtain another equivalent form
\begin{subequations}
\begin{align}\label{Lem:BStep1}
	(A_1G_1^{-1}E_{22}-I)G_1^{-1}
	&=
	G_2^{-1}(E_{22}A_1G_1^{-1}-I),\\
\label{Lem:BStep2}
	G_1A_1G_1^{-1}G_2
	&=
	G_1G_2^{-1}A_2G_2.
\end{align}
\end{subequations}
Comparing the $(2,2)$-blocks of \eqref{Lem:BStep1} gives
\begin{align*}
	\bigl((A_1G_1^{-1})_{22}-I\bigr)(G_1^{-1})_{22}
	&=
	(G_2^{-1})_{22}
	\bigl((A_1G_1^{-1})_{22}-I\bigr).
\end{align*}
Condition~\eqref{Cond:G1G2disj} shows that this identity is
equivalent to $(A_1G_1^{-1})_{22} = I$.
Comparing the remaining blocks of \eqref{Lem:BStep1} and applying
the block matrix inverse formula shows that
\eqref{Lem:BStep1} and~\eqref{Lem:BStep2} are equivalent to
\eqref{Cond:G11} and
\begin{align}\label{Exp:GA2AG1}
	G_2^{-1}A_2
	&=
	A_1G_1^{-1}
	=
	\begin{pmatrix}
		\ast & -G_{11}^{-1}(G_1)_{12}\\
		-(G_2)_{21}G_{11}^{-1} & I
	\end{pmatrix}.
\end{align}
Using the block matrix inverse formula, the identities for the
$(1,2)$- and $(2,2)$-blocks of $A_1G_1^{-1}$ in
\eqref{Exp:GA2AG1} are equivalent to the first identity in
\eqref{DiagA1A2}.
Similarly, the identities for the $(2,1)$- and $(2,2)$-blocks of
$G_2^{-1}A_2$ are equivalent to the second identity in
\eqref{DiagA1A2}.

Under \eqref{DiagA1A2} and \eqref{Cond:G11}, the identity for the
$(2,1)$-block of $A_1G_1^{-1}$ in \eqref{Exp:GA2AG1} is equivalent
to \eqref{Cond:G21Exp}.
Indeed,
\begin{align*}
	(G_1^{-1})_{11}
	&=
	G_{11}^{-1}
	+
	(G_1^{-1})_{12}(G_1^{-1})_{22}^{-1}
	(G_1^{-1})_{21},\\
	-(G_2)_{21}G_{11}^{-1}
	&=
	(A_1)_{21}(G_1^{-1})_{11}
	+
	(A_1)_{22}(G_1^{-1})_{21}\\
	&=
	(A_1)_{21}G_{11}^{-1}
	+
	\left(
		(A_1)_{21}(G_1^{-1})_{12}
		(G_1^{-1})_{22}^{-1}
		+
		(A_1)_{22}
	\right)
	(G_1^{-1})_{21}\\
	&=
	(A_1)_{21}G_{11}^{-1}
	+
	(G_1^{-1})_{22}^{-1}(G_1^{-1})_{21}\\
	&=
	(A_1)_{21}G_{11}^{-1}
	-
	(G_1)_{21}G_{11}^{-1}.
\end{align*}
Similarly, the identity for the $(1,2)$-block of
$G_2^{-1}A_2$ in \eqref{Exp:GA2AG1} is equivalent to
\eqref{Cond:G12Exp}.

Next, we note that
\begin{align*}
	(J_1)_{11}
	&=
	G_{11}(A_1G_1^{-1})_{11}
	-
	(G_1)_{12}(G_2)_{21}G_{11}^{-1},\\
	(J_2)_{11}
	&=
	(G_2^{-1}A_2)_{11}G_{11}
	-
	G_{11}^{-1}(G_1)_{12}(G_2)_{21}.
\end{align*}
Thus, the identity for the $(1,1)$-block in
\eqref{Exp:GA2AG1} is equivalent to \eqref{Cond:G11Prop}.
Finally, using \eqref{DiagA1A2}, we compare the $(1,1)$-blocks in
\begin{align*}
	G_1A_1
	=
	(G_1A_1G_1^{-1})G_1,\qquad
	A_2G_2
	=
	G_2(G_2^{-1}A_2G_2).
\end{align*}
This gives \eqref{Exp:J11} and completes the proof.
\end{prf}

The $q$-isomonodromy equation~\eqref{qisoeq}
can be reduced to~\eqref{Eq:BasicLem} in
Lemma~\ref{Lem:EquiLAG}.
Together with Lemma~\ref{Lem:UniLAG}, this gives the following
local uniqueness result.

\begin{cor}\label{Cor:UniVk}
Let $(U,\Phi;W;\boldsymbol{\alpha})$ be a
$q$-monodromy system.
Suppose that its formal exponents
$\boldsymbol{\alpha}$ satisfy the following
\textbf{separation conditions}:
\begin{subequations}\label{Cond:Sep}
\begin{align}
\label{Cond:SpecDisj}
	\mathrm{Spec}(u_k^{-1}q^{\alpha_k})
	\cap
	\left(
	\bigcup_{s\neq k}
	\mathrm{Spec}(u_s^{-1}q^{\alpha_s})
	\right)
	&=
	\varnothing,\\
\label{Cond:SpecnoShift1}
	\mathrm{Spec}(q^{\alpha_k})
	\cap
	\mathrm{Spec}(q^{-1}q^{\alpha_k})
	&=
	\varnothing.
\end{align}
\end{subequations}
Then there is at most one triple
$(T_k\Phi,T_kW;V_k)$ satisfying the
$q$-isomonodromy equations~\eqref{qisouk}
in the $u_k$-direction, with $V_k$ satisfying the
spectral condition~\eqref{Cond:SpecVk}
determined by $\boldsymbol{\alpha}$.
Moreover, for any such triple,
\begin{align}\label{Eq:Vkkhatk}
	(V_k)_{k\hat{k}}
	&=
	\Phi_{k\hat{k}}
	\frac{(q-1)u_k}{u_kI-q^{-1}U_{\hat{k}\hat{k}}}.
\end{align}
\end{cor}

\begin{prf}
Without loss of generality, let $k=n$.
With respect to the block decomposition $(\hat n,n)$, set
\begin{subequations}\label{Eq:BlockReduction}
\begin{align}
	A_1
	&=
	U^{-1}\bigl(I+(q-1)\Phi\bigr),\\
	G_1
	&=
	\begin{pmatrix}
		I & 0\\
		0 & u_n^{-1}I
	\end{pmatrix}
	V_n,\\
	G_2
	&=
	q^{-E_n}U^{-1}G_1U,\\
	A_2
	&=
	\begin{pmatrix}
		I & 0\\
		0 & u_n^{-1}I
	\end{pmatrix}
	(T_nU)^{-1}
	\left(
		I+(q-1)
		q^{\operatorname{ad}(\boldsymbol{\beta}E_n)}
		T_n\Phi
	\right)
	\begin{pmatrix}
		I & 0\\
		0 & u_nI
	\end{pmatrix}.
\end{align}
\end{subequations}
A direct calculation transforms
the $q$-isomonodromy equation~\eqref{qisoeq}
into~\eqref{Eq:BasicLem}.

Setting $x=0$ in the spectral condition~\eqref{Cond:SpecVk}
shows that $V_n$ is invertible.
Hence $G_1$ and~$G_2$ are also invertible.
The normalization condition~\eqref{Cond:NorVk} gives
\begin{align*}
	(G_1)_{11}
	&=
	(G_2)_{11}
	=
	I.
\end{align*}
The block matrix inverse formula and
the spectral condition~\eqref{Cond:SpecVk} give
\begin{align*}
	\mathrm{Spec}\bigl((G_1^{-1})_{22}^{-1}\bigr)
	&=
	\mathrm{Spec}\left(
		u_n^{-1}
		\left(
			(V_n)_{nn}
			-
			(V_n)_{n\hat n}(V_n)_{\hat n n}
		\right)
	\right)=
	\mathrm{Spec}(u_n^{-1}q^{\alpha_n}).
\end{align*}
By the definition of the formal exponents,
condition~\eqref{Cond:SpecDisj} is precisely
condition~\eqref{Cond:Sdisj} for
$S =
\mathrm{Spec}(u_n^{-1}q^{\alpha_n})$.
Moreover,
\begin{align*}
	(G_2^{-1})_{22} = q(G_1^{-1})_{22}.
\end{align*}
Hence condition~\eqref{Cond:SpecnoShift1} is precisely
condition~\eqref{Cond:G1G2disj}.

Lemma~\ref{Lem:EquiLAG} now gives
\eqref{DiagA1A2} and
\eqref{Cond:G11}--\eqref{Cond:G12Exp}.
Applying Lemma~\ref{Lem:UniLAG} to $A_1$, $G_1$, and~$S$,
and using the definition of~$G_1$ and~\eqref{DiagA1A2},
we find that
\begin{align*}
	(G_1)_{11}^{-1}(G_1)_{12} &=
	(V_n)_{\hat n n},\\
	(G_1^{-1})_{22} \cdot
	(G_1^{-1})_{22}^{-1} \cdot
	(G_1^{-1})_{22}^{-1} &=
	(G_1^{-1})_{22}^{-1}
	=
	u_n^{-1}
	\left(
		(V_n)_{nn}
		-
		(V_n)_{n\hat n}(V_n)_{\hat n n}
	\right)
\end{align*}
are uniquely determined.
Under the above identification,
equation~\eqref{Cond:G21Exp} becomes
\begin{align*}
	u_n^{-1}(V_n)_{n\hat n}
	&=
	u_n^{-1}(q-1)\Phi_{n\hat n}
	+
	q^{-1}u_n^{-2}
	(V_n)_{n\hat n}U_{\hat n\hat n}.
\end{align*}
Therefore,
\begin{align*}
	(V_n)_{n\hat n}
	&=
	\Phi_{n\hat n}
	\frac{(q-1)u_n}{u_nI-q^{-1}U_{\hat n\hat n}}.
\end{align*}
This proves~\eqref{Eq:Vkkhatk}.
Finally,
\begin{align*}
	(V_n)_{nn}
	&=
	u_n(G_1^{-1})_{22}^{-1}
	+
	(V_n)_{n\hat n}(V_n)_{\hat n n},
\end{align*}
so every block of~$V_n$ is uniquely determined.
Setting $x=0$ in~\eqref{qisoeq} gives
\begin{align*}
	q^{\operatorname{ad}(\boldsymbol{\beta}E_n)}
	T_n\Phi
	&=
	V_n\Phi V_n^{-1},
\end{align*}
so $T_n\Phi$ is uniquely determined by~$V_n$.
Equation~\eqref{qisoeqW} then uniquely determines~$T_nW$.
This completes the proof.
\end{prf}

Corollary~\ref{Cor:UniVk} gives local uniqueness under
	the separation conditions~\eqref{Cond:Sep}, without assuming
	the regularized condition~\eqref{Cond:RegTkPhi}.
When the regularized condition fails,
	Proposition~\ref{Pro:qUni} no longer guarantees that the full
	$q$-monodromy data determine $T_k\Phi$ and $T_kW$ uniquely.
Corollary~\ref{Cor:UniVk} nevertheless gives local uniqueness
	for the $q$-isomonodromy equations under the separation
	conditions~\eqref{Cond:Sep}.
Thus, the $q$-isomonodromy equations may select
	a deformation that is not determined by
	the $q$-monodromy data alone.
Moreover, these equations remain meaningful even when
	the corresponding $q$-monodromy data are not defined.

This motivates the following generalized notion of
$q$-isomonodromic deformation, defined by the
$q$-isomonodromy equations alone without assuming that the
corresponding $q$-monodromy data are defined.
Here, $\Phi$ need not be $q$-nonresonant, and no branch of
$[\Phi]^q$ needs to be chosen.

\begin{defi}\label{Def:GenqID}
We call the triple
$(\Phi(u;\boldsymbol{\beta}),W(u;\boldsymbol{\beta});
\boldsymbol{\alpha})$
a \textbf{generalized $q$-isomonodromic deformation} with respect to
the formal monodromy $\boldsymbol{\beta}$ if, for every~$u$, the
identity
\begin{align}
	\mathrm{Spec}
	\bigl(
		U^{-1}(I+(q-1)\Phi(u;\boldsymbol{\beta}))
	\bigr)
	&=
	\{
	u_k^{-1}q^{\alpha_i^{(k)}}:
	1\leqslant k\leqslant n,\,
	1\leqslant i\leqslant m_k
	\}
\end{align}
holds and, for each $k=1,\ldots,n$, there exists a matrix~$V_k(u)$
satisfying the spectral condition~\eqref{Cond:SpecVk}
determined by $\boldsymbol{\alpha}$ such that
$(T_k\Phi(u;\boldsymbol{\beta}),
T_kW(u;\boldsymbol{\beta});V_k(u))$
is a solution of the $q$-isomonodromy equations~\eqref{qisouk}
in the $u_k$-direction.
\end{defi}

\noindent
In what follows, the term \textbf{$q$-isomonodromic deformation}
refers to a generalized $q$-isomonodromic deformation in the sense
of Definition~\ref{Def:GenqID}.

\begin{rmk}
	If $(\Phi(u;0),W(u;0);\boldsymbol{\alpha})$
	is a generalized $q$-isomonodromic deformation
	with respect to the formal monodromy $0$,
	then
	$(U^{-\mathrm{ad}\boldsymbol{\beta}}\Phi(u;0),
	U^{-\boldsymbol{\beta}}W(u;0);\boldsymbol{\alpha})$
	is a generalized $q$-isomonodromic deformation
	with respect to the formal monodromy $\boldsymbol{\beta}$.
\end{rmk}

\subsection{Global Solutions}
\label{Sect:GlobalQIso}

In this subsection, we consider the $q$-isomonodromy equations
where all $u_k$ have multiplicity one.
We give explicit formulas for their local solutions and prove the
compatibility of the deformation in different directions.
We thus obtain global solutions.

For each $k$, define $V_k(u;\Phi,\boldsymbol{\alpha};q)$ by
\begin{subequations}\label{Exp:Vk}
\begin{align}
	V_k(u;\Phi,\boldsymbol{\alpha};q)_{\hat{k}\hat{k}}
	&:=
	I,
	\label{Exp:Vkhatkk}\\
	V_k(u;\Phi,\boldsymbol{\alpha};q)_{k\hat{k}}
	&:=
	\Phi_{k\hat{k}}
	\frac{(q-1)u_k}{
		u_kI-q^{-1}U_{\hat{k}\hat{k}}
	},
	\label{Exp:Vkkhatk}
	\\
	V_k(u;\Phi,\boldsymbol{\alpha};q)_{\hat k k} & :=
            \frac{(q-1) q^{-\alpha_{k}} u_k}{
            q^{-\alpha_k} u_k
            \left( I + (q-1) \Phi_{\hat{k}\hat{k}} \right)
            - U_{\hat{k}\hat{k}} }
            \Phi_{\hat{k}k},
	\label{Exp:Vk12}\\
	V_k(u;\Phi,\boldsymbol{\alpha};q)_{kk}
	&:=
	q^{\alpha_k}
	+
	V_k(u;\Phi,\boldsymbol{\alpha};q)_{k\hat{k}}
	V_k(u;\Phi,\boldsymbol{\alpha};q)_{\hat{k}k}.
	\label{Exp:Vkkk}
\end{align}
\end{subequations}

\begin{thm}\label{Thm:qiso}
For any $q$-monodromy system
$(U^0,\Phi^0;W^0;\boldsymbol{\alpha})$,
there exist functions
$(\Phi(u;\boldsymbol{\beta}),W(u;\boldsymbol{\beta}))$
on their maximal domain of definition in
$(u_1^0q^{\mathbb Z})
\times\cdots\times
(u_n^0q^{\mathbb Z})$
such that
\begin{align}\label{Eq:qIsoInitial}
	\bigl(
	\Phi(u^0;\boldsymbol{\beta}),
	W(u^0;\boldsymbol{\beta})
	\bigr)
	&=
	(\Phi^0,W^0),
\end{align}
and
\begin{subequations}\label{ExpclitPhiW}
\begin{align}
	T_k\Phi(u;\boldsymbol{\beta})
	&=
	\mathrm{Ad}\left(
		q^{-\boldsymbol{\beta}E_k}
		V_k(u;\Phi(u;\boldsymbol{\beta}),\boldsymbol{\alpha};q)
	\right)
	\Phi(u;\boldsymbol{\beta}),
	\label{ExpclitPhi}\\
	T_k^{-1}\Phi(u;\boldsymbol{\beta})
	&=
	\mathrm{Ad}\left(
		q^{\boldsymbol{\beta}E_k}
		V_k\left(
		-u;-q\Phi(u;\boldsymbol{\beta})^\top,
		-\boldsymbol{\alpha};q^{-1}
		\right)^{-\top}
	\right)
	\Phi(u;\boldsymbol{\beta}),
	\label{ExpclitPhiInv}\\
	T_kW(u;\boldsymbol{\beta})
	&=
	q^{-\boldsymbol{\beta}E_k}
	V_k(u;\Phi(u;\boldsymbol{\beta}),\boldsymbol{\alpha};q)\cdot
	W(u;\boldsymbol{\beta}),
	\label{ExpclitW}\\
	T_k^{-1}W(u;\boldsymbol{\beta})
	&=
	q^{\boldsymbol{\beta}E_k}
	V_k\left(
	-u;-q\Phi(u;\boldsymbol{\beta})^\top,
	-\boldsymbol{\alpha};q^{-1}
	\right)^{-\top}\cdot
	W(u;\boldsymbol{\beta}).
	\label{ExpclitWInv}
\end{align}
\end{subequations}
Then
$(\Phi(u;\boldsymbol{\beta}),
W(u;\boldsymbol{\beta});
\boldsymbol{\alpha})$
is a $q$-isomonodromic deformation with respect to the formal
monodromy $\boldsymbol{\beta}$.
If, in addition, the following
separation condition holds:
\begin{align}\label{Cond:GlobalSep}
	\frac{
		(u_i^0)^{-1}q^{\alpha_i}
	}{
		(u_j^0)^{-1}q^{\alpha_j}
	}
	&\notin
	q^{\mathbb Z},
	\qquad
	1\leqslant i\neq j\leqslant n,
\end{align}
then it is the unique $q$-isomonodromic deformation
satisfying~\eqref{Eq:qIsoInitial}.
\end{thm}

\begin{prf}
We first derive the formulas in~\eqref{Exp:Vk}.
Equation~\eqref{Exp:Vkhatkk} follows from the normalization
condition~\eqref{Cond:NorVk}.
Comparing the $(k,\hat{k})$-blocks of the coefficients of~$x$
on the two sides of~\eqref{qisoeq} gives~\eqref{Exp:Vkkhatk}.
Equation~\eqref{Exp:Vkkk} follows from the block matrix inverse
formula and the spectral condition~\eqref{Cond:SpecVk}.

To derive~\eqref{Exp:Vk12}, use the block
reduction~\eqref{Eq:BlockReduction}, with $n$ replaced by~$k$.
Equations~\eqref{Exp:Vkhatkk} and~\eqref{Exp:Vkkk}, together with
the block matrix inverse formula, give
\begin{align*}
	(G_1^{-1})_{\hat{k}k}
	=
	-q^{-\alpha_k}u_k(V_k)_{\hat{k}k},\qquad
	(G_1^{-1})_{kk}
	&=
	q^{-\alpha_k}u_k.
\end{align*}
Multiplying the first identity in~\eqref{DiagA1A2} by $G_1^{-1}$
from the left and taking the $(\hat{k},k)$-block, we obtain
\begin{align*}
	(A_1G_1^{-1})_{\hat{k}k}
	&=
	(G_1^{-1})_{\hat{k}k}
	(G_1^{-1})_{kk}^{-1}
	=
	-(V_k)_{\hat{k}k}.
\end{align*}
Substituting the definition of~$A_1$ from~\eqref{Eq:BlockReduction}
and the above formulas for~$G_1^{-1}$ gives
\begin{align*}
	&-q^{-\alpha_k}u_k
	U_{\hat{k}\hat{k}}^{-1}
	\bigl(I+(q-1)\Phi_{\hat{k}\hat{k}}\bigr)
	(V_k)_{\hat{k}k}
	+(q-1)q^{-\alpha_k}u_k
	U_{\hat{k}\hat{k}}^{-1}\Phi_{\hat{k}k}
	=
	-(V_k)_{\hat{k}k}.
\end{align*}
Multiplying by~$U_{\hat{k}\hat{k}}$ from the left and rearranging,
we obtain
\begin{align*}
	\left(
		q^{-\alpha_k}u_k
		\bigl(I+(q-1)\Phi_{\hat{k}\hat{k}}\bigr)
		-
		U_{\hat{k}\hat{k}}
	\right)
	(V_k)_{\hat{k}k}
	&=
	(q-1)q^{-\alpha_k}u_k\Phi_{\hat{k}k},
\end{align*}
which is equivalent to~\eqref{Exp:Vk12}.

We next verify that the matrix~$V_k$ defined by~\eqref{Exp:Vk}
satisfies the $q$-isomonodromy equations~\eqref{qisouk}.
Equation~\eqref{qisoeqW} is precisely~\eqref{ExpclitW}, and the
constant term of~\eqref{qisoeq} follows from~\eqref{ExpclitPhi}.
Since $U$ is diagonal, the coefficients of~$x^2$ also agree.
For the coefficient of~$x$, the $(\hat{k},\hat{k})$-block holds
identically, while the $(k,\hat{k})$-block is
\eqref{Exp:Vkkhatk}.
After using~\eqref{Exp:Vk12} and~\eqref{Exp:Vkkk}, the
$(\hat{k},k)$- and $(k,k)$-blocks both reduce to
\begin{align}\label{Eq:VkSpectralIdentity}
	1+(q-1)\Phi_{kk}
	-
	(q-1)\Phi_{k\hat{k}}(V_k)_{\hat{k}k}
	&=
	q^{\alpha_k}.
\end{align}
Since $\boldsymbol{\alpha}$ are the formal exponents of
$(U,\Phi)$, Definition~\ref{Def:alpha} gives
\begin{align*}
	\det\left(
		I+(q-1)\Phi
		-
		u_k^{-1}q^{\alpha_k}U
	\right)
	&=
	0.
\end{align*}
On the domain of definition of~\eqref{Exp:Vk12}, the
$(\hat{k},\hat{k})$-block of the matrix in this determinant is
invertible.
The Schur complement formula therefore gives
\eqref{Eq:VkSpectralIdentity}.
Hence $V_k$ gives a local solution of the $q$-isomonodromy
equations.

The inverse-transpose duality~\eqref{Eq:CDuality}
in Proposition~\ref{Pro:CTransformation}, together with
Definition~\ref{Def:alpha}, identifies
\begin{align*}
	(-U,-q\Phi^\top;W^{-\top};-\boldsymbol{\alpha})
\end{align*}
as the dual system with base~$q^{-1}$ and formal monodromy
$-\boldsymbol{\beta}^{\top}$.
Applying the forward formulas~\eqref{ExpclitPhi}
and~\eqref{ExpclitW} to this system gives
\eqref{ExpclitPhiInv} and~\eqref{ExpclitWInv}, respectively.

To establish compatibility, we first
impose~\eqref{Cond:GlobalSep}.
On this dense open subset,
Proposition~\ref{Pro:qUni} shows that a $q$-monodromy
system with fixed $U$ and formal exponents is uniquely determined
by its $q$-monodromy data.
Lemma~\ref{Lem:qBasic} and the duality~\eqref{Eq:CDuality} show
that the forward and inverse transformations constructed above
preserve these data.
Since $T_iT_jU=T_jT_iU$, Proposition~\ref{Pro:qUni} gives
\begin{align*}
	T_iT_j\Phi
	&=
	T_jT_i\Phi,\qquad
	T_iT_jW
	=
	T_jT_iW.
\end{align*}
The same argument shows that each inverse transformation is
inverse to the corresponding forward transformation.
Since both sides of the commutation and inverse identities are
rational functions of the initial data by~\eqref{Exp:Vk}
and~\eqref{ExpclitPhiW}, the identities extend from this dense
open subset to their maximal common domain of definition.
Consequently, the transformations define a consistent
$\mathbb Z^n$-action and yield the required functions
$\Phi(u;\boldsymbol{\beta})$ and~$W(u;\boldsymbol{\beta})$
on the full $q^{\mathbb Z}$-lattice.
Under~\eqref{Cond:GlobalSep}, Proposition~\ref{Pro:qUni} also
proves their uniqueness subject to~\eqref{Eq:qIsoInitial}.
This completes the proof.
\end{prf}

\begin{cor}
Under the assumptions of Theorem~\ref{Thm:qiso},
\begin{subequations}\label{Exp:TkPhiBlocks}
\begin{align}
	(T_k\Phi)_{\hat{k}k}
	&=
	\frac{
		u_kI-U_{\hat{k}\hat{k}}
	}{
		(q-1)u_k
	}
	(V_k)_{\hat{k}k}q^{\beta_k},
	\\
	(T_k\Phi)_{kk}
	&=
	\Phi_{kk}
	-
	q^{-1}\Phi_{k\hat{k}}
	\frac{
		(q-1)U_{\hat{k}\hat{k}}
	}{
		u_kI-q^{-1}U_{\hat{k}\hat{k}}
	}
	(V_k)_{\hat{k}k},
	\\
	(T_k\Phi)_{\hat{k}\hat{k}}
	&=
	\Phi_{\hat{k}\hat{k}}
	+
	\frac{1}{(q-1)u_k}
	\left(
		U_{\hat{k}\hat{k}}
		(V_k)_{\hat{k}k}(V_k)_{k\hat{k}}
		-
		(V_k)_{\hat{k}k}(V_k)_{k\hat{k}}
		q^{-1}U_{\hat{k}\hat{k}}
	\right),
	\\
	(T_k\Phi)_{k\hat{k}}
	&=
	q^{-\beta_k}\Phi_{k\hat{k}}
	\left(
		q^{\alpha_k}I
		+
		\frac{q-1}{I-q^{-1}
		\frac{U_{\hat{k}\hat{k}}}{u_k}}
		\left(
			(T_k\Phi)_{\hat{k}\hat{k}}
			-
			[\alpha_k]_qI
		\right)
	\right).
\end{align}
\end{subequations}
\end{cor}

The compatibility of the transformations~\eqref{ExpclitPhiW}
established in Theorem~\ref{Thm:qiso} yields the following
algebraic identity.

\begin{cor}[Compatibility condition]\label{CompCond}
Denote $V_k(u;\Phi,\boldsymbol{\alpha};q)$ by $V_k(u;\Phi)$.
Let $\Phi$ be a generic $n\times n$ matrix such that
$\boldsymbol{\alpha}$ are the formal exponents of $(U,\Phi)$.
Then
\begin{align}\label{Eq:CompCond}
	V_i
	\left(
	T_ju;
		\operatorname{Ad}\left(V_j(u;\Phi)\right)\Phi
	\right)
	V_j(u;\Phi)
	&=
	V_j
	\left(
	T_iu;
		\operatorname{Ad}\left(V_i(u;\Phi)\right)\Phi
	\right)
	V_i(u;\Phi).
\end{align}
\end{cor}

\begin{prf}
For $\boldsymbol{\beta}=0$, the
transformations~\eqref{ExpclitPhiW} in
Theorem~\ref{Thm:qiso} reduce to
\begin{subequations}\label{Eq:CompShifts}
\begin{align}
	T_k\Phi(u)
	&=
		\operatorname{Ad}\left(V_k(u;\Phi(u))\right)\Phi(u),
	\label{Eq:CompShiftPhi}\\
	T_kW(u)
	&=
	V_k(u;\Phi(u))W(u).
	\label{Eq:CompShiftW}
\end{align}
\end{subequations}
Substituting~\eqref{Eq:CompShiftW} into the commutativity
$T_jT_iW=T_iT_jW$
established in Theorem~\ref{Thm:qiso} gives
\begin{align*}
	V_i\left(T_ju;T_j\Phi(u)\right)
	V_j(u;\Phi(u))W(u)
	&=
	V_j\left(T_iu;T_i\Phi(u)\right)
	V_i(u;\Phi(u))W(u).
\end{align*}
Cancelling $W(u)$ and using~\eqref{Eq:CompShiftPhi}
proves~\eqref{Eq:CompCond}.
\end{prf}

The compatibility established above allows us to define composite
shifts independently of their ordering.
We conclude this subsection with two identities describing these
shifts.
For an index set
$\boldsymbol{i}=\{i_1,\ldots,i_s\}$ and a diagonal matrix
$D=\operatorname{diag}(d_1,\ldots,d_n)$,
we use the notation
\begin{align}\label{Nota:Eiset}
	T_{\boldsymbol{i}}
	&\assign
	T_{i_1}\cdots T_{i_s},
	\qquad
	E_{\boldsymbol{i}}
	\assign
	E_{i_1}+\cdots+E_{i_s},
	\qquad
	D_{\boldsymbol{i}}
	\assign
	d_{i_1}E_{i_1}+\cdots+d_{i_s}E_{i_s}.
\end{align}
For a $q$-isomonodromic deformation with respect to the formal
monodromy~$\boldsymbol{\beta}$, define
\begin{align}\label{Def:Vseti}
	V_{\boldsymbol{i}}(u)
	&\assign
	q^{\boldsymbol{\beta}E_{\boldsymbol{i}}}
	+
	(q-1)
	\left(
		\left(
		q^{\operatorname{ad}
		(\boldsymbol{\beta}E_{\boldsymbol{i}})}
		T_{\boldsymbol{i}}H_1^{[\infty]}(u)
		\right)
		U_{\boldsymbol{i}}
		-
		U_{\boldsymbol{i}}H_1^{[\infty]}(u)
	\right).
\end{align}
For a singleton, equation~\eqref{Eq:VkHinf} gives
$V_{\{k\}}(u)=V_k(u)$.

\begin{lem}\label{Lem:CompositeShift}
For any disjoint index sets $\boldsymbol{i}$ and~$\boldsymbol{j}$,
one has
\begin{align}\label{Eq:CompositeShift}
	U_{\boldsymbol{i}\sqcup\boldsymbol{j}}x
	+
	V_{\boldsymbol{i}\sqcup\boldsymbol{j}}(u)
	&=
	\left(
	U_{\boldsymbol{i}}x
	+
	q^{\operatorname{ad}
	(\boldsymbol{\beta}E_{\boldsymbol{j}})}
	T_{\boldsymbol{j}}V_{\boldsymbol{i}}(u)
	\right)
	\left(
	U_{\boldsymbol{j}}x
	+
	V_{\boldsymbol{j}}(u)
	\right).
\end{align}
\end{lem}

\begin{prf}
Following the argument from~\eqref{Eq:Rxu} to
\eqref{Eq:RPolynomial} in the proof of
Lemma~\ref{Lem:betaEq}, set
\begin{align*}
	R_{\boldsymbol{i}}(x,u)
	:=
	T_{\boldsymbol{i}}Y^{[\infty]}(x,u)
	Y^{[\infty]}(x,u)^{-1}=
	T_{\boldsymbol{i}}
	\left(
	Y^{[0]}(x,u)W(u)
	\right)
	\left(
	Y^{[0]}(x,u)W(u)
	\right)^{-1}.
\end{align*}
The same comparison of the expansions at $x=\infty$ and~$x=0$
gives
\begin{align}\label{Eq:Ri}
	R_{\boldsymbol{i}}(x,u)
	&=
	q^{-\boldsymbol{\beta}E_{\boldsymbol{i}}}
	\left(
	(q-1)U_{\boldsymbol{i}}x
	+
	V_{\boldsymbol{i}}(u)
	\right).
\end{align}
For disjoint $\boldsymbol{i}$ and~$\boldsymbol{j}$, we have
\begin{align}\label{Eq:Rij}
	R_{\boldsymbol{i}\sqcup\boldsymbol{j}}(x,u)
	&=
	T_{\boldsymbol{j}}R_{\boldsymbol{i}}(x,u)\cdot
	R_{\boldsymbol{j}}(x,u).
\end{align}
Substituting \eqref{Eq:Ri} into \eqref{Eq:Rij} gives~\eqref{Eq:CompositeShift}.
\end{prf}

For the full index set, the composite shift takes a particularly
simple form.

\begin{lem}\label{Lem:FullShift}
Let $(\Phi(u),W(u);\boldsymbol{\alpha})$ be a
$q$-isomonodromic deformation with respect to
$\boldsymbol{\beta}$.
Then
\begin{align}
\label{eq:FullT}
	T_{\{1,\ldots,n\}}
	(\Phi(u),W(u);\boldsymbol{\alpha})
	&=
	\left(
	q^{-\boldsymbol{\beta}}\Phi(u)q^{\boldsymbol{\beta}},
	q^{-\boldsymbol{\beta}}W(u)q^\Lambda;
	\boldsymbol{\alpha}
	\right),\\
\label{eq:FullV}
	V_{\{1,\ldots,n\}}(u)
	&=
	I+(q-1)\Phi(u).
\end{align}
\end{lem}

\begin{prf}
Taking $r=q$ and $D=q^{\boldsymbol{\beta}}$ in
Corollary~\ref{Cor:Basic}, and using the canonical
forms~\eqref{Solatinf} and~\eqref{Solatzero} together
with~\eqref{Jordan}, gives
\begin{subequations}
\begin{align}\label{eq:TfullYinf}
	Y^{[\infty]}
	\left(
	x;qU,
	q^{-\boldsymbol{\beta}}\Phi(u)q^{\boldsymbol{\beta}};
	\boldsymbol{\beta}
	\right)
	&=
	q^{-\boldsymbol{\beta}}
	Y^{[\infty]}
	\left(
	qx;U,\Phi(u);
	\boldsymbol{\beta}
	\right),\\
	Y^{[0]}
	\left(
	x;qU,
	q^{-\boldsymbol{\beta}}\Phi(u)q^{\boldsymbol{\beta}}
	\right)
	q^{-\boldsymbol{\beta}}W(u)q^\Lambda
	&=
	q^{-\boldsymbol{\beta}}
	Y^{[0]}(qx;U,\Phi(u))W(u).
\end{align}
\end{subequations}
Consequently,
$\left(
	qU,
	q^{-\boldsymbol{\beta}}\Phi(u)q^{\boldsymbol{\beta}};
	q^{-\boldsymbol{\beta}}W(u)q^\Lambda;
	\boldsymbol{\alpha}
	\right)$
has the same $q$-monodromy data as
$(U,\Phi(u);W(u);\boldsymbol{\alpha})$.
On the generic locus where~\eqref{NotUniPhi} holds,
Proposition~\ref{Pro:qUni} gives~\eqref{eq:FullT}.
Since both sides of~\eqref{eq:FullT} are rational functions
of the initial data by~\eqref{ExpclitPhiW}, the identity extends
to their maximal common domain of definition.

Comparing the coefficients of $x^{-1}$ in~\eqref{eq:TfullYinf}
gives
\begin{align*}
	T_{\{1,\ldots,n\}}H_1^{[\infty]}
	&=
	q^{-1}q^{-\boldsymbol{\beta}}
	H_1^{[\infty]}q^{\boldsymbol{\beta}}.
\end{align*}
Moreover, the recurrence~\eqref{Eq:HinfRec} with $\nu=0$ gives
\begin{align*}
	q^{-1}H_1^{[\infty]}U
	-
	UH_1^{[\infty]}
	&=
	\Phi(u)-[\boldsymbol{\beta}]_q.
\end{align*}
Substituting these identities into~\eqref{Def:Vseti} with
$\boldsymbol{i}=\{1,\ldots,n\}$ yields
\begin{align*}
	V_{\{1,\ldots,n\}}(u)
	&=
	q^{\boldsymbol{\beta}}
	+
	(q-1)
	\left(
	q^{-1}H_1^{[\infty]}U
	-
	UH_1^{[\infty]}
	\right)\\
	&=
	q^{\boldsymbol{\beta}}
	+
	(q-1)
	\left(
	\Phi(u)-[\boldsymbol{\beta}]_q
	\right) =
	I+(q-1)\Phi(u),
\end{align*}
which proves~\eqref{eq:FullV}.
\end{prf}

\subsection{Explicit Rank-Two Solutions}
\label{Sect:Eqiso22}

In this subsection, following Section~\ref{Sect:ExampCan2}, we give
explicit solutions of the $2\times2$ $q$-isomonodromy equations
\eqref{qisouk}.

Let
\begin{align*}
	U^0 =
	\begin{pmatrix}
		u_1^0&0\\
		0&u_2^0
	\end{pmatrix},\qquad
	\Phi^0 =
	\begin{pmatrix}
		\varphi_{11}^0&\varphi_{12}^0\\
		\varphi_{21}^0&\varphi_{22}^0
	\end{pmatrix}.
\end{align*}
Assume that
$u_1^0,u_2^0\in\mathbb C^\ast$,
$u_1^0\notin u_2^0q^{\mathbb Z}$, and that $\Phi^0$ is
$q$-nonresonant with distinct eigenvalues
$[\lambda_1]_q$ and $[\lambda_2]_q$, neither of which is equal to
$\varphi_{11}^0$.
Set
	$\Lambda=\operatorname{diag}(\lambda_1,\lambda_2)$,
	and choose $W^0=(w_{ij}^0)\in\operatorname{GL}_2(\mathbb C)$
	and a branch of $[\Phi^0]^q$ such that
\begin{align}
	(W^0)^{-1}[\Phi^0]^qW^0
	&=
	\Lambda,\qquad
	\frac{w_{1i}^0}{w_{2i}^0}
	=
	\frac{\varphi_{12}^0}
	{[\lambda_i]_q-\varphi_{11}^0},
	\quad i=1,2.
	\label{Eq:RankTwoInitialW}
\end{align}
Choose $\boldsymbol{\alpha}
=\operatorname{diag}(\alpha_1,\alpha_2)$ so that
\[
	\operatorname{Spec}
	\left(
		(U^0)^{-1}(I+(q-1)\Phi^0)
	\right)
	=
	\left\{
		(u_1^0)^{-1}q^{\alpha_1},
		(u_2^0)^{-1}q^{\alpha_2}
	\right\},
\]
and fix an arbitrary formal monodromy
$\boldsymbol{\beta}=\operatorname{diag}(\beta_1,\beta_2)$.

For $N_1,N_2\in\mathbb Z$, set
$u_1=q^{N_1}u_1^0$ and
$u_2=q^{N_2}u_2^0$.
For brevity, we write
\[
	\Phi(u_1,u_2)
	:=
	\Phi(u_1,u_2;\boldsymbol{\beta})
	=
	\begin{pmatrix}
		\varphi_{11}(u_1,u_2)&\varphi_{12}(u_1,u_2)\\
		\varphi_{21}(u_1,u_2)&\varphi_{22}(u_1,u_2)
	\end{pmatrix}.
\]
The entries of $\Phi(u_1,u_2)$ are given by
\begin{subequations}\label{Exp:RankTwoPhi}
\begin{align}
	\varphi_{11}(u_1,u_2)
	&=
	\frac{u_2}{u_2-u_1}[\alpha_1]_q
	-
	\frac{u_1}{u_2-u_1}
	\left(
		[\lambda_1]_q+[\lambda_2]_q-[\alpha_2]_q
	\right),
	\label{qPhi11}\\
	\varphi_{22}(u_1,u_2)
	&=
	\frac{u_2}{u_2-u_1}
	\left(
		[\lambda_1]_q+[\lambda_2]_q-[\alpha_1]_q
	\right)
	-
	\frac{u_1}{u_2-u_1}[\alpha_2]_q,
	\\
	\varphi_{12}(u_1,u_2)
	&=
	\left(\frac{u_1}{u_1^0}\right)^{\alpha_1-\beta_1}
	\left(\frac{u_2}{u_2^0}\right)^{\beta_2-\alpha_1}
	\frac{
		\left(
			q^{\lambda_2-\alpha_1+1}
			\frac{u_1^0}{u_2^0};q
		\right)_{N_1-N_2}
		\left(
			q^{\lambda_1-\alpha_1+1}
			\frac{u_1^0}{u_2^0};q
		\right)_{N_1-N_2}
	}{
		\left(
			q\frac{u_1^0}{u_2^0};q
		\right)_{N_1-N_2}^2
	}
	\varphi_{12}^0
	\nonumber\\
	&=
	\left(\frac{u_1}{u_1^0}\right)^{\alpha_2-\beta_1}
	\left(\frac{u_2}{u_2^0}\right)^{\beta_2-\alpha_2}
	\frac{
		\left(
			\frac{u_2^0}{u_1^0};q
		\right)_{N_2-N_1}^2
	}{
		\left(
			q^{\lambda_1-\alpha_2}
			\frac{u_2^0}{u_1^0};q
		\right)_{N_2-N_1}
		\left(
			q^{\lambda_2-\alpha_2}
			\frac{u_2^0}{u_1^0};q
		\right)_{N_2-N_1}
	}
	\varphi_{12}^0,
	\label{qPhi12}\\
	\varphi_{21}(u_1,u_2)
	&=
	\left(\frac{u_1}{u_1^0}\right)^{\beta_1-\alpha_1}
	\left(\frac{u_2}{u_2^0}\right)^{\alpha_1-\beta_2}
	\frac{
		\left(
			\frac{u_1^0}{u_2^0};q
		\right)_{N_1-N_2}^2
	}{
		\left(
			q^{\lambda_2-\alpha_1}
			\frac{u_1^0}{u_2^0};q
		\right)_{N_1-N_2}
		\left(
			q^{\lambda_1-\alpha_1}
			\frac{u_1^0}{u_2^0};q
		\right)_{N_1-N_2}
	}
	\varphi_{21}^0
	\nonumber\\
	&=
	\left(\frac{u_1}{u_1^0}\right)^{\beta_1-\alpha_2}
	\left(\frac{u_2}{u_2^0}\right)^{\alpha_2-\beta_2}
	\frac{
		\left(
			q^{\lambda_1-\alpha_2+1}
			\frac{u_2^0}{u_1^0};q
		\right)_{N_2-N_1}
		\left(
			q^{\lambda_2-\alpha_2+1}
			\frac{u_2^0}{u_1^0};q
		\right)_{N_2-N_1}
	}{
		\left(
			q\frac{u_2^0}{u_1^0};q
		\right)_{N_2-N_1}^2
	}
	\varphi_{21}^0.
	\label{qPhi21}
\end{align}
\end{subequations}
Similarly, write
\[
	W(u_1,u_2)
	:=
	W(u_1,u_2;\boldsymbol{\beta})
	=
	\begin{pmatrix}
		w_{11}(u_1,u_2)&w_{12}(u_1,u_2)\\
		w_{21}(u_1,u_2)&w_{22}(u_1,u_2)
	\end{pmatrix}.
\]
Its entries are
\begin{subequations}\label{Exp:RankTwoW}
\begin{align}
	w_{11}(u_1,u_2)
	&=
	\left(\frac{u_1}{u_1^0}\right)^{\alpha_1-\beta_1}
	\left(\frac{u_2}{u_2^0}\right)^{\lambda_1-\alpha_1}
	\frac{
		\left(
			q^{\lambda_1-\alpha_1+1}
			\frac{u_1^0}{u_2^0};q
		\right)_{N_1-N_2}
	}{
		\left(
			q\frac{u_1^0}{u_2^0};q
		\right)_{N_1-N_2}
	}
	w_{11}^0
	\nonumber\\
	&=
	\left(\frac{u_1}{u_1^0}\right)^{\lambda_1-\beta_1}
	\frac{
		\left(
			\frac{u_2^0}{u_1^0};q
		\right)_{N_2-N_1}
	}{
		\left(
			q^{\lambda_2-\alpha_2}
			\frac{u_2^0}{u_1^0};q
		\right)_{N_2-N_1}
	}
	w_{11}^0,
	\\
	w_{12}(u_1,u_2)
	&=
	\left(\frac{u_1}{u_1^0}\right)^{\alpha_1-\beta_1}
	\left(\frac{u_2}{u_2^0}\right)^{\lambda_2-\alpha_1}
	\frac{
		\left(
			q^{\lambda_2-\alpha_1+1}
			\frac{u_1^0}{u_2^0};q
		\right)_{N_1-N_2}
	}{
		\left(
			q\frac{u_1^0}{u_2^0};q
		\right)_{N_1-N_2}
	}
	w_{12}^0
	\nonumber\\
	&=
	\left(\frac{u_1}{u_1^0}\right)^{\lambda_2-\beta_1}
	\frac{
		\left(
			\frac{u_2^0}{u_1^0};q
		\right)_{N_2-N_1}
	}{
		\left(
			q^{\lambda_1-\alpha_2}
			\frac{u_2^0}{u_1^0};q
		\right)_{N_2-N_1}
	}
	w_{12}^0,
	\\
	w_{21}(u_1,u_2)
	&=
	\left(\frac{u_2}{u_2^0}\right)^{\lambda_1-\beta_2}
	\frac{
		\left(
			\frac{u_1^0}{u_2^0};q
		\right)_{N_1-N_2}
	}{
		\left(
			q^{\lambda_2-\alpha_1}
			\frac{u_1^0}{u_2^0};q
		\right)_{N_1-N_2}
	}
	w_{21}^0
	\nonumber\\
	&=
	\left(\frac{u_1}{u_1^0}\right)^{\lambda_1-\alpha_2}
	\left(\frac{u_2}{u_2^0}\right)^{\alpha_2-\beta_2}
	\frac{
		\left(
			q^{\lambda_1-\alpha_2+1}
			\frac{u_2^0}{u_1^0};q
		\right)_{N_2-N_1}
	}{
		\left(
			q\frac{u_2^0}{u_1^0};q
		\right)_{N_2-N_1}
	}
	w_{21}^0,
	\\
	w_{22}(u_1,u_2)
	&=
	\left(\frac{u_2}{u_2^0}\right)^{\lambda_2-\beta_2}
	\frac{
		\left(
			\frac{u_1^0}{u_2^0};q
		\right)_{N_1-N_2}
	}{
		\left(
			q^{\lambda_1-\alpha_1}
			\frac{u_1^0}{u_2^0};q
		\right)_{N_1-N_2}
	}
	w_{22}^0
	\nonumber\\
	&=
	\left(\frac{u_1}{u_1^0}\right)^{\lambda_2-\alpha_2}
	\left(\frac{u_2}{u_2^0}\right)^{\alpha_2-\beta_2}
	\frac{
		\left(
			q^{\lambda_2-\alpha_2+1}
			\frac{u_2^0}{u_1^0};q
		\right)_{N_2-N_1}
	}{
		\left(
			q\frac{u_2^0}{u_1^0};q
		\right)_{N_2-N_1}
	}
	w_{22}^0.
	\label{2t2W22}
\end{align}
\end{subequations}
These functions satisfy
\[
	\Phi(u_1^0,u_2^0)=\Phi^0,
	\qquad
	W(u_1^0,u_2^0)=W^0.
\]

\section{Deformations with Prescribed Asymptotics}
\label{Sect:Asym}

In the remainder of this paper, unless otherwise stated,
we make the following assumptions and choice:
\begin{itemize}
	\item each $u_k$ has multiplicity one;
	\item $u_i\notin u_jq^{\mathbb{Z}\setminus\{0\}}$
	for all $i\neq j$, and
	$u_1,\ldots,u_n$ are all nonzero;
	\item the formal monodromy is taken to be
	$\boldsymbol{\beta}=\boldsymbol{\alpha}$, where
	$\boldsymbol{\alpha}$ denotes the formal exponents.
\end{itemize}
Applying Theorem~\ref{Thm:qiso},
	we construct a $q$-isomonodromic deformation
	$(\Phi(u;\boldsymbol{\alpha}),W(u;\boldsymbol{\alpha});
	\boldsymbol{\alpha})$
	depending continuously on $u$ and
	having prescribed
	asymptotic leading term $(\Phi_0,W_0)$.
The construction requires
	the eigenvalues of
	the upper-left submatrices of $\Phi_0$ to satisfy
	an additional shrinking condition.
This deformation is uniquely
	determined by the leading term $(\Phi_0,W_0)$.
We call a $q$-isomonodromic deformation
	obtained in this way a
	\textbf{shrinking solution} of
	the $q$-isomonodromy equations.
The choice
	$\boldsymbol{\beta}=\boldsymbol{\alpha}$
	makes the recursive asymptotic formulas
	take their simplest form.

The construction is recursive.
Starting from a constant pair $(\Phi_0,W_0)$ satisfying the
conditions specified below, we use Picard iteration to construct
the sequence
$(\Phi_k(u),W_k(u);\boldsymbol{\alpha})$,
for $k=1,\ldots,n$.
The $k$-th triple depends only on $u_1,\ldots,u_k$, and the final
triple is
\[
(\Phi_n(u),W_n(u);\boldsymbol{\alpha})
=
(\Phi(u;\boldsymbol{\alpha}),W(u;\boldsymbol{\alpha});
\boldsymbol{\alpha}).
\]
Here, the $n\times n$ matrices $\Phi_k(u)$ and $W_k(u)$ satisfy
the following $q$-difference equations
for $j=1,\ldots,k$:
\begin{subequations}\label{qkjiso}
\begin{align}
\label{qkjisoP}
T_j \Phi_k(u) & =
	V^{(n)}_{k;j}(u;\Phi_k(u),\alpha_j) \cdot \Phi_k(u) \cdot
	V^{(n)}_{k;j}(u;\Phi_k(u),\alpha_j)^{-1},\\
\label{qkjisoW}
T_j W_k(u) & =
	V^{(n)}_{k;j}(u;\Phi_k(u),\alpha_j) \cdot W_k(u),
\end{align}
\end{subequations}
with the following relations:
\begin{subequations}\label{Rel:kneq}
\begin{align}
    W_k(u)^{-1} \Phi_k(u) W_k(u) & =
	[\Lambda]_q, \\
    \left \{ u_1^{-1} q^{\alpha_1}, \ldots, u_k^{-1} q^{\alpha_k} \right \}
	& = 
    \mathrm{Spec}
    \left(
        (U^{[k]})^{-1}(I + (q-1)\Phi_k^{[k]})
    \right), \\
    q^{\alpha_s}
    & =
    \frac{
        \mathrm{det}( I + (q-1)\Phi_k^{[s]} )
    }{
        \mathrm{det}( I + (q-1)\Phi_k^{[s-1]} )
    }, \quad \text{for }s = k+1, \ldots, n.
\end{align}
\end{subequations}
The upper-left $s\times s$ submatrix of a matrix $A$ is denoted
by $A^{[s]}$. For $1\leqslant j\leqslant k$, set
\begin{align}\label{Def:Vqiso}
V_{k;j}^{(n)}(u;\Phi_k,\alpha_j)
:=
\begin{pmatrix}
q^{-\boldsymbol{\alpha}^{[k]}E_j}
V_j\bigl(
u_1,\ldots,u_k;
\Phi_k^{[k]},\boldsymbol{\alpha}^{[k]};q
\bigr)
& 0\\
0 & I_{n-k}
\end{pmatrix},
\end{align}
where $V_j$, defined by~\eqref{Exp:Vk}, is the normalized shift
matrix associated with $T_j$ in the rank-$k$
$q$-isomonodromy equations \eqref{qisouk}.
We call \eqref{qkjiso},
together with the preceding relations \eqref{Rel:kneq},
the \textbf{$(k,n)$-equations}.
A triple $(\Phi_k(u),W_k(u);\boldsymbol{\alpha})$ satisfying these
equations is called a
\textbf{$q$-isomonodromic deformation at level $(k,n)$}.

Throughout this section, we assume
that $(\Phi_{k-1}(u),W_{k-1}(u);\boldsymbol{\alpha})$ is a
$q$-isomonodromic deformation at level $(k-1,n)$.
In addition, we always impose the following conditions on
$\Phi_{k-1}$.
\begin{itemize}
	\item We assume that
	\begin{align}\label{PhikInvert}
		I+(q-1)\Phi_{k-1}^{[k-1]}
		\text{ is invertible},
	\end{align}
	and write
	\[
	\mathrm{Spec}\left(
		\Phi_{k-1}^{[k-1]}
	\right)
	=
	\{
	[\lambda^{(k-1)}_1]_q,\ldots,[\lambda^{(k-1)}_{k-1}]_q
	\}.
	\]

	\item The following
	\textbf{shrinking condition at level $k-1$} holds:
	\begin{align}\label{Cond:less1}
		\underset{1\leqslant i,j\leqslant k-1}{\mathrm{max}}
		\left|
		\mathrm{Re}(\lambda^{(k-1)}_i-\lambda^{(k-1)}_j)
		-
		\frac{\mathrm{arg}q}{\mathrm{ln}|q|}
		\mathrm{Im}(\lambda^{(k-1)}_i-\lambda^{(k-1)}_j)
		\right|
		&<
		1-\varepsilon_{k-1}<1.
	\end{align}
\end{itemize}
Let $\delta_k$ denote
	the following nonlinear projection on
	$\mathrm{Mat}_{n\times n}(\mathbb{C})$,
\begin{align}
\alpha_s(A)
&:=
\log_q
\frac{
\det\bigl(I+(q-1)A^{[s]}\bigr)
}{
\det\bigl(I+(q-1)A^{[s-1]}\bigr)
},
\qquad 1\leqslant s \leqslant n,\\
\label{Def:deltak}
\delta_k A
&:=
\begin{pmatrix}
A^{[k]} & & & 0\\
& [\alpha_{k+1}(A)]_q & &\\
& & \ddots &\\
0 & & & [\alpha_n(A)]_q
\end{pmatrix},\qquad
0\leqslant k \leqslant n,
\end{align}
where $A^{[k]}$ denotes the upper-left $k\times k$ submatrix of
$A$, and we use the convention
$\det\bigl(I+(q-1)A^{[0]}\bigr)=1$.

\begin{rmk}
The shrinking condition \eqref{Cond:less1} is independent of the
choices of
$\lambda^{(k-1)}_1,\ldots,\lambda^{(k-1)}_{k-1}$.
\end{rmk}

The remainder of this section is organized as follows.
In Section~\ref{Sect:RecConstruction}, we use Picard iteration to
construct a solution with prescribed asymptotic behavior along one
deformation direction. This is the first step in the recursive
construction.
In Section~\ref{Sect:RecCompatibility}, we combine these recursive
steps to construct the desired $q$-isomonodromic deformation.
Finally, in Section~\ref{Sect:AsymRankTwo}, we describe the
asymptotic behavior of the solutions in the $2\times2$ case.

\subsection{Recursive Construction}
\label{Sect:RecConstruction}

In this subsection, we give the recursive construction needed to
obtain the desired $q$-isomonodromic deformation.

For a scalar- or matrix-valued function $f$ defined on a fixed
	$q$-spiral $u_0q^{\mathbb Z}$,
	define the \textbf{$q$-Jackson integral}
\begin{align}\label{qJackson}
\int_\infty^u f(t)\,\mathrm d_qt
:=
\begin{cases}
(1-q)u\displaystyle\sum_{k=0}^{\infty}
q^k f(uq^k),& |q|>1,\\[6pt]
(q-1)u\displaystyle\sum_{k=1}^{\infty}
q^{-k}f(uq^{-k}),& |q|<1.
\end{cases}
\end{align}
For a scalar-valued function $g$, define
\begin{align*}
\int_\infty^u g(t)\,|\mathrm d_qt|
:=
\begin{cases}
|(1-q)u|\displaystyle\sum_{k=0}^{\infty}
|q|^k g(uq^k),& |q|>1,\\[6pt]
|(q-1)u|\displaystyle\sum_{k=1}^{\infty}
|q|^{-k}g(uq^{-k}),& |q|<1.
\end{cases}
\end{align*}
Moreover,
\begin{align}
\left\lVert\int_\infty^u f(t)\,\mathrm d_qt\right\rVert
&\leqslant
\int_\infty^u\lVert f(t)\rVert\,|\mathrm d_qt|,\\
\label{BasicqInt}
\int_\infty^u|t|^r\,|\mathrm d_qt|
&=
\frac{|1-q|}{|1-|q|^{r+1}|}|u|^{r+1}
=
\frac{|1-q|}{|1-|q||}
\frac{|u|^{r+1}}{|[r+1]_{|q|}|},
\qquad r<-1.
\end{align}

We first establish a Picard iteration lemma for the integral equations associated with the $(k,n)$-equations \eqref{qkjiso}.

\begin{lem}\label{Lem:Picard}
Let $A\in\mathrm{Mat}_{(n-1)\times(n-1)}$ be a constant matrix
	such that $I+(q-1)A$ is invertible.
Fix a branch of $[A]^q$ and assume that
\begin{align}\label{Cond:PicardGap}
	\mathrm{max}\left\{
	\left|
	\mathrm{Re}(\lambda_i-\lambda_j)
	-
	\frac{\mathrm{arg}q}{\mathrm{ln}|q|}
	\mathrm{Im}(\lambda_i-\lambda_j)
	\right|:
	\lambda_i,\lambda_j\in\mathrm{Spec}([A]^q)
	\right\}
	&=
	1-\varepsilon_A<1.
\end{align}
Set
	$X_{\hat{n}\hat{n}}(u)
	=u^{\mathrm{ad}[A]^q}Z_{\hat{n}\hat{n}}(u)$.
Consider the following system for the matrix-valued functions
	$Z_{\hat{n}\hat{n}}(u)$,
	$X_{\hat{n}n}(u)$, and $X_{n\hat{n}}(u)$:
\begin{subequations}\label{Eq:PicardSystem}
	\begin{align}\label{Lem:Eq11}
	\frac{\partial_q}{\partial_q u}
	Z_{\hat{n}\hat{n}} & =
		u^{-2}
		f(Z_{\hat{n}\hat{n}}, X_{ \hat{n}n }, X_{ n\hat{n} }, u), \\
	\label{Lem:Eq12}
	\frac{\partial_q}{\partial_q u}
	X_{\hat{n}n} & =
		u^{-1}
		f_1(X_{\hat{n}\hat{n}}, X_{\hat{n}n}, X_{n\hat{n}}, u), \\
	\label{Lem:Eq21}
	\frac{\partial_q}{\partial_q u}
	X_{n\hat{n}} & =
		u^{-1}
		f_2(X_{\hat{n}\hat{n}}, X_{\hat{n}n}, X_{n\hat{n}}, u),
	\end{align}
\end{subequations}
	where
	\begin{align}\nonumber
	f(Z_{\hat{n}\hat{n}},X_{\hat{n}n},X_{n\hat{n}},u)
	&=
	U_1\cdot h_1(Z_{\hat{n}\hat{n}},u)
	u^{-[A]^q}
	g_1(X_{\hat{n}n},X_{n\hat{n}},u)
	u^{[A]^q}\\
	&\quad+
	h_2(Z_{\hat{n}\hat{n}},u)
	u^{-[A]^q}
	g_2(X_{\hat{n}n},X_{n\hat{n}},u)
	u^{[A]^q}\cdot U_2,
	\label{Def:Picardf}
	\end{align}	
	and $U_1,U_2\in\mathrm{Mat}_{(n-1)\times(n-1)}$
	are fixed matrices.
For $x=u^{\mathrm{ad}[A]^q}z$, set
	$\eta_i(x,u)=u^{\mathrm{ad}[A]^q}h_i(z,u)$.
Assume that, for some $0<\varepsilon<\varepsilon_A$, there exist
	$C_0,C_1>0$, with $C_0$ sufficiently small, and constants
	$K_\ast>0$, independent of $u$,
	such that for all sufficiently large $|u|$,
	when
	\[
	\lVert z^{(\ast)}-A\rVert,\
	\lVert x^{(\ast)}-A\rVert < C_0,
	\qquad
	\lVert y_1^{(\ast)}\rVert,\
	\lVert y_2^{(\ast)}\rVert < C_1,
	\]
	the coefficient functions
	$h_i$, $\eta_i$, $g_i$, and $f_i$
	satisfy the following estimates:
\begin{subequations}\label{Est:PicardConditions}
	\begin{align}
	\label{Est:Bound}
		\lVert U_i \rVert,
		\lVert h_i(z^{(\ast)},u) \rVert,
		\lVert \eta_i(x^{(\ast)},u) \rVert,
		\lVert g_i(y_1^{(\ast)}, y_2^{(\ast)}, u) \rVert
		& \leqslant
			K_0,\\
	\label{Est:fiBase}
		\lVert f_i(A,y_1^{(0)},y_2^{(0)},u)\rVert
		&\leqslant
		K_0 \cdot |u|^{-\varepsilon},\\
		\label{Est:h}
		\lVert
		h_i( z^{(1)},u) -
		h_i( z^{(2)},u) \rVert
		& \leqslant
			K_1\cdot
			\lVert
			z^{(1)} -
			z^{(2)} \rVert, \\
		\label{Est:eta}
		\lVert
		\eta_i( x^{(1)},u) -
		\eta_i( x^{(2)},u) \rVert
		& \leqslant
			K_1\cdot
			\lVert
			x^{(1)} -
			x^{(2)} \rVert, \\
		\label{Est:g1}
		\lVert
		g_i(y_1^{(1)}, y_2^{(0)}, u) -
		g_i(y_1^{(2)}, y_2^{(0)}, u) \rVert
		& \leqslant
			K_1\cdot
			\lVert
			y_1^{(1)} -
			y_1^{(2)} \rVert, \\
		\label{Est:g2}
		\lVert
		g_i(y_1^{(0)}, y_2^{(1)}, u) -
		g_i(y_1^{(0)}, y_2^{(2)}, u) \rVert
		& \leqslant
			K_1\cdot
			\lVert
			y_2^{(1)} -
			y_2^{(2)} \rVert,\\
		\label{Est:fix}
		\lVert
		f_i(x^{(1)}, y_1^{(0)}, y_2^{(0)}, u) -
		f_i(x^{(2)}, y_1^{(0)}, y_2^{(0)}, u)
		\rVert
		& \leqslant
			K_1\cdot
			\lVert
			x^{(1)} -
			x^{(2)} \rVert,\\
		\label{Est:fiy1}
		\lVert
		f_i(x^{(0)}, y_1^{(1)}, y_2^{(0)}, u) -
		f_i(x^{(0)}, y_1^{(2)}, y_2^{(0)}, u)
		\rVert
		& \leqslant
			\left(
			K_{11} \lVert x^{(0)}-A \rVert +
			K_{12}|u|^{-\varepsilon} \right)
			\cdot
			\lVert
			y_1^{(1)} -
			y_1^{(2)} \rVert,\\
		\label{Est:fiy2}
		\lVert
		f_i(x^{(0)}, y_1^{(0)}, y_2^{(1)}, u) -
		f_i(x^{(0)}, y_1^{(0)}, y_2^{(2)}, u)
		\rVert
		& \leqslant
			\left(
			K_{11} \lVert x^{(0)}-A \rVert +
			K_{12}|u|^{-\varepsilon} \right)
			\cdot
			\lVert
			y_2^{(1)} -
			y_2^{(2)} \rVert,
	\end{align}
\end{subequations}
Then, for every pair of constant matrices
	$B_{\hat{n}n}\in\mathrm{Mat}_{(n-1)\times1}$ and
	$B_{n\hat{n}}\in\mathrm{Mat}_{1\times(n-1)}$
	satisfying
	$\lVert B_{\hat{n}n}\rVert,
	 \lVert B_{n\hat{n}}\rVert<C_1$,
	there exists a unique solution of
	\eqref{Eq:PicardSystem} for sufficiently large $|u|$ such that
\begin{equation}\label{Lim:Picard}
\begin{alignedat}{2}
Z_{\hat{n}\hat{n}}(u)
	&=
	A + O(u^{-\varepsilon}),\qquad &
X_{\hat{n}\hat{n}}(u)
	&=
	A + O(u^{-\varepsilon}),\\
X_{\hat{n}n}(u)
	&=
	B_{\hat{n}n} + O(u^{-\varepsilon}),\qquad &
X_{n\hat{n}}(u)
	&=
	B_{n\hat{n}} + O(u^{-\varepsilon}).
\end{alignedat}
\end{equation}	
Moreover, if the coefficient functions are analytic in their
	matrix variables and in $[A]^q$, and the estimates
	\eqref{Est:PicardConditions} hold locally uniformly with respect
	to the asymptotic data
	$([A]^q,B_{\hat{n}n},B_{n\hat{n}})$,
	then the solution depends analytically on these data.
\end{lem}

\begin{prf}
We first introduce the constants needed for the estimates
	and then set up the standard Picard iteration.
	We verify the base case and then the inductive step,
	and the resulting convergence gives existence.
	Uniqueness follows from the standard Picard argument
	and is omitted.

Let the Jordan--Chevalley decomposition of $[A]^q$ and the
	spectral decomposition of its semisimple part $\mathscr{S}$ be
\begin{align*}
[A]^q = \mathscr{S} + \mathscr{N} =
		\lambda_1\mathscr{P}_1+\cdots+
		\lambda_{n-1}\mathscr{P}_{n-1}+\mathscr{N},
\end{align*}
where $\mathscr{P}_1,\ldots,\mathscr{P}_{n-1}$ are the
corresponding spectral projections.
This decomposition gives
	\begin{align}\label{SpecDecomp}
		u^{-[A]^q}Xu^{[A]^q} =
		\sum_{j_1,j_2,k_1,k_2}
		(-1)^{k_1}
		\left(
		\frac{\mathscr{P}_{j_1} \mathscr{N}^{k_1}}{k_1!} X
		\frac{\mathscr{P}_{j_2} \mathscr{N}^{k_2}}{k_2!}
		\right) 
		u^{\lambda_{j_2}-\lambda_{j_1}}
		(\mathrm{ln} u)^{k_1+k_2}.
	\end{align}
	Take a constant $K_A>0$ such that
	\begin{align}\label{Const:KA}
		\mathrm{e}^{2\lVert \mathscr{N} \rVert}
		\sum_{j_1, j_2}
			\lVert \mathscr{P}_{j_1} \rVert
			\lVert \mathscr{P}_{j_2} \rVert
		\leqslant K_A,\qquad
		\mathrm{e}^{4\lVert \mathscr{N} \rVert}
		\sum_{j_1, j_2, j_3}
			\lVert \mathscr{P}_{j_1} \rVert
			\lVert \mathscr{P}_{j_2} \rVert
			\lVert \mathscr{P}_{j_3} \rVert
		\leqslant K_A.
	\end{align}
	For $u\in u_0q^\mathbb{Z}$ we have
	\begin{align}\label{Abs:uPower}
		|u^a| = |u_0|^{ \frac{\mathrm{arg} q}{\mathrm{ln}|q|}\mathrm{Im}a }
		\mathrm{e}^{-\mathrm{arg}u_0 \cdot \mathrm{Im}a } \cdot
		|u|^{\mathrm{Re}a - \frac{\mathrm{arg} q}{\mathrm{ln}|q|}\mathrm{Im}a}.
	\end{align}
Therefore, for any $0<\varepsilon<\varepsilon_A$, there exists
a constant $K_{u_0}>0$, depending on the spectrum of $[A]^q$
and the $q$-spiral $u_0q^{\mathbb{Z}}$, such that the following
estimates hold for all integers
$0\leqslant k_i\leqslant 2(n-1)$ and
$\lambda_i\in\mathrm{Spec}([A]^q)$:
	\begin{subequations}\label{Est:PicardSpectral}
	\begin{align}\label{Const:Ku0}
		\int_\infty^u \left| 
		t^{-2-\varepsilon}
		t^{\lambda_1-\lambda_2}
		(\mathrm{ln} t)^{k_1} \right|
		|\mathrm{d}_qt|
		& \leqslant
		K_{u_0} |u|^{-2\varepsilon}, \\
		\label{Const:Ku01}
		\left| u^{\lambda_3 - \lambda_2} (\mathrm{ln}u)^{k_2} \right|
		\int_\infty^u \left|
		t^{-2-\varepsilon} t^{\lambda_2-\lambda_1}
		(\mathrm{ln} t)^{k_1} \right|
		|\mathrm{d}_qt|
		& \leqslant
		K_{u_0} |u|^{-2\varepsilon},\\
		\label{Const:Ku02}
		\left| u^{\lambda_1 - \lambda_3} (\mathrm{ln}u)^{k_2} \right|
		\int_\infty^u \left|
		t^{-2-\varepsilon} t^{\lambda_2-\lambda_1}
		(\mathrm{ln} t)^{k_1} \right|
		|\mathrm{d}_qt|
		& \leqslant
		K_{u_0} |u|^{-2\varepsilon}.
	\end{align}
	\end{subequations}
We fix such an $\varepsilon$ and a corresponding constant
$K_{u_0}$ for the remainder of the proof.
	
Define the sequences
	$(Q_i)_{i\in\mathbb{N}}$,
	$({P}_i)_{i\in\mathbb{N}}$,
	$(P^{(1)}_i)_{i\in\mathbb{N}}$,
	$(P^{(2)}_i)_{i\in\mathbb{N}}$ by
	\begin{align}\label{Init:Picard}
		Q_0(u) = A,\quad
		P^{(1)}_0(u) = B_{ \hat{n}n },\quad
		P^{(2)}_0(u) = B_{ n\hat{n} },
	\end{align}
	and the recursive relations
	\begin{subequations}\label{Rec:}
	\begin{align}
	\label{Rec:P0}
	{P}_{i}(u) & =
		u^{\mathrm{ad}[A]^q} Q_{i}(u), \\
	\label{Rec:Q}
	Q_{i+1}(u) & = A + \int_{\infty}^u
		t^{-2}
		f(Q_i(t), P^{(1)}_i(t), P^{(2)}_i(t), t)
		\mathrm{d}_qt, \\
	\label{Rec:P1}
	P^{(1)}_{i+1}(u) & = B_{ \hat{n}n } + \int_{\infty}^u
		t^{-1}
		f_1({P}_{i+1}(t), P^{(1)}_i(t), P^{(2)}_i(t), t)
		\mathrm{d}_qt, \\
	\label{Rec:P2}
	P^{(2)}_{i+1}(u) & = B_{ n\hat{n} } + \int_{\infty}^u
		t^{-1}
		f_2({P}_{i+1}(t), P^{(1)}_i(t), P^{(2)}_i(t), t)
		\mathrm{d}_qt.
	\end{align}
	\end{subequations}
For the fixed $0<\varepsilon<\varepsilon_A$,
	choose $K>0$ sufficiently large and $0<r<1$,
	both independent of $i$ and $u$.
	We prove by induction that the estimates
\begin{subequations}\label{Induc:i}
	\begin{align}\label{Induc:i1}
		\lVert Q_{i+1}(u) - Q_i(u) \rVert,\
		\lVert P_{i+1}(u) - P_i(u) \rVert
		&\leqslant
		K r^i \cdot |u|^{-\varepsilon},\\
	\label{Induc:i2}
		\lVert P^{(\ast)}_{i+1}(u) - P^{(\ast)}_i(u) \rVert
		&\leqslant
		K r^i \cdot |u|^{-\varepsilon}
	\end{align}
\end{subequations}
hold for all $i\geqslant 0$.

We begin with the base case $i=0$:
\begin{subequations}\label{Base:Picard}
	\begin{align}\label{Base:Picard1}
		\lVert Q_1(u)-A \rVert,\
		\lVert P_1(u)-A \rVert
		&\leqslant
		K |u|^{-\varepsilon},\\
	\label{Base:Picard2}
		\lVert P^{(1)}_1(u)-B_{\hat{n}n} \rVert,\
		\lVert P^{(2)}_1(u)-B_{n\hat{n}} \rVert
		&\leqslant
		K |u|^{-\varepsilon}.
	\end{align}
\end{subequations}
From \eqref{Init:Picard}, \eqref{Rec:P0}, and \eqref{Rec:Q},
	using \eqref{BasicqInt}, \eqref{Cond:PicardGap},
	\eqref{Def:Picardf}, \eqref{Est:Bound},
	\eqref{SpecDecomp}, \eqref{Const:KA},
	and \eqref{Abs:uPower},
	we obtain \eqref{Base:Picard1}.
Together with \eqref{Base:Picard1}, applying
	\eqref{BasicqInt}, \eqref{Est:fiBase}, and \eqref{Est:fix}
	to \eqref{Rec:P1} and \eqref{Rec:P2}
	gives \eqref{Base:Picard2}.
Thus, \eqref{Base:Picard} holds for all sufficiently large
	$|u|$ by our choice of $K$.
	By \eqref{Init:Picard} and \eqref{Rec:P0},
	this is precisely the case $i=0$ of
	\eqref{Induc:i1} and \eqref{Induc:i2}.
	
Assume now that \eqref{Induc:i1} and \eqref{Induc:i2}
	hold up to some $i$. We verify them for $i+1$.	
First, for any sufficiently large $N$ and
	$|t|\geqslant|u|>N$, summing \eqref{Induc:i1}
	from \eqref{Init:Picard} and using \eqref{Rec:P0}
	gives \eqref{DifftoAi} and \eqref{DifftoAplus}.
	Similarly, \eqref{Init:Picard} and \eqref{Induc:i2}
	give \eqref{SizeofP}.
\begin{subequations}\label{DifftoA}
	\begin{align}\label{DifftoAi}
		\lVert Q_i(t)-A \rVert,\
		\lVert {P}_i(t)-A \rVert
		& \leqslant 
		\frac{K}{1-r}|t|^{-\varepsilon} \leqslant
		\frac{K}{1-r}N^{-\varepsilon} < C_0, \\
	\label{DifftoAplus}
		\lVert Q_{i+1}(t)-A \rVert,\
		\lVert {P}_{i+1}(t)-A \rVert
		& \leqslant 
		\frac{K}{1-r}|t|^{-\varepsilon} \leqslant
		\frac{K}{1-r}N^{-\varepsilon} < C_0, \\
		\lVert P^{(\ast)}_i(t) \rVert,\
		\lVert P^{(\ast)}_{i+1}(t) \rVert
		& \leqslant
		\max\left\{
			\lVert B_{\hat{n}n}\rVert,
			\lVert B_{n\hat{n}}\rVert
		\right\}
		+
		\frac{K}{1-r}N^{-\varepsilon} <C_1.
	\label{SizeofP}
	\end{align}
\end{subequations}
From \eqref{Rec:Q}, using
	\eqref{Def:Picardf}, \eqref{Est:Bound}, \eqref{Est:h},
	\eqref{Est:g1}, \eqref{Est:g2},
	\eqref{SpecDecomp}, \eqref{Const:KA}, \eqref{Const:Ku0},
	\eqref{Induc:i}, and
	\eqref{DifftoA}, we obtain
	\begin{align}\label{Q0i2diff}
		\lVert Q_{i+2}(u)-Q_{i+1}(u) \rVert & \leqslant
		Kr^i\cdot
		6K_0^2K_1K_AK_{u_0}\cdot
		|u|^{-2\varepsilon}.
	\end{align}
Substituting \eqref{Rec:Q} into \eqref{Rec:P0} and using
	\eqref{SpecDecomp}, we obtain
	\begin{align}\nonumber
		{P}_{i+1}(u) & =
			A + u^{\mathrm{ad}[A]^q} \int_{\infty}^u
				t^{-2} U_1 t^{-[A]^q}
				\eta_1({P}_i(t),t)
				g_1(P^{(1)}_i(t), P^{(2)}_i(t), t)
				t^{[A]^q}
			\mathrm{d}_qt \\
		\nonumber
		& \quad
			+ u^{\mathrm{ad}[A]^q} \int_{\infty}^u
				t^{-2} t^{-[A]^q}
				\eta_2({P}_i(t),t)
				g_2(P^{(1)}_i(t), P^{(2)}_i(t), t)
				t^{[A]^q} U_2
			\mathrm{d}_qt \\
		\nonumber
		& =
			A + \sum_{j_\ast, k_\ast} (-1)^{k_1+k_4}
			u^{\lambda_{j_3}-\lambda_{j_2}} (\mathrm{ln} u)^{k_3+k_4}\\
		\nonumber
		& \qquad
			\int_{\infty}^u
				t^{ -2+\lambda_{j_2}-\lambda_{j_1}} (\mathrm{ln} t)^{k_1+k_2}
				\frac{\mathscr{P}_{j_3} \mathscr{N}^{k_3}}{k_3!}
				U_1
				\frac{\mathscr{P}_{j_1} \mathscr{N}^{k_1}}{k_1!}
				\eta_1(\ast)
				g_1(\ast)
				\frac{\mathscr{P}_{j_2} \mathscr{N}^{k_2+k_4}}{k_2!k_4!}
			\mathrm{d}_qt \\
		\nonumber
		& \quad
			+ \sum_{j_\ast, k_\ast} (-1)^{k_1+k_3}
			u^{\lambda_{j_1}-\lambda_{j_3}} (\mathrm{ln} u)^{k_3+k_4}\\
		& \qquad
			\int_{\infty}^u
				t^{-2+\lambda_{j_2}-\lambda_{j_1}} (\mathrm{ln} t)^{k_1+k_2}
				\frac{\mathscr{P}_{j_1} \mathscr{N}^{k_1+k_4}}{k_1!k_4!}
				\eta_2(\ast) g_2(\ast) 
				\frac{\mathscr{P}_{j_2} \mathscr{N}^{k_2}}{k_2!}
				U_2 
				\frac{\mathscr{P}_{j_3} \mathscr{N}^{k_3}}{k_3!}
			\mathrm{d}_qt.
		\label{Rec:adP0}
	\end{align}
From \eqref{Rec:adP0}, using
	\eqref{Est:Bound}, \eqref{Est:eta},
	\eqref{Est:g1}, \eqref{Est:g2},
	\eqref{Const:KA}, \eqref{Const:Ku01}, \eqref{Const:Ku02},
	\eqref{Induc:i}, and \eqref{DifftoA}, we obtain
\begin{align}\label{P0i2diff}
	\lVert {P}_{i+2}(u)-{P}_{i+1}(u) \rVert \leqslant
	Kr^i\cdot
	6K_0^2K_1K_AK_{u_0}\cdot
	|u|^{-2\varepsilon}.
\end{align}
From \eqref{Rec:P1} and \eqref{Rec:P2}, using
	\eqref{BasicqInt}, \eqref{Est:fix},
	\eqref{Est:fiy1}, \eqref{Est:fiy2},
	\eqref{Induc:i2}, \eqref{DifftoAplus},
	\eqref{SizeofP}, and \eqref{P0i2diff}, we obtain
	\begin{align}\nonumber
		\lVert P^{(\ast)}_{i+2}(u)-P^{(\ast)}_{i+1}(u) \rVert
		& \leqslant
		\int_\infty^u
			K_1 |t|^{-1} \lVert {P}_{i+2}(t)-{P}_{i+1}(t) \rVert
		|\mathrm{d}_qt| \\
		\nonumber
		& \quad +
		\int_\infty^u
			|t|^{-1}
			(K_{11}\lVert P_{i+1}(t)-A \rVert + K_{12}|t|^{-\varepsilon})
			\lVert P^{(1)}_{i+1}(t)-P^{(1)}_i(t) \rVert
		|\mathrm{d}_qt| \\
		\nonumber
		& \quad +
		\int_\infty^u
			|t|^{-1}
			(K_{11}\lVert P_{i+1}(t)-A \rVert + K_{12}|t|^{-\varepsilon})
			\lVert P^{(2)}_{i+1}(t)-P^{(2)}_i(t) \rVert
		|\mathrm{d}_qt| \\
		& \leqslant
		K r^i\cdot 
		\left(
			6K_0^2K_1^2K_AK_{u_0} +
			2\left(K_{11}\frac{K}{1-r} + K_{12}\right)
		\right) \cdot
		\frac{|1-q|}{|1-|q||}
		\frac{|u|^{-2\varepsilon}}{|[-2\varepsilon]_{|q|}|}.
	\label{P12i2diff}
	\end{align}
Finally, since these constants are independent of $N$,
	we can choose
	$N$ sufficiently large so that
	\begin{align}
		6K_0^2K_1K_AK_{u_0}\cdot N^{-\varepsilon}
		& \leqslant r, \label{Choice:PicardN1}\\
		\left(
			6K_0^2K_1^2K_AK_{u_0} +
			2\left(K_{11}\frac{K}{1-r} + K_{12}\right)
		\right) \cdot
		\frac{|1-q|}{|1-|q||}
		\frac{N^{-\varepsilon}}{ |[-2\varepsilon]_{|q|}| }
		& \leqslant r. \label{Choice:PicardN2}
	\end{align}
By \eqref{Q0i2diff}, \eqref{P0i2diff}, and
	\eqref{Choice:PicardN1}, \eqref{Induc:i1}
	holds with $i$ replaced by $i+1$.
	Similarly, \eqref{P12i2diff} and
	\eqref{Choice:PicardN2} give \eqref{Induc:i2}
	with $i$ replaced by $i+1$.
	This completes the induction.
	
By \eqref{Induc:i}, the sequences
	$Q_i$, $P_i$, $P^{(1)}_i$, and $P^{(2)}_i$
	converge uniformly for $|u|>N$ as $i\to\infty$. Set
	\begin{align}\label{Def:PicardLimits}
		Z_{\hat{n}\hat{n}}(u)
		&= \lim_{i\to\infty} Q_i(u),\qquad
		X_{\hat{n}n}(u)
		= \lim_{i\to\infty} P^{(1)}_i(u),\qquad
		X_{n\hat{n}}(u)
		= \lim_{i\to\infty} P^{(2)}_i(u).
	\end{align}
	Equation \eqref{Rec:P0} gives
	$\underset{i\to\infty}{\lim}P_i(u)
	=u^{\mathrm{ad}[A]^q}Z_{\hat{n}\hat{n}}(u)$.
The estimates above justify passing to the limit in
	\eqref{Rec:Q}, \eqref{Rec:P1}, and \eqref{Rec:P2},
	which give \eqref{Lem:Eq11}, \eqref{Lem:Eq12},
	and \eqref{Lem:Eq21}, respectively.
	Summing \eqref{Induc:i} gives \eqref{Lim:Picard}.
	
Under the additional assumptions, \eqref{Init:Picard} and
	\eqref{Rec:} show inductively that the iterates
\begin{align*}
	(Q_i(u),P_i(u),P^{(1)}_i(u),P^{(2)}_i(u))
\end{align*}
	depend analytically on the initial data
	$([A]^q,B_{\hat{n}n},B_{n\hat{n}})$.
The strict inequality \eqref{Cond:PicardGap} and the local
	uniformity of \eqref{Est:PicardConditions} allow the constants
	in \eqref{Base:Picard}, \eqref{Q0i2diff},
	\eqref{P0i2diff}, and \eqref{P12i2diff}, as well as $N$ in
	\eqref{Choice:PicardN1} and \eqref{Choice:PicardN2},
	to be chosen locally uniformly with respect to these data.
Hence, the convergence in \eqref{Def:PicardLimits} is locally
	uniform, and the limiting solution depends analytically on them.
\end{prf}

\noindent
Although the Picard iteration in
Lemma~\ref{Lem:Picard} is carried out on a single $q$-spiral,
the prescribed asymptotic behavior determines a natural analytic
extension of the resulting solution to a neighborhood of infinity
in $\widetilde{\mathbb{C}^{\ast}}$.

We now apply the Picard iteration of Lemma~\ref{Lem:Picard}
to construct the $u_n$-dependent part of the recursive step
from level $(n-1,n)$ to level $(n,n)$.

\begin{lem}\label{Lem:RecStep1}
Suppose that $\Phi_{n-1}$ satisfies
\eqref{PhikInvert} and
and the shrinking condition \eqref{Cond:less1}.
Choose $\alpha_1,\ldots,\alpha_n$ such that
\begin{subequations}\label{Lem:Step1cond}
\begin{align}
\label{Lem:Step1cond2}
\left\{
	u_1^{-1}q^{\alpha_1},\ldots,
	u_{n-1}^{-1}q^{\alpha_{n-1}}
\right\}
&=
\mathrm{Spec}\left(
	U_{\hat n\hat n}^{-1}
	\bigl(I+(q-1)(\Phi_{n-1})_{\hat n\hat n}\bigr)
\right),\\
\label{Lem:Step1cond3}
q^{\alpha_n}
&=
\frac{
	\det\bigl(I+(q-1)\Phi_{n-1}\bigr)
}{
	\det\bigl(I+(q-1)\Phi_{n-1}^{[n-1]}\bigr)
}.
\end{align}
\end{subequations}
Fix a branch of
$[\Phi_{n-1}^{[n-1]}]^q$, and set
\[
	[\delta_{n-1}\Phi_{n-1}]^q
	:=
	\begin{pmatrix}
		[\Phi_{n-1}^{[n-1]}]^q&0\\
		0&\alpha_n
	\end{pmatrix}.
\]
Then there exists a unique function
$\Phi_n(u_n)$, analytic in $u_n$ on a neighborhood of $\infty$
in $\widetilde{\mathbb C^\ast}$, such that
\begin{subequations}
\begin{align}
\label{Picardk}
T_n\Phi_n(u_n)
&=
V_{n;n}^{(n)}(u_n;\Phi_n(u_n),\alpha_n) \cdot
\Phi_n(u_n) \cdot
V_{n;n}^{(n)}(u_n;\Phi_n(u_n),\alpha_n)^{-1},\\
\label{Lim:adPhin}
u_n^{\operatorname{ad}[\delta_{n-1}\Phi_{n-1}]^q}
\Phi_n(u_n)
&=
\Phi_{n-1} + O(u_n^{-\varepsilon_{n-1}}),\qquad
u_n\to\infty,\\
\label{Lem:Recalpha}
\left\{
	u_1^{-1}q^{\alpha_1},\ldots,
	u_n^{-1}q^{\alpha_n}
\right\}
&=
\mathrm{Spec}\left(
	U^{-1}\bigl(I+(q-1)\Phi_n(u_n)\bigr)
\right).
\end{align}
\end{subequations}
Fix also a branch of $[\Phi_{n-1}]^q$.
Then for any $W_{n-1}\in\mathrm{GL}_n(\mathbb C)$,
there exist a unique function $W_n(u_n)$,
analytic in $u_n$ on a neighborhood of $\infty$ in
$\widetilde{\mathbb C^\ast}$, and a unique branch
$[\Phi_n(u_n)]^q$ such that
\begin{subequations}
\begin{align}
\label{PicardWn}
T_nW_n(u_n)
&=
V_{n;n}^{(n)}(u_n;\Phi_n(u_n),\alpha_n) \cdot W_n(u_n),\\
\label{Lim:adWn}
u_n^{[\delta_{n-1}\Phi_{n-1}]^q}
u_n^{-[\Phi_n(u_n)]^q}
W_n(u_n)
&=
W_{n-1} + O(u_n^{-\varepsilon_{n-1}}),\qquad
u_n\to\infty,\\
\label{Lem:DiagW}
W_n(u_n)^{-1} \cdot [\Phi_n(u_n)]^q \cdot W_n(u_n)
&=
W_{n-1}^{-1} \cdot [\Phi_{n-1}]^q \cdot W_{n-1}.
\end{align}
\end{subequations}
Moreover, the pair
$(\Phi_n(u_n),W_n(u_n))$
depends analytically on the initial data
$(\Phi_{n-1},W_{n-1})$.
\end{lem}

\begin{prf}
Denote $[A]^q=[\Phi_{n-1}^{[n-1]}]^q$.
To construct a solution of \eqref{Picardk}, set
\begin{subequations}
	\begin{alignat}{2}
		Z_{\hat{n}\hat{n}}(u_n) & =
			\Phi_n(u_n)_{\hat{n}\hat{n}},& \quad
		X_{\hat{n}\hat{n}}(u_n) & =
			u_n^{\mathrm{ad}[A]^q}\Phi_n(u_n)_{\hat{n}\hat{n}}, \\
		X_{\hat{n}n}(u_n) & =
			u_n^{[A]^q-\alpha_n I}\Phi_n(u_n)_{\hat{n}n},& \quad
		X_{n\hat{n}}(u_n) & =
			\Phi_n(u_n)_{n\hat{n}}u_n^{\alpha_n I-[A]^q}.
	\end{alignat}
\end{subequations}
The condition
	$u_n^{-1}q^{\alpha_n}\in
	\mathrm{Spec}(U^{-1}(I+(q-1)\Phi_n(u_n)))$
	is equivalent to
	\begin{align}\nonumber
		\varphi_{nn}(u_n) & = [\alpha_n]_q +
		\Phi(u_n)_{n\hat{n}}
		\frac{q-1}
		{I+(q-1) \Phi(u_n)_{\hat{n}\hat{n}} -
		q^{\alpha_n}\frac{U_{\hat{n}\hat{n}}}{u_n} }
		\Phi(u_n)_{\hat{n}n} \\
		\label{Def:pnn}
		& = [\alpha_n]_q +
		X_{n\hat{n}}(u_n)
		\frac{q-1}
		{I+(q-1) X_{\hat{n}\hat{n}}(u_n) -
		q^{\alpha_n} u_n^{\mathrm{ad}[A]^q}
		\frac{U_{\hat{n}\hat{n}}}{u_n}}
		X_{\hat{n}n}(u_n).
	\end{align}
Substituting \eqref{Def:pnn} into \eqref{Picardk}
	eliminates $\varphi_{nn}(u_n)$, and the remaining three block
	equations form the closed system
	\begin{subequations}\label{Lem:PhiEq}
	\begin{align}\label{Lem:PhiEq11}
	\frac{\partial_q}{\partial_q u_n}
	Z_{\hat{n}\hat{n}} & =
		u_n^{-2}
		f  (Z_{\hat{n}\hat{n}}, X_{\hat{n}n}, X_{n\hat{n}}, u_n), \\
	\frac{\partial_q}{\partial_q u_n}
	X_{\hat{n}n} & =
		u_n^{-1}
		f_1(X_{\hat{n}\hat{n}}, X_{\hat{n}n}, X_{n\hat{n}}, u_n), \\
	\label{Lem:PhiEq21}
	\frac{\partial_q}{\partial_q u_n}
	X_{n\hat{n}} & =
		u_n^{-1}
		f_2(X_{\hat{n}\hat{n}}, X_{\hat{n}n}, X_{n\hat{n}}, u_n).
	\end{align}
	\end{subequations}
	Here
	\begin{subequations}
	\begin{align}\notag
		f( Z_{\hat{n}\hat{n}}, X_{\hat{n}n}, X_{n\hat{n}}, u_n )
			& = U_{\hat{n}\hat{n}} \cdot
			h(Z_{\hat{n}\hat{n}}, u_n) u_n^{-[A]^q}
			g(X_{\hat{n}n}, X_{n\hat{n}}, u_n) u_n^{[A]^q} \\
		& \quad -
			h(Z_{\hat{n}\hat{n}}, u_n) u_n^{-[A]^q}
			g(X_{\hat{n}n}, X_{n\hat{n}}, u_n) u_n^{[A]^q}\cdot
			q^{-1}U_{\hat{n}\hat{n}},\\
		f_1( X_{\hat{n}\hat{n}}, X_{\hat{n}n}, X_{n\hat{n}}, u_n )
			& =
			( A - X_{\hat{n}\hat{n}} +
			([\alpha_n]_qI-A)\tilde{U}(u_n) )
			\frac{1}{
				I+(q-1)X_{\hat{n}\hat{n}} - q^{\alpha_n} \tilde{U}(u_n)
			}
			X_{\hat{n}n}, \\
		\nonumber
		f_2( X_{\hat{n}\hat{n}}, X_{\hat{n}n}, X_{n\hat{n}}, u_n )
			& = X_{n\hat{n}}
			\frac{1}{I - q^{-1} \tilde{U}(u_n)}
			\Bigg(
			( X_{\hat{n}\hat{n}}-A + 
			q^{-1}\tilde{U}(u_n) ( A-[\alpha_n]_qI ) )
        	\\
        	& \hspace{1em} +
        	(q-1)
        	u_n^{-1}u_n^{\mathrm{ad}[A]^q}
        	f( Z_{\hat{n}\hat{n}}, X_{\hat{n}n}, X_{n\hat{n}}, u_n )
    		\Bigg) q^{-[A]^q},
	\end{align}
	\end{subequations}
	where we denote
	$\tilde{U}(u_n) =
	u_n^{-1}u_n^{\mathrm{ad}[A]^q} U_{\hat{n}\hat{n}}$
	and
	\begin{align}
		h(Z_{\hat{n}\hat{n}}, u_n) & =
        	\frac{1}{
        	I+(q-1) Z_{\hat{n}\hat{n}}
        	- q^{\alpha_n} \frac{U_{\hat{n}\hat{n}}}{u_n} }, \\
        g(X_{\hat{n}n}, X_{n\hat{n}}, u_n) & =
        	X_{\hat{n}n} X_{n\hat{n}}
        	\frac{1}{I-q^{-1} \tilde{U}(u_n)}.
	\end{align}
Denote
	$\eta(X_{\hat{n}\hat{n}}, u_n) :=
	 u_n^{\mathrm{ad}[A]^q}
	 h(Z_{\hat{n}\hat{n}}, u_n)$.
Since
	\begin{align}\nonumber
		u_n^{\mathrm{ad}[A]^q}
		f( Z_{\hat{n}\hat{n}}, X_{\hat{n}n}, X_{n\hat{n}}, u_n )
			& = u_n^{[A]^q} U_{\hat{n}\hat{n}} u_n^{-[A]^q} \cdot
			\eta(X_{\hat{n}\hat{n}}, u_n)
			g(X_{\hat{n}n}, X_{n\hat{n}}, u_n) \\
		& \quad -
			\eta(X_{\hat{n}\hat{n}}, u_n)
			g(X_{\hat{n}n}, X_{n\hat{n}}, u_n) \cdot
			q^{-1} u_n^{[A]^q} U_{\hat{n}\hat{n}} u_n^{-[A]^q},
	\end{align}
the choices
	$U_1=U_{\hat n\hat n}$, $U_2=-q^{-1}U_{\hat n\hat n}$,
	$h_1=h_2=h$, $\eta_1=\eta_2=\eta$, and $g_1=g_2=g$
	identify this expression with \eqref{Def:Picardf}.
For the remaining hypotheses of Lemma \ref{Lem:Picard},
	equation \eqref{PhikInvert} gives the invertibility of $I+(q-1)A$,
	while the shrinking condition \eqref{Cond:less1} gives
	$\widetilde U(u_n)=O(u_n^{-\varepsilon_{n-1}})$.
The formulas above then satisfy \eqref{Est:PicardConditions}.
Hence there exists a unique solution
$(Z_{\hat n\hat n}(u_n),
X_{\hat n n}(u_n),
X_{n\hat n}(u_n))$
of equations \eqref{Lem:PhiEq} such that
\begin{subequations}\label{Eq:lem43}
\begin{alignat}{2}
Z_{\hat n\hat n}(u_n) &=
	\Phi_{n-1}^{[n-1]}
	+O(u_n^{-\varepsilon_{n-1}}),\qquad &
X_{\hat n\hat n}(u_n) &=
	\Phi_{n-1}^{[n-1]}
	+O(u_n^{-\varepsilon_{n-1}}),\\
X_{\hat n n}(u_n) &=
	(\Phi_{n-1})_{\hat n n}
	+O(u_n^{-\varepsilon_{n-1}}),\qquad &
X_{n\hat n}(u_n) &=
	(\Phi_{n-1})_{n\hat n}
	+O(u_n^{-\varepsilon_{n-1}}).
\end{alignat}
\end{subequations}
Together with \eqref{Def:pnn}, this solution gives
	$\Phi_n(u_n)$ satisfying \eqref{Picardk}.
The estimates above verify \eqref{Lim:adPhin} in every block except
	the $(n,n)$-block.
For the remaining block,
	\eqref{Lem:Step1cond3} gives
	\begin{align*}
		(\Phi_{n-1})_{nn} = [\alpha_n]_q +
		(\Phi_{n-1})_{n\hat{n}}
		\frac{q-1}{I+(q-1)\Phi_{n-1}^{[n-1]}}
		(\Phi_{n-1})_{\hat{n}n}.
	\end{align*}
Combining \eqref{Def:pnn}, \eqref{Lem:Step1cond3}, and
\eqref{Eq:lem43} with
$\widetilde U(u_n)=O(u_n^{-\varepsilon_{n-1}})$,
we obtain
\[
\varphi_{nn}(u_n)
=
(\Phi_{n-1})_{nn}
+
O(u_n^{-\varepsilon_{n-1}}).
\]
Thus \eqref{Lim:adPhin} holds.

Set
\[
	\widehat{\chi}(x,u_n)
	:=
	(x-u_n^{-1}q^{\alpha_n})^{-1}
	\mathrm{det}
	\left(
		xI-U^{-1}(I+(q-1)\Phi_n(u_n))
	\right).
\]
By \eqref{Def:pnn}, this is a monic polynomial in $x$ of degree
	$n-1$.
Equations \eqref{Picardk} and \eqref{Def:pnn}, together with the
	definition of $V_{n;n}^{(n)}$, give the matrix polynomial identity
\begin{align*}
\left(
		T_nU x+I+(q-1)T_n\Phi_n
	\right)
	\left(
		q^{-\alpha_n}UE_nx+V_{n;n}^{(n)}
	\right) =
	\left(
		q^{-\alpha_n}UE_nqx+V_{n;n}^{(n)}
	\right)
	\left(
		Ux+I+(q-1)\Phi_n
	\right).
\end{align*}
Taking determinants in this identity gives
$\widehat{\chi}(x,qu_n)=\widehat{\chi}(x,u_n)$.
Hence $\widehat{\chi}(x,u_n)$ is
	a $q$-constant with respect to $u_n$.
The block determinant formula,
the shrinking condition \eqref{Cond:less1}, and
	\eqref{Lim:adPhin} give
\[
	\lim_{u_n\to\infty}\widehat{\chi}(x,u_n)
	=
	\mathrm{det}\left(
		xI-U_{\hat n\hat n}^{-1}
		(I+(q-1)\Phi_{n-1}^{[n-1]})
	\right).
\]
Since $\widehat{\chi}(x,u_n)$ is a $q$-constant, the displayed
	limit shows that it is in fact independent of $u_n$ and equal to
	the determinant on the right-hand side.
By \eqref{Lem:Step1cond2} and the definition of
	$\widehat{\chi}$, this proves \eqref{Lem:Recalpha}.
	
By \eqref{Lim:adPhin}, the fixed branch
$[\Phi_{n-1}]^q$ determines a unique branch
$[\Phi_n(u_n)]^q$ such that
\begin{align}\label{Lim:adqPhin}
	\lim_{u_n\to\infty}
	u_n^{\operatorname{ad}[\delta_{n-1}\Phi_{n-1}]^q}
	[\Phi_n(u_n)]^q
	&=
	[\Phi_{n-1}]^q.
\end{align}
Moreover, \eqref{Picardk} and the uniqueness of this branch give
\begin{align}\label{TransPq}
T_n[\Phi_n(u_n)]^q
	=
	V_{n;n}^{(n)}(u_n;\Phi_n(u_n),\alpha_n) \cdot
	[\Phi_n(u_n)]^q \cdot
	V_{n;n}^{(n)}(u_n;\Phi_n(u_n),\alpha_n)^{-1}.
\end{align}
Denote $X(u_n) = u_n^{ [\delta_{n-1}\Phi_{n-1}]^q }
	u_n^{ -[\Phi_n(u_n)]^q }
	W_n(u_n)$.
Taking $\varepsilon=\varepsilon_{n-1}$ in
\eqref{Lim:Picard},
and using the shrinking condition
	\eqref{Cond:less1},
	\eqref{Def:pnn}, \eqref{Lem:Step1cond3},
	and the definition of $V_{n;n}^{(n)}$ together
	with \eqref{Exp:Vk} give
\begin{align}\label{Eq:useful}
	u_n^{\mathrm{ad}[\delta_{n-1}\Phi_{n-1}]^q}
	\left(
		q^{[\delta_{n-1}\Phi_{n-1}]^q}
		V_{n;n}^{(n)}(u_n)
	\right)
	&=
	I+(q-1)\Phi_{n-1}+O(u_n^{-\varepsilon_{n-1}}),\\
	u_n^{\mathrm{ad}[\delta_{n-1}\Phi_{n-1}]^q}
	q^{[\Phi_n(u_n)]^q}
	&=
	I+(q-1)\Phi_{n-1}+O(u_n^{-\varepsilon_{n-1}}).
\end{align}
Moreover, \eqref{PhikInvert} and \eqref{Lem:Step1cond3}
show that the common leading term is invertible.
Hence \eqref{PicardWn} can be rewritten as
\begin{align}\nonumber
	\frac{\partial_q}{\partial_q u_n} X(u_n)
	&=
	u_n^{\mathrm{ad}[\delta_{n-1}\Phi_{n-1}]^q}
	\left(
	\frac{
			q^{[\delta_{n-1}\Phi_{n-1}]^q}V_{n;n}^{(n)}(u_n)
			-q^{[\Phi_n(u_n)]^q}
		}{
			(q-1)u_n
		}
		q^{-[\Phi_n(u_n)]^q}
	\right)X(u_n)\\
	&=
	O(u_n^{-1-\varepsilon_{n-1}})\cdot X(u_n).
	\label{Eq:AdWnun}
\end{align}
By a similar Picard iteration argument, for any
$W_{n-1}\in\mathrm{GL}_n(\mathbb C)$,
there exists a unique solution $X(u_n)$ of
\eqref{Eq:AdWnun} such that
$X(u_n) = W_{n-1} + O(u_n^{-\varepsilon_{n-1}})$.
Equivalently, there exists a unique $W_n(u_n)$ satisfying
\eqref{PicardWn} and \eqref{Lim:adWn}.

Equations \eqref{TransPq} and \eqref{PicardWn} show that
$W_n(u_n)^{-1}[\Phi_n(u_n)]^qW_n(u_n)$
is a $q$-constant with respect to $u_n$.
By the definition of $X(u_n)$,
\[
	W_n(u_n)^{-1}[\Phi_n(u_n)]^qW_n(u_n)
	=
	X(u_n)^{-1}
	\left(
		u_n^{\operatorname{ad}[\delta_{n-1}\Phi_{n-1}]^q}
		[\Phi_n(u_n)]^q
	\right)
	X(u_n).
\]
Taking the limit in this identity and applying
\eqref{Lim:adqPhin} and \eqref{Lim:adWn}, we obtain
\[
	W_n(u_n)^{-1}[\Phi_n(u_n)]^qW_n(u_n)
	=
	W_{n-1}^{-1}[\Phi_{n-1}]^qW_{n-1},
\]
which proves \eqref{Lem:DiagW}.
The analytic dependence follows directly from
Lemma~\ref{Lem:Picard}.
\end{prf}

\subsection{Compatibility of the Recursive Construction}
\label{Sect:RecCompatibility}

In this subsection, we show that the recursive construction in the
$u_n$-direction given by Lemma~\ref{Lem:RecStep1} is compatible with
the $q$-isomonodromy equations in the other directions.
The technical estimate in Lemma~\ref{Lem:RecImprotant} allows us to
pass from a $q$-isomonodromic deformation at level $(n-1,n)$ to one
at level $(n,n)$. Extending this step to arbitrary levels and
iterating the construction, we obtain the desired
$q$-isomonodromic deformation.

The following technical lemma will be used to show that the solution
constructed in the $u_n$-direction in Lemma~\ref{Lem:RecStep1} also
satisfies the $q$-isomonodromy equations in the other directions.

\begin{lem}\label{Lem:RecImprotant}
Let $(\Phi_n(u_n),W_n(u_n);\boldsymbol{\alpha})$ be the triple
derived in Lemma \ref{Lem:RecStep1}.
If $\varepsilon_{n-1} > \frac{1}{2}$,
then
the following limit holds
for each $1\leqslant k \leqslant n-1$:
\begin{align}\label{Lim:RecImportant}
	\lim_{u_n\to\infty}
	u_n^{\mathrm{ad}[\delta_{n-1}\Phi_{n-1}]^q}
	\left(
		V_{n-1;k}^{(n)}(u;\Phi_{n-1})^{-1}
		V_{n;k}^{(n)}(u;\Phi_{n}(u_n))
	\right) = I.
\end{align}
\end{lem}

\begin{prf}
Denote $[A]^q=[\Phi_{n-1}^{[n-1]}]^q$.
By \eqref{Lim:Picard} and the definitions of
	$X_{\hat n n}$ and $X_{n\hat n}$ in the proof of
	Lemma \ref{Lem:RecStep1}, we have
\begin{equation}\label{Est:RecPicardData}
\begin{alignedat}{2}
\Phi_n(u_n)_{\hat n\hat n}
	&=A+O(u_n^{-\varepsilon}), &
\Phi_n(u_n)_{nn}&=O(1), \\
X_{\hat n n}
	:=u_n^{[A]^q-\alpha_nI}\Phi_n(u_n)_{\hat n n}
	&=O(1), &
X_{n\hat n}
	:=\Phi_n(u_n)_{n\hat n}u_n^{\alpha_nI-[A]^q}
	&=O(1).
\end{alignedat}
\end{equation}
Write
\[
	\widehat{kn}:=\{1,\ldots,n\}\setminus\{k,n\}.
\]
Since we choose $\varepsilon$ such that
	$\frac12<\varepsilon<\varepsilon_{n-1}$,
it is enough to prove the following five estimates:
\begin{subequations}\label{Est:RecVEntries}
\begin{align}
V_{n;k}^{(n)}(u;\Phi_n)_{\widehat{kn}k}
&=
V_{n-1;k}^{(n-1)}(u;\Phi_{n-1})_{\hat k k}
+
O(u_n^{-\varepsilon}),
\label{Est:RecVCol}
\\
V_{n;k}^{(n)}(u;\Phi_n)_{k\widehat{kn}}
&=
V_{n-1;k}^{(n-1)}(u;\Phi_{n-1})_{k\hat k}
+
O(u_n^{-\varepsilon}),
\label{Est:RecVRow}
\\
V_{n;k}^{(n)}(u;\Phi_n)_{kk}
&=
V_{n-1;k}^{(n-1)}(u;\Phi_{n-1})_{kk}
+
O(u_n^{-\varepsilon}),
\label{Est:RecVkk}
\\
u_n^{[A]^q-\alpha_nI}
\left(
V_{n-1;k}^{(n-1)}(u;\Phi_{n-1})
\right)^{-1}
E_{kk}
V_{n;k}^{(n)}(u;\Phi_n)_{\hat n n}
&=
O(u_n^{-\varepsilon}),
\label{Est:RecVUpper}
\\
V_{n;k}^{(n)}(u;\Phi_n)_{n\hat n}
E_{kk}
u_n^{\alpha_nI-[A]^q}
&=
O(u_n^{-\varepsilon}).
\label{Est:RecVLower}
\end{align}
\end{subequations}
By the shrinking condition \eqref{Cond:less1}
	and \eqref{Est:RecPicardData},
	for every bounded matrix $C=O(1)$, we have
\begin{subequations}\label{Est:RecPairing}
\begin{align}
u_n^{[A]^q-\alpha_nI}
C\Phi_n(u_n)_{\hat n n}
&=
\left(
u_n^{\mathrm{ad}[A]^q}C
\right)
X_{\hat n n}
=
O(u_n^{1-\varepsilon}),
\label{Est:RecPairingUpper}
\\
\Phi_n(u_n)_{n\hat n}
C u_n^{\alpha_nI-[A]^q}
&=
X_{n\hat n}
\left(
u_n^{\mathrm{ad}[A]^q}C
\right)
=
O(u_n^{1-\varepsilon}),
\label{Est:RecPairingLower}
\\
\Phi_n(u_n)_{n\hat n}
C\Phi_n(u_n)_{\hat n n}
&=
X_{n\hat n}
\left(
u_n^{\mathrm{ad}[A]^q}C
\right)
X_{\hat n n}
=
O(u_n^{1-\varepsilon}),
\label{Est:RecPairingCross}\\
\Phi_n(u_n)_{\hat n n}
\Phi_n(u_n)_{n\hat n}
&=
u_n^{-\operatorname{ad}[A]^q}
\left(X_{\hat n n}X_{n\hat n}\right)
=
O(u_n^{1-\varepsilon}).
\label{Est:RecPairingCross2}
\end{align}
\end{subequations}
We shall repeatedly use \eqref{Est:RecPairing}
	to establish \eqref{Est:RecVEntries}.

Write the denominator in
	\eqref{Exp:Vk12} in the block form
\begin{equation*}
\begin{pmatrix}
B_{\hat n\hat n}&B_{\hat n n}\\
B_{n\hat n}&b_{nn}
\end{pmatrix}
:=
\left(
\begin{array}{c|c}
q^{-\alpha_k}u_k
\left(
I+(q-1)
\Phi_n(u_n)_{\widehat{kn}\widehat{kn}}
\right)
-U_{\widehat{kn}\widehat{kn}}
&
(q-1)q^{-\alpha_k}u_k
\Phi_n(u_n)_{\widehat{kn}n}
\\[4pt]
\hline
(q-1)q^{-\alpha_k}u_k
\Phi_n(u_n)_{n\widehat{kn}}
&
q^{-\alpha_k}u_k
\left(
1+(q-1)\varphi_{nn}(u_n)
\right)
-u_n
\end{array}
\right).
\end{equation*}
Here $B_{\hat n\hat n}$ is indexed by $\widehat{kn}$.
The block inversion formula gives
\begin{align}\label{Def:RecSchur}
S_{nn}
&:=
b_{nn}
-
B_{n\hat n}
B_{\hat n\hat n}^{-1}
B_{\hat n n}\\
\label{Inv:RecVkDen}
\begin{pmatrix}
B_{\hat n\hat n}&B_{\hat n n}\\
B_{n\hat n}&b_{nn}
\end{pmatrix}^{-1}
&=
\begin{pmatrix}
B_{\hat n\hat n}^{-1}
+
B_{\hat n\hat n}^{-1}
B_{\hat n n}
S_{nn}^{-1}
B_{n\hat n}
B_{\hat n\hat n}^{-1}
&
-
B_{\hat n\hat n}^{-1}
B_{\hat n n}
S_{nn}^{-1}
\\[4pt]
-
S_{nn}^{-1}
B_{n\hat n}
B_{\hat n\hat n}^{-1}
&
S_{nn}^{-1}
\end{pmatrix}.
\end{align}
By \eqref{Est:RecPicardData},
\begin{align}\label{Eq:Bhatnhatn}
B_{\hat n\hat n}^{-1}
&=
\left(
q^{-\alpha_k}u_k
\left(
I+(q-1)A_{\widehat{kn}\widehat{kn}}
\right)
-U_{\widehat{kn}\widehat{kn}}
\right)^{-1}
+
O(u_n^{-\varepsilon})
=
O(1).
\end{align}
Hence \eqref{Est:RecPairingCross}
and $b_{nn}=-u_n+O(1)$ give
\begin{equation}\label{Est:RecSchur}
S_{nn}
=
-u_n + O(u_n^{1-\varepsilon}),
\qquad
S_{nn}^{-1}
=
O(u_n^{-1}).
\end{equation}
We first prove \eqref{Est:RecVUpper} and
	\eqref{Est:RecVLower}.

By \eqref{Exp:Vkkhatk}, \eqref{Exp:Vk12},
and \eqref{Inv:RecVkDen},
\begin{align}
V_{n;k}^{(n)}(u;\Phi_n)_{\hat n n}
&=
E_{kk}\Phi_n(u_n)_{\hat n n}
\frac{(q-1)q^{-\alpha_k}u_k}
{u_k-q^{-1}u_n},
\label{Eq:Vhatknn}\\
\label{Vk:Recnk}
V_{n;k}^{(n)}(u;\Phi_n)_{nk}
&=
(q-1)q^{-\alpha_k}u_k
S_{nn}^{-1}
\left(
\Phi_n(u_n)_{nk}
-
B_{n\hat n}
B_{\hat n\hat n}^{-1}
\Phi_n(u_n)_{\widehat{kn}k}
\right).
\end{align}
Since $(u_k-q^{-1}u_n)^{-1}=O(u_n^{-1})$,
	\eqref{Est:RecPairingUpper} gives
	\eqref{Est:RecVUpper}.
Equations \eqref{Exp:Vkhatkk} and \eqref{Vk:Recnk}
allow us to choose a matrix $C$, indexed by $\hat n$, with
\[
C_{\widehat{kn}k}
=
-(q-1)q^{-\alpha_k}u_k
B_{\hat n\hat n}^{-1}
\Phi_n(u_n)_{\widehat{kn}k},
\qquad
C_{kk}=1,
\]
and all its other entries equal to zero, such that
\[
V_{n;k}^{(n)}(u;\Phi_n)_{n\hat n}E_{kk}
=
(q-1)q^{-\alpha_k}u_k
S_{nn}^{-1}
\Phi_n(u_n)_{n\hat n}C.
\]
Since \eqref{Est:RecPicardData} and \eqref{Eq:Bhatnhatn}
show that $C=O(1)$,
	\eqref{Est:RecPairingLower} and \eqref{Est:RecSchur}
yield \eqref{Est:RecVLower}.

By \eqref{Exp:Vkkhatk}, we have
\begin{align*}
V_{n;k}^{(n)}(u;\Phi_n)_{k\widehat{kn}}
&=
q^{-\alpha_k}
\Phi_n(u_n)_{k\widehat{kn}}
\frac{(q-1)u_k}
{u_kI-q^{-1}U_{\widehat{kn}\widehat{kn}}}.
\end{align*}
Thus \eqref{Est:RecPicardData} gives \eqref{Est:RecVRow}.
By \eqref{Exp:Vk12} and \eqref{Inv:RecVkDen}, we have
\begin{align}\nonumber
V_{n;k}^{(n)}(u;\Phi_n)_{\widehat{kn}k}
&=
(q-1)q^{-\alpha_k}u_k
B_{\hat n\hat n}^{-1}
\Phi_n(u_n)_{\widehat{kn}k}
\\
&\quad+
(q-1)q^{-\alpha_k}u_k
B_{\hat n\hat n}^{-1}B_{\hat n n}S_{nn}^{-1}
\left(
B_{n\hat n}B_{\hat n\hat n}^{-1}
\Phi_n(u_n)_{\widehat{kn}k}
-
\Phi_n(u_n)_{nk}
\right).
\label{Eq:Vhatknk}
\end{align}
By \eqref{Est:RecPairingCross2},
\[
B_{\hat n n}B_{n\hat n}
=
O(u_n^{1-\varepsilon}),
\qquad
B_{\hat n n}\Phi_n(u_n)_{nk}
=
O(u_n^{1-\varepsilon}).
\]
Equations \eqref{Eq:Bhatnhatn},
	\eqref{Est:RecPicardData}, and \eqref{Est:RecSchur}
then show that the second term in \eqref{Eq:Vhatknk}
is $O(u_n^{-\varepsilon})$.
Moreover, \eqref{Est:RecPicardData},
	\eqref{Eq:Bhatnhatn}, and \eqref{Exp:Vk12} give
\[
(q-1)q^{-\alpha_k}u_k
B_{\hat n\hat n}^{-1}
\Phi_n(u_n)_{\widehat{kn}k}
=
V_{n-1;k}^{(n-1)}(u;\Phi_{n-1})_{\hat k k}
+
O(u_n^{-\varepsilon}).
\]
This proves \eqref{Est:RecVCol}.

Finally, \eqref{Eq:Vhatknn} and \eqref{Vk:Recnk} give
\begin{align*}
V_{n;k}^{(n)}(u;\Phi_n)_{kn}
V_{n;k}^{(n)}(u;\Phi_n)_{nk}
=
\frac{((q-1)q^{-\alpha_k}u_k)^2}
{u_k-q^{-1}u_n}
\Phi_n(u_n)_{kn}S_{nn}^{-1}
\left(
\Phi_n(u_n)_{nk}
-
B_{n\hat n}B_{\hat n\hat n}^{-1}
\Phi_n(u_n)_{\widehat{kn}k}
\right).
\end{align*}
By \eqref{Est:RecPairingCross2},
\[
\Phi_n(u_n)_{kn}\Phi_n(u_n)_{nk}
=
O(u_n^{1-\varepsilon}),
\qquad
\Phi_n(u_n)_{kn}B_{n\hat n}
=
O(u_n^{1-\varepsilon}).
\]
Together with \eqref{Eq:Bhatnhatn},
	\eqref{Est:RecPicardData}, and \eqref{Est:RecSchur},
these estimates imply
\[
V_{n;k}^{(n)}(u;\Phi_n)_{kn}
V_{n;k}^{(n)}(u;\Phi_n)_{nk}
=
O(u_n^{-1-\varepsilon}).
\]
Equations \eqref{Exp:Vkkk},
\eqref{Est:RecVRow}, and \eqref{Est:RecVCol} now give
\begin{align*}
V_{n;k}^{(n)}(u;\Phi_n)_{kk}
&=
1+
V_{n;k}^{(n)}(u;\Phi_n)_{k\widehat{kn}}
V_{n;k}^{(n)}(u;\Phi_n)_{\widehat{kn}k}
+
V_{n;k}^{(n)}(u;\Phi_n)_{kn}
V_{n;k}^{(n)}(u;\Phi_n)_{nk}
\\
&=
1+
V_{n-1;k}^{(n-1)}(u;\Phi_{n-1})_{k\hat k}
V_{n-1;k}^{(n-1)}(u;\Phi_{n-1})_{\hat k k}
+
O(u_n^{-\varepsilon})
\\
&=
V_{n-1;k}^{(n-1)}(u;\Phi_{n-1})_{kk}
+
O(u_n^{-\varepsilon}).
\end{align*}
Thus \eqref{Est:RecVkk} follows.
\end{prf}

\begin{lem}\label{Lem:Phin}
Suppose that
	$(\Phi_{n-1}(u),W_{n-1}(u);\boldsymbol{\alpha})$
	is a $q$-isomonodromic deformation at level $(n-1,n)$
	and that $\Phi_{n-1}(u)$ satisfies \eqref{PhikInvert}
	and the shrinking condition \eqref{Cond:less1}.
Fix branches of
$[\Phi_{n-1}^{[n-1]}]^q$ and $[\Phi_{n-1}]^q$.
Then there exists a unique $q$-isomonodromic deformation
$(\Phi_n(u),W_n(u);\boldsymbol{\alpha})$,
together with a unique branch $[\Phi_n(u)]^q$, such that
\begin{subequations}\label{Lem:PWnm1}
	\begin{align}
		\label{Lem:Pnm1}
		\lim_{u_n\to\infty}
		u_n^{ \mathrm{ad}[\delta_{n-1}\Phi_{n-1}]^q }
		[\Phi_n(u)]^q & =
		[\Phi_{n-1}(u)]^q, \\
		\label{Lem:Wnm1}
		\lim_{u_n\to\infty}
		u_n^{ [\delta_{n-1}\Phi_{n-1}]^q }
		u_n^{ -[\Phi_n]^q }
		W_n(u) & =
		W_{n-1}(u).
	\end{align}
\end{subequations}
Moreover, the pair
$(\Phi_n(u),W_n(u))$
depends analytically on the initial data
$(\Phi_{n-1}(u),W_{n-1}(u))$.
\end{lem}

\begin{prf}
Take $(\Phi_{n-1}(u), W_{n-1}(u); \boldsymbol{\alpha})$
	as a $q$-isomonodromic deformation at level $(n-1,n)$.
Let $(\Phi_n(u), W_n(u); \boldsymbol{\alpha})$
	denote the triple obtained from the initial data
	$\Phi_{n-1}(u)$ and $W_{n-1}(u)$ by the above construction.
	
We aim to prove that $T_k\Phi_n(u)$ and
$V_{n;k}^{(n)}(u;\Phi_n(u))\Phi_n(u)
 V_{n;k}^{(n)}(u;\Phi_n(u))^{-1}$
coincide. First, Corollary~\ref{CompCond} shows that both matrices
satisfy the same $q$-difference equation:
\begin{align}\label{Eq:eqTk}
	T_n X(u)
	&=
	V_{n;n}^{(n)}(T_ku;X(u)) \cdot X(u) \cdot
	V_{n;n}^{(n)}(T_ku;X(u))^{-1}.
\end{align}
Lemma~\ref{Lem:Picard} further guarantees that
\eqref{Eq:eqTk} admits a unique solution
$X(u_n)=T_k\Phi_n(u_n)$ satisfying
\begin{align}
	\lim_{u_n\to\infty}\mathrm{Ad}
	\left(
		u_n^{[\delta_{n-1}T_k\Phi_{n-1}]^q}
	\right)X(u_n)
	&=T_k\Phi_{n-1},
	\label{Lim:TkPhin}\\
	\lim_{u_n\to\infty}\mathrm{Ad}
	\left(
		u_n^{[\delta_{n-1}\Phi_{n-1}]^q}
		V_{n-1;k}^{(n)}(u;\Phi_{n-1})^{-1}
	\right)X(u_n)
	&=\Phi_{n-1}.
	\label{Lim:ConjTkPhin}
\end{align}
By Lemma~\ref{Lem:RecImprotant}, specifically
\eqref{Lim:RecImportant}, together with \eqref{Lim:adPhin}, the matrix
$V_{n;k}^{(n)}(u;\Phi_n(u))\Phi_n(u)
 V_{n;k}^{(n)}(u;\Phi_n(u))^{-1}$
satisfies \eqref{Lim:ConjTkPhin}, and hence also
\eqref{Lim:TkPhin}, when
$\varepsilon_{n-1}>\frac{1}{2}$.
Since it also satisfies \eqref{Eq:eqTk}, the uniqueness in
Lemma~\ref{Lem:Picard} gives
\begin{align}\label{Prf:TkP}
	T_k\Phi_n(u)
	&=
	V_{n;k}^{(n)}(u;\Phi_n(u))\Phi_n(u)
	V_{n;k}^{(n)}(u;\Phi_n(u))^{-1}.
\end{align}
By Lemma~\ref{Lem:RecStep1}, both sides of \eqref{Prf:TkP}
depend analytically on $\Phi_{n-1}$. Hence, by analytic continuation
in $\Phi_{n-1}$, the identity \eqref{Prf:TkP} holds under the original
condition $\varepsilon_{n-1}>0$.

Corollary~\ref{CompCond}, together with \eqref{Prf:TkP} and
\eqref{PicardWn}, shows that $T_kW_n(u)$ and
$V_{n;k}^{(n)}(u;\Phi_n(u))W_n(u)$ satisfy
the same $q$-difference equation:
\begin{align*}
	T_nY(u)
	&=
	V_{n;n}^{(n)}(T_ku;T_k\Phi_n(u)) \cdot Y(u).
\end{align*}
Moreover, \eqref{Lim:RecImportant} and \eqref{Lim:adWn} show that
both matrices satisfy
\begin{align*}
	\lim_{u_n\to\infty}
	u_n^{[\delta_{n-1}T_k\Phi_{n-1}]^q}
	u_n^{-[T_k\Phi_n]^q}Y(u)
	&=
	T_kW_{n-1}(u)
\end{align*}
when $\varepsilon_{n-1}>\frac{1}{2}$.
The uniqueness in Lemma~\ref{Lem:RecStep1} therefore gives
\begin{align}\label{Prf:TkW}
	T_kW_n(u)
	&=
	V_{n;k}^{(n)}(u;\Phi_n(u)) \cdot W_n(u).
\end{align}
By Lemma~\ref{Lem:RecStep1}, both sides of \eqref{Prf:TkW}
depend analytically on $\Phi_{n-1}$. Hence analytic continuation
in $\Phi_{n-1}$ extends \eqref{Prf:TkW} to
$\varepsilon_{n-1}>0$.

Together with \eqref{Lem:DiagW} and \eqref{Lem:Recalpha},
equations \eqref{Prf:TkP} and \eqref{Prf:TkW} prove that
$(\Phi_n(u),W_n(u);\boldsymbol{\alpha})$
is a $q$-isomonodromic deformation.
Equations \eqref{Lim:adqPhin} and \eqref{Lim:adWn}
give \eqref{Lem:PWnm1}.
The analytic dependence follows directly from
Lemma~\ref{Lem:RecStep1}.
\end{prf}

\begin{cor}\label{Cor:Reck}
Suppose that
$(\Phi_{k-1}(u),W_{k-1}(u);\boldsymbol{\alpha})$
is a $q$-isomonodromic deformation at level $(k-1,n)$
and that $\Phi_{k-1}(u)$ satisfies \eqref{PhikInvert}
and the shrinking condition \eqref{Cond:less1}.
Fix branches of
$[\Phi_{k-1}^{[k-1]}]^q$ and
$[\Phi_{k-1}^{[k]}]^q$.
Then there exists a unique $q$-isomonodromic deformation
$(\Phi_k(u),W_k(u);\boldsymbol{\alpha})$
at level $(k,n)$, together with a unique branch
$[\Phi_k(u)^{[k]}]^q$, such that
\begin{subequations}\label{Lem:PWkm1}
\begin{align}
\label{Lem:Pkm0}
\lim_{u_k\to\infty}
u_k^{\mathrm{ad}[\delta_{k-1}\Phi_{k-1}]^q}
[\delta_k\Phi_k(u)]^q
&=
[\delta_k\Phi_{k-1}(u)]^q,\\
\label{Lem:Pkm1}
\lim_{u_k\to\infty}
u_k^{\mathrm{ad}[\delta_{k-1}\Phi_{k-1}]^q}
u_k^{-\mathrm{ad}[\delta_k\Phi_k]^q}
\Phi_k(u)
&=
\Phi_{k-1}(u),\\
\label{Lem:Wkm1}
\lim_{u_k\to\infty}
u_k^{[\delta_{k-1}\Phi_{k-1}]^q}
u_k^{-[\delta_k\Phi_k]^q}
W_k(u)
&=
W_{k-1}(u).
\end{align}
\end{subequations}
Moreover, the pair
$(\Phi_k(u),W_k(u))$
depends analytically on the initial data
$(\Phi_{k-1}(u),W_{k-1}(u))$.
\end{cor}

\begin{prf}
Let
$(\Phi_{k-1}(u),W_{k-1}(u);\boldsymbol{\alpha})$
be a $q$-isomonodromic deformation at level $(k-1,n)$.
From this deformation, we construct
$W_{k-1}^{(k)}(u)$ such that
\[
	(\Phi_{k-1}(u)^{[k]},W_{k-1}^{(k)}(u);
	\boldsymbol{\alpha}^{[k]})
\]
is a $q$-isomonodromic deformation at level $(k-1,k)$.
Applying Lemma~\ref{Lem:Phin} to this deformation with the fixed
branches $[\Phi_{k-1}^{[k-1]}]^q$ and
$[\Phi_{k-1}^{[k]}]^q$, we obtain a
$q$-isomonodromic deformation at level $(k,k)$, which we denote by
\[
	(\Phi_k(u)^{[k]},W_k^{(k)}(u);
	\boldsymbol{\alpha}^{[k]}),
\]
together with a unique branch $[\Phi_k(u)^{[k]}]^q$.
Equation \eqref{Lem:Pnm1},
applied at rank $k$, gives \eqref{Lem:Pkm0}.
It remains to extend this deformation to a
$q$-isomonodromic deformation at level $(k,n)$
satisfying \eqref{Lem:Pkm1} and \eqref{Lem:Wkm1}.
By Lemma~\ref{Lem:Phin}, the matrix $\Phi_k(u)^{[k]}$ is uniquely
determined by $\Phi_{k-1}(u)^{[k]}$. Since
\[
	V_{k;j}^{(n)}(u;\Phi_k(u))
	=
	V_{k;j}^{(n)}(u;\Phi_k(u)^{[k]}),
\]
the coefficient matrices in \eqref{qkjiso}
are uniquely determined.
Hence these equations are fixed.

Denote $W_{k;\ast}^{(n)}(u) :=
	\begin{pmatrix}
		W_\ast^{(k)}(u) & 0 \\
		0 & I_{n-k}
	\end{pmatrix}$.
Suppose that there exists a matrix $W_k(u)$ satisfying
\eqref{qkjisoW} and \eqref{Lem:Wkm1}.
Since $W_{k;k}^{(n)}(u)$ also satisfies \eqref{qkjisoW},
	$W_{k;k}^{(n)}(u)^{-1}W_k(u)$
	is a $q$-constant with respect to $u_1,\ldots,u_k$.
By Lemma~\ref{Lem:Phin}, specifically \eqref{Lem:Wnm1}
applied at rank $k$, together with \eqref{Lem:Wkm1},
\begin{align*}
	\lim_{u_k\to\infty}
	W_{k;k}^{(n)}(u)^{-1}W_k(u)
	=
	W_{k;k-1}^{(n)}(u)^{-1}W_{k-1}(u).
\end{align*}
It follows that
\begin{equation}\label{Eq:WkExtension}
	W_k(u)=
	W_{k;k}^{(n)}(u)
	W_{k;k-1}^{(n)}(u)^{-1}
	W_{k-1}(u).
\end{equation}
A direct verification shows that the matrix $W_k(u)$ defined by
\eqref{Eq:WkExtension} satisfies
\eqref{qkjisoW} and \eqref{Lem:Wkm1}.
Moreover,
\begin{align*}
W_{k;k}^{(n)}(u)W_{k;k-1}^{(n)}(u)^{-1}
\end{align*}
is independent of the choice of $W_{k;k-1}^{(n)}(u)$.
Hence such a matrix $W_k(u)$ exists and is unique.

Set
\begin{align}\nonumber
\widetilde{\Phi}_k(u)
&=
	W_k(u)W_{k-1}(u)^{-1}\Phi_{k-1}(u)
	W_{k-1}(u)W_k(u)^{-1} \\
&=
	W_{k;k}^{(n)}(u)
	W_{k;k-1}^{(n)}(u)^{-1}\Phi_{k-1}(u)
	W_{k;k-1}^{(n)}(u)
	W_{k;k}^{(n)}(u)^{-1}.
\label{Eq:tPisP}
\end{align}
If a matrix $\Phi_k(u)$ satisfying the required conditions exists,
then it must coincide with $\widetilde{\Phi}_k(u)$.
The second equality directly shows that
$\widetilde{\Phi}_k(u)^{[k]}=\Phi_k(u)^{[k]}$.
Hence such a matrix $\Phi_k(u)$ indeed exists and is unique.

We now take $\Phi_k(u)=\widetilde{\Phi}_k(u)$.
Note that \eqref{Eq:tPisP} also guarantees
	\begin{align*}
		q^{\alpha_s} & =
    	\frac{
        	\mathrm{det}( I + (q-1)\Phi_{k-1}(u)^{[s]} )
    	}{
        	\mathrm{det}( I + (q-1)\Phi_{k-1}(u)^{[s-1]} )
    	} =
    	\frac{
        	\mathrm{det}( I + (q-1)\Phi_k(u)^{[s]} )
    	}{
        	\mathrm{det}( I + (q-1)\Phi_k(u)^{[s-1]} )
    	},\quad
    	s=k+1,\ldots,n.
	\end{align*}
We thus obtain a unique $q$-isomonodromic deformation
	$(\Phi_k(u),W_k(u);\boldsymbol{\alpha})$
	at level $(k,n)$
that satisfies \eqref{Lem:PWkm1}.
The analytic dependence follows directly from
Lemma~\ref{Lem:Phin}.
\end{prf}

At this point, starting from a $q$-isomonodromic deformation at level $(0, n)$, that is, a constant matrix triple $(\Phi_0, W_0; \boldsymbol{\alpha})$,
we can uniquely construct a $q$-isomonodromic deformation
$(\Phi(u), W(u); \boldsymbol{\alpha})$
with prescribed asymptotic behavior.
This is stated in the following theorem.

\begin{thm}\label{Thm:AsymBehavior}
Suppose that
	$\Phi_0\in\mathrm{Mat}_{n\times n}(\mathbb C)$ and
	$W_0\in\mathrm{GL}_n(\mathbb C)$.
For each $1\leqslant k\leqslant n$,
	assume that
\begin{align}\label{Thm:Cond1}
I+(q-1)\Phi_0^{[k]} \text{ is invertible.}
\end{align}
For each $1\leqslant k\leqslant n$,
fix a branch of $[\Phi_0^{[k]}]^q$ and write
\[
	\mathrm{Spec}\bigl([\Phi_0^{[k]}]^q\bigr)
	=
	\{
		\lambda_1^{(k)},\ldots,\lambda_k^{(k)}
	\},
\]
and assume further that there exists
$0<\varepsilon_{k-1}<1$
such that the following
shrinking condition holds:
\begin{align}\label{Cond:Shrinking}
	\underset{1\leqslant i,j\leqslant k-1}{\mathrm{max}}
	\left|
	\mathrm{Re}\bigl(\lambda_i^{(k-1)}-\lambda_j^{(k-1)}\bigr)
	-
	\frac{\mathrm{arg}q}{\mathrm{ln}|q|}
	\mathrm{Im}\bigl(\lambda_i^{(k-1)}-\lambda_j^{(k-1)}\bigr)
	\right|
	<
	1-\varepsilon_{k-1}.
\end{align}
Choose
$\boldsymbol{\alpha}
=\operatorname{diag}(\alpha_1,\ldots,\alpha_n)$
such that
\begin{align}\label{Thm:defalpha}
	q^{\alpha_k}
	&=
	\frac{
		\det\bigl(I+(q-1)\Phi_0^{[k]}\bigr)
	}{
		\det\bigl(I+(q-1)\Phi_0^{[k-1]}\bigr)
	},
	\qquad
	k=1,\ldots,n.
\end{align}
Then there exists a unique $q$-isomonodromic deformation
$(\Phi(u),W(u);\boldsymbol{\alpha})$,
together with functions
\[
	([\Phi_k(u)]^q,
	[\Phi_k(u)^{[n-1]}]^q,
	\ldots,
	[\Phi_k(u)^{[k]}]^q,
	W_k(u)),
	\qquad
	k=1,\ldots,n,
\]
such that for each $1\leqslant k \leqslant j\leqslant n$
\begin{subequations}\label{Thm:Asymptotics}
\begin{align}
\lim_{u_2\to\infty}\cdots\lim_{u_k\to\infty}
	u_1^{\operatorname{ad}\boldsymbol{\alpha}}
	\left(
	\mathop{\prod}\limits_{s=1}^{\overset{\longrightarrow}{k-1}}
	\left(
		\frac{u_{s+1}}{u_s}
	\right)^{
		\operatorname{ad}[\delta_s\Phi_s(u)]^q
	}
	\right)
	[\delta_k\Phi_k(u)]^q
	&=
	[\delta_k\Phi_0]^q,
	\\
\lim_{u_2\to\infty}\cdots\lim_{u_k\to\infty}
	u_1^{\operatorname{ad}\boldsymbol{\alpha}}
	\left(
	\mathop{\prod}\limits_{s=1}^{\overset{\longrightarrow}{k-1}}
	\left(
		\frac{u_{s+1}}{u_s}
	\right)^{
		\operatorname{ad}[\delta_s\Phi_s(u)]^q
	}
	\right)
	u_k^{-\operatorname{ad}[\delta_k\Phi_k(u)]^q}
	[\delta_j\Phi_k(u)]^q
	&=
	[\delta_j\Phi_0]^q,
	\\
\lim_{u_2\to\infty}\cdots\lim_{u_k\to\infty}
	u_1^{\boldsymbol{\alpha}}
	\left(
	\mathop{\prod}\limits_{s=1}^{\overset{\longrightarrow}{k-1}}
	\left(
		\frac{u_{s+1}}{u_s}
	\right)^{
		[\delta_s\Phi_s(u)]^q
	}
	\right)
	u_k^{-[\delta_k\Phi_k(u)]^q}
	W_k(u)
	&=
	W_0,
\end{align}
and
\begin{align}
	(\Phi(u),W(u))
	&=
	(\Phi_n(u),W_n(u)),\\
	W(u)^{-1} \cdot [\Phi(u)]^q \cdot W(u)
	&=
	W_0^{-1} \cdot [\Phi_0]^q \cdot W_0.
\end{align}
\end{subequations}
Here, for each $0\leqslant k \leqslant j \leqslant n$,
\[
	[\delta_j\Phi_k]^q
	:=
	\begin{pmatrix}
		[\Phi_k^{[j]}]^q & 0\\
		0 &
		\operatorname{diag}
		(\alpha_{j+1},\ldots,\alpha_n)
	\end{pmatrix}.
\]
Moreover, the pair
$(\Phi(u),W(u))$
depends analytically on the initial data
$(\Phi_0,W_0)$.
\end{thm}

\begin{proof}
By \eqref{Thm:defalpha},
$(\Phi_0,W_0;\boldsymbol{\alpha})$
is a $q$-isomonodromic deformation at level $(0,n)$.
We apply Corollary~\ref{Cor:Reck} recursively.
At the $(k-1)$-st step, we have obtained the unique
$q$-isomonodromic deformation
$(\Phi_{k-1},W_{k-1};\boldsymbol{\alpha})$
at level $(k-1,n)$, together with the branches
\[
	[\Phi_{k-1}^{[k-1]}]^q,\ldots,
	[\Phi_{k-1}^{[n]}]^q.
\]
Applying Corollary~\ref{Cor:Reck} to these data, we obtain the
unique $q$-isomonodromic deformation
$(\Phi_k,W_k;\boldsymbol{\alpha})$
at level $(k,n)$,
together with the branches
\[
	[\Phi_{k}^{[k]}]^q,\ldots,
	[\Phi_{k}^{[n]}]^q,
\]
satisfying \eqref{Lem:PWkm1} and
\begin{subequations}\label{Eq:Thm4prf}
\begin{align}
\lim_{u_k\to\infty}
	u_k^{\operatorname{ad}[\delta_{k-1}\Phi_{k-1}]^q}
	u_k^{-\operatorname{ad}[\delta_k\Phi_k]^q}
	[\delta_j\Phi_k(u)]^q
	&=
	[\delta_j\Phi_{k-1}(u)]^q,
	\qquad
	k<j\leqslant n,\\
W_k^{-1}[\Phi_k(u)]^qW_k
	&=
	W_{k-1}^{-1}[\Phi_{k-1}(u)]^qW_{k-1}.
\end{align}
\end{subequations}
At level $1$, \eqref{Exp:Vk} gives
$V_{1;1}^{(n)}=I$, and hence
$T_1\Phi_1=\Phi_1$ and $T_1W_1=W_1$.
Thus for $1\leqslant j\leqslant n$ we have identities
\begin{align}\label{Eq:lv1}
u_1^{\operatorname{ad}\boldsymbol{\alpha}}
	u_1^{-\operatorname{ad}[\delta_1\Phi_1]^q}
	[\delta_j\Phi_1(u)]^q
	=
	[\delta_j\Phi_0]^q,\qquad
u_1^{\boldsymbol{\alpha}}
	u_1^{-[\delta_1\Phi_1]^q}W_1(u)
	=
	W_0.
\end{align}
Taking the resulting unique pair
$(\Phi_n,W_n)$ as $(\Phi,W)$
and combining \eqref{Lem:PWkm1} with
\eqref{Eq:Thm4prf} and \eqref{Eq:lv1}, we obtain
\eqref{Thm:Asymptotics}.
The analytic dependence follows directly from
Corollary~\ref{Cor:Reck},
which completes the proof.
\end{proof}

\begin{defi}
We call the $q$-isomonodromic deformation
	$(\Phi(u),W(u);\boldsymbol{\alpha})$
	constructed in Theorem~\ref{Thm:AsymBehavior}
	a \textbf{shrinking solution}
	of the $q$-isomonodromy equations.
\end{defi}

\begin{rmk}
If the formal monodromy is instead taken to be
$\boldsymbol{\beta}
=\operatorname{diag}(\beta_1,\ldots,\beta_n)$,
and $\delta_k$ is replaced by
\[
[\delta_k^{\boldsymbol{\beta}}A]^q
:=
\operatorname{diag}
\left(
	[A^{[k]}]^q,
	\beta_{k+1},\ldots,\beta_n
\right),
\qquad
0\leqslant k\leqslant n,
\]
then the asymptotic formulas
\eqref{Lem:PWkm1} in Corollary~\ref{Cor:Reck}
remain unchanged, with every $\delta_k$ replaced by
$\delta_k^{\boldsymbol{\beta}}$.
In this notation, the corresponding first limit in
\eqref{Thm:Asymptotics} takes the form
\begin{subequations}
\begin{align}
\lim_{u_2\to\infty}\cdots\lim_{u_k\to\infty}
	u_1^{\operatorname{ad}\boldsymbol{\beta}}
	\left(
		\mathop{\prod}\limits_{s=1}^{\overset{\longrightarrow}{k-1}}
		\left(
			\frac{u_{s+1}}{u_s}
		\right)^{
			\operatorname{ad}
			[\delta_s^{\boldsymbol{\beta}}\Phi_s(u)]^q
		}
	\right)
	[\delta_k^{\boldsymbol{\beta}}\Phi_k(u)]^q
	&=
	[\delta_k^{\boldsymbol{\beta}}\Phi_0]^q,
	\\
\lim_{u_2\to\infty}\cdots\lim_{u_k\to\infty}
	u_1^{\operatorname{ad}\boldsymbol{\beta}}
	\left(
		\mathop{\prod}\limits_{s=1}^{\overset{\longrightarrow}{k-1}}
		\left(
			\frac{u_{s+1}}{u_s}
		\right)^{
			\operatorname{ad}
			[\delta_s^{\boldsymbol{\beta}}\Phi_s(u)]^q
		}
	\right)
	u_k^{-\operatorname{ad}
		[\delta_k^{\boldsymbol{\beta}}\Phi_k(u)]^q}
	[\delta_j^{\boldsymbol{\beta}}\Phi_k(u)]^q
	&=
	[\delta_j^{\boldsymbol{\beta}}\Phi_0]^q,
	\\
\lim_{u_2\to\infty}\cdots\lim_{u_k\to\infty}
	u_1^{\boldsymbol{\beta}}
	\left(
		\mathop{\prod}\limits_{s=1}^{\overset{\longrightarrow}{k-1}}
		\left(
			\frac{u_{s+1}}{u_s}
		\right)^{
			[\delta_s^{\boldsymbol{\beta}}\Phi_s(u)]^q
		}
	\right)
	u_k^{-[\delta_k^{\boldsymbol{\beta}}\Phi_k(u)]^q}
	W_k(u)
	&=
	W_0.
\end{align}
\end{subequations}
In fact, if the two recursive constructions start from the same
leading term $(\Phi_0,W_0)$,
then for each
$1\leqslant k\leqslant n$ we have
\begin{align*}
\Phi_k(u;\boldsymbol{\beta})
&=
U^{
	\operatorname{ad}\left(
		(\boldsymbol{\alpha}-\boldsymbol{\beta})
		(E_1+\cdots+E_k)
	\right)
}
\Phi_k(u;\boldsymbol{\alpha}),\\
W_k(u;\boldsymbol{\beta})
&=
U^{
	(\boldsymbol{\alpha}-\boldsymbol{\beta})
	(E_1+\cdots+E_k)
}
W_k(u;\boldsymbol{\alpha}).
\end{align*}
\end{rmk}

\subsection{Explicit Rank-Two Asymptotics}
\label{Sect:AsymRankTwo}

In this subsection,
	we give an explicit $2\times2$ example of
	the construction in Theorem~\ref{Thm:AsymBehavior}.
	
Fix constant matrices
\[
	\Phi_0
	=
	\begin{pmatrix}
		\varphi_{0,11} & \varphi_{0,12}\\
		\varphi_{0,21} & \varphi_{0,22}
	\end{pmatrix}
	\in\operatorname{Mat}_{2\times2}(\mathbb C),
	\qquad
	W_0
	=
	\begin{pmatrix}
		w_{0,11} & w_{0,12}\\
		w_{0,21} & w_{0,22}
	\end{pmatrix}
	\in\operatorname{GL}_2(\mathbb C).
\]
Assume that
$1+(q-1)\varphi_{0,11}\neq0$ and
$I+(q-1)\Phi_0$ is invertible.
Fix branches of
$[\varphi_{0,11}]^q$ and $[\Phi_0]^q$, and write
\[
	\lambda_1^{(1)}
	:=
	[\varphi_{0,11}]^q,
	\qquad
	\operatorname{Spec}\bigl([\Phi_0]^q\bigr)
	=
	\{
		\lambda_1^{(2)},\lambda_2^{(2)}
	\},\qquad
	\Lambda
	:=
	\operatorname{diag}
	\bigl(\lambda_1^{(2)},\lambda_2^{(2)}\bigr).
\]
Choose $\boldsymbol{\alpha} =
	\operatorname{diag}(\alpha_1,\alpha_2)$ such that
\[
	q^{\alpha_1}
	=
	1+(q-1)\varphi_{0,11},
	\qquad
	q^{\alpha_2}
	=
	\frac{
		\det\bigl(I+(q-1)\Phi_0\bigr)
	}{
		1+(q-1)\varphi_{0,11}
	}.
\]
Denote $[\delta_1\Phi_0]^q :=
	\operatorname{diag}(\lambda_1^{(1)},\alpha_2)$.
	
Assume that
	$\lambda_1^{(2)}$ and $\lambda_2^{(2)}$
	are distinct and chosen $W_0$ such that
	$W_0^{-1}[\Phi_0]^qW_0=\Lambda$.
Theorem~\ref{Thm:AsymBehavior} then gives the unique
$q$-isomonodromic deformation
$\bigl(
		\Phi_2(u_1,u_2),
		W_2(u_1,u_2);
		\boldsymbol{\alpha}
	\bigr)$
with prescribed leading term $(\Phi_0,W_0)$.
\begin{align}
\nonumber
\lim_{u_2\to\infty}
	u_1^{\mathrm{ad} \boldsymbol{\alpha}}
	\left(
		\frac{u_2}{u_1}
	\right)^{\mathrm{ad}[\delta_1\Phi_0]^q}
	\Phi_2(u_1,u_2)
	&=
	\Phi_0,\\
\nonumber
\lim_{u_2\to\infty}
	u_1^{\boldsymbol{\alpha}}
	\left(
		\frac{u_2}{u_1}
	\right)^{[\delta_1\Phi_0]^q}
	u_2^{-[\Phi_2]^q}W_2(u_1,u_2)
	&=W_0.
\end{align}
We have
\begin{subequations}\label{Exp:RankTwoPhiFromAsym}
\begin{align}
\Phi_2(u_1,u_2)_{11}
	&=
	\frac{u_2}{u_2-u_1}[\alpha_1]_q
	-
	\frac{u_1}{u_2-u_1}
	\left(
		[\lambda_1^{(2)}]_q
		+
		[\lambda_2^{(2)}]_q
		-
		[\alpha_2]_q
	\right),\\
\Phi_2(u_1,u_2)_{22}
	&=
	\frac{u_2}{u_2-u_1}
	\left(
		[\lambda_1^{(2)}]_q
		+
		[\lambda_2^{(2)}]_q
		-
		[\alpha_1]_q
	\right)
	-
	\frac{u_1}{u_2-u_1}[\alpha_2]_q,\\
\Phi_2(u_1,u_2)_{12}
	&=
	u_1^{\lambda_1^{(1)}-\alpha_1}
	u_2^{\alpha_2-\lambda_1^{(1)}}
	\frac{
		\left(
			q\frac{u_1}{u_2};q
		\right)_\infty^2
	}{
		\left(
			q^{\lambda_1^{(2)}-\alpha_1+1}
			\frac{u_1}{u_2};q
		\right)_\infty
		\left(
			q^{\lambda_2^{(2)}-\alpha_1+1}
			\frac{u_1}{u_2};q
		\right)_\infty
	}
	\varphi_{0,12},\\
\Phi_2(u_1,u_2)_{21}
	&=
	u_1^{\alpha_1-\lambda_1^{(1)}}
	u_2^{\lambda_1^{(1)}-\alpha_2}
	\frac{
		\left(
			q^{\lambda_1^{(2)}-\alpha_1}
			\frac{u_1}{u_2};q
		\right)_\infty
		\left(
			q^{\lambda_2^{(2)}-\alpha_1}
			\frac{u_1}{u_2};q
		\right)_\infty
	}{
		\left(
			\frac{u_1}{u_2};q
		\right)_\infty^2
	}
	\varphi_{0,21},
\end{align}
\end{subequations}
and
\begin{subequations}\label{Exp:RankTwoWFromAsym}
\begin{align}
W_2(u_1,u_2)_{11}
	&=
	u_1^{\lambda_1^{(1)}-\alpha_1}
	u_2^{\lambda_1^{(2)}-\lambda_1^{(1)}}
	\frac{
		\left(
			q\frac{u_1}{u_2};q
		\right)_\infty
	}{
		\left(
			q^{\lambda_1^{(2)}-\alpha_1+1}
			\frac{u_1}{u_2};q
		\right)_\infty
	}
	w_{0,11},\\
W_2(u_1,u_2)_{12}
	&=
	u_1^{\lambda_1^{(1)}-\alpha_1}
	u_2^{\lambda_2^{(2)}-\lambda_1^{(1)}}
	\frac{
		\left(
			q\frac{u_1}{u_2};q
		\right)_\infty
	}{
		\left(
			q^{\lambda_2^{(2)}-\alpha_1+1}
			\frac{u_1}{u_2};q
		\right)_\infty
	}
	w_{0,12},\\
W_2(u_1,u_2)_{21}
	&=
	u_2^{\lambda_1^{(2)}-\alpha_2}
	\frac{
		\left(
			q^{\lambda_2^{(2)}-\alpha_1}
			\frac{u_1}{u_2};q
		\right)_\infty
	}{
		\left(
			\frac{u_1}{u_2};q
		\right)_\infty
	}
	w_{0,21},\\
W_2(u_1,u_2)_{22}
	&=
	u_2^{\lambda_2^{(2)}-\alpha_2}
	\frac{
		\left(
			q^{\lambda_1^{(2)}-\alpha_1}
			\frac{u_1}{u_2};q
		\right)_\infty
	}{
		\left(
			\frac{u_1}{u_2};q
		\right)_\infty
	}
	w_{0,22}.
\end{align}
\end{subequations}
Recall that in Definition~\ref{Def:q} we extended the $q$-Pochhammer
symbol to the case $|q|>1$. Hence the above formulas hold for all
$q$ with $|q|\neq1$.

\section{Explicit Formulas for $q$-Monodromy Data}
\label{Sect:Dec}

Throughout this section, we consider the $q$-isomonodromic deformation
$(\Phi(u;\boldsymbol{\alpha}),W(u;\boldsymbol{\alpha});
\boldsymbol{\alpha})$
constructed from the initial data
$(\Phi_0,W_0;\boldsymbol{\alpha})$
in Theorem~\ref{Thm:AsymBehavior}.
Our goal is to derive an explicit formula, in terms of these initial
data, for the associated central connection matrix
$C(x;U,\Phi(u);\boldsymbol{\alpha},W(u))$.
The theoretical framework for this approach is established in
	Section 4 of \cite{tang_boundary_nodate}.

We begin with a single recursive step.
For the remainder of this section, let
	$(\Phi_{n-1},W_{n-1};\boldsymbol{\alpha})$
	denote the constant initial data
	for a single recursive step,
	with $\Phi_{n-1}$ satisfying
	\eqref{PhikInvert}
	and the shrinking condition \eqref{Cond:less1},
	and let
	$(\Phi_n(u_n),W_n(u_n);\boldsymbol{\alpha})$
	denote the one-variable family
	constructed from these data in
	Lemma~\ref{Lem:RecStep1}.
Its asymptotic formulas \eqref{Lim:adPhin} and \eqref{Lim:adWn},
	together with the block estimate \eqref{Eq:lem43},
	yield the following relations as $u_n\to\infty$:
\begin{subequations}\label{Lim:dk1ktok1}
\begin{align}
\label{Lim:dk1Pktodk1Pk1}
	[\delta_{n-1}\Phi_n(u_n)]^q
	&=
	[\delta_{n-1}\Phi_{n-1}]^q
	+O(u_n^{-\varepsilon_{n-1}}),\\
\label{Lim:dk1PktoPk1}
	u_n^{\operatorname{ad}[\delta_{n-1}\Phi_{n-1}]^q}
	[\Phi_n(u_n)]^q
	&=
	[\Phi_{n-1}]^q
	+O(u_n^{-\varepsilon_{n-1}}),\\
\label{Lim:dk1WktoWk1}
	u_n^{[\delta_{n-1}\Phi_{n-1}]^q}
	u_n^{-[\Phi_n(u_n)]^q}
	W_n(u_n)
	&=
	W_{n-1}
	+O(u_n^{-\varepsilon_{n-1}}).
\end{align}
\end{subequations}
Moreover, as in the proof of Lemma~\ref{Lem:qBasic},
	the $q$-isomonodromy equations satisfied by this family
	are the compatibility conditions for the linear
	$q$-difference system
\begin{equation}\label{Eq:TargetSystem}
\left\{
\begin{aligned}
    \frac{\partial_q}{\partial_q x}Y(x,u_n)
    &=
	\left(
		U+\frac{\Phi_n(u_n)}{x}
	\right)Y(x,u_n),\\
    \frac{\partial_q}{\partial_q u_n}Y(x,u_n)
    &=
	\left(
		q^{-\alpha_n}E_nx
		+
		\frac{
		V_{n;n}^{(n)}(u_n;\Phi_n(u_n),\alpha_n)-I
		}{
		(q-1)u_n
		}
	\right)Y(x,u_n).
\end{aligned}
\right.
\end{equation}
Theorem~\ref{Thm:qIsoMD} shows that the associated central connection
	matrix
	$C(x;U,\Phi_n(u_n);\boldsymbol{\alpha},W_n(u_n))$
	is a $q$-constant with respect to $u_n$.
To determine this matrix, we shall factorize the canonical fundamental
	solutions of \eqref{Eq:TargetSystem} and thereby express it explicitly
	in terms of the initial data
	$(\Phi_{n-1},W_{n-1};\boldsymbol{\alpha})$,
	completing the recursive step.
	
The remainder of this section is organized as follows.
In Section~\ref{Sect:FactCan}, we introduce the canonical fundamental
solutions required for the factorization formulas and establish their
basic properties.
In Section~\ref{Sect:CanFact}, we prove these factorization formulas.
Finally, in Section~\ref{Sect:CenFact}, we use these factorization
formulas to obtain an explicit formula for the central connection
matrix in terms of the initial data..

\subsection{Canonical Fundamental Solutions for the Factorization}
\label{Sect:FactCan}

In this subsection,
we introduce the canonical fundamental solutions
required for the factorization.
We establish the results below under the following additional
open condition on $\Phi_{n-1}$:
\begin{align}\label{Cond:less1new}
\max_{1\leqslant i\leqslant n-1}
\left|
\mathrm{Re}\bigl(\lambda_i^{(n-1)}-\alpha_n\bigr)
-
\frac{\arg q}{\ln|q|}
\mathrm{Im}\bigl(\lambda_i^{(n-1)}-\alpha_n\bigr)
\right|
&<
1-\varepsilon_\alpha,
\qquad
\text{for some }0<\varepsilon_\alpha<1,
\end{align}
and denote
$\epsilon_\alpha := \min\{\varepsilon_\alpha,\varepsilon_{n-1}\}$.
This condition implies the block estimate
\begin{align*}
\Phi_n(u_n)_{\hat n n} =
O(u_n^{1-\varepsilon_\alpha}),\qquad
\Phi_n(u_n)_{n \hat n} =
O(u_n^{1-\varepsilon_\alpha}),
\qquad
u_n\to\infty.
\end{align*}
This estimate, together with \eqref{Lim:dk1ktok1},
	provides the growth control needed below.
The resulting factorization formulas depend analytically on the
initial data $\Phi_{n-1}$. Therefore, Lemma~\ref{Lem:RecStep1}
ensures that the additional open condition
\eqref{Cond:less1new} is not essential.

We introduce the coordinates
\begin{equation}\label{Eq:Trans}
\left\{
\begin{aligned}
    x_n &= x, \\
    \zeta_n &= x u_n.
\end{aligned}
\right.
\end{equation}
In these coordinates, the $x_n$-equation in the limit
$\zeta_n\to\infty$ and the $\zeta_n$-equation in the limit
$x_n\to0$ give rise to two auxiliary $q$-difference systems.
We shall introduce the canonical fundamental solutions of
\eqref{Eq:TargetSystem} and the transformed system.
We shall also introduce those of the two limiting systems
used in the factorization.

\begin{lem}\label{Lem:TargetSysNew}
We denote the system obtained from
\eqref{Eq:TargetSystem} via the coordinate transformation
\eqref{Eq:Trans} as follows:
\begin{subequations}\label{Eq:TargetSysNew}
\begin{empheq}[left=\empheqlbrace]{align}
\frac{\partial_q}{\partial_q x_n}
Y(x_n,\zeta_n)
&=
A_1(x_n,\zeta_n)\cdot
Y(x_n,\zeta_n),
\label{Eq:TargetSysNew-x}
\\
\frac{\partial_q}{\partial_q \zeta_n}
\left(
\left(
\frac{\zeta_n}{x_n}
\right)^{[\delta_{n-1}\Phi_{n-1}]^q}
Y(x_n,\zeta_n)
\right)
&=
A_2(x_n,\zeta_n)\cdot
\left(
\left(
\frac{\zeta_n}{x_n}
\right)^{[\delta_{n-1}\Phi_{n-1}]^q}
Y(x_n,\zeta_n)
\right).
\label{Eq:TargetSysNew-zeta}
\end{empheq}
\end{subequations}
There exists a constant $C>0$,
	independent of sufficiently small $x_n$ and
	sufficiently large $\zeta_n$, such that
\begin{align}
	\left\|
	A_2(x_n,\zeta_n) -
		\left( E_n + \frac{\Phi_{n-1}}{\zeta_n}
		\right)
	\right\|
	&\leqslant
	C
	|x_n|^{\varepsilon_{n-1}}
	|\zeta_n|^{-1-\varepsilon_{n-1}},
\label{Eq:EstA2}
\end{align}
Denote $U_{\hat n}:=
\begin{pmatrix}
U_{\hat n\hat n}&0\\
0&0
\end{pmatrix}$.
Under the additional open condition \eqref{Cond:less1new},
	there exists a constant $C>0$,
	independent of sufficiently small
	$x_n$ and sufficiently large $\zeta_n$, such that
\begin{subequations}\label{Eq:EstA1andad}
\begin{align}
\left\|
A_1(x_n,\zeta_n)
- \left( U_{\hat n} +
\frac{\delta_{n-1}\Phi_{n-1}}{x_n}
\right)
\right\|
&\leqslant
C
|x_n|^{-1+\epsilon_\alpha}
|\zeta_n|^{-\epsilon_\alpha},
\label{Eq:EstA1}\\
\left\|
\left(
	\frac{\zeta_n}{x_n}
\right)^{\operatorname{ad}[\delta_{n-1}\Phi_{n-1}]^q}
\left(
A_1(x_n,\zeta_n)
- \left( U_{\hat n} +
\frac{\delta_{n-1}\Phi_{n-1}}{x_n}
\right) \right)
\right\|
&\leqslant
C
|x_n|^{-1+\epsilon_\alpha}
|\zeta_n|^{-\epsilon_\alpha}.
\label{Eq:EstadA1}
\end{align}
\end{subequations}
\end{lem}

\begin{prf}
Applying \eqref{Eq:TargetSystem}
	under the coordinate transformation \eqref{Eq:Trans} gives
\begin{subequations}
\begin{align}
\label{Eq:A1xz}
I+(q-1)x_nA_1(x_n,q\zeta_n)
&=
\left(
	I+(q-1)\bigl(Ux_n+\Phi_n(u_n)\bigr)
\right)
\left(
	V_{n;n}^{(n)}(u_n)
	+(q-1)q^{-\alpha_n}u_nx_nE_n
\right)^{-1},
\\
\label{Eq:A2xz}
A_2(x_n,\zeta_n)
&=
E_n+
\frac{
u_n^{\operatorname{ad}[\delta_{n-1}\Phi_{n-1}]^q}
\left(
q^{[\delta_{n-1}\Phi_{n-1}]^q}
V_{n;n}^{(n)}(u_n)
\right)-I
}{
(q-1)\zeta_n
}.
\end{align}
\end{subequations}
Equation \eqref{Eq:useful}, substituted into
\eqref{Eq:A2xz}, together with
$u_n=\zeta_n/x_n$, gives \eqref{Eq:EstA2}.

In fact, the right-hand side of \eqref{Eq:A1xz}
	can be rewritten as a linear polynomial in $x_n$.
By the definition of $V_{n;n}^{(n)}$ and
	\eqref{Det:beta} in Lemma~\ref{Lem:betaEq}, we have
\begin{equation*}
\det\left(
	V_{n;n}^{(n)}(u_n)
	+(q-1)q^{-\alpha_n}u_nx_nE_n
\right)
=
1+(q-1)q^{-\alpha_n}u_nx_n.
\end{equation*}
Using \eqref{Exp:Vkhatkk} and \eqref{Exp:Vkkk},
the block inverse formula gives
\begin{align}\label{Inv:A1V}
&\left(
	V_{n;n}^{(n)}(u_n)
	+(q-1)q^{-\alpha_n}u_nx_nE_n
\right)^{-1}
\notag\\
&\quad=
\begin{pmatrix}
I&0\\
0&0
\end{pmatrix}
+
\frac{1}{
1+(q-1)q^{-\alpha_n}u_nx_n
}
\begin{pmatrix}
\bigl(V_{n;n}^{(n)}(u_n)\bigr)_{\hat n n}
\bigl(V_{n;n}^{(n)}(u_n)\bigr)_{n\hat n}
&
-\bigl(V_{n;n}^{(n)}(u_n)\bigr)_{\hat n n}
\\
-\bigl(V_{n;n}^{(n)}(u_n)\bigr)_{n\hat n}
&
1
\end{pmatrix}.
\end{align}
Substituting \eqref{Exp:Vkkhatk}, \eqref{Exp:Vk12},
\eqref{Def:pnn}, and \eqref{Inv:A1V} into
\eqref{Eq:A1xz} gives the following block identities
for $A_1$:
\begin{subequations}\label{Eq:A1Block}
\begin{align}
A_1(x_n,q\zeta_n)_{\hat n\hat n}
&=
U_{\hat n\hat n}
+
\frac{\Phi_n(u_n)_{\hat n\hat n}}{x_n} +
\frac{q^{\alpha_n}}{(q-1)\zeta_n}
U_{\hat n\hat n}
V_{n;n}^{(n)}(u_n)_{\hat n n}
V_{n;n}^{(n)}(u_n)_{n \hat n},
\\
A_1(x_n,q\zeta_n)_{\hat n n}
&=
-\frac{q^{\alpha_n}}{(q-1)\zeta_n}
U_{\hat n\hat n}
V_{n;n}^{(n)}(u_n)_{\hat n n},
\\
A_1(x_n,q\zeta_n)_{n\hat n}
&=
-\frac{q^{\alpha_n-1}}{(q-1)\zeta_n}
V_{n;n}^{(n)}(u_n)_{n\hat n}
U_{\hat n\hat n},
\\
A_1(x_n,q\zeta_n)_{nn}
&=
\frac{[\alpha_n]_q}{x_n}.
\end{align}
\end{subequations}
Applying \eqref{Lim:dk1Pktodk1Pk1} and
	\eqref{Lim:dk1PktoPk1} to \eqref{Exp:Vk}
	under \eqref{PhikInvert}
	and the shrinking condition \eqref{Cond:less1},
	together with the off-diagonal blocks of
	\eqref{Eq:useful}, gives
\begin{subequations}\label{Est:A1BasicBlocks}
\begin{alignat}{2}
(\Phi_n)_{\hat n\hat n}
&=
	\Phi_{n-1}^{[n-1]}
	+O(u_n^{-\varepsilon_{n-1}}),
	\qquad &
(V_{n;n}^{(n)})_{\hat n n}
(V_{n;n}^{(n)})_{n\hat n}
&=
	O(u_n^{1-\varepsilon_{n-1}}),
\\
u_n^{\operatorname{ad}[\Phi_{n-1}^{[n-1]}]^q}
(\Phi_n)_{\hat n\hat n}
&=
	\Phi_{n-1}^{[n-1]}
	+O(u_n^{-\varepsilon_{n-1}}),
	\qquad &
u_n^{\operatorname{ad}[\Phi_{n-1}^{[n-1]}]^q}
U_{\hat n\hat n}
&=
	O(u_n^{1-\varepsilon_{n-1}}),
\\
u_n^{[\Phi_{n-1}^{[n-1]}]^q-\alpha_nI}
(V_{n;n}^{(n)})_{\hat n n}
&=
	O(1),
	\qquad &
(V_{n;n}^{(n)})_{n\hat n}
u_n^{\alpha_nI-[\Phi_{n-1}^{[n-1]}]^q}
&=
	O(1).
\end{alignat}
\end{subequations}
The additional open condition \eqref{Cond:less1new}
	further gives
\begin{align}\label{Est:VBlocksA1}
V_{n;n}^{(n)}(u_n)_{\hat n n}
&=
O(u_n^{1-\varepsilon_\alpha}),\qquad
V_{n;n}^{(n)}(u_n)_{n\hat n}
=
O(u_n^{1-\varepsilon_\alpha}).
\end{align}
Substituting \eqref{Est:A1BasicBlocks} and
\eqref{Est:VBlocksA1} into \eqref{Eq:A1Block}
and using $u_n=\zeta_n/x_n$ gives
\eqref{Eq:EstA1andad}.
\end{prf}

\begin{rmk}
In fact, Lemma~\ref{Lem:CompositeShift} guarantees that
$I+(q-1)x_nA_1(x_n,q\zeta_n)$ is a linear polynomial in $x_n$.
\end{rmk}

We now introduce the canonical fundamental solutions used in the
factorization and analyze their singularities in detail.
The recursive equations \eqref{ExpclitPhiW} extend $\Phi(u)$ to a
meromorphic function on its natural domain
$(\widetilde{\mathbb C^\ast})^n$.
Hence all its finite singularities are poles.
The forward and backward evolution formulas in
	Theorem~\ref{Thm:qiso} show how
	the fixed $q$-resonance poles of
	$\Phi_n(u_n)$ propagate along
	the corresponding half $q$-spirals.
We therefore separate these poles from
	the movable ones as follows.

\begin{defi}\label{Def:PhinSingularLoci}
Set
\begin{align*}
\mathfrak P(\Phi_{n-1};u_1,\ldots,u_{n-1})
:=
\begin{cases}
\displaystyle
\operatorname{Pole}\bigl(\Phi_n(u_n)\bigr)
\setminus
\left(
\{u_1,\ldots,u_{n-1}\}q^{\mathbb N}
\right),
& |q|<1,
\\
\displaystyle
\operatorname{Pole}\bigl(\Phi_n(u_n)\bigr)
\setminus
\left(
\{u_1,\ldots,u_{n-1}\}q^{-\mathbb N}
\right),
& |q|>1.
\end{cases}
\end{align*}
The elements of
$\mathfrak P(\Phi_{n-1};u_1,\ldots,u_{n-1})$
are called the \textbf{movable poles} of $\Phi_n(u_n)$.
\end{defi}

\begin{lem}\label{Lem:MovablePolesFromZeros}
Denote
\begin{align*}
\mathfrak Z(\Phi_{n-1};u_1,\ldots,u_{n-1})
:=
\left\{
u_n\in
\widetilde{\mathbb C^\ast}
:
\det\left(
I+(q-1)\Phi_n(u_n)_{\hat n\hat n}
-q^{\alpha_n}\frac{U_{\hat n\hat n}}{u_n}
\right)
=0
\right\}.
\end{align*}
Then
\begin{subequations}
\begin{align}
\operatorname{Pole}\bigl(\Phi_n(u_n)\bigr)
&\subseteq
\begin{cases}
\mathfrak P(\Phi_{n-1};u_1,\ldots,u_{n-1})
\cup
\{u_1,\ldots,u_{n-1}\}q^{\mathbb N},
& |q|<1,
\\
\mathfrak P(\Phi_{n-1};u_1,\ldots,u_{n-1})
\cup
\{u_1,\ldots,u_{n-1}\}q^{-\mathbb N},
& |q|>1,
\end{cases}
\label{Eq:PhinPoleClassification}
\\
\mathfrak P(\Phi_{n-1};u_1,\ldots,u_{n-1})
&\subseteq
\begin{cases}
\mathfrak Z(\Phi_{n-1};u_1,\ldots,u_{n-1})
q^{\mathbb N + 1},
& |q|<1,
\\
\mathfrak Z(\Phi_{n-1};u_1,\ldots,u_{n-1})
q^{-\mathbb N - 1},
& |q|>1.
\end{cases}
\label{Eq:MovablePolesFromZeros}
\end{align}
\end{subequations}
Moreover, for generic initial data, the inclusion
\eqref{Eq:PhinPoleClassification} becomes an equality.
\end{lem}

\begin{prf}
This follows directly from \eqref{Exp:Vk} and \eqref{ExpclitPhiW}.
\end{prf}

\begin{lem}\label{Lem:CanXn0}
For each fixed
$\zeta_n\in\widetilde{\mathbb C^\ast}$,
the $x_n$-equation \eqref{Eq:TargetSysNew-x}
admits a unique canonical fundamental solution
in a punctured neighborhood of $x_n=0$ of the form
\begin{subequations}\label{Eq:Yxn0}
\begin{align}
\label{Eq:Yxn0Y}
\mathcal Y_{x_n}^{[0]}(x_n,\zeta_n)
&=
\mathcal H_{x_n}^{[0]}(x_n,\zeta_n)\cdot
x_n^{[\delta_{n-1}\Phi_{n-1}]^q},
\\
\label{Eq:Yxn0H}
\left(
\frac{\zeta_n}{x_n}
\right)^{
\operatorname{ad}[\delta_{n-1}\Phi_{n-1}]^q}
\mathcal H_{x_n}^{[0]}(x_n,\zeta_n)
&=
I+O(x_n^{\varepsilon_{n-1}}),
\qquad x_n\to0.
\end{align}
\end{subequations}
The estimate in \eqref{Eq:Yxn0H}
	is locally uniform in $\zeta_n$.
Moreover,
$\mathcal H_{x_n}^{[0]}(x_n,\zeta_n)^{\pm 1}$
admit meromorphic continuations in $x_n$ to
$\widetilde{\mathbb C^\ast}$.
Their pole sets as functions of $x_n$ are as follows.
\begin{subequations}\label{Pole:Yxn0H}
\begin{align}\label{Pole:Yxn0H1}
\operatorname{Pole}_{x_n}
\left(
\mathcal H_{x_n}^{[0]}
\right)
&\subseteq
\begin{cases}
\left(\zeta_n
\operatorname{Pole}\left(\Phi_n\right)^{-1}
\cup
\displaystyle
\bigcup_{1\leqslant k<n}
\left\{
\frac{q^{\alpha_k}}{(1-q)u_k}
\right\}
\right) q^{-\mathbb N},
& |q|<1,
\\
\zeta_n
\operatorname{Pole}\left(\Phi_n\right)^{-1} q^{\mathbb N},
& |q|>1,
\end{cases}
\\
\label{Pole:Yxn0Hm1}
\operatorname{Pole}_{x_n}
\left(
(\mathcal H_{x_n}^{[0]})^{-1}
\right)
&\subseteq
\begin{cases}
\zeta_n
\operatorname{Pole}\left(\Phi_n\right)^{-1} q^{-\mathbb N},
& |q|<1,
\\
\left(
\zeta_n
\operatorname{Pole}\left(\Phi_n\right)^{-1}
\cup
\displaystyle
\bigcup_{1\leqslant k<n}
\left\{
\frac{q^{\alpha_k}}{(1-q)q^{-1}u_k}
\right\}
\right) q^{\mathbb N},
& |q|>1.
\end{cases}
\end{align}
\end{subequations}
\end{lem}

\begin{prf}
Using the shrinking condition \eqref{Cond:less1}
	and \eqref{Lim:dk1PktoPk1},
	we have
\begin{align}\label{Est:PhiWeightedTop}
u_n^{
\operatorname{ad}[\Phi_{n-1}^{[n-1]}]^q}
U_{\hat n\hat n}
=
O(u_n^{1-\varepsilon_{n-1}}),\qquad
u_n^{
\operatorname{ad}[\Phi_{n-1}^{[n-1]}]^q}
\Phi_n(u_n)_{\hat n\hat n} =
\Phi_{n-1}^{[n-1]} +
O(u_n^{-\varepsilon_{n-1}}).
\end{align}
Taking the off-diagonal blocks in \eqref{Eq:useful} gives
\begin{align}\label{Est:VWeightedOffDiag}
u_n^{[\Phi_{n-1}^{[n-1]}]^q-\alpha_n I}
V_{n;n}^{(n)}(u_n)_{\hat n n}
=
O(1),\qquad
V_{n;n}^{(n)}(u_n)_{n\hat n}
u_n^{\alpha_n I-[\Phi_{n-1}^{[n-1]}]^q}
=
O(1).
\end{align}
Since $\zeta_n=x_nu_n$, equations
\eqref{Est:PhiWeightedTop},
\eqref{Est:VWeightedOffDiag}, and
\eqref{Eq:A1Block} give
\begin{align}
\label{Est:A1Weighted}
\left(
\frac{\zeta_n}{x_n}
\right)^{
\operatorname{ad}[\delta_{n-1}\Phi_{n-1}]^q}
\left(
A_1(x_n,\zeta_n)
-
\frac{\delta_{n-1}\Phi_{n-1}}{x_n}
\right)
&=
O(x_n^{-1+\varepsilon_{n-1}}).
\end{align}
Substituting \eqref{Eq:Yxn0Y} into
\eqref{Eq:TargetSysNew-x} and applying
\eqref{Est:A1Weighted}, we obtain
\begin{align*}
\frac{\partial_q}{\partial_q x_n}
\left(
\left(
\frac{\zeta_n}{x_n}
\right)^{
\operatorname{ad}[\delta_{n-1}\Phi_{n-1}]^q}
\mathcal H_{x_n}^{[0]}(x_n,\zeta_n)
\right)
&=
O(x_n^{-1+\varepsilon_{n-1}})\cdot
\left(
\frac{\zeta_n}{x_n}
\right)^{
\operatorname{ad}[\delta_{n-1}\Phi_{n-1}]^q}
\mathcal H_{x_n}^{[0]}(x_n,\zeta_n).
\end{align*}
The Picard iteration therefore gives a unique solution
satisfying \eqref{Eq:Yxn0H}.
The estimates above are uniform for $\zeta_n$ in compact subsets
of $\widetilde{\mathbb C^\ast}$, so the estimate in
\eqref{Eq:Yxn0H} is locally uniform in $\zeta_n$.

For the meromorphic continuation, set
\[
B(x_n,\zeta_n):=
I+(q-1)x_nA_1(x_n,\zeta_n).
\]
Then \eqref{Eq:TargetSysNew-x} and \eqref{Eq:Yxn0Y} give
\begin{align}\label{Eq:Yxn0Recurrence}
\mathcal H_{x_n}^{[0]}(qx_n,\zeta_n)
&=
B(x_n,\zeta_n)
\mathcal H_{x_n}^{[0]}(x_n,\zeta_n)
q^{-[\delta_{n-1}\Phi_{n-1}]^q}.
\end{align}
Equations \eqref{Eq:A1xz} and \eqref{Eq:A1Block},
together with \eqref{Exp:Vk} and \eqref{ExpclitPhiW}, give
\begin{subequations}\label{Pole:Yxn0B}
\begin{align}
\operatorname{Pole}_{x_n}
\left(
B(x_n,\zeta_n)
\right)
&\subseteq
\operatorname{Pole}_{x_n}
\left(
\Phi_n\left(\frac{\zeta_n}{x_n}\right)
\right) \{1,q^{-1}\},
\\
\operatorname{Pole}_{x_n}
\left(
B(x_n,\zeta_n)^{-1}
\right)
&\subseteq
\operatorname{Pole}_{x_n}
\left(
\Phi_n\left(\frac{\zeta_n}{x_n}\right)
\right) \{1,q^{-1}\}
\cup
\bigcup_{1\leqslant k<n}
\left\{
\frac{q^{\alpha_k}}{(1-q)u_k}
\right\}.
\end{align}
\end{subequations}
Since $\mathcal H_{x_n}^{[0]}(x_n,\zeta_n)^{\pm1}$
are holomorphic near $x_n=0$, iterating
\eqref{Eq:Yxn0Recurrence} in the appropriate direction and using
\eqref{Pole:Yxn0B} gives their meromorphic continuations together
with \eqref{Pole:Yxn0H1} and \eqref{Pole:Yxn0Hm1}.
\end{prf}

We next introduce notation for the canonical fundamental solutions
of \eqref{Eq:TargetSystem} and of the two limiting systems used in
the factorization.

\begin{itemize}

\item In the $(x,u_n)$-coordinates,
Proposition~\ref{Can:zero} gives the canonical fundamental solution
of the same $x$-equation at $x=0$:
\begin{subequations}
\begin{align}
Y^{[0]}(x,u_n)
&=
H^{[0]}(x,u_n)\cdot
x^{[\Phi_n(u_n)]^q},\\
H^{[0]}(x,u_n)
&=
I+O(x),
\qquad x\to0.
\end{align}
\end{subequations}
Proposition~\ref{Can:inf}
gives the canonical fundamental solution with formal monodromy
$\boldsymbol{\alpha}$ of the $x$-equation in
\eqref{Eq:TargetSystem} at $x=\infty$:
\begin{subequations}
\begin{align}
\label{Eq:YInfNormal}
Y^{[\infty]}(x,u_n)
&=
H^{[\infty]}(x,u_n)\cdot
x^{\boldsymbol{\alpha}}
e_q(q^{-\boldsymbol{\alpha}}Ux),\\
H^{[\infty]}(x,u_n)
&=
I+O(x^{-1}),
\qquad x\to\infty.
\end{align}
\end{subequations}

\item The limit $\zeta_n\to\infty$ in \eqref{Eq:EstA1}
gives the limiting system
\begin{align}\label{Eq:TargetSysLim-x}
\frac{\partial_q}{\partial_q x_n}T(x_n)
&=
\left(
U_{\hat n} +
\frac{\delta_{n-1}\Phi_{n-1}}{x_n}
\right)T(x_n),\qquad
U_{\hat n}:=
\begin{pmatrix}
U_{\hat n\hat n}&0\\
0&0
\end{pmatrix}.
\end{align}
Proposition~\ref{Can:zero} gives its canonical fundamental solution
at $x_n=0$:
\begin{subequations}
\begin{align}
T_{x_n}^{[0]}(x_n)
&=
R_{x_n}^{[0]}(x_n)\cdot
x_n^{[\delta_{n-1}\Phi_{n-1}]^q},\\
R_{x_n}^{[0]}(x_n)
&=
I+O(x_n),
\qquad x_n\to0.
\end{align}
\end{subequations}
Proposition~\ref{Can:inf} gives its canonical fundamental solution
with formal monodromy $\boldsymbol{\alpha}$ at $x_n=\infty$:
\begin{subequations}
\begin{align}
\label{Eq:TxnInfNormal}
T_{x_n}^{[\infty]}(x_n)
&=
R_{x_n}^{[\infty]}(x_n)\cdot
x_n^{\boldsymbol{\alpha}}
\begin{pmatrix}
e_q\!\left(
q^{-\boldsymbol{\alpha}_{\hat n\hat n}}
U_{\hat n\hat n}x_n
\right)&0\\
0&1
\end{pmatrix},\\
R_{x_n}^{[\infty]}(x_n)
&=
I+O(x_n^{-1}),
\qquad x_n\to\infty.
\end{align}
\end{subequations}

\item The limit $x_n\to0$ in \eqref{Eq:EstA2}
gives the limiting system
\begin{align}\label{Eq:TargetSysLim-zeta}
\frac{\partial_q}{\partial_q\zeta_n}T(\zeta_n)
&=
\left(
E_n+\frac{\Phi_{n-1}}{\zeta_n}
\right)T(\zeta_n).
\end{align}
Proposition~\ref{Can:zero} gives its canonical fundamental solution
at $\zeta_n=0$:
\begin{subequations}
\begin{align}
T_{\zeta_n}^{[0]}(\zeta_n)
&=
R_{\zeta_n}^{[0]}(\zeta_n)\cdot
\zeta_n^{[\Phi_{n-1}]^q},\\
R_{\zeta_n}^{[0]}(\zeta_n)
&=
I+O(\zeta_n),
\qquad \zeta_n\to0.
\end{align}
\end{subequations}
Proposition~\ref{Can:inf} gives its formal fundamental solution
with formal monodromy
$[\delta_{n-1}\Phi_{n-1}]^q$ at $\zeta_n=\infty$:
\begin{subequations}
\begin{align}
T_{\zeta_n}^{[\infty]}(\zeta_n)
&=
R_{\zeta_n}^{[\infty]}(\zeta_n)\cdot
\zeta_n^{[\delta_{n-1}\Phi_{n-1}]^q}
\begin{pmatrix}
I&0\\
0&e_q(q^{-\alpha_n}\zeta_n)
\end{pmatrix},\\
R_{\zeta_n}^{[\infty]}(\zeta_n)
&=
I+\sum_{\nu=1}^{\infty}
R_{\zeta_n,\nu}^{[\infty]}\zeta_n^{-\nu}.
\end{align}
\end{subequations}
Proposition~\ref{Can:inf}
	shows that the last column
	$(R_{\zeta_n}^{[\infty]})_{\ast n}$
	is convergent when $|q|<1$, whereas the last row
	$(R_{\zeta_n}^{[\infty] -1})_{n\ast}$
	is convergent when $|q|>1$.

\end{itemize}

\begin{lem}\label{Lem:CanZetaInf}
For each fixed
$x_n\in
\widetilde{\mathbb C^\ast}
\setminus
\underset{1\leqslant k<n}{\bigcup}
\frac{q^{\alpha_k}}{(1-q)u_k}
q^{\mathbb Z}$,
define
\begin{align}\label{Eq:CanZetaInfDef}
\mathcal Y_{\zeta_n}^{[\infty]}(x_n,\zeta_n)
&:=
Y^{[\infty]}
\left(x_n,\frac{\zeta_n}{x_n}\right)
T_{x_n}^{[\infty]}(x_n)^{-1}
x_n^{[\delta_{n-1}\Phi_{n-1}]^q}.
\end{align}
Then $\mathcal Y_{\zeta_n}^{[\infty]}(x_n,\zeta_n)$
is a fundamental solution of the $\zeta_n$-equation
\eqref{Eq:TargetSysNew-zeta}. 
When the additional open condition \eqref{Cond:less1new} holds,
this solution has the following form:
\begin{subequations}
\begin{align}
\left(
	\frac{\zeta_n}{x_n}
\right)^{[\delta_{n-1}\Phi_{n-1}]^q}
\mathcal Y_{\zeta_n}^{[\infty]}(x_n,\zeta_n)
&=
\mathcal H_{\zeta_n}^{[\infty]}(x_n,\zeta_n)
\zeta_n^{[\delta_{n-1}\Phi_{n-1}]^q}
\begin{pmatrix}
I&0\\
0&e_q(q^{-\alpha_n}\zeta_n)
\end{pmatrix},
\label{Eq:CanZetaInfNormal}
\\
\left(
	\frac{\zeta_n}{x_n}
\right)^{
	-\mathrm{ad}[\delta_{n-1}\Phi_{n-1}]^q
}
\mathcal H_{\zeta_n}^{[\infty]}(x_n,\zeta_n)
&=
I+
O\!\left(
\zeta_n^{-\epsilon_\alpha}
\right),
\qquad
\zeta_n\to\infty.
\label{Eq:CanZetaInfAsym}
\end{align}
\end{subequations}
The estimate in \eqref{Eq:CanZetaInfAsym}
is locally uniform in $x_n$.
Moreover,
$\mathcal H_{\zeta_n}^{[\infty]}(x_n,\zeta_n)^{\pm1}$
are meromorphic in $\zeta_n$ on
$\widetilde{\mathbb C^\ast}$, and\begin{align}\label{Pole:CanZetaInfGauge}
\operatorname{Pole}
\left(
(\mathcal H_{\zeta_n}^{[\infty]})^{\pm1}
\right)
&\subseteq
\left\{(x_n,\zeta_n):
\frac{\zeta_n}{x_n}
\in
\operatorname{Pole}\left(\Phi_n\right) \cup
\{u_1,\ldots,u_{n-1}\}
q^{\mathbb Z\setminus\{0\}}
\right\}
\notag\\
&\quad\cup
\begin{cases}
\displaystyle
\left\{(x_n,\zeta_n):
x_n\in
\bigcup_{1\leqslant k<n}
\frac{q^{\alpha_k}}{(1-q)u_k}q^{\mathbb N+1}
\quad\text{or}\quad
\zeta_n\in
\frac{q^{\alpha_n}}{1-q}q^{\mathbb N+1}
\right\},
& |q|<1,
\\[8pt]
\displaystyle
\left\{(x_n,\zeta_n):
x_n\in
\bigcup_{1\leqslant k<n}
\frac{q^{\alpha_k}}{(1-q)u_k}q^{-\mathbb N}
\quad\text{or}\quad
\zeta_n\in
\frac{q^{\alpha_n}}{1-q}q^{-\mathbb N}
\right\},
& |q|>1.
\end{cases}
\end{align}
\end{lem}

\begin{prf}
Since $Y^{[\infty]}(x,u_n)$ is a common fundamental solution of
the two equations in \eqref{Eq:TargetSystem}, it follows directly
from \eqref{Eq:Trans} and \eqref{Eq:CanZetaInfDef} that
$\mathcal Y_{\zeta_n}^{[\infty]}(x_n,\zeta_n)$ is a fundamental
solution of the $\zeta_n$-equation
\eqref{Eq:TargetSysNew-zeta}.

Substituting \eqref{Eq:YInfNormal} and
\eqref{Eq:TxnInfNormal} into \eqref{Eq:CanZetaInfDef} gives
\begin{align}
\left(\frac{\zeta_n}{x_n}\right)^{
-\operatorname{ad}[\delta_{n-1}\Phi_{n-1}]^q}
\mathcal H_{\zeta_n}^{[\infty]}(x_n,\zeta_n)
&=
H^{[\infty]}
\left(x_n,\frac{\zeta_n}{x_n}\right)
R_{x_n}^{[\infty]}(x_n)^{-1}.
\label{Eq:CanZetaInfWeightedIdentity}
\end{align}
The existence conditions in Proposition~\ref{Can:inf},
together with Proposition~\ref{Pro:SolY} applied to
$H^{[\infty]}$ and $R_{x_n}^{[\infty]}$,
give \eqref{Pole:CanZetaInfGauge}.

Write
\begin{align*}
H^{[\infty]}(x,u_n)
R_{x_n}^{[\infty]}(x)^{-1}
&=
I+\sum_{\nu=1}^{\infty}F_\nu(u_n)x^{-\nu}.
\end{align*}
Comparing the $x$-equation in \eqref{Eq:TargetSystem} with
\eqref{Eq:TargetSysLim-x}, one directly verifies that, for every
$\nu\geqslant0$,
\begin{align}\label{Eq:RecFun}
UF_{\nu+1}(u_n)
-q^{-(\nu+1)}F_{\nu+1}(u_n)U
&=
F_\nu(u_n)
\left(
[-\nu]_qI+
q^{-\nu}\delta_{n-1}\Phi_{n-1}
\right)
-\Phi_n(u_n)F_\nu(u_n).
\end{align}
Here we set $F_0=I$.
We first assume that $|q|<1$.
There exists a constant $C>0$,
	independent of $\nu$, such that
\begin{align}
	\left\|
		\frac{q^{-\nu}}
		{u_nI-q^{-(\nu+1)}U_{\hat n\hat n}}
	\right\|
	\leqslant C,\qquad
	\left\|
		\frac{1}
		{u_nI-q^{-(\nu+1)}U_{\hat n\hat n}}
	\right\|
	&\leqslant C|u_n|^{-1}.
\end{align}
Using \eqref{Lim:dk1ktok1} and
	the additional open condition \eqref{Cond:less1new},
	it follows by induction from \eqref{Eq:RecFun} that,
	for every $\nu\geqslant1$,
	the following estimates hold with a constant $M>0$ 
	independent of $\nu$:
\begin{align}
F_\nu(u_n)
&=
\begin{pmatrix}
	O\left( M^\nu u_n^{-\varepsilon_{n-1}}
	\right) &
	O\left( M^\nu u_n^{-\nu\varepsilon_\alpha}
	\right) \\
	O\left( M^\nu u_n^{-\varepsilon_\alpha}
	\right) &
	O\left( M^\nu u_n^{-1-(\nu-1)\varepsilon_\alpha}
	\right)
\end{pmatrix},\qquad
u_n\to\infty,
\end{align}
together with
\begin{align}
\Phi_n(u_n)_{\hat n n}F_\nu(u_n)_{n\hat n}
&=
O\left(M^\nu u_n^{-\varepsilon_{n-1}}\right),
&
\Phi_n(u_n)_{n\hat n}F_\nu(u_n)_{\hat n n}
&=
O\left(M^\nu u_n^{-\nu\varepsilon_\alpha}\right).
\end{align}
It follows that, for sufficiently large $x_n$,
\begin{align}
H^{[\infty]}(x_n,u_n)
R_{x_n}^{[\infty]}(x_n)^{-1}
&=
I+
\begin{pmatrix}
	O\left( u_n^{-\varepsilon_{n-1}}
		\right) &
	O\left( u_n^{-\varepsilon_\alpha}
		\right) \\
	O\left( u_n^{-\varepsilon_\alpha}
		\right) &
	O\left( u_n^{-1}
		\right)
\end{pmatrix}.
\end{align}
The case $|q|>1$ follows similarly.
Combining this with \eqref{Eq:CanZetaInfWeightedIdentity} and
$u_n=\zeta_n/x_n$ gives \eqref{Eq:CanZetaInfAsym}.
This completes the proof.
\end{prf}

Finally,
	we apply the $q$-Borel summation introduced in
	Section~\ref{Sect:qBorel} to
	the formal divergent series
	$R_{\zeta_n}^{[\infty]}(\zeta_n)$ above,
	and thereby obtain actual fundamental solutions of 
	\eqref{Eq:TargetSysLim-zeta} at $\zeta_n=\infty$.
What is particularly unusual here is that,
	in order to obtain the factorization
	required for our final formula,
	we do not use a single
	summation direction for the entire formal matrix.
Instead, to obtain the desired singularity structure,
	the summation directions are assigned after
	a suitable rearrangement of the matrix entries,
	as described below.
	
\begin{defi}\label{Def:ReassembledQBorelSum}
For a $q$-monodromy system
$(U,\Phi;W;\boldsymbol{\alpha})$, define its
\textbf{normalized central connection matrix} by
\begin{subequations}
\label{Def:NormalizedCentralConnection}
\begin{align}
\Omega(x;U,\Phi) &:=
Y^{[\infty]}(x;U,\Phi)^{-1}
Y^{[0]}(x;U,\Phi),\\
\Omega_d(x;U,\Phi) &:=
Y^{[\infty]}_d(x;U,\Phi)^{-1}
Y^{[0]}(x;U,\Phi),\qquad d\in\mathbb{C}^\ast.
\end{align}
\end{subequations}
For $\boldsymbol d=(
	d_1,\ldots, d_n) \in (\mathbb C^\ast)^n$
	and a formal matrix series $\widehat{F}(x)$,
	denote
\begin{subequations}
\begin{align}
\mathcal S_{\boldsymbol d}^{(+)}\widehat F &:=
\begin{pmatrix}
	\mathcal S_{ d_1 }\widehat F_{\ast 1} &
	\cdots &
	\mathcal S_{ d_n }\widehat F_{\ast n}
\end{pmatrix},
\label{Def:ColumnwiseQBorelSum}\\
\mathcal S_{\boldsymbol d}^{(-)}\widehat F &:=
\left(
	\mathcal S_{\boldsymbol d}^{(+)}
	(\widehat F^{-\top})
\right)^{-\top}.
\end{align}
\end{subequations}
When the $j$-th column is convergent, the choice of $d_j$ does
not affect its $q$-Borel sum. In this case, we write $d_j=\ast$.
Denote
\begin{subequations}
\begin{align}
\mathcal SR_{\zeta_n}^{[\infty]}(\zeta_n;x_n) &:=
\begin{cases}
\displaystyle
	\mathcal S_{\boldsymbol d}^{(+)}
	\left(
	R_{\zeta_n}^{[\infty]}(\zeta_n) \cdot
	\zeta_n^{[\delta_{n-1}\Phi_{n-1}]^q}
	\Omega(x_n)^{-1}
	\right)
	\Omega(x_n)
	\zeta_n^{-[\delta_{n-1}\Phi_{n-1}]^q},
	& |q|<1,
\\
\displaystyle
	\mathcal S_{\boldsymbol d}^{(-)}
	\left(
	R_{\zeta_n}^{[\infty]}(\zeta_n) \cdot
	\zeta_n^{[\delta_{n-1}\Phi_{n-1}]^q}
	\Omega(x_n)^{-1}
	\right)
	\Omega(x_n)
	\zeta_n^{-[\delta_{n-1}\Phi_{n-1}]^q},
& |q|>1,
\end{cases}
\label{eq:second-summed-direction}
\\
\mathcal ST_{\zeta_n}^{[\infty]}(\zeta_n;x_n)
&:=
\mathcal SR_{\zeta_n}^{[\infty]}(\zeta_n;x_n) \cdot
\zeta_n^{[\delta_{n-1}\Phi_{n-1}]^q}
\begin{pmatrix}
I&0\\
0&e_q(q^{-\alpha_n}\zeta_n)
\end{pmatrix},
\label{Eq:SummedZetaInfSolution}
\end{align}
\end{subequations}
where $\boldsymbol{d} = (
	-x_nu_1,\ldots,
	-x_nu_{n-1},\ast)$
and $\Omega(x_n) =
\Omega(x_n;U_{\hat n},\delta_{n-1}\Phi_{n-1})$.
\end{defi}

Notice that each entry of
	$\zeta_n^{[\delta_{n-1}\Phi_{n-1}]^q}$
	can be written as a finite sum in terms of
	$\zeta_n^\lambda(\log\zeta_n)^j$.
In \eqref{eq:second-summed-direction},
	the $q$-Borel summations
	are understood by first writing
	$\zeta_n^{[\delta_{n-1}\Phi_{n-1}]^q}$
entrywise as the finite sums described above, then forming the
matrix product
\[
R_{\zeta_n}^{[\infty]}(\zeta_n)
\zeta_n^{[\delta_{n-1}\Phi_{n-1}]^q}
\Omega(x_n)^{-1},
\]
	and finally taking
	the $q$-Borel sum of each formal series
	appearing in the entries.
Propositions~\ref{Pro:SolY} and
	\ref{Pro:qBorelSummability}
	show that \eqref{eq:second-summed-direction}
	is well defined
	for the values of $x_n$ considered below.

\begin{lem}\label{Lem:SummedZetaInf}
For each fixed
$x_n\in
\widetilde{\mathbb C^\ast}
\setminus
\left(
\underset{1\leqslant k<n}{\bigcup}
\frac{q^{\alpha_k}}
{(1-q)u_k}q^{\mathbb Z}
\cup
\underset{1\leqslant i,j<n}{\bigcup}
\frac{q^{\lambda_i^{(n-1)}}}
{(1-q)u_j}q^{\mathbb Z}
\right)$,
the matrix function
$\mathcal ST_{\zeta_n}^{[\infty]}(\zeta_n;x_n)$
defined in \eqref{Eq:SummedZetaInfSolution}
is a fundamental solution of
\eqref{Eq:TargetSysLim-zeta}.
The matrix function
$\mathcal SR_{\zeta_n}^{[\infty]}(\zeta_n;x_n)$
admits
$R_{\zeta_n}^{[\infty]}(\zeta_n)$
as its $q$-Gevrey asymptotic expansion of order one at
$\zeta_n=\infty$, away from
$\{x_nu_1,\ldots,x_nu_{n-1}\}q^{\mathbb Z}$.
In particular, for every $N\geqslant0$ we have
\begin{align*}
\mathcal SR_{\zeta_n}^{[\infty]}(\zeta_n;x_n)
&=
I+\sum_{\nu=1}^{N}
R_{\zeta_n,\nu}^{[\infty]}\zeta_n^{-\nu}
+O(\zeta_n^{-N-1}),
\qquad
\zeta_n\to\infty.
\end{align*}
Moreover,
$\mathcal SR_{\zeta_n}^{[\infty]}(\zeta_n;x_n)^{\pm1}$
are meromorphic in $\zeta_n$ on
$\widetilde{\mathbb C^\ast}$.
For every $1\leqslant k<n$, we have
\begin{subequations}\label{Pole:SummedZetaInf}
\begin{align}
\operatorname{Pole}_{\zeta_n}
\left(
\left(
\mathcal SR_{\zeta_n}^{[\infty]}(\zeta_n;x_n)
\zeta_n^{[\delta_{n-1}\Phi_{n-1}]^q}
\Omega(x_n)^{-1}
\right)_{\ast k}
\right)
&\subseteq
x_nu_kq^{\mathbb Z},
&& |q|<1,
\\
\operatorname{Pole}_{\zeta_n}
\left(
\left(
\left(
\mathcal SR_{\zeta_n}^{[\infty]}(\zeta_n;x_n)
\zeta_n^{[\delta_{n-1}\Phi_{n-1}]^q}
\Omega(x_n)^{-1}
\right)^{-1}
\right)_{k\ast}
\right)
&\subseteq
x_nu_kq^{\mathbb Z},
&& |q|>1,
\end{align}
\end{subequations}
and all these poles have order at most one.
\end{lem}

\begin{prf}
By the standard morphism property of the
$q$-Borel summation
\cite{dreyfus2015building},
$\mathcal ST_{\zeta_n}^{[\infty]}(\zeta_n;x_n)$
is a fundamental solution of
\eqref{Eq:TargetSysLim-zeta}, and
$\mathcal SR_{\zeta_n}^{[\infty]}(\zeta_n;x_n)$
admits
$R_{\zeta_n}^{[\infty]}(\zeta_n)$
as its $q$-Gevrey asymptotic expansion of order one at
$\zeta_n=\infty$.
The remaining assertions follow directly from
Proposition~\ref{Pro:qBorelSummability}.
\end{prf}

\subsection{Factorization of the Canonical Fundamental Solutions}
\label{Sect:CanFact}

In this subsection,
we prove the following three factorization formulas for
the canonical fundamental solutions:
\begin{align}
Y^{[\infty]}(x,u_n) &=
\mathcal Y_{\zeta_n}^{[\infty]}(x_n,\zeta_n) \cdot
	x_n^{-[\delta_{n-1}\Phi_{n-1}]^q}
	T_{x_n}^{[\infty]}(x_n),
\tag{I}\label{Eq:FirstFactorization}\\
\mathcal Y_{x_n}^{[0]}(x_n,\zeta_n) \cdot
	\zeta_n^{-[\delta_{n-1}\Phi_{n-1}]^q}
	\mathcal ST_{\zeta_n}^{[\infty]}(\zeta_n;x_n) &=
\mathcal Y_{\zeta_n}^{[\infty]}(x_n,\zeta_n) \cdot
	x_n^{-[\delta_{n-1}\Phi_{n-1}]^q}
	T_{x_n}^{[0]}(x_n),
\tag{II}\label{Eq:SecondFactorization}\\
\mathcal Y_{x_n}^{[0]}(x_n,\zeta_n) \cdot
	\zeta_n^{-[\delta_{n-1}\Phi_{n-1}]^q}
	T_{\zeta_n}^{[0]}(\zeta_n)W_{n-1} &=
Y^{[0]}(x,u_n) \cdot W_n(u_n).
\tag{III}\label{Eq:ThirdFactorization}
\end{align}
\noindent
Identity \eqref{Eq:FirstFactorization}
follows directly from the definition \eqref{Eq:CanZetaInfDef}
in Lemma~\ref{Lem:CanZetaInf}.

We first note that the matrix functions on both sides of
\eqref{Eq:FirstFactorization}--\eqref{Eq:ThirdFactorization}
are common solutions of \eqref{Eq:TargetSysNew}.

\begin{lem}\label{Lem:SecondCommonSolution}
If $T(\zeta_n)$ is a solution of
\eqref{Eq:TargetSysLim-zeta}, then
$\mathcal Y_{x_n}^{[0]}(x_n,\zeta_n)
\zeta_n^{-[\delta_{n-1}\Phi_{n-1}]^q}
T(\zeta_n)$
is a common solution of \eqref{Eq:TargetSysNew}.
\end{lem}

\begin{prf}
It suffices to verify that
$\mathcal Y_{x_n}^{[0]}(x_n,\zeta_n)
\zeta_n^{-[\delta_{n-1}\Phi_{n-1}]^q}
T(\zeta_n)$
satisfies \eqref{Eq:TargetSysNew-zeta}.
Denote the shift matrices of
\eqref{Eq:TargetSysNew-zeta} and
\eqref{Eq:TargetSysLim-zeta} by
\begin{align*}
\mathcal B_2(x_n,\zeta_n)
:=
I+(q-1)\zeta_nA_2(x_n,\zeta_n),\qquad
\mathcal B_\infty(\zeta_n)
:=
I+(q-1)(\zeta_nE_n+\Phi_{n-1}),
\end{align*}
respectively. 
The compatibility of the two equations in
\eqref{Eq:TargetSysNew} implies that
\begin{align}
&\left(
\left(\frac{q\zeta_n}{x_n}\right)^{
[\delta_{n-1}\Phi_{n-1}]^q}
\mathcal Y_{x_n}^{[0]}(x_n,q\zeta_n)
\right)^{-1}
\mathcal B_2(x_n,\zeta_n)
\left(\frac{\zeta_n}{x_n}\right)^{
[\delta_{n-1}\Phi_{n-1}]^q}
\mathcal Y_{x_n}^{[0]}(x_n,\zeta_n)
\label{Eq:SecondCommonShiftQuotient}
\end{align}
is a $q$-constant with respect to $x_n$.
By Lemma~\ref{Lem:CanXn0},
\begin{align*}
\left(\frac{\zeta_n}{x_n}\right)^{
	[\delta_{n-1}\Phi_{n-1}]^q}
	\mathcal Y_{x_n}^{[0]}(x_n,\zeta_n)
	\zeta_n^{-[\delta_{n-1}\Phi_{n-1}]^q}
	&=
	I+O(x_n^{\varepsilon_{n-1}}),
	\qquad x_n\to0.
\end{align*}
Together with \eqref{Eq:EstA2} in
	Lemma~\ref{Lem:TargetSysNew},
this shows that the limit of
\eqref{Eq:SecondCommonShiftQuotient} as $x_n\to0$ is
\begin{align}
	(q\zeta_n)^{-[\delta_{n-1}\Phi_{n-1}]^q}
	\mathcal B_\infty(\zeta_n)
	\zeta_n^{[\delta_{n-1}\Phi_{n-1}]^q}.
\label{Eq:SecondCommonShiftQuotient2}
\end{align}
Therefore, \eqref{Eq:SecondCommonShiftQuotient}
	is equal to \eqref{Eq:SecondCommonShiftQuotient2}.
A direct calculation completes the proof.
\end{prf}

To prove identity \eqref{Eq:SecondFactorization},
we need the following remainder estimate.
\begin{lem}\label{Lem:ZeroCanonicalStability}
For each fixed
$x_n\in
\widetilde{\mathbb C^\ast}
\setminus
\begin{cases}
\underset{1\leqslant k<n}{\bigcup}
\frac{q^{\alpha_k}}{(1-q)u_k}q^{-\mathbb N},
& |q|<1,
\\[8pt]
\underset{1\leqslant k<n}{\bigcup}
\frac{q^{\alpha_k}}{(1-q)u_k}q^{\mathbb N+1},
& |q|>1.
\end{cases}$,
under the additional open condition \eqref{Cond:less1new},
	we have
\begin{align}
	\mathcal H_{x_n}^{[0]}(x_n,\zeta_n)
	R_{x_n}^{[0]}(x_n)^{-1} &=
	I+O(
	\zeta_n^{-\epsilon_\alpha} ),
	\qquad
	\zeta_n\to\infty.
	\label{Eq:ZeroCanonicalStability}
\end{align}
\end{lem}

\begin{prf}
For each $x_n$ in the domain specified in the statement
and all sufficiently large $\zeta_n$, denote
\begin{align*}
G(x_n,\zeta_n)
&:=
x_n^{-\operatorname{ad}[\delta_{n-1}\Phi_{n-1}]^q}
\left(
	R_{x_n}^{[0]}(x_n)^{-1}
	\mathcal H_{x_n}^{[0]}(x_n,\zeta_n)
\right).
\end{align*}
Since the coefficient matrix in
\eqref{Eq:TargetSysLim-x} is block diagonal,
the uniqueness in Proposition~\ref{Can:zero} shows that
$R_{x_n}^{[0]}(x_n)$ is block diagonal.
Hence
$R_{x_n}^{[0]}(x_n)=I+O(x_n)$ and the shrinking condition
\eqref{Cond:less1} imply
\[
x_n^{-\operatorname{ad}[\delta_{n-1}\Phi_{n-1}]^q}
R_{x_n}^{[0]}(x_n)^{-1}
=
I+O(x_n^{\varepsilon_{n-1}}).
\]
Together with equation \eqref{Eq:Yxn0H} in
Lemma~\ref{Lem:CanXn0}, this gives
\begin{align}
G(x_n,\zeta_n)&\longrightarrow I,
\qquad x_n\to0,
\label{Lim:ZeroStabilityBoundary}
\end{align}
for each fixed $\zeta_n$.
The $x_n$-equation
	\eqref{Eq:TargetSysNew-x} and its limiting equation
	\eqref{Eq:TargetSysLim-x} give
	the following $q$-Jackson integral equation for $G$
	as a function of $x_n$:
\begin{align}
G(x_n,\zeta_n)
&=
I+
\int_0^{x_n}
T_{x_n}^{[0]}(qx_n)^{-1}
\left(
A_1(x_n,\zeta_n)
-
\left(
	U_{\hat n} +
	\frac{\delta_{n-1}\Phi_{n-1}}{x_n}
\right)
\right)
T_{x_n}^{[0]}(x_n)
G(x_n,\zeta_n)
\,\mathrm d_qx_n.
\label{Eq:ZeroStabilityPicardIntegral}
\end{align}
We now establish the required estimate for
	the kernel of the integral operator in 
	\eqref{Eq:ZeroStabilityPicardIntegral}.
The standard Picard iteration then gives the unique solution
satisfying \eqref{Lim:ZeroStabilityBoundary}.
The resulting solution has the asymptotic behavior required for
\eqref{Eq:ZeroCanonicalStability}.

As in the proof of Lemma~\ref{Lem:Picard}, denote the
Jordan--Chevalley decomposition of
$[\delta_{n-1}\Phi_{n-1}]^q$ and the spectral decomposition of
its semisimple part $\mathscr S$ by
\begin{align*}
[\delta_{n-1}\Phi_{n-1}]^q
&=
\mathscr S+\mathscr N
=
\lambda^{(n-1)}_1\mathscr P_1 + \cdots +
\lambda^{(n-1)}_{n-1}\mathscr P_{n-1}+
\lambda^{(n-1)}_{n}\mathscr P_{n}+
\mathscr N,
\end{align*}
where $\mathscr P_i$ denotes the spectral projection
corresponding to $\lambda_i^{(n-1)}$ for $1\leqslant i<n$,
and we further denote
\begin{align*}
\lambda_n^{(n-1)}:=\alpha_n,\qquad
\mathscr P_n:=E_n.
\end{align*}
Denote
\begin{align*}
\sigma_{ij}
:=
\operatorname{Re}(
	\lambda^{(n-1)}_i -
	\lambda^{(n-1)}_j)-
\frac{\arg q}{\ln|q|}
\operatorname{Im}(
	\lambda^{(n-1)}_i -
	\lambda^{(n-1)}_j).
\end{align*}
Applying \eqref{Eq:EstA1} and \eqref{Eq:EstadA1} gives
\begin{subequations}
\begin{align}
	\left\|
	\mathscr P_i
	\left(
	x_n^{-\operatorname{ad}[\delta_{n-1}\Phi_{n-1}]^q}
	\left(
	A_1(x_n,\zeta_n)
-
\left(
	U_{\hat n} +
	\frac{\delta_{n-1}\Phi_{n-1}}{x_n}
\right)
	\right)
	\right)
	\mathscr P_j
	\right\|
	&\leqslant
	C|x_n|^{-1+\epsilon_\alpha}
	|\zeta_n|^{-\epsilon_\alpha}
	|x_n|^{-\sigma_{ij}},
\label{Est:A1}\\
	\left\|
	\mathscr P_i
	\left(
	x_n^{-\operatorname{ad}[\delta_{n-1}\Phi_{n-1}]^q}
	\left(
	A_1(x_n,\zeta_n)
-
\left(
	U_{\hat n} +
	\frac{\delta_{n-1}\Phi_{n-1}}{x_n}
\right)
	\right)
	\right)
	\mathscr P_j
	\right\|
	&\leqslant
	C|x_n|^{-1+\epsilon_\alpha}
	|\zeta_n|^{-\epsilon_\alpha}
	|\zeta_n|^{-\sigma_{ij}}.
\label{Est:A1ad}
\end{align}
\end{subequations}
Since $0<|x_n|<1<|\zeta_n|$, we have
\begin{align*}
\min\left\{
|x_n|^{-\sigma_{ij}},
|\zeta_n|^{-\sigma_{ij}}
\right\}
\leqslant1.
\end{align*}
The logarithmic factors arising from $\mathscr N$ are absorbed
by slightly increasing $\varepsilon_{n-1}$ and
$\varepsilon_\alpha$, as permitted by the strict inequalities
in \eqref{Cond:less1} and \eqref{Cond:less1new}.
Taking the minimum of \eqref{Est:A1} and
\eqref{Est:A1ad},
	and summing over $1\leqslant i,j\leqslant n$, gives
\begin{subequations}
\begin{align}
\left\|
x_n^{-\operatorname{ad}[\delta_{n-1}\Phi_{n-1}]^q}
\left(
A_1(x_n,\zeta_n)
-
\left(
	U_{\hat n} +
	\frac{\delta_{n-1}\Phi_{n-1}}{x_n}
\right)
\right)
\right\|
&\leqslant
C|x_n|^{-1+\epsilon_\alpha}
|\zeta_n|^{-\epsilon_\alpha},
\label{Est:ZeroStabilityConjugatedA1}\\
\left\|
T_{x_n}^{[0]}(qx_n)^{-1}
\left(
A_1(x_n,\zeta_n)
-
\left(
	U_{\hat n} +
	\frac{\delta_{n-1}\Phi_{n-1}}{x_n}
\right)
\right)
T_{x_n}^{[0]}(x_n)
\right\|
&\leqslant
C|x_n|^{-1+\epsilon_\alpha}
|\zeta_n|^{-\epsilon_\alpha}.
\label{Est:ZeroStabilityConjugatedTA1}
\end{align}
\end{subequations}
For each fixed $x_n$ in the domain specified in the statement,
since the kernel of the integral operator in
\eqref{Eq:ZeroStabilityPicardIntegral} satisfies
\eqref{Est:ZeroStabilityConjugatedTA1},
the standard Picard iteration completes the proof.
\end{prf}

\begin{pro}\label{Lem:SecondFactorization}
Identity \eqref{Eq:SecondFactorization}
holds for
\begin{align*}
x_n\notin
\underset{1\leqslant k<n}{\bigcup}
\frac{q^{\alpha_k}}{(1-q)u_k}q^{\mathbb Z},\qquad
\zeta_n\notin
\frac{q^{\alpha_n}}{1-q}q^{\mathbb Z},\qquad
\frac{\zeta_n}{x_n}\notin
\operatorname{Pole}\left(\Phi_n\right)
\cup
\{u_1,\ldots,u_{n-1}\}
q^{\mathbb Z\setminus\{0\}}.
\end{align*}
\end{pro}

\begin{prf}
Denote
$\Omega(x_n) :=
\Omega(x_n;U_{\hat n},\delta_{n-1}\Phi_{n-1})$.
Lemmas~\ref{Lem:SummedZetaInf} and
	\ref{Lem:SecondCommonSolution}
	show that there exists a
	$q$-constant matrix $C(x_n,\zeta_n)$, 
	such that
\begin{align}\label{II:step1}
\mathcal Y_{x_n}^{[0]}(x_n,\zeta_n)
\zeta_n^{-[\delta_{n-1}\Phi_{n-1}]^q}
\mathcal ST_{\zeta_n}^{[\infty]}(\zeta_n;x_n)
\Omega(x_n)^{-1}
&=
Y^{[\infty]}
\left(x_n,\frac{\zeta_n}{x_n}\right)
C(x_n,\zeta_n).
\end{align}
It remains to prove that $C(x_n,\zeta_n)=I$.

Since $\Omega(x_n)$ commutes with
$\begin{pmatrix}
I&0\\
0&e_q(q^{-\alpha_n}\zeta_n)
\end{pmatrix}$.
Substituting the canonical forms into \eqref{II:step1} gives
\begin{align*}
&
\begin{pmatrix}
I&0\\
0&e_q(q^{-\alpha_n}\zeta_n)
\end{pmatrix}
C(x_n,\zeta_n)
\begin{pmatrix}
I&0\\
0&e_q(q^{-\alpha_n}\zeta_n)^{-1}
\end{pmatrix}
\\
&\quad=
T_{x_n}^{[\infty]}(x_n)^{-1}
\left(
\left(
\frac{\zeta_n}{x_n}
\right)^{
-\operatorname{ad}[\delta_{n-1}\Phi_{n-1}]^q
}
\mathcal H_{\zeta_n}^{[\infty]}(x_n,\zeta_n)
\right)^{-1}
\\
&\qquad\quad\times
\left(
\mathcal H_{x_n}^{[0]}(x_n,\zeta_n)
R_{x_n}^{[0]}(x_n)^{-1}
\right)
T_{x_n}^{[0]}(x_n)
\\
&\qquad\quad\times
\left(
\zeta_n^{-[\delta_{n-1}\Phi_{n-1}]^q}
\mathcal SR_{\zeta_n}^{[\infty]}(\zeta_n;x_n)
\zeta_n^{[\delta_{n-1}\Phi_{n-1}]^q}
\right)
\Omega(x_n)^{-1}.
\end{align*}
By Lemma~\ref{Lem:CanZetaInf},
	Lemma~\ref{Lem:ZeroCanonicalStability}, and
	Lemma~\ref{Lem:SummedZetaInf},
	we obtain
\begin{align*}
&
\begin{pmatrix}
I&0\\
0&e_q(q^{-\alpha_n}\zeta_n)
\end{pmatrix}
C(x_n,\zeta_n)
\begin{pmatrix}
I&0\\
0&e_q(q^{-\alpha_n}\zeta_n)^{-1}
\end{pmatrix}
=
I+o(1).
\end{align*}
Since $C(x_n,\zeta_n)$ is a $q$-constant,
the growth of $e_q(q^{-\alpha_n}\zeta_n)$ gives
\begin{align*}
C(x_n,\zeta_n)
&=
\begin{cases}
\begin{pmatrix}
I&0\\
c(x_n,\zeta_n)&1
\end{pmatrix},
& |q|<1,
\\[10pt]
\begin{pmatrix}
I&c(x_n,\zeta_n)\\
0&1
\end{pmatrix},
& |q|>1.
\end{cases}
\end{align*}
It remains to prove that $c(x_n,\zeta_n)=0$.

Assume that $|q|<1$ in what follows.
Write
$c(x_n,\zeta_n)
=
(c_1(x_n,\zeta_n),\ldots,c_{n-1}(x_n,\zeta_n))$.
Comparing the $j$-th columns on both sides of
\eqref{II:step1}, we have
\begin{align}
&
\mathcal H_{x_n}^{[0]}(x_n,\zeta_n)
\left(
	\frac{\zeta_n}{x_n}
\right)^{-[\delta_{n-1}\Phi_{n-1}]^q}
\left(
\mathcal SR_{\zeta_n}^{[\infty]}(\zeta_n;x_n)
\zeta_n^{[\delta_{n-1}\Phi_{n-1}]^q}
\Omega(x_n)^{-1}
\right)_{\ast j}
\notag\\
&=
H^{[\infty]}
\left(x_n,\frac{\zeta_n}{x_n}\right)_{\ast j}
x_n^{\alpha_j}
e_q(q^{-\alpha_j}u_jx_n) +
c_j(x_n,\zeta_n)
H^{[\infty]}
\left(x_n,\frac{\zeta_n}{x_n}\right)_{\ast n}
x_n^{\alpha_n}
e_q(q^{-\alpha_n}\zeta_n).
\label{Eq:SecondComparisonColumn}
\end{align}
Fix
\begin{align*}
x_n\in\widetilde{\mathbb C^\ast}
\setminus
\Biggl(
&\bigcup_{1\leqslant k<n}
\left\{
	\frac{q^{\alpha_k}}{(1-q)u_k},
	\frac{q^{\alpha_n}}{(1-q)u_k}
\right\} q^{\mathbb Z}
\cup
\bigcup_{1\leqslant i,j<n}
\frac{q^{\lambda_i^{(n-1)}}}
{(1-q)u_j}q^{\mathbb Z} \cup
\frac{q^{\alpha_n}}{1-q}
\operatorname{Pole}(\Phi_n)^{-1}q^{-\mathbb N}
\Biggr).
\end{align*}
Since $c_j(x_n,\zeta_n)$ is a $q$-constant with respect to
both $x_n$ and $\zeta_n$, its poles and their orders are invariant
under $\zeta_n\mapsto q\zeta_n$.
We may therefore assume that both $x_n$ and
$\zeta_n/x_n$ are sufficiently large so that
$H^{[\infty]}
\left(x_n,\frac{\zeta_n}{x_n}\right)_{\ast n}
\neq0$.
By the analyticity of $\Phi_n(u_n)$ at $u_n=\infty$
and Lemmas~\ref{Lem:CanXn0} and
\ref{Lem:SummedZetaInf},
the possible poles of the left-hand side of
\eqref{Eq:SecondComparisonColumn}
lie on $x_nu_jq^{\mathbb Z}$ and have order at most one.
On the other hand, the column recurrence
\eqref{Eq:HinfRec} in the proof of
Proposition~\ref{Can:inf} shows that the possible poles of
$H^{[\infty]}
\left(x_n,\frac{\zeta_n}{x_n}\right)_{\ast j}$
also lie on $x_nu_jq^{\mathbb Z}$ and
have order at most one.
Since $e_q(q^{-\alpha_n}\zeta_n)$ has no zeros,
\eqref{Eq:SecondComparisonColumn} now shows that
\begin{align*}
\operatorname{Pole}_{\zeta_n}(c_j)
&\subseteq
x_nu_jq^{\mathbb Z},
\end{align*}
and every pole has order at most one.

By the choice of $x_n$,
equation
\eqref{Eq:SecondComparisonColumn} shows that
$c_j(x_n,\zeta_n)
H^{[\infty]}
\left(x_n,\frac{\zeta_n}{x_n}\right)_{\ast n}
x_n^{\alpha_n}
e_q(q^{-\alpha_n}\zeta_n)$
is holomorphic at
$\zeta_n = \frac{q^{\alpha_n}}{1-q}$.
Since $e_q(q^{-\alpha_n}\zeta_n)$ has a simple pole there and
$H^{[\infty]}
\left(x_n,\frac{\zeta_n}{x_n}\right)$ is invertible,
we obtain
\begin{align}\label{II:cis0}
c_j\left(x_n,\frac{q^{\alpha_n}}{1-q}\right)
&=0.
\end{align}
The monodromy relations of the fundamental solutions in
\eqref{II:step1} give
\begin{align*}
c_j(x_n\mathrm e^{2\pi\mathrm i},
\zeta_n\mathrm e^{2\pi\mathrm i})
&=
\mathrm e^{2\pi\mathrm i(\alpha_j-\alpha_n)}
c_j(x_n,\zeta_n).
\end{align*}
Hence
\begin{align*}
	c_j(x_n,\zeta_n) \cdot
	\frac{
		\theta_q\left(
		-\zeta_n/(x_nu_j)
		\right)
	}{
		\theta_q\left(
		-q^{\alpha_j-\alpha_n}\zeta_n/(x_nu_j)
		\right)
	}
	\zeta_n^{\alpha_n-\alpha_j}
\end{align*}
is a single-valued $q$-constant on
$\mathbb C^\ast/q^{\mathbb Z}$ with at most one simple pole.
It is therefore constant, and \eqref{II:cis0} shows that
$c_j=0$.

The case $|q|>1$ follows in the same way.
The analytic dependence on the initial data given by
Lemma~\ref{Lem:RecStep1}
extends the identity to all admissible initial data.
This completes the proof.
\end{prf}

\begin{pro}\label{Lem:ThirdFactorization}
Suppose in addition that $\Phi_{n-1}$ is
$q$-nonresonant.
Identity \eqref{Eq:ThirdFactorization}
holds for
\begin{align*}
x_n \notin
\begin{cases}
\displaystyle
\underset{1\leqslant k<n}{\bigcup}
\frac{q^{\alpha_k}}{(1-q)u_k}q^{-\mathbb N},
& |q|<1,\\
\displaystyle
\underset{1\leqslant k<n}{\bigcup}
\frac{q^{\alpha_k}}{(1-q)u_k}q^{\mathbb N+1},
& |q|>1,
\end{cases},\qquad
\zeta_n \notin
\begin{cases}
\displaystyle
\frac{q^{\alpha_n}}{1-q}q^{-\mathbb N},
& |q|<1,\\
\displaystyle
\frac{q^{\alpha_n}}{1-q}q^{\mathbb N+1},
& |q|>1,
\end{cases},\qquad
\frac{\zeta_n}{x_n} \notin
\operatorname{Pole}\left(\Phi_n\right).
\end{align*}
\end{pro}

\begin{prf}
By Lemma~\ref{Lem:SecondCommonSolution},
the left-hand side of \eqref{Eq:ThirdFactorization}
is a common fundamental solution of
\eqref{Eq:TargetSysNew}.
The $q$-isomonodromy equations show that
$Y^{[0]}(x,u_n)W_n(u_n)$ is a common fundamental
solution of \eqref{Eq:TargetSystem}.
Under \eqref{Eq:Trans} it is therefore a common fundamental
solution of \eqref{Eq:TargetSysNew}.
Hence there is a $q$-constant matrix $C(x_n,\zeta_n)$
	such that
\begin{align}\label{III:Comparison}
\mathcal Y_{x_n}^{[0]}(x_n,\zeta_n)
\zeta_n^{-[\delta_{n-1}\Phi_{n-1}]^q}
T_{\zeta_n}^{[0]}(\zeta_n)W_{n-1}
&=
Y^{[0]}\left(x_n,\frac{\zeta_n}{x_n}\right)
W_n\left(\frac{\zeta_n}{x_n}\right)
C(x_n,\zeta_n).
\end{align}
It remains to prove that $C(x_n,\zeta_n)=I$.

Write
\begin{align*}
H^{[0]}(x,u_n)
=
I+\sum_{\nu=1}^{\infty}
H_\nu^{[0]}(u_n)x^\nu,\qquad
R_{\zeta_n}^{[0]}(\zeta_n)
=
I+\sum_{\nu=1}^{\infty}
R_{\zeta_n,\nu}^{[0]}\zeta_n^\nu.
\end{align*}
For $\nu\geqslant1$, denote
\begin{align*}
\widehat H_\nu(u_n)
&:=
u_n^{-\nu}
u_n^{\operatorname{ad}[\delta_{n-1}\Phi_{n-1}]^q}
H_\nu^{[0]}(u_n),\\
\widehat\Phi_n(u_n)
&:=
u_n^{\operatorname{ad}[\delta_{n-1}\Phi_{n-1}]^q}
\Phi_n(u_n),\\
\widehat U(u_n)
&:=
u_n^{-1}
u_n^{\operatorname{ad}[\delta_{n-1}\Phi_{n-1}]^q}U.
\end{align*}
The coefficient recurrence in the proof of
Proposition~\ref{Can:zero} gives
\begin{align}
\mathcal L_{\nu,\Psi}(X)
&:=
\frac{1}{q-1}
\left(
q^\nu X\bigl(I+(q-1)\Psi\bigr)
-
\bigl(I+(q-1)\Psi\bigr)X
\right),\\
\widehat H_\nu(u_n)
&=
\mathcal L_{\nu,\widehat\Phi_n(u_n)}^{-1}
\left(
\widehat U(u_n)\widehat H_{\nu-1}(u_n)
\right).
\label{III:ScaledRecurrence}
\end{align}
For both $|q|<1$ and $|q|>1$, one verifies directly that
	a constant $M_0>0$ can be chosen
	independently of $\nu$ and locally uniformly
	with respect to $q$-nonresonant $\Psi$
	such that
\begin{align*}
\left\|
\mathcal L_{\nu,\Psi}^{-1}
\right\|
&\leqslant M_0,
\qquad \nu\geqslant1.
\end{align*}
Equation~\eqref{Lim:dk1PktoPk1} gives
	$\widehat\Phi_n(u_n)=
	\Phi_{n-1} +
	O(u_n^{-\varepsilon_{n-1}})$.
The shrinking condition \eqref{Cond:less1} gives
	$\widehat U(u_n) =
	E_n +
	O(u_n^{-\varepsilon_{n-1}})$.
Since $\Phi_{n-1}$ is $q$-nonresonant,
	an induction based on \eqref{III:ScaledRecurrence}
	shows that a constant $M>0$ can be chosen
	independently of $\nu$ and uniformly
	for all sufficiently large $u_n$ such that
\begin{align*}
\widehat H_\nu(u_n)
&=
R_{\zeta_n,\nu}^{[0]}
+
O\left(
M^\nu u_n^{-\varepsilon_{n-1}}
\right),
\qquad \nu\geqslant1.
\end{align*}
It follows that
\begin{subequations}\label{III:ScaledHLimit}
\begin{align}
u_n^{\operatorname{ad}[\delta_{n-1}\Phi_{n-1}]^q}
H^{[0]}(x,u_n)
&=
R_{\zeta_n}^{[0]}(xu_n)
+
O(u_n^{-\varepsilon_{n-1}}),
\qquad u_n\to\infty, \\
\left(
\frac{\zeta_n}{x_n}
\right)^{\operatorname{ad}[\delta_{n-1}\Phi_{n-1}]^q}
H^{[0]}\left(
x_n,\frac{\zeta_n}{x_n}
\right)
&=
R_{\zeta_n}^{[0]}(\zeta_n)
+
O(x_n^{\varepsilon_{n-1}}),
\qquad x_n\to0,
\end{align}
\end{subequations}
locally uniformly for $\zeta_n$.

In the $(x_n,\zeta_n)$-coordinates,
	fix $\zeta_n$ and let $x_n\to0$.
Then $u_n=\zeta_n/x_n\to\infty$.
Equations \eqref{III:ScaledHLimit},
	\eqref{Lim:dk1PktoPk1} and
	\eqref{Lim:dk1WktoWk1} give
\begin{subequations}
\begin{align}
\nonumber
u_n^{\mathrm{ad}[\delta_{n-1}\Phi_{n-1}]^q}
\left(
	Y^{[0]}(x,u_n)
	u_n^{[\Phi_n(u_n)]^q}
\right) &=
\left(
u_n^{\operatorname{ad}[\delta_{n-1}\Phi_{n-1}]^q}
H^{[0]}(x,u_n)
\right)
\zeta_n^{
u_n^{\operatorname{ad}[\delta_{n-1}\Phi_{n-1}]^q}
[\Phi_n(u_n)]^q} \\
&\longrightarrow
R_{\zeta_n}^{[0]}(\zeta_n)
\zeta_n^{[\Phi_{n-1}]^q}
=
T_{\zeta_n}^{[0]}(\zeta_n),\\
u_n^{[\delta_{n-1}\Phi_{n-1}]^q}
Y^{[0]}\left(
x_n,\frac{\zeta_n}{x_n}
\right)
W_n\left(\frac{\zeta_n}{x_n}\right)
&\longrightarrow
T_{\zeta_n}^{[0]}(\zeta_n)W_{n-1}.
\label{III:RightBoundary}
\end{align}
\end{subequations}
Equation~\eqref{Eq:Yxn0H} in
Lemma~\ref{Lem:CanXn0} also gives
\begin{align}
u_n^{[\delta_{n-1}\Phi_{n-1}]^q}
\mathcal Y_{x_n}^{[0]}(x_n,\zeta_n)
\zeta_n^{-[\delta_{n-1}\Phi_{n-1}]^q}
&=
u_n^{\operatorname{ad}[\delta_{n-1}\Phi_{n-1}]^q}
\mathcal H_{x_n}^{[0]}(x_n,\zeta_n)
\longrightarrow I.
\label{III:LeftBoundary}
\end{align}
Multiplying \eqref{III:Comparison} on the left by
$u_n^{[\delta_{n-1}\Phi_{n-1}]^q}$
and applying \eqref{III:RightBoundary} and
\eqref{III:LeftBoundary}, we obtain
\begin{align*}
\lim_{x_n\to0}C(x_n,\zeta_n)
&=
I.
\end{align*}
Since $C(x_n,\zeta_n)$ is a $q$-constant
with respect to $x_n$, we have
$C(x_n,\zeta_n)=I$
for every sufficiently small admissible $\zeta_n$.

The identity extends to the full common domain by
meromorphic continuation.
The analytic dependence on the initial data given by
Lemma~\ref{Lem:RecStep1}
extends the identity to all admissible initial data.
This completes the proof.
\end{prf}

\subsection{Factorization of the Central Connection Matrix}
\label{Sect:CenFact}

In this subsection, we apply the three factorization formulas
\eqref{Eq:FirstFactorization}--\eqref{Eq:ThirdFactorization}
to obtain a one-step factorization of the central connection matrix.
Iterating this factorization yields an explicit formula in terms of
the initial data $(\Phi_0,W_0)$.

For convenience,
	we work below only in the $(x,u_n)$-coordinates.
Denote the (normalized) central connection matrix of
	the system \eqref{Eq:TargetSysLim-zeta} by
\begin{subequations}
\label{Def:ZetaNormalizedCentralConnection}
\begin{align}
\Omega_{\boldsymbol d}(xu_n;E_n,\Phi_{n-1})
&:=
\mathcal ST_{\zeta_n}^{[\infty]}(\zeta_n;x_n)^{-1} \cdot
T_{\zeta_n}^{[0]}(\zeta_n),\\
C_{\boldsymbol d}(xu_n;E_n,\Phi_{n-1};W_{n-1})
&:=
\Omega_{\boldsymbol d}(xu_n;E_n,\Phi_{n-1})\cdot
W_{n-1},
\end{align}
\end{subequations}
where
$\boldsymbol d=(-x u_1,\ldots,-x u_{n-1},\ast)$
is as in Definition~\ref{Def:ReassembledQBorelSum}.

\begin{cor}\label{Cor:CentralConnectionFirstStep}
The following identity holds on the maximal natural domain
on which both sides are defined:
\begin{subequations}
\label{Eq:CentralConnectionFirstStep}
\begin{align}
\lim_{u_n\to\infty}
\Omega(x;U,\Phi_n)
u_n^{[\Phi_n]^q}
u_n^{-[\delta_{n-1}\Phi_{n-1}]^q}
&=
\Omega(x;U_{\hat n},\delta_{n-1}\Phi_{n-1})
\Omega_{\boldsymbol d}(xu_n;E_n,\Phi_{n-1}),
\\
C(x;U,\Phi_n;W_n)
&=
\Omega(x;U_{\hat n},\delta_{n-1}\Phi_{n-1})
C_{\boldsymbol d}(xu_n;E_n,\Phi_{n-1};W_{n-1}).
\end{align}
\end{subequations}
\end{cor}

\begin{prf}
The result follows directly from
\eqref{Eq:FirstFactorization}--\eqref{Eq:ThirdFactorization}
and \eqref{Lim:dk1WktoWk1}.
\end{prf}

In what follows,
	let $(\Phi(u),W(u);\boldsymbol{\alpha})$ be
	the $q$-isomonodromic deformation constructed in
	Theorem~\ref{Thm:AsymBehavior} from
	the initial data $(\Phi_0,W_0;\boldsymbol{\alpha})$
	and the choices of
	$[\delta_1\Phi_0]^q,\ldots,
	[\delta_n\Phi_0]^q=[\Phi_0]^q$.
Assume further that $\Phi_0^{[k]}$ is $q$-nonresonant
	for every $1\leqslant k\leqslant n$.
Next we choose
	$W_{0;1},\ldots,W_{0;n}\in\mathrm{GL}_n(\mathbb C)$
	recursively as follows.
Denote $\Phi_{0;0}:=\Phi_0$ and
\begin{align}
	\Phi_{0;k-1}
	&:=
	(W_{0;1}\cdots W_{0;k-1})^{-1}\cdot
	\Phi_0\cdot
	(W_{0;1}\cdots W_{0;k-1}),
\end{align}
and choose the branch of
$[\delta_k\Phi_{0;k-1}]^q$ by
\begin{align}
	[\delta_k\Phi_{0;k-1}]^q
	&:=
	(W_{0;1}\cdots W_{0;k-1})^{-1} \cdot
	[\delta_k\Phi_0]^q \cdot
	(W_{0;1}\cdots W_{0;k-1}).
\end{align}
For $1\leqslant k\leqslant n$, choose
	$W_{0;k} =
	\diag(W_{0;k}^{[k]},I_{n-k})$
	such that $W_{0;k}$ conjugates
	$[\delta_k\Phi_{0;k-1}]^q$
	to Jordan normal form, and
\begin{align}
	W_{0;n}^{-1} \cdot
	[\Phi_{0;n-1}]^q \cdot
	W_{0;n} =
	W_0^{-1} \cdot
	[\Phi_0]^q \cdot
	W_0.
\end{align}
Finally, denote
\begin{align}
\mathscr W_0
&:=
(W_{0;1}\cdots W_{0;n})^{-1}W_0.
\end{align}
We obtain the following theorem.

\begin{thm}\label{Thm:qCaterpillar}
For $|q|<1$ we have
\begin{subequations}\label{Eq:qCaterpillarl1}
\begin{align}
\label{Eq:qCaterpillarl1O}
	\lim_{u_2\to\infty}\cdots\lim_{u_n\to\infty}&
	\left(
	\mathop{\prod}
	\limits_{k=2}^{\overset{\longrightarrow}{n}}
	u_k^{[\delta_{k-1}\Phi_{k-1}(u)]^q}
	u_k^{-[\delta_k\Phi_k(u)]^q}
	\right)
	\Omega(x;U,\Phi(u))^{-1}
	\notag\\
	&=
	\sum_{j=1}^{n}
	\left(
	\mathop{\prod}
	\limits_{k=1}^{\overset{\longleftarrow}{n}}
	\Omega_{-xu_j}
	(xu_k;E_k,\delta_k\Phi_0)^{-1}
	\right)E_j,
\\
\label{Eq:qCaterpillarl1C}
	C(x;U,\Phi(u);W(u))^{-1}
	&=
	\mathscr W_0^{-1}
	\sum_{j=1}^{n}
	\left(
	\mathop{\prod}
	\limits_{k=1}^{\overset{\longleftarrow}{n}}
	C_{-xu_j}
	(xu_k;E_k,\delta_k\Phi_{0;k-1};W_{0;k})^{-1}
	\right)
	E_j.
\end{align}
\end{subequations}
For $|q|>1$ we have
\begin{subequations}\label{Eq:qCaterpillarg1}
\begin{align}
\label{Eq:qCaterpillarg1O}
	\lim_{u_2\to\infty}\cdots\lim_{u_n\to\infty}&
	\Omega(x;U,\Phi(u))
	\left(
	\mathop{\prod}
	\limits_{k=2}^{\overset{\longleftarrow}{n}}
	u_k^{[\delta_k\Phi_k(u)]^q}
	u_k^{-[\delta_{k-1}\Phi_{k-1}(u)]^q}
	\right)
	\notag\\
	&=
	\sum_{j=1}^{n}
	E_j
	\left(
	\mathop{\prod}
	\limits_{k=1}^{\overset{\longrightarrow}{n}}
	\Omega_{-xu_j}(xu_k;E_k,\delta_k\Phi_0)
	\right),
\\
\label{Eq:qCaterpillarg1C}
	C(x;U,\Phi(u);W(u))
	&=
	\sum_{j=1}^{n}
	E_j
	\left(
	\mathop{\prod}
	\limits_{k=1}^{\overset{\longrightarrow}{n}}
	C_{-xu_j}
	(xu_k;E_k,\delta_k\Phi_{0;k-1};W_{0;k})
	\right)
	\mathscr W_0.
\end{align}
\end{subequations}
The limits in \eqref{Eq:qCaterpillarl1O} and
\eqref{Eq:qCaterpillarg1O} are taken along $q$-spirals.
\end{thm}

\begin{prf}
Without loss of generality, it suffices to prove
\eqref{Eq:qCaterpillarl1O}.
For every admissible direction $d$, denote by
$T_{\zeta_n,d}^{[\infty]}(\zeta_n)$
the canonical fundamental solution obtained from
$T_{\zeta_n}^{[\infty]}(\zeta_n)$ by
$q$-Borel summation in the direction $d$.
Definition~\ref{Def:ReassembledQBorelSum} gives
\begin{align}
&
\mathcal ST_{\zeta_n}^{[\infty]}(\zeta_n;x)
\Omega(x;U_{\hat n},\delta_{n-1}\Phi_{n-1})^{-1}
=
\sum_{j=1}^{n}
T_{\zeta_n,-xu_j}^{[\infty]}(\zeta_n) \cdot
\Omega(x;U_{\hat n},\delta_{n-1}\Phi_{n-1})^{-1}
E_j.
\label{eq:omega-En-columnwise-sum}
\end{align}
Since the last column of
$T_{\zeta_n}^{[\infty]}(\zeta_n)$ is convergent,
the choice of its corresponding direction has no effect.
For convenience, we take it to be $-xu_n$.

Multiplying \eqref{eq:omega-En-columnwise-sum} on the left by
$T_{\zeta_n}^{[0]}(\zeta_n)^{-1}$
and applying
Corollary~\ref{Cor:CentralConnectionFirstStep} gives
\begin{align}
&
\lim_{u_n\to\infty}
u_n^{[\delta_{n-1}\Phi_{n-1}]^q}
u_n^{-[\Phi_n]^q}
\Omega(x;U,\Phi_n)^{-1} =
\sum_{j=1}^{n}
\Omega_{-xu_j}
(xu_n;E_n,\Phi_{n-1})^{-1}
\Omega(x;U_{\hat n},\delta_{n-1}\Phi_{n-1})^{-1}
E_j.
\label{eq:omega-En-recursive-closed-form}
\end{align}
For $1\leqslant i\leqslant n$, denote
$U_{\{1,\ldots,i\}}:=\sum_{k=1}^{i}u_kE_k$.
We prove by descending induction on $i$ that
\begin{align}
&
\lim_{u_{i+1}\to\infty}\cdots\lim_{u_n\to\infty}
\left(
\mathop{\prod}
\limits_{k=i+1}^{\overset{\longrightarrow}{n}}
u_k^{[\delta_{k-1}\Phi_{k-1}(u)]^q}
u_k^{-[\delta_k\Phi_k(u)]^q}
\right)
\Omega(x;U,\Phi_n(u))^{-1}
\notag\\
&\quad=
\sum_{j=1}^{n}
\left(
\mathop{\prod}
\limits_{k=i+1}^{\overset{\longleftarrow}{n}}
\Omega_{-xu_j}
(xu_k;E_k,\delta_k\Phi_i(u))^{-1}
\right)
\Omega(x;U_{\{1,\ldots,i\}},\delta_i\Phi_i(u))^{-1}E_j.
\label{eq:qcat-complete-inverse-expanded}
\end{align}
Equation~\eqref{eq:omega-En-recursive-closed-form}
gives the induction base $i=n-1$.

Suppose that \eqref{eq:qcat-complete-inverse-expanded}
	holds for some $1\leqslant i\leqslant n-1$.
For every $i<k\leqslant n$ and
$1\leqslant j\leqslant k$, the asymptotics
\eqref{Lem:PWkm1} in Corollary~\ref{Cor:Reck}
and the conjugation covariance of the canonical fundamental
solutions give
\begin{align}
&
\lim_{u_i\to\infty}
u_i^{\operatorname{ad}[\delta_{i-1}\Phi_{i-1}(u)]^q}
u_i^{-\operatorname{ad}[\delta_i\Phi_i(u)]^q}
\Omega_{-xu_j}
(xu_k;E_k,\delta_k\Phi_i(u))^{-1} =
\Omega_{-xu_j}
(xu_k;E_k,\delta_k\Phi_{i-1}(u))^{-1}.
\label{eq:qcat-later-factor-limit}
\end{align}
Equation~\eqref{eq:omega-En-recursive-closed-form},
with $n$ replaced by $i$, also gives
\begin{align}
&
\lim_{u_i\to\infty}
u_i^{[\delta_{i-1}\Phi_{i-1}(u)]^q}
u_i^{-[\delta_i\Phi_i(u)]^q}
\Omega
\left(
x;U_{\{1,\ldots,i\}},\delta_i\Phi_i(u)
\right)^{-1}
\notag\\
&\quad=
\sum_{j=1}^{n}
\Omega_{-xu_j}
(xu_i;E_i,\delta_i\Phi_{i-1}(u))^{-1}
\Omega
\left(
x;U_{\{1,\ldots,i-1\}},
\delta_{i-1}\Phi_{i-1}(u)
\right)^{-1}
E_j.
\label{eq:qcat-lower-step}
\end{align}
Combining these identities with
\eqref{eq:qcat-complete-inverse-expanded}
gives the same formula with $i$ replaced by $i-1$.
This completes the induction.
Taking $i=0$ gives \eqref{Eq:qCaterpillarl1O}.
Applying the $q\leftrightarrow q^{-1}$ duality gives
\eqref{Eq:qCaterpillarg1O}.
\end{prf}

We have thus obtained the explicit formula for
the central connection matrix in terms of the initial data.
Assume further that $\Phi_0^{[k]}$ is diagonalizable and
	$W_{0;k}$ diagonalizes $[\delta_k\Phi_{0;k-1}]^q$
	for every $1\leqslant k\leqslant n$.
The formulas in \cite{lin_ma_xu_qstokes_2024} then
	give explicit expressions for
	$C_{-xu_j}(xu_k;E_k,\delta_k\Phi_{0;k-1};W_{0;k})$.

\section{Applications}
\label{Sect:Applications}

In this section,
	we present two applications of
	Theorems~\ref{Thm:AsymBehavior}
	and~\ref{Thm:qCaterpillar}.
In Section~\ref{Sect:InvGen},
	we construct the inverse monodromy map and prove that
	the shrinking solutions constructed in
	Theorem~\ref{Thm:AsymBehavior}
	form a generic family .
In Section~\ref{Sect:Conflu},
	we show that the mathematical objects constructed in this paper
	are natural in a certain sense.
In the differential limit as $q\to1$,
	they converge to the corresponding objects
	in the differential setting.

\subsection{Inverse Monodromy and Genericity}
\label{Sect:InvGen}

In this subsection, we prove that
	the shrinking solutions constructed in
	Theorem~\ref{Thm:AsymBehavior} from
	the initial data $(\Phi_0,W_0;\boldsymbol{\alpha})$
	form a generic family of
	$q$-isomonodromic deformations.
For convenience,
	we assume throughout this subsection that $|q|<1$.
The $|q|>1$ case follows from the
	$q\leftrightarrow q^{-1}$ duality.

To construct the inverse monodromy map from
	the $q$-monodromy data to
	the initial data $(\Phi_0,W_0)$,
	we introduce new coordinates for $\Phi_0$ that
	are more convenient than
	the standard entry coordinates.
For each $1\leqslant k\leqslant n$,
	fix a branch of $[\Phi_0^{[k]}]^q$ and an ordering
\[
\mathrm{Spec}\bigl([\Phi_0^{[k]}]^q\bigr)
=
\{\lambda_1^{(k)},\ldots,\lambda_k^{(k)}\}.
\]
Introduce the
\textbf{Gelfand--Zeitlin eigenvalue--minor coordinates}
$(\lambda_i^{(k)},\eta_j^{(m)})_{
\substack{
	1\leqslant i \leqslant k \leqslant n\\
	1\leqslant j\leqslant m \leqslant n-1}}$,
where
\begin{align*}
\eta_j^{(m)}
:=
\frac{
\det(
\Phi_0-[\lambda_j^{(m)}]_qI
)^{1,\ldots,m}_{1,\ldots,m-1,m+1}
}{
\displaystyle
\prod_{\substack{1\leqslant s\leqslant m\\s\neq j}}
\bigl(
[\lambda_s^{(m)}]_q-[\lambda_j^{(m)}]_q
\bigr)
},\qquad
1\leqslant j\leqslant m \leqslant n-1.
\end{align*}
The following genericity conditions on $\Phi_0$ will be used below:
\begin{subequations}\label{GenP0}
\begin{itemize}
\item For every $1\leqslant k\leqslant n$,
	the upper-left $k\times k$ submatrix
	$\Phi_0^{[k]}$ is $q$-nonresonant,
	and its spectrum is simple:
\begin{align}\label{GenCond:1}
	q^{\lambda_i^{(k)}}/
	q^{\lambda_j^{(k)}}
	\notin q^{\mathbb Z},
	\qquad
	1\leqslant i,j\leqslant k,\quad
	i\neq j.
\end{align}

\item For every $2\leqslant k\leqslant n$,
	the spectra at levels $k-1$ and $k$ are $q$-disjoint:
\begin{align}\label{GenCond:2}
	q^{\lambda_i^{(k)}}/
	q^{\lambda_j^{(k-1)}}
	\notin q^{\mathbb Z},
	\qquad
	1\leqslant i\leqslant k,\quad
	1\leqslant j\leqslant k-1.
\end{align}

\item All the minor coordinates are nonzero:
\begin{align}\label{GenCond:3}
	\eta_j^{(m)}\neq0,
	\qquad
	1\leqslant j\leqslant m<n.
\end{align}

\end{itemize}
\end{subequations}
If the genericity conditions
	\eqref{GenP0} hold,
	then the matrices $W_{0;k}$ required in
	Theorem~\ref{Thm:qCaterpillar}
	may be chosen as follows.
For $1\leqslant k\leqslant n$,
denote
\begin{subequations}\label{Eq:W0kEigenminor}
\begin{align}
W_{0;k}
&=
\operatorname{diag}\bigl(
W_{0;k}^{[k]},I_{n-k}
\bigr),
\\
(W_{0;k})_{ij}
&=
\begin{cases}
\displaystyle
\frac{
\eta_i^{(k-1)}
}{
[\lambda_j^{(k)}]_q-
[\lambda_i^{(k-1)}]_q
},
&
1\leqslant i \leqslant k-1, \
1\leqslant j\leqslant k
\\
1, &
i=k, \
1\leqslant j\leqslant k
\end{cases}.
\end{align}
\end{subequations}
In this case, $\mathscr W_0$ is diagonal,
and we have
\begin{subequations}\label{Eq:EMtoEntry}
\begin{align}
\Phi_0
&=
W_{0;1}\cdots W_{0;n} \cdot
\operatorname{diag}(
[\lambda_1^{(n)}]_q,\ldots,
[\lambda_n^{(n)}]_q) \cdot
W_{0;n}^{-1}\cdots W_{0;1}^{-1},\\
W_0
&=
W_{0;1}\cdots W_{0;n} \cdot \mathscr W_0.
\end{align}
\end{subequations}
The formulas \eqref{Eq:W0kEigenminor} and
	\eqref{Eq:EMtoEntry}
	give explicit formulas for
	the entry coordinates of $(\Phi_0,W_0)$ in terms of
	the eigenvalue--minor coordinates of $\Phi_0$
	and $\mathscr W_0$.

We also need to determine the structure of
	the space of $q$-monodromy data.
The space of $(C(x),\Lambda;\boldsymbol{\alpha})$
	itself has dimension $n^2+n$,
	of which $n$ dimensions arise from
	the choice of $W$.
However, the parameters themselves have dimension $n^2+2n$,
	so $n$ further constraints are required.
These are precisely the following conditions
\eqref{Eq:CThetaDeterminantConditions} and
\eqref{Eq:CThetaBalance}.

\begin{lem}\label{Lem:CThetaStructure}
Let $(U,\Phi;W;\boldsymbol{\alpha})$ be a
$q$-monodromy system.
Suppose that its $q$-monodromy data
$(C(x),\Lambda;\boldsymbol{\alpha})$ satisfy
\[
\Lambda
=
\operatorname{diag}(\lambda_1,\ldots,\lambda_n).
\]
Then the following relations hold.
\begin{subequations}
\begin{enumerate}
\item There exists a unique constant matrix
$K=K(U,\Phi;\boldsymbol{\alpha},W)$ such that
\begin{align}
\bigl(C(x)^{-1}\bigr)_{ij}
&=
K_{ij} \cdot
\frac{
\theta_q\bigl(
-q^{-\lambda_i}(1-q)xu_j
\bigr)
}{
\theta_q\bigl(
-q^{-\alpha_j}(1-q)xu_j
\bigr)
}
\bigl((1-q)xu_j\bigr)^{
\alpha_j-\lambda_i}.
\label{Eq:CThetaStructure}
\end{align}

\item For every $1\leqslant s\leqslant n$, we have
\begin{align}
\det\left(
K_{ij} \cdot
\theta_q\left(
- q^{\alpha_s-\lambda_i}
\frac{u_j}{u_s}
\right)
u_j^{-\lambda_i}
\right)_{1\leqslant i,j\leqslant n}
&=0.
\label{Eq:CThetaDeterminantConditions}
\end{align}

\item The formal monodromy satisfies
\begin{align}
\prod_{j=1}^n q^{\alpha_j}
&=
\prod_{i=1}^n q^{\lambda_i}.
\label{Eq:CThetaBalance}
\end{align}
\end{enumerate}
\end{subequations}
Moreover, at least one of the $n$ identities in
\eqref{Eq:CThetaDeterminantConditions} is redundant.
\end{lem}

\begin{prf}
Since $|q|<1$,
Proposition~\ref{Pro:SolY} shows that
$H^{[0]}(x)^{-1}$ is holomorphic on $\mathbb C^\ast$.
By Lemma~\ref{Lem:BG} and Proposition~\ref{Pro:SolY},
the possible poles of the $j$-th column of
$H^{[\infty]}(x)$ lie on
$\frac{q^{\alpha_j}}{(1-q)u_j}q^{\mathbb N+1}$,
and every pole has order at most one. Moreover,
$e_q(q^{-\alpha_j}u_jx)$ has simple poles on
$\frac{q^{\alpha_j}}{(1-q)u_j}q^{-\mathbb N}$.
Consequently, the possible poles of
$\bigl(C(x)^{-1}\bigr)_{ij}$ lie on
$\frac{q^{\alpha_j}}{(1-q)u_j}q^{\mathbb Z}$,
and every pole has order at most one.
The single-valuedness of
$H^{[0]}(x)^{-1}$ and $H^{[\infty]}(x)$ gives
\begin{align*}
\bigl(C(x\mathrm e^{2\pi\mathrm i})^{-1}\bigr)_{ij}
&=
\mathrm e^{2\pi\mathrm i(\alpha_j-\lambda_i)}
\bigl(C(x)^{-1}\bigr)_{ij}.
\end{align*}
This monodromy relation shows that
\begin{align*}
\bigl(C(x)^{-1}\bigr)_{ij}
\frac{
	\theta_q\bigl(
	-q^{-\alpha_j}(1-q)xu_j
	\bigr)
}{
	\theta_q\bigl(
	-q^{-\lambda_i}(1-q)xu_j
	\bigr)
}
\bigl((1-q)xu_j\bigr)^{
\lambda_i-\alpha_j}
\end{align*}
is a single-valued $q$-constant,
and has at most one simple pole
$\frac{q^{\lambda_i}}{(1-q)u_j} q^{\mathbb Z}$
on $\mathbb C^\ast/q^{\mathbb Z}$.
It is therefore constant.
Denoting this constant by $K_{ij}$ proves
\eqref{Eq:CThetaStructure}.

Applying \eqref{Det:H} in Proposition~\ref{Pro:SolY}
and \eqref{Eq:CThetaStructure}, we obtain
\begin{align*}
\det C(x)^{-1}
&=
\frac{1}{\det W}
x^{\sum_{j=1}^n\alpha_j-\sum_{i=1}^n\lambda_i}
\\
&=
x^{\sum_{j=1}^n\alpha_j-\sum_{i=1}^n\lambda_i}
\prod_{j=1}^n
\frac{
\bigl((1-q)u_j\bigr)^{\alpha_j}
}{
\theta_q\bigl(
-q^{-\alpha_j}(1-q)xu_j
\bigr)
}
\\
&\quad\times
\det\left(
K_{ij} \cdot
\theta_q\bigl(
-q^{-\lambda_i}(1-q)xu_j
\bigr)
\bigl((1-q)u_j\bigr)^{-\lambda_i}
\right)_{1\leqslant i,j\leqslant n},
\end{align*}
and we have
\begin{align*}
&
\det\left(
K_{ij} \cdot
\theta_q\bigl(
-q^{-\lambda_i}(1-q)xu_j
\bigr) 
\bigl((1-q)u_j\bigr)^{-\lambda_i}
\right)_{1\leqslant i,j\leqslant n} =
\frac{1}{\det W}
\prod_{j=1}^n
\frac{
	\theta_q\bigl(
	-q^{-\alpha_j}(1-q)xu_j
	\bigr)
}{
	\bigl((1-q)u_j\bigr)^{\alpha_j}
}.
\end{align*}
Taking
$x=\frac{q^{\alpha_s}}{(1-q)u_s}$
in this identity gives
\eqref{Eq:CThetaDeterminantConditions}.

Finally,
\[
\prod_{i=1}^nq^{\lambda_i}
=
\det\bigl(I+(q-1)\Phi\bigr)
=
\det U \cdot
\prod_{j=1}^n
u_j^{-1}q^{\alpha_j}
=
\prod_{j=1}^nq^{\alpha_j}.
\]
This proves \eqref{Eq:CThetaBalance}.
Equation~\eqref{Eq:CThetaBalance} also shows that
for every constant matrix
$\widetilde K=(\widetilde K_{ij})$,
the function
\begin{align*}
F(x)
&:=
\det\left(
\widetilde K_{ij} \cdot
\theta_q\bigl(
-q^{-\lambda_i}(1-q)xu_j
\bigr)
\bigl((1-q)u_j\bigr)^{-\lambda_i}
\right)_{1\leqslant i,j\leqslant n}
\prod_{j=1}^n
\theta_q\bigl(
-q^{-\alpha_j}(1-q)xu_j
\bigr)^{-1}.
\end{align*}
is a single-valued $q$-constant.
If two of the points
	$u_1^{-1}q^{\alpha_1},\ldots,u_n^{-1}q^{\alpha_n}$
	lie on the same $q$-spiral,
	two of the $n$ identities in 
	\eqref{Eq:CThetaDeterminantConditions} are equivalent.
Otherwise, any $n-1$ of these identities imply that
	$F(x)$ has at most one simple pole
	on $\mathbb C^\ast/q^{\mathbb Z}$,
	and hence is constant.
This gives the same conclusion.
\end{prf}  

\noindent
The conditions
	\eqref{Eq:CThetaDeterminantConditions} on
	the constant matrix $K$
	are invariant under $q$-shifts of
	$q^{\alpha_i}$, $q^{\lambda_i}$, and $u_i$.

The formulas in \cite{lin_ma_xu_qstokes_2024}
give the following result.

\begin{pro}\label{Pro:OneStepConnectionEntries}
Suppose that \eqref{GenP0} hold,
	and let $W_{0;k}$ be chosen as in
	\eqref{Eq:W0kEigenminor}.
Write
\begin{align*}
C_{-xu_s}
(xu_k;E_k,\delta_k\Phi_{0;k-1};W_{0;k})
&=
\operatorname{diag}(C,I_{n-k}),\qquad
1\leqslant k,s\leqslant n.
\end{align*}
\begin{subequations}
For $1\leqslant i\leqslant k$ and
$1\leqslant j \leqslant k-1$, we have
\begin{align}
\left(C^{-1}\right)_{ij}
&=
\frac{
[\lambda_i^{(k)}]_q-
[\lambda_j^{(k-1)}]_q
}{
\eta_j^{(k-1)}
}
\prod_{\substack{1\leqslant\ell\leqslant k\\
\ell\neq i}}
\frac{
\Gamma_q\bigl(
\lambda_\ell^{(k)}-\lambda_i^{(k)}
\bigr)
}{
\Gamma_q\bigl(
\lambda_\ell^{(k)}-\lambda_j^{(k-1)}
\bigr)
}
\prod_{\substack{1\leqslant\ell \leqslant k-1\\
\ell\neq j}}
\frac{
\Gamma_q\bigl(
\lambda_\ell^{(k-1)}-\lambda_j^{(k-1)}
\bigr)
}{
\Gamma_q\bigl(
\lambda_\ell^{(k-1)}-\lambda_i^{(k)}
\bigr)
}
\notag\\
&\quad\times
\frac{
\theta_q\bigl(
-q^{\lambda_j^{(k-1)}-\lambda_i^{(k)}}u_k/u_s
\bigr)
}{
\theta_q(-u_k/u_s)
}
\frac{
\theta_q\bigl(
-q^{-\lambda_i^{(k)}}(1-q)xu_s
\bigr)
}{
\theta_q\bigl(
-q^{-\lambda_j^{(k-1)}}(1-q)xu_s
\bigr)
}
\bigl((1-q)xu_k\bigr)^{
\lambda_j^{(k-1)}-\lambda_i^{(k)}}.
\label{Eq:OneStepConnectionFirstColumns}
\end{align}
For $1\leqslant i\leqslant k$, we have
\begin{align}
\left(C^{-1}\right)_{ik}
&=
\frac{
\displaystyle
\prod_{\substack{1\leqslant\ell\leqslant k\\
\ell\neq i}}
\Gamma_q\bigl(
\lambda_\ell^{(k)}-\lambda_i^{(k)}
\bigr)
}{
\displaystyle
\prod_{1\leqslant\ell\leqslant k-1}
\Gamma_q\bigl(
\lambda_\ell^{(k-1)}-\lambda_i^{(k)}
\bigr)
}
\frac{
\theta_q\bigl(
-q^{-\lambda_i^{(k)}}(1-q)xu_k
\bigr)
}{
\theta_q\bigl(
-q^{-\alpha_k}(1-q)xu_k
\bigr)
}
\bigl((1-q)xu_k\bigr)^{
\alpha_k-\lambda_i^{(k)}}.
\label{Eq:OneStepConnectionLastColumn}
\end{align}
\end{subequations}
\end{pro}

We now reconstruct the initial data from the
	$q$-monodromy data.
Fix $\lambda_\ast^{(n)}$, $\boldsymbol{\alpha}$, and
	$u_1,\ldots,u_n$.
We assume throughout that
	\eqref{GenP0} hold.
Denote
\begin{subequations}
\begin{align}
\eta^{(n)}
&:=
\mathscr W_0
=
\diag(
\eta^{(n)}_1,\ldots,\eta^{(n)}_n),
\\
	\eta^{(s-1)} &:=
	\diag(\eta^{(s-1)}_1,\ldots,\eta^{(s-1)}_{s-1},
	1,\ldots,1),\\
	\mathscr K_d^{(s)}(
	\lambda_\ast^{(s)},\lambda_\ast^{(s-1)}) &:=
	\diag(
	\mathscr K_d^{(s)}(
	\lambda_\ast^{(s)},\lambda_\ast^{(s-1)})^{[s]},
	1,\ldots,1).
\end{align}
\end{subequations}
The entries of $\mathscr K_d^{(s)}$ are given by
\begin{subequations}
\begin{align}
(\mathscr K_d^{(s)})_{ij}
&:=
\left(
[\lambda_i^{(s)}]_q-
[\lambda_j^{(s-1)}]_q
\right)
\prod_{\substack{1\leqslant\ell\leqslant s\\
\ell\neq i}}
\frac{
\Gamma_q\bigl(
\lambda_\ell^{(s)}-\lambda_i^{(s)}
\bigr)
}{
\Gamma_q\bigl(
\lambda_\ell^{(s)}-\lambda_j^{(s-1)}
\bigr)
}
\prod_{\substack{1\leqslant\ell\leqslant s-1\\
\ell\neq j}}
\frac{
\Gamma_q\bigl(
\lambda_\ell^{(s-1)}-\lambda_j^{(s-1)}
\bigr)
}{
\Gamma_q\bigl(
\lambda_\ell^{(s-1)}-\lambda_i^{(s)}
\bigr)
}
\notag\\
&\quad\times
\frac{
\theta_q\bigl(
-q^{\lambda_j^{(s-1)}-\lambda_i^{(s)}}d
\bigr)
}{
\theta_q(-d)
}
d^{\lambda_j^{(s-1)}-\lambda_i^{(s)}},
\qquad
1\leqslant i\leqslant s, \quad
1\leqslant j \leqslant s-1,
\\
(\mathscr K_d^{(s)})_{is}
&:=
\frac{
\displaystyle
\prod_{\substack{1\leqslant\ell\leqslant s\\
\ell\neq i}}
\Gamma_q\bigl(
\lambda_\ell^{(s)}-\lambda_i^{(s)}
\bigr)
}{
\displaystyle
\prod_{1\leqslant\ell\leqslant s-1}
\Gamma_q\bigl(
\lambda_\ell^{(s-1)}-\lambda_i^{(s)}
\bigr)
},
\qquad
1\leqslant i\leqslant s.
\end{align}
\end{subequations}

\begin{cor}\label{Cor:ConnectionFromEigenminorCoordinates}
Suppose that $\Phi_0$ satisfies
\eqref{GenP0}.
Choose $W_{0;s}$ as in
\eqref{Eq:W0kEigenminor}
and fix the choices of $\lambda_i^{(s)}$.
Let $K$ be the constant matrix in
\eqref{Eq:CThetaStructure}
for the $q$-monodromy system constructed from these
initial data in Theorem~\ref{Thm:AsymBehavior}.
Then $K$ is given by
\begin{align}
	KE_j
	&=
	\left(
	\mathop{\prod}
	\limits_{s=j}^{\overset{\longleftarrow}{n}}
	(\eta^{(s)})^{-1} \cdot
	\mathscr K_{u_s/u_j}^{(s)}
	\bigl(
	\lambda_\ast^{(s)},\lambda_\ast^{(s-1)}
	\bigr)
	\right)E_j,
	\qquad
	1\leqslant j\leqslant n.
\label{Eq:KColumnFromEigenminorCoordinates}
\end{align}
\end{cor}

\begin{prf}
Substitute the formulas in
Proposition~\ref{Pro:OneStepConnectionEntries}
into Theorem~\ref{Thm:qCaterpillar}.
\end{prf}

\begin{lem}\label{Lem:InverseMonodromy}
Let $(K,\Lambda;\boldsymbol{\alpha})$ be a generic triple
satisfying \eqref{Eq:CThetaDeterminantConditions} and
\eqref{Eq:CThetaBalance}.
Then there exists a unique shrinking solution
with $q$-monodromy data
$(C(x),\Lambda;\boldsymbol{\alpha})$
satisfying \eqref{Eq:CThetaStructure}.
\end{lem}

\begin{prf}
We first outline the construction.
Write $K^{(n)}:=K$ and
$\lambda_i^{(n)}:=\lambda_i$.
Once $K^{(s)}$ has been constructed,
we determine $\lambda_\ast^{(s-1)}$ by solving a system
of equations.
The $s$-th column of $K^{(s)}$ then determines
$\eta^{(s)}$.
We use it to construct $K^{(s-1)}$.
Repeating this step recursively recovers the
eigenvalue--minor coordinates of $\Phi_0$.

We first determine $\lambda_\ast^{(n-1)}$.
For $1\leqslant i\leqslant n$, denote
\begin{align*}
g_i^{(n)}(\lambda_\ast^{(n)},\lambda_\ast^{(n-1)})
&:=
\prod_{\substack{1\leqslant\ell\leqslant n\\
\ell\neq i}}
	\Gamma_q\bigl(
	\lambda_\ell^{(n)}-\lambda_i^{(n)}
	\bigr)
\prod_{1\leqslant\ell \leqslant n-1}
	\Gamma_q\bigl(
	\lambda_\ell^{(n-1)}-\lambda_i^{(n)}
	\bigr)^{-1}.
\end{align*}
The formula
	\eqref{Eq:KColumnFromEigenminorCoordinates} in 
	Corollary~\ref{Cor:ConnectionFromEigenminorCoordinates}
	naturally suggests the following system of equations.
\begin{subequations}
\begin{align}
\label{Eq:EtaRecovery}
\eta_i^{(n)}
&=
(K_{in})^{-1} \cdot
g_i^{(n)}(\lambda_\ast^{(n)},\lambda_\ast^{(n-1)}),\\
\left(
\left(
\mathscr K_{u_n/u_j}^{(n)}
\bigl(
\lambda_\ast^{(n)},\lambda_\ast^{(n-1)}
\bigr)
\right)^{-1}
\begin{pmatrix}
\displaystyle
\frac{K_{1j}}{K_{1n}} \cdot
g_1^{(n)}(\lambda_\ast^{(n)},\lambda_\ast^{(n-1)})
\\
\vdots
\\
\displaystyle
\frac{K_{nj}}{K_{nn}} \cdot
g_n^{(n)}(\lambda_\ast^{(n)},\lambda_\ast^{(n-1)})
\end{pmatrix}
\right)_n
&= 0,\qquad
1\leqslant j\leqslant n-1.
\label{Eq:FirstSpectralRecovery}
\end{align}
\end{subequations}
The $q$-Gamma reflection formula shows that
\eqref{Eq:FirstSpectralRecovery} is equivalent to
\begin{subequations}
\begin{align}
\label{Eq:FirstSpectralRecovery2}
\det\begin{pmatrix}
	\displaystyle
	\frac{
	\theta_q\left(
	-\frac{q^{\lambda_1^{(n-1)}}u_n}
	{q^{\lambda_1^{(n)}}u_j}
	\right)
	}{
	\theta_q\left(
	-\frac{q^{\lambda_1^{(n-1)}}}
	{q^{\lambda_1^{(n)}}}
	\right)
	}
	&
	\cdots
	&
	\displaystyle
	\frac{
	\theta_q\left(
	-\frac{q^{\lambda_{n-1}^{(n-1)}}u_n}
	{q^{\lambda_1^{(n)}}u_j}
	\right)
	}{
	\theta_q\left(
	-\frac{q^{\lambda_{n-1}^{(n-1)}}}
	{q^{\lambda_1^{(n)}}}
	\right)
	}
	&
	\displaystyle
	\left(\frac{u_n}{u_j}\right)^{\lambda_1^{(n)}}
	\frac{K_{1j}}{K_{1n}}
	\\
	\vdots
	&
	\ddots
	&
	\vdots
	&
	\vdots
	\\
	\displaystyle
	\frac{
	\theta_q\left(
	-\frac{q^{\lambda_1^{(n-1)}}u_n}
	{q^{\lambda_n^{(n)}}u_j}
	\right)
	}{
	\theta_q\left(
	-\frac{q^{\lambda_1^{(n-1)}}}
	{q^{\lambda_n^{(n)}}}
	\right)
	}
	&
	\cdots
	&
	\displaystyle
	\frac{
	\theta_q\left(
	-\frac{q^{\lambda_{n-1}^{(n-1)}}u_n}
	{q^{\lambda_n^{(n)}}u_j}
	\right)
	}{
	\theta_q\left(
	-\frac{q^{\lambda_{n-1}^{(n-1)}}}
	{q^{\lambda_n^{(n)}}}
	\right)
	}
	&
	\displaystyle
	\left(\frac{u_n}{u_j}\right)^{\lambda_n^{(n)}}
	\frac{K_{nj}}{K_{nn}}
\end{pmatrix}
&=0,\\
\label{Eq:FirstSpectralRecoveryGen}
\det\begin{pmatrix}
	\displaystyle
	\frac{
	\theta_q\left(
	-\frac{q^{\lambda_1^{(n-1)}}u_n}
	{q^{\lambda_1^{(n)}}u_j}
	\right)
	}{
	\theta_q\left(
	-\frac{q^{\lambda_1^{(n-1)}}}
	{q^{\lambda_1^{(n)}}}
	\right)
	}
	&
	\cdots
	&
	\displaystyle
	\frac{
	\theta_q\left(
	-\frac{q^{\lambda_{n-1}^{(n-1)}}u_n}
	{q^{\lambda_1^{(n)}}u_j}
	\right)
	}{
	\theta_q\left(
	-\frac{q^{\lambda_{n-1}^{(n-1)}}}
	{q^{\lambda_1^{(n)}}}
	\right)
	}
	&
	\displaystyle
	\left(\frac{u_n}{u_j}\right)^{\lambda_1^{(n)}}
	\\
	\vdots
	&
	\ddots
	&
	\vdots
	&
	\vdots
	\\
	\displaystyle
	\frac{
	\theta_q\left(
	-\frac{q^{\lambda_1^{(n-1)}}u_n}
	{q^{\lambda_n^{(n)}}u_j}
	\right)
	}{
	\theta_q\left(
	-\frac{q^{\lambda_1^{(n-1)}}}
	{q^{\lambda_n^{(n)}}}
	\right)
	}
	&
	\cdots
	&
	\displaystyle
	\frac{
	\theta_q\left(
	-\frac{q^{\lambda_{n-1}^{(n-1)}}u_n}
	{q^{\lambda_n^{(n)}}u_j}
	\right)
	}{
	\theta_q\left(
	-\frac{q^{\lambda_{n-1}^{(n-1)}}}
	{q^{\lambda_n^{(n)}}}
	\right)
	}
	&
	\displaystyle
	\left(\frac{u_n}{u_j}\right)^{\lambda_n^{(n)}}
\end{pmatrix}
&\neq 0.
\end{align}
\end{subequations}
To solve \eqref{Eq:FirstSpectralRecovery2}, denote
\begin{align*}
H(z):=
\prod_{r=1}^{n-1}
\theta_q\left(
-\frac{z}{q^{\lambda_r^{(n-1)}}}
\right),\qquad
w:=
\prod_{r=1}^{n-1}
q^{\lambda_r^{(n-1)}}.
\end{align*}
On the locus defined by
	\eqref{Eq:FirstSpectralRecoveryGen},
	the elliptic partial fraction formula
	\cite{rosengren_elliptic_root_systems_2004}
	shows that \eqref{Eq:FirstSpectralRecovery2}
	is equivalent to
\begin{align}
\sum_{i=1}^{n}
\frac{
u_n^{\lambda_i^{(n)}}
}{
K_{in}
\displaystyle
\prod_{\substack{1\leqslant k\leqslant n\\k\neq i}}
\theta_q\left(
-\frac{q^{\lambda_i^{(n)}}}
{q^{\lambda_k^{(n)}}}
\right)
}\cdot
K_{ij}
\theta_q\left(
-\frac{
q^{\lambda_i^{(n)}}u_n
}{
\left(\prod_{k=1}^{n}q^{\lambda_k^{(n)}}
\right)
u_j
}w
\right)
u_j^{-\lambda_i^{(n)}} \cdot
H\bigl(q^{\lambda_i^{(n)}}\bigr)
&=0.
\label{Eq:FirstSpectralThetaEvaluation}
\end{align}
Condition~\eqref{Eq:CThetaDeterminantConditions}
	with $s=n$ guarantees that
	for $w=q^{-\alpha_n}
	\prod_{k=1}^{n}q^{\lambda_k^{(n)}}$,
	equation~\eqref{Eq:FirstSpectralThetaEvaluation}
	has a nonzero solution as
	a homogeneous linear system in
	$H(q^{\lambda_1^{(n)}}),\ldots,
	H(q^{\lambda_n^{(n)}})$.
Since theta functions $H(z)$ of order $n-1$ on
	$\mathbb C^\ast/q^{\mathbb Z}$
	form a space of dimension $n-1$
	and satisfy the $j=n$ case of
	\eqref{Eq:FirstSpectralThetaEvaluation},
	this nonzero solution determines
	a theta function $H(z)$ of order $n-1$
	up to a nonzero scalar.
Writing its zeros as
	$q^{\lambda_1^{(n-1)}}q^{\mathbb Z},\ldots,
	q^{\lambda_{n-1}^{(n-1)}}q^{\mathbb Z}$
	gives a partial solution of
\eqref{Eq:FirstSpectralRecovery2}.
Moreover,
\begin{align}
\prod_{r=1}^{n-1}q^{\lambda_r^{(n-1)}}
&=
q^{-\alpha_n}
\prod_{k=1}^{n}q^{\lambda_k^{(n)}}.
\label{Eq:FirstSpectralBalance}
\end{align}
For generic data, these $n-1$ $q$-spirals admit
a unique choice of representatives satisfying
\eqref{Eq:FirstSpectralBalance} and the shrinking condition
\eqref{Cond:Shrinking}.
This recovers
$q^{\lambda_1^{(n-1)}},\ldots,
q^{\lambda_{n-1}^{(n-1)}}$.

In particular,
	\eqref{Eq:EtaRecovery} determines $\eta^{(n)}$.
Define $K^{(n-1)}$ by
\begin{align*}
	K^{(n-1)}E_j
	&:=
	\left(
	\mathscr K_{u_n/u_j}^{(n)}
	\bigl(
	\lambda_\ast^{(n)},\lambda_\ast^{(n-1)}
	\bigr)
	\right)^{-1}
	\eta^{(n)}K^{(n)}E_j,
	\qquad
	1\leqslant j\leqslant n-1.
\end{align*}
One can directly verify that the upper-left $(n-1)\times(n-1)$
	submatrix of $K^{(n-1)}$ satisfies the same initial conditions
	for the recursion, so we can continue the recursive construction
	until $s=1$.
Condition~\eqref{Eq:CThetaBalance} then gives
	$q^{\lambda_1^{(1)}}=q^{\alpha_1}$.
Thus we obtain the required initial data $(\Phi_0,W_0)$.
Although the choices of $\lambda_i^{(k)}$
	are not canonical
	and the recovered initial data need not be unique,
	Proposition~\ref{Pro:qUni} shows that
	the shrinking solution is unique.
\end{prf}

\begin{thm}\label{Thm:GenericShrinking}
The shrinking solutions constructed in
	Theorem~\ref{Thm:AsymBehavior}
	form a generic family of
	$q$-isomonodromic deformations.
\end{thm}

\begin{prf}
By Lemma~\ref{Lem:InverseMonodromy},
	every generic set of
	$q$-monodromy data is realized by a shrinking solution 
	constructed in Theorem~\ref{Thm:AsymBehavior}.
Proposition~\ref{Pro:qUni} shows that the $q$-monodromy data
	determine the corresponding $q$-isomonodromic deformation
	uniquely on the generic locus.
Hence every generic $q$-isomonodromic deformation
	is obtained by this construction.
\end{prf}

\subsection{Differential Limit}
\label{Sect:Conflu}

In this subsection,
	we show that the objects considered in this paper
	return to the differential case as $q\to1$.
More precisely,
	the canonical fundamental solutions of
	the linear $q$-difference system
	\eqref{qLinear} considered here
	should return to
	the canonical fundamental solutions
	of the following linear differential system
	as $q\to1$.
\begin{align}\label{dLinear}
	\frac{\mathrm{d}}{\mathrm{d} x}Y(x)
	= \left(U + \frac{\Phi}{x} \right) \cdot Y(x).
\end{align}
The $q$-isomonodromic deformations
	studied in this paper
	should also return to
	the corresponding isomonodromic deformations
	of the system above:
\begin{subequations}\label{diso}
\begin{align}
\frac{\partial}{\partial u_k}\Phi(u)
	&=
	[\operatorname{ad}_U^{-1}
	\operatorname{ad}_{E_k}\Phi(u), \Phi(u)],\\
\frac{\partial}{\partial u_k}W(u)
	&=
	\operatorname{ad}_U^{-1}
	\operatorname{ad}_{E_k}\Phi(u)\cdot W(u).
\end{align}
\end{subequations}
Here $k=1,\ldots,n$, and
	$\operatorname{ad}_U^{-1}$ denotes
	the inverse of
	$\operatorname{ad}_U=[U,\cdot]$
	on the subspace of matrices
	with vanishing diagonal blocks.
	
In \cite{tang_boundary_nodate},
	we similarly introduced the notion of
	a shrinking solution
	in the differential setting,
	proved that these solutions also form a generic family,
	and gave a formula for their monodromy data.
In fact,
	much of the technical foundation and motivation
	for the present paper
	comes from that work.
\begin{defi}\label{def:dif}
Denote
\begin{align}
	\delta^\mathrm{d}\Phi
	:=
	\diag(\Phi_{11},\ldots,\Phi_{nn}).
\end{align}
\begin{subequations}
Denote the \textbf{canonical fundamental solution}
	of system~\eqref{dLinear} at $x=0$ by
\begin{align}
	Y^{[0]}(x;U,\Phi)
	&=
	H^{[0]}(x;U,\Phi)\cdot
	x^\Phi.
\end{align}
For each admissible direction $d$,
	denote the \textbf{canonical fundamental solution}
	of system~\eqref{dLinear} at $x=\infty$
	in direction $d$ by
\begin{align}
	Y_d^{[\infty]}(x;U,\Phi)
	&=
	H_d^{[\infty]}(x;U,\Phi)\cdot
	x^{\delta^\mathrm{d}\Phi}\e^{Ux}.
\end{align}
\end{subequations}
Let $W$ be an invertible matrix such that
\begin{align}
	W^{-1}\Phi W=\Lambda.
\end{align}
The corresponding \textbf{central connection matrix}
	is defined by
\begin{align}
	C_d(U,\Phi;W)
	&:=
	Y_d^{[\infty]}(x;U,\Phi)^{-1}\cdot
	Y^{[0]}(x;U,\Phi)W.
\end{align}
If $\Phi(u)$ and $W(u)$ satisfy
	the differential isomonodromy equations~\eqref{diso},
	we call
	$(\Phi(u),W(u);\delta^\mathrm{d}\Phi)$
	a \textbf{differential isomonodromic deformation}.
For each admissible direction $d$,
	its \textbf{monodromy data} are defined by
\begin{align}
	\mathrm{MD}^{\mathrm d}_d(U,\Phi;W)
	&:=
	\bigl(
	C_d(U,\Phi;W),
	\Lambda;
	\delta^\mathrm{d}\Phi
	\bigr).
\end{align}
\end{defi}

We now turn to the differential limit as $q\to1$.
Fix $0<\epsilon<\frac{\pi}{2}$.
We shall use the following two increasingly
	restrictive conditions on the approach $q\to1$.
\begin{subequations}
\begin{align}
	|\arg(1-q)|<\frac{\pi}{2}-\epsilon
	\qquad &\text{or}\qquad
	|\arg(q-1)|<\frac{\pi}{2}-\epsilon,
\label{qreg}\\
	\arg(1-q)\ln|1-q| \longrightarrow 0
	\qquad &\text{or}\qquad
	\arg(q-1)\ln|q-1|\longrightarrow 0.
\label{qthetaradin}
\end{align}
\end{subequations}
These conditions play different roles.
Condition~\eqref{qreg} allows $q\to 1$
	with nonzero argument.
Along any such approach for which
	$\frac{\arg q}{\ln|q|}\to c$,
	the discrete $q$-spiral $q^{\mathbb Z}$ has the logarithmic spiral
	$\mathrm{e}^{(1+\mathrm{i}c)\mathbb{R}}$
	as its continuum limit.
Accordingly, the asymptotic behavior of the limiting differential
	isomonodromic deformation is naturally understood along this
	logarithmic spiral.
The same phenomenon also appears in the critical
	behavior of Painlev\'e~VI solutions
	\cite{guzzetti_elliptic_2001}.
Under the stronger condition~\eqref{qthetaradin},
	the rescaled $q$-spiral
	$\frac{1}{|1-q|}q^{\mathbb Z}$
	converges to $\mathbb R_{>0}$.
	
\begin{lem}\label{Lem:qIsoConfluence}
Suppose that \eqref{qreg} holds.
Let
	$\Phi_0^{\mathrm d}\in
	\mathrm{Mat}_{n\times n}(\mathbb C)$
	and
	$W_0^{\mathrm d}\in\mathrm{GL}_n(\mathbb C)$.
Suppose further that
\begin{align}
	\lim_{q\to1}
	\Phi_0(q)=\Phi_0^{\mathrm d},\qquad
	\lim_{q\to1}
	W_0(q)=W_0^{\mathrm d}.
\end{align}
For each $1\leqslant k\leqslant n$,
	fix a branch of $[\Phi_0(q)^{[k]}]^q$
	and denote its eigenvalues by
\begin{align}
	\operatorname{Spec}\bigl(
	[\Phi_0(q)^{[k]}]^q\bigr)
	=
	\{
	\lambda_1^{(k)}(q),\ldots,
	\lambda_k^{(k)}(q)
	\},
\end{align}
	in such a way that
\begin{align}
	\lim_{q\to1}
	[\Phi_0(q)^{[k]}]^q
	=
	(\Phi_0^{\mathrm d})^{[k]},\qquad
	\lim_{q\to1}
	\lambda_i^{(k)}(q)
	=:
	\lambda_i^{(k)},
	\qquad 1\leqslant i\leqslant k.
\end{align}
For each $2\leqslant k\leqslant n$,
	assume that there exists
	$0<\varepsilon_{k-1}<1$ independent of $q$
	such that the following
	shrinking condition holds:
\begin{align}\label{Cond:qShrinking}
	\underset{1\leqslant i,j\leqslant k-1}{\mathrm{max}}
	\left|
	\mathrm{Re}\bigl(
		\lambda_i^{(k-1)}(q)-
		\lambda_j^{(k-1)}(q)
	\bigr)
	-
	\frac{\mathrm{arg}q}{\mathrm{ln}|q|}
	\mathrm{Im}\bigl(
		\lambda_i^{(k-1)}(q)-
		\lambda_j^{(k-1)}(q)
	\bigr)
	\right|
	<
	1-\varepsilon_{k-1}.
\end{align}
Choose
$\boldsymbol{\alpha}(q)
=\operatorname{diag}(\alpha_1(q),\ldots,\alpha_n(q))$
such that
\begin{align}\label{Thm:qdefalpha}
	q^{\alpha_k(q)}
	&=
	\frac{
		\det\bigl(I+(q-1)\Phi_0(q)^{[k]}\bigr)
	}{
		\det\bigl(I+(q-1)\Phi_0(q)^{[k-1]}\bigr)
	},
	\qquad
	\lim_{q\to 1} \alpha_k(q) =
	(\Phi_0^{\mathrm d})_{kk},\qquad
	k=1,\ldots,n.
\end{align}
Then the $q$-isomonodromic deformation
	$(\Phi(u;q),W(u;q);\boldsymbol{\alpha}(q))$
	constructed in Theorem~\ref{Thm:AsymBehavior}
	satisfies
\begin{align}
	\lim_{q\to 1}
	(\Phi(u;q),W(u;q);\boldsymbol{\alpha}(q)) =
	(\Phi^{\mathrm d}(u),W^{\mathrm d}(u);
	\delta^{\mathrm d}\Phi^{\mathrm d})
\end{align}
	locally uniformly with respect to $u$.
Here
	$(\Phi^{\mathrm d}(u),W^{\mathrm d}(u);
	\delta^{\mathrm d}\Phi^{\mathrm d})$
	is the differential isomonodromic deformation
	constructed from
	$(\Phi_0^{\mathrm d},W_0^{\mathrm d};
	\delta^{\mathrm d}\Phi_0^{\mathrm d})$.
\end{lem}

\begin{prf}
It follows from \eqref{qreg} that
\begin{align}
	\limsup_{q\to1}
	\frac{|1-q|}{|1-|q||}
	&\leqslant
	\csc\epsilon,
	\qquad
	\limsup_{q\to1}
	\left|
	\frac{\mathrm{arg}q}
	{\mathrm{ln}|q|}
	\right|
	\leqslant
	\cot\epsilon..
\end{align}
Together with the convergence assumptions
	and \eqref{Cond:qShrinking},
	these bounds make the Picard estimates in
	Lemma~\ref{Lem:Picard}
	uniform as $q\to1$.
The $q$-Jackson integral equations
	converge along the same logarithmic spirals
	to the corresponding integral equations
	defining the differential solutions
	on these logarithmic spirals.
Applying this argument recursively in
	Corollary~\ref{Cor:Reck}
	proves the result.
\end{prf}

The following limit formulas are well known
\cite{sauloy_basic_2024}.
\begin{lem}\label{Lem:qto1}
If \eqref{qreg} holds,
	then the following limits
	hold locally uniformly:
\begin{subequations}
\begin{align}\label{gammaqlimit}
	\lim_{q\to 1}
	\Gamma_q(z) &= \Gamma(z),\qquad
	z\in\mathbb C\setminus\mathbb Z_{\leqslant0},\\
	\lim_{q\to1}
	\frac{(q^a z;q)_\infty}{(z;q)_\infty}
	&=
	(1-z)^{-a},\qquad
	z\in\mathbb C\setminus
	\left\{
	z\in\mathbb C:
	|z|\geqslant1,\
	\frac{
		|\operatorname{arg}z|
	}{
		\ln|z|
	}
	\leqslant
	\cot\epsilon\,
	\right\},
	\quad a\in\mathbb{C},
\end{align}
	where $\arg(1-z) \in(-\pi,\pi)$.
If \eqref{qthetaradin} holds,
	then the following limits
	hold locally uniformly:
\begin{align}\label{thetaqlimitin}
	\lim_{q\to1}
	(1-q)^a
	\frac{
		\theta_q((1-q)q^az)
	}{
		\theta_q((1-q)z)
	}
	=
	z^{-a},\qquad
	\arg z\in(-\pi,\pi).
\end{align}
\end{subequations}
\end{lem}

\begin{cor}\label{Cor:CentralConnectionConfluence}
Under the assumptions and notation of
	Lemma~\ref{Lem:qIsoConfluence},
	suppose further that \eqref{qthetaradin} holds.
Then
\begin{subequations}\label{Eq:QConConflu}
\begin{align}
\lim_{q\to1}
	C(x;U,\Phi(u;q);
	\boldsymbol{\alpha}(q),W(u;q))^{-1}
	&=
	\sum_{j=1}^{n}
	C_{-\arg(-u_j)}
	(U,\Phi^{\mathrm d}(u);W^{\mathrm d}(u))^{-1}
	E_j,
	\qquad |q|<1,\\
\lim_{q\to1}
	C(x;U,\Phi(u;q);
	\boldsymbol{\alpha}(q),W(u;q))
	&=
	\sum_{j=1}^{n}
	E_j
	C_{-\arg(-u_j)}
	(U,\Phi^{\mathrm d}(u);W^{\mathrm d}(u)),
	\qquad |q|>1.
\end{align}
\end{subequations}
Here,
	for each $1\leqslant j\leqslant n$,
	the branch of $-\arg(-u_j)$ is chosen such that
\begin{align}
	\left|
		-\arg(-u_j) - \arg x
	\right| < \pi.
\end{align}
Both limits are locally uniform with respect to $(x,u)$,
	provided that all the directions $-\arg(-u_j)$
	are admissible
	and the above branch choices remain fixed.
\end{cor}

\begin{prf}
In \cite[Section~4]{tang_boundary_nodate},
	the corresponding
	differential central connection matrices
	are similarly expressed in terms of
	$\Phi_0^{\mathrm d}$ and $W_0^{\mathrm d}$.
Applying Theorem~\ref{Thm:qCaterpillar}
	and substituting the limit formulas
	in Lemma~\ref{Lem:qto1}
	gives \eqref{Eq:QConConflu}.
\end{prf}

\begin{lem}\label{Lem:InverseMonodromyConfluence}
Suppose that \eqref{qreg} holds.
Let $(U,\Phi;W)$ be a generic triple such that
\begin{align}\label{Eq:lemq1}
	W^{-1}\Phi W
	&=
	\Lambda
	:=
	\operatorname{diag}
	\bigl(
	\lambda_1^{(n)},\ldots,\lambda_n^{(n)}
	\bigr).
\end{align}
Choose a branch $[\Lambda]^q$ and formal monodromy
	$\boldsymbol{\alpha}(q)
	=
	\diag(\alpha_1(q),\ldots,\alpha_n(q))$
	such that
\begin{subequations}\label{Eq:lemq2}
\begin{alignat}{2}
	[\Lambda]^q
	&=
	\operatorname{diag}
	\bigl(
	\lambda_1^{(n)}(q),\ldots,
	\lambda_n^{(n)}(q)
	\bigr),&\qquad
	\lim_{q\to 1}
	[\Lambda]^q &=
	\Lambda,\\
	\operatorname{Spec}\bigl(
	U^{-1}(I+(q-1)\Phi)
	\bigr)
	&=
	\bigl\{
	u_1^{-1}q^{\alpha_1(q)},\ldots,
	u_n^{-1}q^{\alpha_n(q)}
	\bigr\},&\qquad
	\lim_{q\to 1}
	\boldsymbol{\alpha}(q) &=
	\delta^{\mathrm{d}}\Phi.
\end{alignat}
\end{subequations}
Let $(\Phi_0(q),W_0(q))$ be the
	asymptotic leading term of the
	$q$-isomonodromic deformation through $(U,\Phi;W)$,
	and let
	$(\Phi_0^{\mathrm d},W_0^{\mathrm d})$ be the
	corresponding asymptotic leading term of the
	differential isomonodromic deformation through
	$(U,\Phi;W)$.
Then
\begin{align}
\lim_{q\to 1}
	\bigl(
	\Phi_0(q),W_0(q);
	\boldsymbol{\alpha}(q)
	\bigr)
	&=
	\bigl(
	\Phi_0^{\mathrm d},W_0^{\mathrm d};
	\delta^{\mathrm d}\Phi_0^{\mathrm d}
	\bigr).
\end{align}
\end{lem}

\begin{prf}
For fixed $U$, let $\mathcal E_q$ denote the analytic map
	from an asymptotic leading term
	$(\Phi_0,W_0)$ of a shrinking solution
	to its value at $U$.
Define $\mathcal E_{\mathrm d}$ in the same way
	for differential shrinking solutions.
Then
\begin{align*}
	\mathcal E_q(\Phi_0(q),W_0(q))
	&=
	\mathcal E_{\mathrm d}
	(\Phi_0^{\mathrm d},W_0^{\mathrm d})
	=
	(\Phi,W).
\end{align*}
Choose a sufficiently small neighborhood
$\mathcal U$ of
$(\Phi_0^{\mathrm d},W_0^{\mathrm d})$
	on which the shrinking construction is defined.
Under condition~\eqref{qreg},
	Lemma~\ref{Lem:qIsoConfluence} gives
\begin{align*}
	\mathcal E_q\longrightarrow\mathcal E_{\mathrm d}
\end{align*}
	locally uniformly on $\mathcal U$.
The differential monodromy formula
	in \cite[Section~5]{tang_boundary_nodate}
	shows that $\mathcal E_{\mathrm d}$
	is locally biholomorphic at
	$(\Phi_0^{\mathrm d},W_0^{\mathrm d})$.
For a suitable choice of representatives
	$\lambda_i^{(k)}$ at all intermediate levels,
	the inverse construction in
	Lemma~\ref{Lem:InverseMonodromy}
	shows that $\mathcal E_q$
	is also locally biholomorphic.
After shrinking $\mathcal U$ if necessary,
	we have
\begin{align*}
	(\Phi_0(q),W_0(q)) =
	\mathcal E_q^{-1}(\Phi,W)
	&\longrightarrow
	\mathcal E_{\mathrm d}^{-1}(\Phi,W) =
	(\Phi_0^{\mathrm d},W_0^{\mathrm d}).
\end{align*}
The differential isomonodromy equations
	preserve the diagonal part of $\Phi$.
Thus
	$\delta^{\mathrm d}\Phi
	=\delta^{\mathrm d}\Phi_0^{\mathrm d}$,
	and the assumed limit
	$\boldsymbol{\alpha}(q)\to\delta^{\mathrm d}\Phi$
	gives
	$\boldsymbol{\alpha}(q)
	\to\delta^{\mathrm d}\Phi_0^{\mathrm d}$.
\end{prf}

\begin{thm}\label{Thm:qto1Confluence}
Suppose that \eqref{qreg} holds.
Let $(U,\Phi;W)$ be a fixed generic triple satisfying
	\eqref{Eq:lemq1}.
Choose a branch $[\Lambda]^q$ and formal monodromy
	$\boldsymbol{\alpha}(q)$ such that
	\eqref{Eq:lemq2} holds.
Let
	$(\Phi(u;q),W(u;q);\boldsymbol{\alpha}(q))$
	be the $q$-isomonodromic deformation
	through $(U,\Phi;W)$,
	and let
	$(\Phi^{\mathrm d}(u),W^{\mathrm d}(u);
	\delta^{\mathrm d}\Phi^{\mathrm d})$
	be the differential isomonodromic deformation
	through the same initial value.
Then the $q$-isomonodromic deformation satisfies
\begin{align}\label{Eq:Thmqisoto1}
\lim_{q\to1}
	(\Phi(u;q),W(u;q);\boldsymbol{\alpha}(q))
	&=
	(\Phi^{\mathrm d}(u),W^{\mathrm d}(u);
	\delta^{\mathrm d}\Phi^{\mathrm d}),
\end{align}
	locally uniformly with respect to $u$.
Suppose further that \eqref{qthetaradin} holds.
Then
\begin{subequations}\label{Eq:ThmCaninfto1}
\begin{align}
\lim_{q\to1}
	Y^{[\infty]}
	(x;U,\Phi;\boldsymbol{\alpha}(q))
	&=
	\sum_{j=1}^{n}
	Y_{-\arg(-u_j)}^{[\infty]}
	(x;U,\Phi)E_j,
	\qquad |q|<1,\\
\lim_{q\to1}
	Y^{[\infty]}
	(x;U,\Phi;\boldsymbol{\alpha}(q))^{-1}
	&=
	\sum_{j=1}^{n}
	E_j
	Y_{-\arg(-u_j)}^{[\infty]}
	(x;U,\Phi)^{-1},
	\qquad |q|>1.
\end{align}
\end{subequations}
Here, for each $1\leqslant j\leqslant n$,
	the branch of $-\arg(-u_j)$ is chosen such that
\begin{align}
	\left|
	-\arg(-u_j)-\arg x
	\right|<\pi.
\end{align}
Both limits are locally uniform with respect to $x$,
	provided that all the directions $-\arg(-u_j)$
	are admissible
	and the above branch choices remain fixed.
\end{thm}

\begin{prf}
Applying Lemma~\ref{Lem:qIsoConfluence}
	to the asymptotic leading terms obtained in
	Lemma~\ref{Lem:InverseMonodromyConfluence}
	gives \eqref{Eq:Thmqisoto1}.
One verifies that the $q$-canonical fundamental solution at $x=0$
	converges to the corresponding
	differential canonical fundamental solution.
Corollary~\ref{Cor:CentralConnectionConfluence}
	then gives \eqref{Eq:ThmCaninfto1}.
\end{prf}

\section*{Acknowledgements}
The authors would like to thank Yiming Ma
	for helpful discussions.
The authors are supported by
	the National Key Research and Development Program of China
	(No.~2021YFA1002000).

\bibliography{20250906.bib}
\bibliographystyle{plain}

\Addresses

\end{document}